\documentclass{amsart}
\usepackage[dvipsnames]{xcolor}

\usepackage{amssymb, amsmath,mathtools}
\usepackage{mathrsfs}
\usepackage{amscd}
\usepackage{verbatim}
\usepackage{stmaryrd}

\usepackage{enumerate}

\usepackage[colorlinks,linkcolor={blue},citecolor={blue},urlcolor={red},]{hyperref}

\usepackage[colorinlistoftodos,prependcaption,textsize=tiny]{todonotes}

\allowdisplaybreaks

\theoremstyle{plain}
\newtheorem{theorem}{Theorem}[section]
\theoremstyle{remark}
\newtheorem{remark}[theorem]{Remark}
\newtheorem{example}[theorem]{Example}
\theoremstyle{plain}
\newtheorem{corollary}[theorem]{Corollary}
\newtheorem{lemma}[theorem]{Lemma}
\newtheorem{proposition}[theorem]{Proposition}
\newtheorem{definition}[theorem]{Definition}

\newtheorem{assumption}[theorem]{Assumption}

\numberwithin{equation}{section}

\def\N{{\mathbb N}}

\def\R{{\mathbb R}}
\def\C{{\mathbb C}}

\newcommand{\nnn}{|\!|\!|}

\newcommand{\E}{{\mathbb E}}
\renewcommand{\P}{{\mathbb P}}
\newcommand{\F}{{\mathscr F}}
\newcommand{\Filtr}{{\mathbb F}}
\newcommand{\A}{{\mathcal A}}
\newcommand{\Vspace}{{\mathcal V}}
\renewcommand{\H}{{\mathscr H}}

\def\Xint#1{\mathchoice
   {\XXint\displaystyle\textstyle{#1}}%
   {\XXint\textstyle\scriptstyle{#1}}%
   {\XXint\scriptstyle\scriptscriptstyle{#1}}%
   {\XXint\scriptscriptstyle\scriptscriptstyle{#1}}%
   \!\int}
\def\XXint#1#2#3{{\setbox0=\hbox{$#1{#2#3}{\int}$}
     \vcenter{\hbox{$#2#3$}}\kern-.5\wd0}}

\newcommand{\g}{\gamma}
\renewcommand{\d}{\delta}
\newcommand{\om}{\omega}
\renewcommand{\O}{\Omega}

\newcommand{\Do}{\mathsf{D}}
\newcommand{\rhos}{\rho^{\star}}

\newcommand{\RR}{\R^d}

\newcommand{\deter}{\mathsf{det}}
\newcommand{\stoc}{\mathsf{sto}}

\newcommand{\Tor}{\mathbb{T}}

\newcommand{\Ext}{\mathsf{E}_T}
\newcommand{\Extz}{\prescript{}{0}{\mathsf{E}_T}}
\newcommand{\Extunit}{\mathsf{E}}

\newcommand{\DN}{\prescript{}{N}{\Delta}}
\newcommand{\Dd}{\prescript{}{D}{\Delta}}
\newcommand{\Dn}{\prescript{}{\n}{\Delta}}

\newcommand{\Aop}{\mathscr{A}}

\newcommand{\HN}{\prescript{}{N}{H}}
\newcommand{\Hn}{\prescript{}{\n}{H}}
\newcommand{\Bn}{\prescript{}{\n}{B}}
\newcommand{\BN}{\prescript{}{N}{B}}

\newcommand{\X}{\mathfrak{X}}
\newcommand{\z}{\mathcal{Z}}
\newcommand{\loc}{loc}

\newcommand{\hz}{\prescript{}{0}{H}}

\newcommand{\Wz}{\prescript{}{0}{W}}

\newcommand{\Tr}{\mathsf{Tr}}
\newcommand{\D}{\mathscr{D}}

\newcommand{\BIP}{\text{\normalfont{BIP}}}
\newcommand{\Xap}{X^{\mathsf{Tr}}_{\a,p}}
\newcommand{\Xapcrit}{X^{\mathsf{Tr}}_{\a_{\crit},p}}
\newcommand{\Xzero}{X^{\mathsf{Tr}}_{2}}
\newcommand{\Xp}{X^{\mathsf{Tr}}_{p}}
\newcommand{\Xzp}{X^{\mathsf{Tr}}_{0,p}}

\newcommand{\HD}{\prescript{}{D}{H}}
\newcommand{\BD}{\prescript{}{D}{B}}
\newcommand{\WD}{\prescript{}{D}{W}}

\newcommand{\n}{\nu}

\renewcommand{\k}{\kappa}
\renewcommand{\a}{\kappa}
\renewcommand{\b}{\beta}

\renewcommand{\l}{\langle}
\renewcommand{\r}{\rangle}
\renewcommand{\ll}{\llbracket}
\newcommand{\rr}{\rrbracket}
\newcommand{\rParen}{\rrparenthesis}
\newcommand{\lParen}{\llparenthesis}
\newcommand{\llo}{\lParen}
\newcommand{\rro}{\rParen}
\newcommand{\tA}{\widetilde{A}}
\newcommand{\tB}{\widetilde{B}}
\newcommand{\wt}{\widetilde}
\newcommand{\wT}{\wt{T}}
\newcommand{\wl}{\wt{\lambda}}

\newcommand{\crit}{\mathsf{crit}}

\newcommand{\angH}{\omega_{H^{\infty}}}
\newcommand{\Dom}{\mathscr{O}}
\newcommand{\tr}{\mathrm{tr}}

\newcommand{\calL}{{\mathscr L}}

\newcommand{\Sch}{{\mathscr S}}
\newcommand{\one}{{{\bf 1}}}
\newcommand{\lb}{\langle}
\newcommand{\rb}{\rangle}

\newcommand{\supp}{\text{\rm supp\,}}

\renewcommand{\div}{\normalfont{\text{div}}}
\newcommand{\T}{\Pi}
\newcommand{\B}{B}

\newcommand{\MRtaz}{\mathcal{SMR}_{p,\a}(T)}
\newcommand{\MRta}{\mathcal{SMR}^{\bullet}_{p,\a}(T)}

\newcommand{\MRtz}{\mathcal{SMR}_p(T)}

\newcommand{\MRtzeroz}{\mathcal{SMR}_{p,0}(T)}
\newcommand{\MRt}{\mathcal{SMR}^{\bullet}_{p}(T)}
\newcommand{\MRtzero}{\mathcal{SMR}^{\bullet}_{p,0}(T)}

\newcommand{\Sol}{\mathscr{R}}

\newcommand{\I}{I}
\newcommand{\J}{J}

\newcommand{\esssup}{\normalfont{\text{ess\,sup}}}

\newcommand{\Four}{\mathcal{F}}

\newcommand{\bb}{\mathcal{B}}

\newcommand{\Hip}{\normalfont{\textbf{(H)}}}
\newcommand{\Hipprime}{\normalfont{\textbf{(H$^{\prime}$)}}}

\renewcommand{\emptyset}{\varnothing}

\newcommand{\dist}{{\rm dist}}

\newcommand{\norm}[1]{{\left\vert\kern-0.25ex\left\vert\kern-0.25ex\left\vert #1
    \right\vert\kern-0.25ex\right\vert\kern-0.25ex\right\vert}}

\newcommand{\mf}{m_F}
\newcommand{\mg}{m_G}

\newcommand{\Progress}{\mathscr{P}}
\newcommand{\Borel}{\mathscr{B}}

\allowdisplaybreaks

\begin{document}

\author{Antonio Agresti}
\address{Institute of Science and Technology Austria (IST Austria)\\ Am Campus 1\\ 3400 Klosterneuburg \\ Austria.} \email{antonio.agresti92@gmail.com}

\author{Mark Veraar}
\address{Delft Institute of Applied Mathematics\\
Delft University of Technology \\ P.O. Box 5031\\ 2600 GA Delft\\The
Netherlands.} \email{M.C.Veraar@tudelft.nl}

\thanks{The second author is supported by the VIDI subsidy 639.032.427 of the Netherlands Organisation for Scientific Research (NWO)}

\date\today
\title[parabolic stochastic evolution equations in critical spaces I]{Nonlinear parabolic stochastic evolution \\ equations in critical spaces part I \\\small{\emph{S\lowercase {tochastic maximal} \lowercase{regularity and local existence}}}}




\begin{abstract}
In this paper we develop a new approach to nonlinear stochastic partial differential equations with Gaussian noise. Our aim is to provide an abstract framework which is applicable to a large class of SPDEs and includes many important cases of nonlinear parabolic problems which are of quasi- or semilinear type. This first part is on local existence and well-posedness.
A second part in preparation is on blow-up criteria and regularization.

Our theory is formulated in an $L^p$-setting, and because of this we can deal with nonlinearities in a very efficient way. Applications to several concrete problems and their quasilinear variants are given. This includes Burgers' equation, the Allen--Cahn equation, the Cahn--Hilliard equation, reaction--diffusion equations, and the porous media equation. The interplay of the nonlinearities and the critical spaces of initial data leads to new results and insights for these SPDEs.

The proofs are based on recent developments in maximal regularity theory for the linearized problem for deterministic and stochastic evolution equations. In particular, our theory can be seen as a stochastic version of the theory of critical spaces due to Pr\"uss--Simonett--Wilke (2018). Sharp weighted time-regularity allow us to deal with rough initial values and obtain instantaneous regularization results. The abstract well-posedness results are obtained by a combination of several sophisticated splitting and truncation arguments.
\end{abstract}

\maketitle

\newpage

\tableofcontents

\newpage

\section{Introduction}

In this article we study parabolic quasilinear and semilinear stochastic evolution equations of the form:
\begin{equation}\label{eq:introSEE}
\begin{cases}
du + A(u) u dt = F(u)dt + (B(u) u+G(u)) dW_{H}, \qquad t\in (0,T), \\
u(0)=u_0.
\end{cases}
\end{equation}
Here $A$ is the leading operator and is of quasilinear type which means that for each $v$ in a suitable interpolation space
\[u\mapsto A(v)u,\]
defines a mapping from $X_1$ into $X_0$,
where $X_1\hookrightarrow X_0$ densely. The problem \eqref{eq:introSEE} includes the semilinear case where the pair $(A(u) u, B(u)u)$ is replaced by $(\tilde{A} u, \tilde{B} u)$. Here $(\tilde{A}, \tilde{B})$ are operators not depending on $u$. The noise term $W_H$ is a cylindrical Brownian motion. The nonlinearities $F$ and $G$ are of semilinear type. Many examples of SPDEs fit in the above framework.

A powerful approach to problems of the form \eqref{eq:introSEE} is the monotone operator approach (see \cite{LiuRock} and references therein), and actually one can even treat some more complicated nonlinearities than $(A,B)$ for the leading operators. In examples the usual coercivity is formulated for the pair $(A,B)$ and ensures that the problem is of parabolic type. Moreover, without any difficulty this method allows to treat problems with $(t,\omega)$-dependent operators, which is important in filtering problems. Here and throughout the paper $\om\in \Omega$ denotes the stochastic variable. There are also some limitations and drawbacks to the method. For example it requires a Hilbert space structure and it does not provide optimal time regularity.
Moreover, for many equations in dimensions $d\geq 3$ (e.g.\ Navier-Stokes, Cahn-Hilliard and Allen-Cahn), $L^p$ or even $L^p(L^q)$-theory seems to be necessary. In some cases  a $C^{\alpha}$-theory could be applied as well. However a $C^{\alpha}$-theory (see e.g.\ \cite{DuLiu,wang2019schauder} for the linear $C^{\alpha}$-theory) requires regularity assumptions on the noise, the coefficients, and on $u_0$ which can be either too restrictive for physical applications, or does not fit the scaling property of the SPDE considered.
Moreover, $L^p(L^q)$-theory provides blow-up criteria that can be combined with \textit{energy estimates} to prove global existence. Energy bounds are usually $L^p$-estimates, and thus,  they do not seem to be exploitable in an $C^{\alpha}$-context. Blow-up criteria and their applications to SPDEs will be the topic of the subsequent parts \cite{AV19_QSEE_2,AV19_QSEE_3} where the results proven here will be of basic importance.

Our aim is to build an $L^p(L^q)$-theory for \eqref{eq:introSEE} in which the coercivity condition can be formulated for an abstract pair $(A,B)$ and where we allow $(t,\omega)$-dependence in the operators in an adapted way. The $L^p$-theory \cite{Kry} is an important step in this direction, and recently an evolution equation approach has been found in \cite{VP18}, which additionally gives the optimal space-time-regularity.
In the case that $A$ is time-independent and $B = 0$, optimal space-time regularity was discovered in the influential paper \cite{MaximalLpregularity}. Moreover, under a smallness condition on $B$, these results can be applied to well-posedness of \eqref{eq:introSEE}. For instance the semilinear case was considered in \cite{Brz2,NVW3} and extended to a {\em maximal regularity} setting (see Section \ref{ss:MRintro}) in \cite{NVW11eq}. The latter was extended to the quasilinear setting in \cite{Hornung}. In this paper we will completely revise the general theory, and our approach has much more flexibility. In particular, we allow:
\begin{itemize}
\item a quasilinear couple $(A,B)$;
\item measurable dependence in $(t,\omega)$;
\item $(A,B)$ without smallness conditions on $B$;
\item weights in the time variable $w_{\kappa}(t) = t^{\kappa}$ with $\kappa\in [0,\frac{p}{2}-1)$;
\item rough initial data: $u_0\in (X_0,X_1)_{1-\frac{1+\a}{p},p}$;
\item nonlinearities $F$ and $G$ defined on interpolation spaces $[X_0,X_1]_{1-\varepsilon}$ which are locally Lipschitz and have polynomial growth;
\item $L^p(0,T;L^q)$-theory and $L^p(0,T;H^{s,q})$-theory for a range of $s\in \R$.
\end{itemize}
In the above $(X_0, X_1)_{\theta,p}$ and $[X_0, X_1]_{\theta}$ denote the real and complex interpolation spaces, respectively. In applications these can be identified with certain Besov spaces and Bessel potential spaces.

Using the weights $w_{\kappa}$ we will introduce a stochastic version of the theory of {\em critical spaces}, which we will briefly discuss in the deterministic setting  in the next section of the introduction.

In the papers \cite{AV19_QSEE_2,AV19_QSEE_3} we will continue the study of \eqref{eq:introSEE}. More specifically, we will study blow-up criteria, regularization phenomena and further applications to SPDEs. In \cite{AV20_NS}, we will show global well--posedness for \eqref{eq:introSEE} with small initial data for the stochastic Navier-Stokes equations.

\subsection{Criticality}
\label{ss:critical_spaces_introduction}
In the literature {\em critical spaces} are often introduced as those spaces which satisfy a scaling invariance similar to the one of the PDE itself, or as those spaces in which the energy bound and nonlinearity are of the same order. More details on this can be found in \cite{Can04,Klain00,LePi,Taodisp,Trie13}, and references therein.
For example for the Navier--Stokes equations on $\R^3$ one can obtain solutions in $L^p(0,T;L^q)$ for small initial data in the critical space $\dot{B}^{-1+\frac3q}_{q,p}$ provided the criticality condition $\frac{2}{p} + \frac{3}{q}=1$ holds, and $q\in (3, \infty)$ (see \cite[p.\ 182]{LePi}).

Another way to introduce criticality would be to consider a specific nonlinearity, e.g.\ $F(u) = |u|^r$ in a given PDE. Typically, some exponent $r$ turns out to be critical in the sense that the ``usual'' estimates are not powerful enough anymore. Below that value of $r$ the problem is usually called {\em subcritical} and above that value it is called {\em supercritical}.

In a recent paper of Pr\"uss--Simonett--Wilke \cite{CriticalQuasilinear} a new viewpoint on critical spaces has been discovered in the deterministic setting. Special cases have been considered before in \cite{Pru17, addendum, PW18}.
The authors consider abstract evolution equations in spaces of the form $L^p((0,T),w_{\kappa};X_0)$, where $p\in (1,\infty)$, $w_{\kappa}(t) = t^{\kappa}$ is a weight function with $\kappa\in [0,p-1)$, and typically $X_0$ is a Sobolev or an $L^q$-space. Assuming {\em maximal regularity} (see Section \ref{ss:MRintro}) for the leading term and several other conditions, the authors establish local well-posedness. The weight can be chosen in correspondence with the polynomial growth rate of the nonlinearity to obtain what they call a critical weight. After the weight exponent $\kappa$ is fixed, the so-called {\em trace space} of initial values which one can consider becomes $(X_0, X_1)_{1-\frac{1+\kappa}{p},p}$, and this space they call critical.

A surprising feature is that in many concrete examples the latter trace space coincides with the critical space from a PDE point of view.
In \cite{CriticalQuasilinear} this leads to several new results for classical PDEs of evolution type such as the Navier--Stokes equation, Cahn--Hilliard equations, convection–-diffusion equations, and many more. A crucial point in their theory is that $F$ does not have to be defined on the real interpolation spaces $(X_0, X_1)_{1-\frac1p,p}$ and one can allow it to be defined on a much smaller space $X_{\theta}$ with $\theta>1-\frac1p$ at the cost of a growth condition on $F$.

In our work we will develop a stochastic version of the above theory. For this many additional difficulties have to be overcome. Some of them are connected to $L^p(\Omega)$-integrability issues for the nonlinearities, and others are connected with the fact that in stochastic maximal regularity (see next section) one needs to work with vector-valued spaces with {\em fractional} smoothness to obtain the right trace theory. Note that in the stochastic case the condition on $\kappa$ becomes more restrictive $\kappa\in [0,\frac{p}{2}-1)$ (in the deterministic case this was $\kappa\in [0,p-1)$). Another issue in the examples is that the stochastic version of maximal $L^p$-regularity theory is more complicated and less developed than the deterministic case. Fortunately, there is a lot of current research in this direction and hopefully our paper will give motivation for further progress.

We will show that our theory can be applied to several classes of parabolic SPDEs. With a hands on approach for each SPDE separately one can often obtain very detailed properties of solutions. Our theory can provide more information as one usually obtains new spaces in which the problem can be analyzed, and thus provides different regularity results which where often not available yet.

Before we continue our discussion on the results of our paper, we will first introduce the reader to
so-called {\em stochastic maximal regularity}, which is one of the key tools in our paper.
\subsection{Stochastic maximal regularity\label{ss:MRintro}}

Maximal regularity has many forms and has always played a fundamental role in modern PDE. Below we will try to explain some of the background in a nontechnical way. The precise definitions can be found in Section \ref{ss:SMR}.

Arguably the most common form of maximal regularity for elliptic equations is: the solution $u$ to $\lambda u-\Delta u=f$ with $f\in L^q(\R^d)$ and $\lambda>0$ satisfies
\[\|u\|_{W^{2,q}(\R^d)}\leq C\|f\|_{L^q(\R^d)},\]
where $q\in (1, \infty)$ and $C$ is a constant depending on $\lambda$ and $q$.
The result fails for the endpoints $q\in \{1, \infty\}$. For $q=2$ this result is simple, and for general $q$ one typically uses Calder\'on--Zygmund theory (see \cite{Grafakos1}).

For the heat equation a similar result holds: the solution $u$ to $\partial_t u -\Delta u=f$, with initial condition $u_0\in B^{2-2/p}_{q,p}(\R^d)$ (Besov space) and $f\in L^p(0,T;L^q(\R^d))$ satisfies
\[\|u\|_{W^{1,p}(0,T;L^q(\R^d))} +\|u\|_{L^p(0,T;W^{2,q}(\R^d))}\eqsim \|u_0\|_{B^{2-2/p}_{q,p}(\R^d)}+\|f\|_{L^p(0,T;L^q(\R^d))},\]
where $p,q\in (1, \infty)$. Again the result fails if $p$ or $q$ are in $\{1, \infty\}$. There are many ways how to deduce the latter results, and again Calder\'on--Zygmund theory plays a central role. The fact that the estimate is two-sided shows that the result is optimal.

An efficient reformulating of the last result is
\[\|u\|_{W^{1,p}(0,T;X_0)} +\|u\|_{L^p(\R_+;X_1)}\eqsim \|u_0\|_{(X_0, X_1)_{1-\frac1p,p}}+\|f\|_{L^p(0,T;X_0)},\]
where $X_0 = L^q(\R^d)$ and $X_1 = W^{2,q}(\R^d)$.
In this form the result can be extended to many other parabolic problems, and this has been an important field of research for decades:
\begin{itemize}
\item For a PDE perspective see \cite{GT83,Krylov2008_book,LSU};
\item For an evolution equation perspective see \cite{DHP,KuWe,pruss2016moving}.
\end{itemize}
This topic is still very active in various schools as is evident from the many recent results (see e.g.\ \cite{DenkKaip,dong2017higher,DKnew, ter2019mathrm, HumLin,piasecki2019maximal,RoiSch,Tolk18}).
As we explained before sharp estimates for the linear setting can be used very effectively in the nonlinear case. In the quasilinear case for deterministic equations the standard reference for this is \cite{ClLi}, and the recent monograph \cite{pruss2016moving}.

In the stochastic situation the above theory is much more recent. If $u$ is a
solution to the stochastic heat equation $d u +\Delta u dt = fdt + g dW$, then for all $p\in (2, \infty)$ and $q\in [2, \infty)$
\begin{align*}
\|u\|_{L^p(\Omega;H^{\theta,p}(0,T;H^{2-2\theta,q}(\R^d)))}
\lesssim \|u_0\|_{L^p(\Omega;B^{2-2/p}_{q,p}(\R^d))}
&+\|f\|_{L^p(\Omega;L^p(0,T;L^q(\R^d)))}\\
& + \|g\|_{L^p(\Omega;L^p(0,T;W^{1,q}(\R^d;\ell^2)))},
\end{align*}
for any $\theta\in [0,1/2)$. Moreover, if $q=2$, then $p=2$ is also allowed. Here $H^{\theta,p}$ denotes the Bessel-potential space with smoothness $\theta$.
The above result was proved in \cite{MaximalLpregularity} by van Neerven, Weis and the second author. The case $\theta=0$ and $p\geq q\geq 2$ was obtained before in \cite{Kry94a,Kry96,Kry,Kry00} with a slightly stronger assumption on $u_0$. Recently, a stochastic version of Calder\'on--Zygmund theory was developed by Lorist and the second author. The latter can be used to derive the full range $p\in (2, \infty)$ and $q\in (2, \infty)$ from the case $p=q$ (see \cite{LoVer}).

As before the evolution equation reformulation is of the form
\begin{align*}
\|u\|_{L^p(\O;H^{\theta,p}(0,T;X_{1-\theta}))} \lesssim \|u_0\|_{L^p(\O;(X_0, X_1)_{1-\frac1p,p})}
&+
\|f\|_{L^p(\O; L^p(0,T;X_0))}\\
& + \|g\|_{L^p(\O; L^p(0,T;\gamma(\ell^2,X_{\frac{1}{2}})))},
\end{align*}
where $X_{1-\theta} = [X_0, X_1]_{1-\theta}$ denotes the complex interpolation spaces and coincides with $\Do(A^{1-\theta})$ for $A = 1-\Delta$ on $X_0$. This is the setting in which in \cite{MaximalLpregularity} the stochastic maximal regularity was proved for a large class of SPDEs. An important difference with the deterministic case is that the estimate does not hold for the end-point $\theta=1/2$. However, the half-open interval $\theta\in [0,1/2)$ is good enough for applications.

It is important to note that the natural form of the above problem is actually $d u +\Delta u dt = fdt + (g + B u) dW$,
where
\[B u \,d W=  \sum_{j=1}^d \sum_{n\geq 1} b_{jn} \partial_j u\,dw^n,\]
where $(b_{jn})_{n\geq 1}\in \ell^2$. Under the parabolicity/coercivity assumption
\begin{equation}\label{eq:stochparaintro}
|\xi|^2 -\frac{1}{2} \sum_{i,j=1}^d \sum_{n\geq 1} b_{in} b_{jn} \xi_i \xi_j\geq \delta |\xi|^2, \ \ \text{ with }\delta>0,
\end{equation}
the above estimates for the stochastic heat equation still hold. See \cite[Theorem 5.3]{VP18} and Section \ref{ss:introduction_motivation_semilinear} for a more general formulation.

Although we will only use stochastic maximal regularity in the $L^p(L^q)$ scale, it is important to note that it can also be considered in different scales such as the Besov scale and H\"older scale. For details on this we refer to
\cite{Brz1,BH09,DL98,LoVer} for the Besov scale and to \cite{DuLiu,DuLiuZhang,wang2019schauder}. The H\"older case can be seen as a stochastic analogue of Schauder theory (see \cite{GT83,KryHolder}).

\subsection{Illustration}
We will illustrate our abstract results Theorems \ref{t:local}, \ref{t:local_Extended} and \ref{t:semilinear} in a simple case in this introduction. We will only do this in the semilinear setting and refer to Section \ref{s:quasi_second_gradient_noise} for examples in the quasilinear setting. Consider the following special case of \eqref{eq:introSEE} on $\R^d$ with $d\geq 3$
\begin{equation}\label{eq:introSEEexample}
\begin{cases}
du - \Delta u dt = \displaystyle u|u|^{2}  dt +  \sum_{n\geq 1}  \Big( \sum_{j=1}^d b_{jn} \partial_j u(x) + g_n u|u|\Big)dw_t^n,  \\
u(0)=u_0,
\end{cases}
\end{equation}
where $(b_{jn})_{n\geq 1},(g_n)_{n\geq 1}\in \ell^2$ and the $b_{jn}$ satisfy \eqref{eq:stochparaintro}. The following is a special case of Theorem \ref{t:critical_space_reaction_diffusion}. The definition of maximal local solution will be given in Section \ref{s:quasi} (see the text below \eqref{eq:choice_AFGC_reaction_diffusion_for_intro} and Definition \ref{def:solution2}).

\begin{theorem}
\label{t:critical_space_reaction_diffusionintro}
Let $d\geq 3$. Assume that $q\in [2, d)$ and $q>d/2$. Let $p\in (2, \infty)$ be such that
$\frac{1}{p}+\frac{d}{2q}\leq 1$ holds, and let $\a_{\crit}=  p(1-\frac{d}{2q}) -1$.
Then for each
$$u_0\in L^0_{\F_0}(\O;B^{\frac{d}{q}-1}_{q,p}(\R^d))$$
there exists a maximal local solution $(u,\sigma)$ to \eqref{eq:introSEEexample}. Moreover, there exists a localizing sequence $(\sigma_n)_{n\geq 1}$ such that a.s.\ for all $n\geq 1$
$$
u\in L^{p}(0,\sigma_n,w_{\a_{\crit}};W^{1,q}(\R^d))
\cap C([0,\sigma_n];B^{\frac{d}{q}-1}_{q,p}(\R^d))\cap C((0,\sigma_n];B^{1-\frac{2}{p}}_{q,p}(\R^d)).
$$
\end{theorem}
Here $L^{p}(0,T,w_{\a})$ denotes the weighted $L^p$-space with weight $w_{\a}(t) = t^{\a}$.
One can check that \eqref{eq:introSEEexample} is invariant under the scaling (see Subsection \ref{sss:critical_reaction_diffusion_l_m})
\[u_{\lambda}(t,x):=\lambda^{1/2} u(\lambda t,\lambda^{1/2}x),\quad \text{ for }\lambda>0,\,x\in \R^d.\]
Moreover, the space of initial data $B^{\frac{d}{q}-1}_{q,p}(\R^d)$ (or actually its homogeneous version) is invariant under this scaling as well. Another interesting feature is that we can obtain $H^{1,q}$-solutions for any initial data with arbitrary low but positive smoothness. Moreover, the process $u:(0,\sigma)\to B^{1-\frac{2}{p}}_{q,p}(\R^d)$ still has continuous paths, and this shows that there is instantaneous regularization if $\a_{\crit}>0$ and the latter holds if the inequality $\frac{1}{p}+\frac{d}{2q}\leq 1$ is strict.

In the above it is important to note that only part of the structure of the nonlinearities $u |u|^2$ and $u|u|$ plays a role in the formulation of the result. In particular, if the nonlinearities have a different growth, then the above spaces need to be changed accordingly (see Theorem \ref{t:critical_space_reaction_diffusion} for details).

The noise term can be allowed to be rougher. For this one can just change the spaces accordingly, and as is classical, one uses the regularizing effect of the leading operator $\Delta$. This is for instance explained for the 1d stochastic Burgers' equation with white noise in Section \ref{ss:Burgers_semilinear} and with rough noise in Section \ref{ss:Burgers_quasilinear}. Here it is important to note that we can still allow critical spaces in many situations.

\subsection{Quasilinear SPDEs}

There exist several papers on quasilinear SPDEs in the literature. Sometimes authors use the wording semilinear, quasilinear and fully nonlinear in different ways. Let us explain our terminology in the example of a stochastic diffusion equation in nondivergence form with nonlinearities $f(u)dt + g(u) dW$ as before
\begin{itemize}
\item Semilinear: the leading operator is $u\mapsto a_{ij} \partial_{ij}^2 u$ with $a_{ij}$ independent of $u$;
\item Quasilinear: the leading operator is $u\mapsto a_{ij}(u,\nabla u) \partial_{ij}^2 u$;
\item Fully nonlinear: the leading operator is $u\mapsto a(u, \nabla u, \nabla^2 u)$.
\end{itemize}
Of course some ellipticity is assumed for the coefficients $a_{ij}$. However, in the quasilinear setting the ellipticity is often allowed to be degenerate, i.e.\ its constant depends on $u$ and is allowed to become zero. In the fully nonlinear case, ellipticity has to be formulated in a more sophisticated way, but we will not consider fully nonlinear problems in this paper. However, our theory can be applied to study fully nonlinear problems as was done by the first author in \cite{A18}.

In the stochastic setting there exist numerous papers in quasilinear setting and even more in the semilinear setting. Clearly, we cannot discuss all of them here. However, some papers which are relevant in the application sections will be mentioned there.

For many concrete equations products of distributions are required and often renormalization is needed (see \cite{GubImkPer} and \cite{Hairer}). We will only treat equations in a more classical context, where this is not required. However, it would be interesting to see what the maximal regularity techniques bring to this theory. Probably the right viewpoint is that our theory gives a very flexible framework in those cases where one does not need renormalization. By using our theory there will be some cases in which renormalization can be avoided, but of course in many cases renormalization is known to be necessary.

\subsubsection*{Overview}

\begin{itemize}
\item In Section \ref{sec:prel} preliminaries will be discussed. This includes functional calculus, some stochastic integration theory and an introduction to the function spaces which will be needed.
\item In Section \ref{ss:SMR} we will introduce stochastic maximal regularity and give sufficient conditions and examples.
\item In Section \ref{s:quasi} we will state and prove the main local well-posedness results for problems of the form \eqref{eq:introSEE}.
\item Applications to semilinear and quasilinear problems are discussed in Sections \ref{s:semilinear_gradient} and \ref{s:quasi_second_gradient_noise} respectively. In particular, the Burgers' equation and porous media equation will be considered there. Other concrete cases such as Allen-Cahn and Cahn-Hilliard are considered in Section \ref{s:AC_CH}.
\item In Section \ref{sec:app} an appendix on interpolation--extrapolation spaces is included, which will be used in the application sections.
\end{itemize}

\subsubsection*{Notation}
\begin{itemize}
\item For any $T\in (0,\infty]$, we set $\I_T=(0,T)$ and $\overline{I}_T$ denotes the closure of $I_T$.
\item We write $A \lesssim_P B$ (resp.\ $A\gtrsim_P B$), whenever there is a constant $C$ only depending on the parameter $P$ such that $A\leq C B$ (resp.\ $A \geq C B$). Moreover, we write $A \eqsim_P B$ if $A \lesssim_P B$ and $A \gtrsim_P B$.
\item If $X,Y$ is an interpolation couple of Banach space, we endow the intersection $X\cap Y$ with the norm $\|\cdot\|_{X\cap Y}:=\|\cdot\|_X+\|\cdot\|_Y$.
\item For any Banach space $Y$, $x\in Y$ and $\eta>0$, we denote the ball of radius $\eta$ and center $x\in Y$ by $\B_{Y}(x,\eta):=\{y\in Y\,:\,\|x-y\|_{Y}<\eta\}$ and $\B_{Y}(\eta)=\B_{Y}(0,\eta)$.
\item $\Borel$ and $\Progress$ denote the Borel and the progressive sigma algebra, respectively. See Subsection \ref{ss:stochastic_integration}.
\item $\MRta$ and $\MRtaz$ (and other variants) denote the set of couple with stochastic maximal $L^p$-regularity (see Definitions \ref{def:SMRz} and \ref{def:SMRgeneralized}).
\end{itemize}

\subsubsection*{Acknowledgements}

The authors thank Emiel Lorist for helpful comments, and Konstantinos Dareiotis for helpful comments on stochastic porous media equations. The
authors thank the anonymous referees for careful reading and helpful suggestions.

\section{Preliminaries}\label{sec:prel}

In this section we collect some useful facts and we give references to the literature for results which are not proven here.
As usual, for $\I\in \{(a,b),(a,b],[a,b),[a,b]\}$, where $0\leq a< b\leq \infty$, and a Banach space $X$, we denote by $C(\I;X)$ the set of all continuous functions $f:\I\to X$. If $b<\infty$, then $C([a,b];X)$ is a Banach space when it is endowed with the norm
\begin{equation}
\label{eq:Continuous_functions_norm}
\|f\|_{C([a,b];X)}:=\sup_{t\in \I} \|f(t)\|_{X}.
\end{equation}

\subsection{Sectorial operators and $H^{\infty}$-calculus}
\label{ss:sectorial_calculus}

Let $A:\Do(A)\subseteq X\to X$ be a closed operator on a Banach space $X$. We say that $A$ is sectorial if the domain and the range of $A$ are dense in $X$ and there exists $\phi\in (0,\pi)$ such that $\sigma(A)\subseteq \overline{\Sigma_{\phi}}$, where $ \Sigma_{\phi}:=\{z\in \C\,:\,|\arg z|< \phi\}$, and there exists $C>0$ such that
\begin{equation}
\label{eq:sectorial}
|\lambda|\|(\lambda-A)^{-1}\|_{\calL(X)}\leq C, \qquad \forall \lambda \in \C\setminus \overline{\Sigma_{\phi}}.
\end{equation}
Moreover, $\om(A):=\inf\{\phi\in (0,\pi)\,:\,\eqref{eq:sectorial}\,\,\text{holds for some}\; C>0\}$ is called the angle of sectoriality of $A$.

Next we define the $H^{\infty}$-calculus for a sectorial operator $A$. Let $\phi\in (0,\pi)$ and let us denote by $H^{\infty}_0(\Sigma_{\phi})$ the set of all holomorphic function $f:\Sigma_{\phi}\to \C$ such that $|f(z)|\lesssim \min\{|z|^{\varepsilon},|z|^{-\varepsilon}\}$ for some $\varepsilon>0$.

For $\phi>\om(A)$ and $f\in H^{\infty}_0(\Sigma_{\phi})$, we set
$$
f(A):=\frac{1}{2\pi i} \int_{\Gamma} f(z)(z-A)^{-1}dz.
$$
Note that $f(A)$ is well defined and $f(A)\in \calL(X)$. We say that $A$ has a bounded $H^{\infty}$-calculus of angle $\phi$ if there exists $C>0$ such that
\begin{equation}
\label{eq:H_infinite_calculus}
\|f(A)\|_{\calL(X)}\leq C\|f\|_{H^{\infty}(\Sigma_{\phi})},\qquad \forall f \in H^{\infty}_0(\Sigma_{\phi}).
\end{equation}
Finally, we set $\angH(A):=\inf\{\phi\in (0,\pi)\,:\,\eqref{eq:H_infinite_calculus}\text{ holds for some }C>0\}$ is the angle of the $H^{\infty}$-calculus of $A$.

For the reader's convenience, we list some operators with a bounded $H^{\infty}$-calculus. However, this list is far from complete. Moreover, there are still many new developments on $H^\infty$-calculus for differential operators.

\begin{example}\
\label{ex:Hinfty}
\begin{enumerate}[{\rm(1)}]
\item Positive self-adjoint operators on Hilbert spaces \cite[Proposition 10.2.23]{Analysis2};
\item $-A$ generates an analytic contraction semigroups on a Hilbert space
\cite[Theorem 10.2.24 and Corollary 10.4.10]{Analysis2};
\item $-A$ generates a positive contraction semigroup on $L^q$ which is analytic and bounded on a sector.
\item Second order uniformly elliptic operators with Dirichlet or Neumann boundary conditions on $L^q(\Dom)$, where $q\in (1,\infty)$ and $\Dom\in \{\R^d,\R^{d-1}\times \R_+\}$ or $\Dom$ is a $C^2$-domain with compact boundary \cite{DDHPV} or \cite{KKW};
\item Second and high-order uniformly elliptic operators with Lopatinskii-Shapiro boundary conditions (see \cite[Chapter 6]{pruss2016moving}) on $L^q(\Dom)$, where $q\in (1,\infty)$ and $\Dom$ is a sufficiently smooth domain with compact boundary \cite{DDHPV};
\item The Stokes operator on $\mathbb{L}^q(\Dom)$ (i.e.\ divergence-free vector fields in $L^q(\Dom;\R^d)$), where $q\in (1,\infty)$ and $\Dom$ is a bounded $C^{2,\alpha}$-domain \cite{KKW,KuWePertErratum};
\item The Stokes operator on $\mathbb{L}^q(\Dom)$, where $|\frac{1}{q}-\frac{1}{2}|\leq \frac{1}{2d}$ and $\Dom$ is a bounded Lipschitz domain \cite{Stokes}.
\end{enumerate}
\end{example}
Some more examples can be found in \cite[Chapter 10]{Analysis2} and in particular the notes to that chapter. Moreover, by interpolation-extrapolation arguments one obtains similar results on other spaces (see Appendix \ref{sec:app}).

Finally, we introduce the class of operators with bounded imaginary powers (or briefly $\BIP$). For details we refer to \cite{Haase:2}. Let $A$ be a sectorial operator on $X$.
The operator $A^{it}$ is defined through the extended functional calculus \cite[Subsection 3.3.2]{pruss2016moving}. We say that $A\in \BIP(X)$ if $A^{it}\in \calL(X)$ for all $t\in \R$. In this case, one can check that $t\mapsto \|A^{it}\|_{\calL(X)}$ has exponential growth and we denote by $\theta_A$ the \textit{power-angle} of $A$, i.e.
$$
\theta_A:=\limsup_{t\uparrow\infty}\frac{1}{t}\log\|A^{it}\|_{\calL(X)}.
$$
For future convenience, let us recall the following properties:
\begin{itemize}
\item If $A\in \BIP(X)$, then $[X,\Do(A)]_{\theta}=\Do(A^{\theta})$ for any $\theta\in (0,1)$, see e.g.\ \cite[Theorem 3.3.7]{pruss2016moving};
\item If $A$ has a bounded $H^{\infty}$-calculus, then $A\in \BIP(X)$ and $\theta_A\leq \angH(A)$.
\end{itemize}

\subsection{Fractional Sobolev spaces with power weights}
Let $X$ be a Banach space. Here and in the rest of the paper for $p\in (1,\infty)$ and $\a\in (-1, p-1)$ we set $w_{\a}(t):=|t|^{\a}$ for $t\in \R$. For $p,\a$ as before and for an open interval $\I$ we denote by
$L^p(\I,w_{\a};X)$  the Banach space of all strongly measurable functions $f:\I\to X$ for which
$$
\|f\|_{L^p(\I,w_{\a};X)}^p:= \int_{\I} \|f(t)\|_X^p w_{\a}(t)dt<\infty.
$$
If $\a=0$, then $w_{\a}=1$ and we write $L^p(\I;X)$ instead of $L^p(\I,w_0;X)$. Moreover, we note that if $0\notin \overline{I}$ and $\I$ is bounded, then $L^p(\I,w_{\a};X)=L^p(I;X)$ isomorphically.
Moreover, for $\I=(a,b)$ and $p,\a$ as above, we set $L^{p}(a,b,w_{\a};X):=L^{p}(\I,w_{\a};X)$. A similar convention will be used for the spaces introduced below.

To introduce Sobolev spaces we need to introduce the space of $X$-valued distributions. For an open subset $\I\subseteq \R$, let $\D(\I):=C^{\infty}_0(\I)$ with the usual topology. Then we define the set of all $X$-valued distribution as $\D'(\I;X):=\calL(\D(\I);X)$. Note that  $L^1_{\loc}(\I;X)\hookrightarrow\D'(\I;X)$ and one can define the distributional derivative $f^{(j)}\in \D'(\I;X)$ for all $j\geq 1$ and $f\in L^1_{\loc}(\I;X)$ in the usual way.

For $n \geq 1$ and an open interval $I\subseteq \R$, we denote by $W^{n,p}(\I,w_{\a};X)$ the set of all $f\in L^p(\I,w_{\a};X)$ such that $f^{(j)} \in L^p(\I,w_{\a};X)$ for all $j\in \{1,\ldots, n\}$, where $f^{(j)}$ denotes the $j$-th distributional derivative of $f$. We endow $W^{n,p}(\I,w_{\a};X)$ with the norm
$$
\|f\|_{W^{n,p}(\I,w_{\a};X)}:=\sum_{j=0}^n \|f^{(j)}\|_{L^p(\I,w_{\a};X)}.
$$
If $\a\in (-1,p-1)$ and $0\in \overline{I}$, then the trace map $f\mapsto f(0)$ is a bounded mapping from $W^{1,p}(\I,w_{\a};X)$ into $X$ (see \cite[Lemma 3.1]{LV18}).

Define a closed subspace of $W^{1,p}(\I,w_{\a};X)$ as
\[\Wz^{1,p}(\I,w_{\a};X) = \{f\in W^{1,p}(\I,w_{\a};X): f(0)=0 \  \text{if $0\in \overline{\I}$}\}.\]
In case $I=(0,T)$ for some $T\in (0,\infty)$, the Poincar\'{e} inequality (see \cite[Lemma 2.12]{MS12}) gives that
\begin{equation}\label{eq:poincare}
\|f\|_{L^p(0,T,w_{\a};X)}\lesssim_{p,\kappa} T\| f'\|_{L^p(0,T,w_{\a};X)},
\quad \forall f\in \Wz^{1,p}(\I,w_{\a};X).
\end{equation}

We introduce fractional Sobolev spaces by complex interpolation as in \cite{MS12} and \cite[Section 3.4.5]{pruss2016moving}.
\begin{definition}
Let $-\infty\leq a<b\leq \infty$, $\I=(a,b)$, $p\in (1,\infty)$, $\a\in (-1,p-1)$ and $\theta \in (0,1)$. Let
\begin{align*}
H^{\theta,p}(\I,w_{\a};X)&:=[L^p(\I,w_{\a};X),W^{1,p}(\I,w_{\a};X)]_{\theta}.
\end{align*}
If $0\in \overline{\I}$ let
\[
\hz^{\theta,p}(\I,w_{\a};X):=[L^p(\I,w_{\a};X),\Wz^{1,p}(\I,w_{\a};X)]_{\theta}.
\]
\end{definition}
As before, $H^{\theta,p}(\I,w_{\a};X)=H^{\theta,p}(\I;X)$ isomorphically if $0\notin \overline{\I}$ and $\I$ is bounded. Furthermore, by interpolation it is immediate that
\begin{align}\label{eq:contractiveembd0}
\hz^{\theta,p}(\I,w_{\a};X)\hookrightarrow H^{\theta,p}(\I,w_{\a};X) \ \ \text{contractively}.
\end{align}
Let us note some further properties of the above spaces.
\begin{proposition}
\label{prop:eliminate_weights}
Let $X$ be a Banach space. Let $\theta\in( 0,1)$, $p\in (1,\infty)$, $\a\in (-1,p-1)$, $\J\subseteq \I\subseteq \R$ intervals, $\I_T = (0,T)$ with $T\in (0,\infty]$, $\varepsilon>0$, and $\A\in \{H,\hz\}$. Then for all $f\in \A^{\theta,p}(\I_T,w_{\a};X)$,
\begin{align*}
\|f\|_{\A^{\theta,p}(\J,w_{\a};X)} & \leq \|f\|_{\A^{\theta,p}(\I,w_{\a};X)}, \ \ \
\\ \|f\|_{H^{\theta,p}(\varepsilon,T;X)} & \leq \varepsilon^{-\a}\|f\|_{\A^{\theta,p}(\I_T,w_{\a};X)},\quad \ \text{if $\kappa\in [0,p-1)$}.
\end{align*}\end{proposition}

\begin{proof}
For convenience of the reader we provide the details.
The first estimate follows by interpolating the restriction operator mapping from $\A^{k,p}(\I,w_{\a};X)$ into $\A^{k,p}(\J,w_{\a};X)$ for $k\in \{0,1\}$.

To prove the second estimate by \eqref{eq:contractiveembd0} it suffices to consider the case $\A= H$. Let $r:f\mapsto f|_{(\varepsilon,T)}$ be the restriction operator on $(\varepsilon,T)$. It is immediate to see that
$$
\|r\|_{\calL(W^{j,p}(\I_T,w_{\a};X)),W^{j,p}(\varepsilon,T;X))}\leq \varepsilon^{-\a},
$$
for $j\in \{0,1\}$. Thus, interpolation gives $r:H^{\theta,p}(\I_T,w_{\a};X)\to H^{\theta,p}(\varepsilon,T;X)$ with norm at most $\varepsilon^{-\a}$.
\end{proof}

\subsubsection{Extension operators}
\label{ss:extension_operators}
Here we discuss extension operators for the spaces just introduced.
In \cite{MS12}, extension operators for the above spaces are already given. However, we found a different and (to our viewpoint) simpler approach to build extension operators. It will give some more information, which will be needed in the following. Let us begin with a definition.

\begin{definition}[Extension operator]
\label{def:extension_operator}
Let $\A\in \{H^{s,p},\hz^{s,p}\}$ for some $s\in [0,1]$, $p\in (1,\infty)$ and let $\a\in (-1,p-1)$. Let $\I_T=(0,T)$ for some $T\in (0,\infty)$. We say that a bounded linear operator
$$\Ext:\A(\I_T,w_{\a};X)\to \A(\R,w_{\a};X),$$
is an extension operator on $\A(\I_T,w_{\a};X)$ if $\Ext f=f$ on $\I_T$.
\end{definition}

Let $\Extunit$ be the extension operator which maps,
\begin{equation}
\A(0,1,w_{\a};X)\to \A(\R,w_{\a};X), \ \ \ \text{ where } \ \ \ \A\in \{L^p,W^{1,p}\}
\end{equation}
given by the classical reflection argument (see e.g.\ \cite[Theorems 5.19 and 5.22]{AF03}), which can be extended to the weighted setting. By construction it follows that
\begin{align}
\label{eq:extension_1_L_p}
\|\Extunit f\|_{L^p(\R,w_{\a};X)}&\leq C_{p,\a} \|f\|_{L^p(0,1,w_{\a};X)},\\
\label{eq:extension_1_W_p}
\|(\Extunit f)'\|_{L^p(\R,w_{\a};X)}&\leq C_{p,\a} (\|f\|_{L^p(0,1,w_{\a};X)}+\|f'\|_{L^p(0,1,w_{\a};X)}),
\end{align}
where $C_{p,\a}$ is a constant which depends only on $p,\a$.

\begin{proposition}
\label{prop:extension}
Let $s\in [0,1]$, $p\in (1,\infty)$, $\a\in (-1,p-1)$ and let $T\in (0,\infty)$. Let $\Ext:L^p(0,T,w_{\a};X)\to L^p(\R,w_{\a};X)$ be the operator given by
$$
\Ext f(t):=\Extunit (f(T \cdot))\Big(\frac{t}{T}\Big), \qquad t\in \R,
$$
where $\Extunit$ is as above. Then the following assertion holds.
\begin{enumerate}[{\rm(1)}]
\item\label{it:extension_0} The restriction $\Extz$ of  $\Ext$ to $\hz^{s,p}(\I_T,w_{\a};X)$ defines a bounded extension operator with values in $\hz^{s,p}(\R,w_{\a};X)$ with
$$
\|\Extz \|_{\calL(\hz^{s,p}(\I_T,w_{\a};X),\,\hz^{s,p}(\R,w_{\a};X))}\leq \prescript{}{0}{C},
$$
where $\prescript{}{0}{C}$ depends only on $p,s,\a$.
\item\label{it:extension} Let $\eta>0$ and $T\in (\eta,\infty]$. Then $\Ext$ induces an extension operator on $H^{s,p}(\I_T,w_{\a};X)$, which will be still denoted by $\Ext$. Moreover,
$$
\|\Ext \|_{\calL(H^{s,p}(\I_T,w_{\a};X),H^{s,p}(\R,w_{\a};X))}\leq C,
$$
where $C$ depends only on $p,s,\a,\eta$.
\end{enumerate}
\end{proposition}

\begin{proof}
\eqref{it:extension_0}:
By a change of variable and \eqref{eq:extension_1_L_p},
\begin{equation*}
\|\Extz f\|_{L^p(\R,w_{\a};X)}
=\Big\|t\mapsto \Extunit (f(T \cdot))\Big(\frac{t}{T}\Big)\Big\|_{L^{p}(\R,w_{\a};X)}
\lesssim \|f\|_{L^p(\I_T,w_{\a};X)},
\end{equation*}
and
\begin{align*}
\|(\Extz f)'\|_{L^p(\R,w_{\a};X)}
&=T^{-1} \Big\|t\mapsto (\Extunit (f(T \cdot)))'\Big(\frac{t}{T}\Big)\Big\|_{L^{p}(\R,w_{\a};X)}\\
&=T^{-1+\frac{1+\a}{p}} \|(\Extunit (f(T \cdot)))'\|_{L^{p}(\R,w_{\a};X)}\\
&\stackrel{(i)}{\lesssim} T^{-1+\frac{1+\a}{p}} (\|f(T\cdot)\|_{L^p(0,1,w_{\a};X)}+\|f'(T\cdot)T\|_{L^p(0,1,w_{\a};X)})\\
&\stackrel{(ii)}{\lesssim}  T^{\frac{1+\a}{p}} \|f'(T\cdot)\|_{L^p(0,1,w_{\a};X)}= \|f'\|_{L^p(\I_T,w_{\a};X)}
\end{align*}
where in $(i)$ we used \eqref{eq:extension_1_W_p} and in $(ii)$ the weighted Poincar\'e inequality \eqref{eq:poincare}.
We can conclude that also $\|\Extz\|_{\calL(\Wz^{1,p}(0,T,w_{\a};X),\Wz^{1,p}(\R,w_{\a};X))}\leq C$ with $C$ independent of $T$.

Now complex interpolation gives that $\Extz$ is a bounded linear operator from $\hz^{s,p}(\I_T,w_{\a};X)$ into $\hz^{s,p}(\R,w_{\a};X)$. Moreover, it has the extension property, i.e.\ $\Extz f=f$ on $\I_T$, which follows from the extension property of $\Extunit$.

\eqref{it:extension}: This follows in the same way, but since we cannot use Poincar\'e inequality, we obtain
$\|\Ext\|_{\calL(W^{1,p}(0,T,w_{\a};X),W^{1,p}(\R,w_{\a};X))}\leq C(1 + T^{-1})$.
\end{proof}

\subsubsection{Embedding results}
In this section we collect some basic embedding results for the spaces introduced in the previous section. To begin, let us introduce Sobolev embeddings and interpolation inequalities for $H^{s,p}$. Some of the following results might also hold for general Banach spaces, but since we will use the UMD property many times we prefer the presentation below. Note that the difficulty in the proofs below is that we want estimates with $T$-independent constants as this is required in fixed point arguments below.

The following result on vector-valued Sobolev spaces follows from \cite[sections 5 and 6]{LMV18}. The scalar unweighted case is simpler, and in that case the result is a special case of \cite{Se}.
\begin{theorem}\label{t:equivalence_h_H}
Let $X$ be a UMD space, $p\in (1, \infty)$, $\a\in (-1,p-1)$, $s\in (0,1)$, and $I \in \{\R,\R_+\}$. If $s\neq \frac{1+\a}{p}$, then
\[\hz^{s,p}(I,w_{\a};X) = \left\{
                     \begin{array}{ll}
                       \{u\in H^{s,p}(I,w_{\a};X): u(0) = 0\}, \quad& \hbox{if $s>\frac{1+\kappa}{p}$,} \\
                       H^{s,p}(I,w_{\a};X), \quad& \hbox{if $s<\frac{1+\kappa}{p}$,}
                     \end{array}
                   \right.
\]
isomorphically.
\end{theorem}

By using the extension operator of Proposition \ref{prop:extension} one can see that Theorem \ref{t:equivalence_h_H} extends to $I = (0,T)$ with $T\in (0,\infty)$.
In particular, if $s\neq \frac{1+\a}{p}$, then $\hz^{s,p}(I,w_{\a};X)$ is a closed subspace of $H^{s,p}(I,w_{\a};X)$. Although this seems very likely, this seems to be highly nontrivial. As a consequence the estimate $\|u\|_{\hz^{s,p}(I,w_{\a};X)}\eqsim \|u\|_{H^{s,p}(I,w_{\a};X)}$ holds, where we need the condition  $u(0)=0$ if $s>\frac{1+\a}{p}$. The theorem will usually be applied through the latter norm equivalence.

\begin{proposition}[Sobolev embedding]
\label{prop:embeddingSobolevWeights}
Let $X$ be a UMD Banach space. Let $T\in (0,\infty]$ and set $I_T = (0,T)$. Assume that $1< p_0\leq p_1<\infty$, $s_0, s_1\in (0,1)$ and $\a_i\in (-1,p_i-1)$ for $i\in \{0,1\}$. Assume $\frac{\a_1}{p_1} \leq \frac{\a_0}{p_0}$ and $s_0-\frac{1+\a_0}{p_0}\geq s_1-\frac{1+\a_1}{p_1}$. Then there is a constant $C$ independent of $T$ such that for all $f\in \hz^{s_0,p_0}(I_T,w_{\a_0};X)$,
\[
\|f\|_{\hz^{s_1,p_1}(I_T,w_{\a_1};X)}\leq C \|f\|_{\hz^{s_0,p_0}(I_T,w_{\a_0};X)}.
\]
The same holds with $\hz^{s_i,p_i}(I_T,w_{\a_i};X)$ replaced by $H^{s_i,p_i}(I_T,w_{\a_i};X)$ with a constant $C$ which depends on $T$.
\end{proposition}
\begin{proof}
First assume $s_1\neq \frac{1+\a_1}{p_1}$.
Let $\Extz$ be as in Proposition \ref{prop:extension}\eqref{it:extension_0}. Then
\begin{align*}
\|f\|_{\hz^{s_1,p_1}(\I_T,w_{\a_1};X)}
{\leq} \|\Extz f\|_{\hz^{s_1,p_1}(\R,w_{\a_1};X)}
\end{align*}
where we used Proposition \ref{prop:eliminate_weights} for $\Extz f$. By Theorem \ref{t:equivalence_h_H} it remains to estimate $\|\Extz f\|_{H^{s_1,p_1}(\R,w_{\a_1};X)}$. By \cite[Propositions 3.2 and 3.7]{MV15}, $\|\Extz f\|_{H^{s_i,p_i}(\R,w_{\a_i};X)}$ is equivalent to $\|\Extz f\|_{\H^{s_i,p_i}(\R,w_{\a_i};X)}$, where $\H$ denotes the Bessel potential space. Therefore, by the weighted Sobolev embedding result \cite[Corollary 1.4]{MV12} we obtain
\begin{align*}
\|\Extz f\|_{H^{s_1,p_1}(\R,w_{\a_1};X)} \lesssim  \|\Extz f\|_{H^{s_0,p_0}(\R,w_{\a_0};X)}.
\end{align*}
By \eqref{eq:contractiveembd0} and Proposition \ref{prop:extension}\eqref{it:extension_0} we obtain
\[
\|\Extz f\|_{H^{s_0,p_0}(\R,w_{\a_0};X)} \leq \|\Extz f\|_{\hz^{s_0,p_0}(\R,w_{\a_0};X)}\lesssim  \|f\|_{\hz^{s_0,p_0}(I_T,w_{\a_0};X)},
\]
and the result follows by combining the estimates.

In the case $s_1-\frac{1+\a_1}{p_1} = 0$ we use an interpolation argument. Let $\varepsilon>0$ be so small that $s_j^{\pm} := s_j\pm \varepsilon\in (0,1)$. Then by the previous considerations
\[\hz^{s_0^{\pm},p_0}(I_T,w_{\a_0};X)\hookrightarrow \hz^{s_1^{\pm},p_1}(I_T,w_{\a_1};X),\]
where the embedding constants can be taken $T$-independent. Interpolating both embeddings gives the desired embedding in the remaining case.

The final assertion can be proved with the same method, but one can avoid Theorem \ref{t:equivalence_h_H}. Moreover, one needs to use the extension operator on $H^{s,p}$ spaces provided by Proposition \ref{prop:extension}.
\end{proof}

Next we prove a version of the mixed derivative result \cite[Theorem 3.18]{LV18}, but with $T$-independent estimates.
\begin{proposition}[Mixed derivative inequality]
\label{prop:Sob_interpolation_h}
Let $(X_0,X_1)$ be an interpolation couple such that both $X_0$ and $X_1$ are UMD spaces.
Let $p_i\in (1,\infty)$, $\a_i\in (-1,p_i-1)$, and $s_i\in (0,1)$ for $i\in \{0,1\}$. For $\theta\in (0,1)$ set
\[s:=s_0(1-\theta)+s_1\theta, \ \ \frac{1}{p}:= \frac{1-\theta}{p_0} + \frac{\theta}{p_1}, \ \ \a := (1-\theta) \frac{p}{p_0} \a_0 + \theta \frac{p}{p_1} \a_1.\]
Assume $T\in (0,\infty]$ and $s \neq \frac{1+\a}{p}$. Then there exists a constant $C>0$ independent of $T\in (0,\infty]$ such that for all $f\in \hz^{s_0,p_0}(\I_T,w_{\a_0};X_0)\cap \hz^{s_1,p_1}(\I_T,w_{\a_1};X_1)$,
\[
\|f\|_{\hz^{s,p}(\I_T,w_{\a};[X_0,X_1]_{\theta})}\leq
C\|f\|_{\hz^{s_0,p_0}(\I_T,w_{\a_0};X_0)}^{1-\theta}\|f\|_{\hz^{s_1,p_1}(\I_T,w_{\a_1};X_1)}^{\theta}.
\]
The same holds with $\hz^{s_i,p_i}(\I_T,w_{\a_i};X_i)$ replaced by $H^{s_i,p_i}(\I_T,w_{\a_i};X_i)$ with a constant $C$ which depends on $T$ in which case $s=\frac{1+\a}{p}$ is also allowed.
\end{proposition}

\begin{proof}
Let $\Extz$ be as in Proposition \ref{prop:extension}\eqref{it:extension_0}. By construction (see Subsection \ref{ss:extension_operators}) $\Extz$ does not depend on $p_i, \a_i, s_i, X_i$. Therefore, Proposition \ref{prop:eliminate_weights} gives
\begin{align*}
\|f\|_{\hz^{s,p}(\I_T,w_{\k};[X_0,X_1]_{\theta})}
&\leq \|\Extz f \|_{\hz^{s,p}(\R,w_{\k};[X_0,X_1]_{\theta})}.
\end{align*}

Since $s \neq \frac{1+\a}{p}$, by Theorem \ref{t:equivalence_h_H} it suffices to estimate $\|\Extz f \|_{H^{s,p}(\R,w_{\k};[X_0,X_1]_{\theta})}$. The interpolation result \cite[Theorem 3.18]{LV18} implies
\begin{align*}
\|\Extz f \|_{H^{s,p}(\R,w_{\k};[X_0,X_1]_{\theta})} \leq  C\|\Extz f\|_{H^{s_0,p_0}(\R,w_{\a_0};X_0)}^{1-\theta}
\|\Extz f\|_{H^{s_1,p_1}(\R,w_{\a_1};X_1)}^{\theta}.
\end{align*}
As in the proof of Proposition \ref{prop:embeddingSobolevWeights} one can check that
\[\|\Extz f\|_{H^{s_i,p_i}(\R,w_{\a_i};X_i)}\leq \|\Extz f\|_{\hz^{s_i,p_i}(\R,w_{\a_i};X_i)}\lesssim  \|f\|_{\hz^{s_i,p_i}(\I_T,w_{\a_i};X_i)},\]
and we can conclude the required embedding holds.

The final assertion can be proved in a similar way.
\end{proof}

\begin{remark}
\label{r:0H_replace_H}
It is to be expected that combining the methods of \cite{LMV18} with \cite[Theorem 3.18]{LV18}, Proposition \ref{prop:Sob_interpolation_h} can be improved to
\begin{equation}\label{eq:interpidentity}
[\hz^{s_0,p_0}(\R_+,w_{\a_0};X_0), \hz^{s_1,p_1}(\R_+,w_{\a_1};X_1)]_{\theta}=\hz^{s,p}(\R_+,w_{\a};[X_0,X_1]_{\theta})
\end{equation}
under the condition $s\neq \frac{1+\kappa}{p}$. In the case that $s = \frac{1+\kappa}{p}$, we expect the embedding $\hz^{s_0,p}(\I_T,w_{\a};X_0) \cap \hz^{s_1,p}(\I_T,w_{\a};X_1) \hookrightarrow \hz^{s,p}(\I_T,w_{\a};[X_0,X_1]_{\theta})$ to be valid with $T$-independent constants as well. This could be proved by a reiteration and interpolation argument using \eqref{eq:interpidentity}.
\end{remark}

We conclude this section by recalling an optimal trace result for anisotropic spaces. This result is a special case of the trace embedding of \cite{ALV19}. In the case that $X_1 =\Do(A)$ where $A$ is an invertible sectorial operator on a UMD Banach space $X_0$, the following is a consequence of \cite[Theorem 1.1]{MV14} and \cite[Corollary 7.6]{AV19}. Moreover, the UMD condition can be avoided.
In the following for an interval $J\subseteq \R_+$ and a Banach space $X$, we denote by $C_0(\overline{J};X)$ the set of all continuous functions $f:\overline{J}\to X$ vanishing at infinity endowed with the norm given by the right-hand side of \eqref{eq:Continuous_functions_norm}.

\begin{proposition}
\label{prop:continuousTrace}
Let $(X_0,X_1)$ be a couple of Banach space such that $X_1\hookrightarrow X_0$. Set $X_{1-\theta}=[X_0,X_1]_{1-\theta}$ or $X_{1-\theta} = (X_0,X_1)_{1-\theta,r}$ with $r\in [1, \infty]$. Assume that $p\in (1,\infty)$, $\k\in [0,p-1)$, $\theta\in (0,1)$ and $T\in (0,\infty]$. Then the following holds:
\begin{enumerate}[{\rm (1)}]
\item\label{it:trace_with_weights_Xap} If $\theta>\frac{1+\a}{p}$, then
$$
H^{\theta,p}(\I_T,w_{\k};X_{1-\theta})\cap L^p(\I_T,w_{\k};X_1)
\hookrightarrow C_0\big(\overline{I}_T;(X_0,X_1)_{1-\frac{1+\a}{p},p}\big);
$$
\item\label{it:trace_without_weights_Xp} If $\theta>\frac{1}{p}$, then for any $0<\varepsilon<T$ and $J_{\varepsilon,T}=(\varepsilon,T)$
$$
H^{\theta,p}(\I_T,w_{\k};X_{1-\theta})\cap L^p(\I_T,w_{\k};X_1)\hookrightarrow
 C_0\big(\overline{J}_{\varepsilon,T};(X_0,X_1)_{1-\frac{1}{p},p}\big).
$$
\end{enumerate}
Moreover, the constants in \eqref{it:trace_with_weights_Xap} and \eqref{it:trace_without_weights_Xp} depend only on $\eta$ if $T\in (\eta,\infty]$.
Furthermore, if we replace $H^{\theta,p}$ by $\hz^{\theta,p}$ in \eqref{it:trace_with_weights_Xap} and \eqref{it:trace_without_weights_Xp} the constants in the embeddings can be chosen independent of $T>0$.
\end{proposition}
Here \eqref{it:trace_with_weights_Xap} follows from the above mentioned references and Proposition \ref{prop:extension}.
To prove \eqref{it:trace_without_weights_Xp} one can reduce to \eqref{it:trace_with_weights_Xap} with $\kappa = 0$ by Proposition \ref{prop:eliminate_weights} and a translation argument. To prove the embeddings \eqref{it:trace_with_weights_Xap} and \eqref{it:trace_without_weights_Xp} for $\hz^{\theta,p}$ by Proposition \ref{prop:extension} it suffices to consider the case $T=\infty$ in which case the result follows from \eqref{it:trace_with_weights_Xap} for $H^{\theta,p}$.

\subsection{Stochastic integration in UMD Banach spaces}
\label{ss:stochastic_integration}
The theory of stochastic integration in UMD Banach spaces with respect to a cylindrical Brownian motion is developed in \cite{BNVW08,NVW1}, see also \cite{NVW13}. Here we recall the results which will be needed in the following.

Throughout the paper $(\O,\mathcal{A},\Filtr=(\F_t)_{t\geq 0},\P)$ will denote a filtered probability space. Recall that a process $\phi:[0,T]\times\Omega \to X$, where $X$ is a Banach space, is called strongly progressively measurable if for all $t\in [0,T]$, $\phi|_{[0,t]}$ is strongly $\Borel([0,t])\otimes \F_t$-measurable (here $\Borel$ denotes the Borel $\sigma$-algebra). The $\sigma$-algebra generated by the strongly progressively measurable processes will be denoted by $\Progress$ and is a subset of $\Borel([0,\infty))\otimes \F_{\infty}$.

In the paper we will consider cylindrical Gaussian noise. In the following $H$ is a separable Hilbert space.
\begin{definition}
\label{def:Cylindrical_BM}
A bounded linear operator $W_H:L^2(\R_+;H)\rightarrow L^2(\Omega)$ is said to be a {\em cylindrical Brownian motion} in $H$ if the following are satisfied:
\begin{itemize}
\item for all $f\in L^2(\R_+;H)$ the random variable $W_H(f)$ is centered Gaussian.
\item for all $t\in \R_+$ and $f\in L^2(\R_+;H)$ with support in $[0,t]$, $W_H(f)$ is $\F_t$-measurable.
\item for all $t\in \R_+$ and $f\in L^2(\R_+;H)$ with support in $[t,\infty]$, $W_H(f)$ is independent of $\F_t$.
\item for all $f_1,f_2\in L^2(\R_+;H)$ we have $\E(W_H(f_1)W_H(f_2))=(f_1,f_2)_{L^2(\R_+;H)}$.
\end{itemize}
\end{definition}
Given a cylindrical Brownian motion in $H$, the process $(W_H(t)h)_{t\geq 0}$, where
\begin{equation}
W_H(t)h:=W_H(\one_{(0,t]}\otimes h), \ \  h\in H,
\end{equation}
is a Brownian motion.

\begin{example}
\label{ex:BM}
Let $(w_n)_{n\geq 1}$ be independent standard Brownian motions. Then $W_{\ell^2} (f) = \sum_{n\geq 1} \int_{\R_+} \lb f, e_n\rb d w_n$ converges in $L^2(\Omega)$ and defines a cylindrical Brownian motion in $\ell^2$, where $e_n=(\delta_{jn})_{n\geq 1}$ and $\delta_{jn}$ denotes the Kronecker's delta.
\end{example}

To introduce stochastic integration in UMD Banach spaces $X$ we first recall the definition of $\g$-radonifying operators (see \cite[Chapter 9]{Analysis2} for details). Let $(\wt{\g}_i)_{i\geq 1}$ be a sequence of independent standard normal random variable on a probability space $(\wt{\O},\wt{\P})$ and $(h_i)_{i\geq 1}$ an orthonormal basis for $H$. We say that a bounded linear operator $T:H\to X$ belongs to $\g(H,X)$ if $\sum_{i=1}^{\infty} \wt{\g}_i Th_i$ converges in $L^2(\Omega;X)$ and in this case we let
$$
\|T\|_{\g(H,X)}^2:=\wt{\E}\Big\|\sum_{i=1}^{\infty} \wt{\g}_i Th_i\Big\|_{X}^2.
$$
Note that for $X = L^p(S)$ with $p\in [1, \infty)$, where $(S,\Sigma,\mu)$ is a measure space one has (see \cite[Proposition 9.3.2]{Analysis2})
\begin{align}\label{eq:gammaidentity}
\gamma(H,X) = L^p(S;H).
\end{align}

At this point, we can define the \textit{stochastic integral with respect to a cylindrical Brownian} motion in $H$ of  the process $\one_{A\times (s,t]}\otimes( h\otimes x)$:
\begin{equation}
\int_0^{\infty} \one_{A\times (s,t]}\otimes( h\otimes x)(s) \,dW_H(s):= \one_{A}\otimes (W_H(t)h-W_H(s)h)\,x\,,
\end{equation}
and we extend it to adapted step processes of finite rank by linearity.

We denote by $L^p_{\Progress}((0,T)\times\Omega;\g(H,X))$ the progressive measurable subspace of $L^p((0,T)\times\Omega;\g(H,X))$. One can show this coincides with the closure of the adapted step processes of finite rank. The next result is well-known and actually valid for the larger class of martingale type $2$ spaces (see \cite[Theorem 4.7]{NVW13} and \cite{Ondrejat04}):
\begin{proposition}
\label{prop:Ito}
Let $T>0$, $p\in (0,\infty)$ and let $X$ be a UMD Banach space with type $2$. Then the mapping $G\mapsto \int_0^T G\,dW_H$ extends to a bounded linear operator from $L^p_{\Progress}((0,T)\times\Omega;\g(H,X))$ into $L^p(\Omega;X)$. Moreover,
\begin{equation*}
\E\sup_{0\leq  t\leq T}\Big\|\int_0^t G(s)\,dW_H(s)\Big\|_{X}^p \lesssim_{p,X,T} \E\|G\|_{L^2(0,T;\g(H,X))}^p.
\end{equation*}
\end{proposition}
A sharp two-sided estimate for the stochastic integral was obtained in \cite{NVW1} and \cite[Theorem 5.5]{NVW13}. It might seem that in the current paper we only use Proposition \ref{prop:Ito}, but typically the maximal regularity estimates we use require these sharper estimates. In particular, this is the case in Theorem \ref{t:H_infinite_SMR} below.

\subsection{Stopping times and related concepts}
A stopping time $\tau$ is a measurable map $\tau:\O\to [0,T]$ such that $\{\tau\leq t\}\in \F_t$ for all $t\in [0,T]$.
We denote by $\ll 0,\sigma\rr$ the stochastic interval
$$
\ll 0,\sigma\rr:=\{(t,\omega)\in [0,T]\times\Omega\,:\,0 \leq t\leq \sigma(\om)\}.
$$
Analogously definitions hold for $\ll 0,\sigma\rro$, $\llo 0,\sigma\rro$ etc.

In accordance with the previous notation, for $A\subseteq \O$ and $\tau,\mu$ two stopping times such that $\tau\leq \mu$, we set
\begin{equation*}
[0,T]\times\Omega \supseteq [\tau,\mu]\times A:=\{(t,\om)\in [0,T]\times A\,:\,\tau(\om)\leq t\leq \mu(\om)\}.
\end{equation*}
Similar definitions are employed for $[\tau,\mu)\times A$, $(\tau,\mu)\times A$ etc. 
In particular, we have $\ll 0,\sigma\rr=[0,\sigma]\times\Omega$.

Let $X$ be a Banach space and let $A\in \mathcal{A}$. We say that $u:A\times [0,\mu]\to X$ is strongly measurable (resp.\ strongly progressively measurable) if the process
\begin{equation}
\one_{A\times [0,\mu]}u:=
\begin{cases}
u,\qquad \text{on }A\times [0,\mu],\\
0,\qquad \text{otherwise},
\end{cases}
\end{equation}
is strongly measurable (resp.\ strongly progressively measurable).

To each stopping time $\tau$ we can associate the $\sigma$-algebra of the $\tau$-past,
$$
\F_{\tau}:=\{A\in \mathcal{A}\,:\,\{\tau\leq t\}\cap A\in \F_t,\;\forall t \in [0,T]\}.
$$

The following well-known results will be used frequently in the paper without further mentioning (see \cite[Lemmas 7.1 and 7.5]{Kal}).
\begin{proposition}
\label{prop:F_sigma_algebra_stopping_times}
Let $\tau$ be a stopping time. Then $\F_{\tau}$ is a $\sigma$-algebra and satisfies the following properties.
\begin{itemize}
\item If $\tau=t$ a.s.\ for some $t\in [0,T]$, then $\F_{\tau}=\F_t$.
\item If $X:[0,T]\times\Omega\to X$ is a strongly progressively measurable process, then the random variable $X_{\tau}(\om):=X(\tau(\om),\om)$ is strongly $\F_{\tau}$-measurable.
\end{itemize}
\end{proposition}

We continue with another measurability lemma.
\begin{lemma}\label{lem:XYtmeas}
Let $X$ be a Banach space. For each $t\in [0,T]$, let $Y_t$ be a space of functions $f:[0,t]\to X$. Assume that for each $f\in Y_T$ and each $t\in [0,T]$,
\begin{itemize}
\item $f|_{[0,t]}\in Y_t$;
\item $t\mapsto \|f|_{[0,t]}\|_{Y_t}$ is increasing;
\end{itemize}
Let $u:\Omega\to Y_T$ be strongly measurable and $\tau$ be a stopping time. Then the map $\omega\mapsto \|u(\omega)|_{[0,\tau(\omega)]}\|_{Y_{\tau(\omega)}}$ is measurable.
\end{lemma}
\begin{proof}
Since $u$ is strongly measurable, we may assume that $Y_T$ is separable.

Let $\Psi:[0,T]\times Y_T\to [0,\infty)$ be given by $\Psi(t,f) = \|f|_{[0,t]}\|_{Y_t}$. Then since for $f\in Y_T$, $\Psi(\cdot,f)$ is increasing, it follows that $\Psi(\cdot,f)$ is measurable. For $t\in [0,T]$ and $f,g\in Y_T$,
\[|\Psi(t,f) - \Psi(t,g)|\leq \|(f-g)|_{[0,t]}\|_{Y_t}\leq \|f-g\|_{Y_T}.\]
Therefore,  $\Psi(t,\cdot)$ is continuous. Since $Y_T$ is separable this implies $\Psi$ is measurable (see \cite[Lemma 4.51]{AliBor}).

On the other hand, $\zeta:\Omega\to [0,T]\times Y_T$ defined by $\zeta(\omega) =(\tau(\omega), u(\omega))$ is measurable.
Since $\|u(\omega)|_{[0,\tau(\omega)]}\|_{Y_{\tau(\omega)}} = \Psi(\zeta(\omega)) = (\Psi\circ \zeta)(\omega)$ the required measurability follows.
\end{proof}

The lemma will be applied to the spaces $Y_t$ such as
\[C([0,t];X), L^p(0,t,w_{\a};X), H^{\theta,p}(\I_t,w_{\a};X), \hz^{\theta,p}(\I_t,w_{\a};X).\]
The first two examples are simple because the norm is actually a continuous function of $t\in [0,T]$. In the cases $H^{\theta,p}$ and $\hz^{\theta,p}$ it is not obvious whether the norms are continuous in $t\in [0,T]$, but fortunately, they are increasing by Proposition \ref{prop:extension}.

The above lemma implies that the following versions of stopped spaces with stopped norms are well-defined.
\begin{definition}
\label{def:phi_spaces}
Let $X$ be a Banach space. Let $T>0$, $p,q\in (1,\infty)$, $r\in \{0\}\cup [1,\infty)$ and $\theta\in [0,1]$. Assume that $\tau$ is a stopping time such that $\tau:\O\to [0,T]$. Let $(Y_t)_{t\in [0,T]}$ be as in Lemma \ref{lem:XYtmeas}.
We say that $u\in L^r_{\Progress}(\O;Y_{\tau})$ if there exists a strongly progressively measurable $\tilde{u}\in L^r(\O;Y_T)$ such that $\tilde{u}|_{\ll 0,\tau\rr}=u$. If in addition $r\in [1,\infty)$, we set
\begin{equation}
\label{eq:norm_phi_spaces}
\begin{aligned}
\|u\|_{L^r(\O;Y_{\tau})}^r := \E\big(\|\tilde{u}|_{[0,\tau]}\|_{Y_{\tau}}^r\big).
\end{aligned}
\end{equation}
Finally, in case $Y_{t}=L^p(I_t,w_{\a};X)$, we set $L^p_{\Progress}(\I_{\tau}\times \O,w_{\a};X):=L^p_{\Progress}(\O;Y_{\tau})$.
\end{definition}

Using Lemma \ref{lem:XYtmeas} one can check that the expectation in \eqref{eq:norm_phi_spaces} is well-defined. Moreover, one can check that the norm does not depend on the choice of $\tilde{u}$.

\section{Stochastic maximal $L^p$-regularity}
\label{ss:SMR}

The following assumptions will be made throughout Sections \ref{ss:SMR} and \ref{s:quasi}.
\begin{assumption}
\label{ass:Xtr}
Let $X_0,X_1$ be UMD Banach spaces with type 2 and assume $X_1\hookrightarrow X_0$ densely. Assume one of the following two settings is satisfied
\begin{itemize}
\item $p\in (2,\infty)$ and $\a\in [0,\frac{p}{2}-1)$;
\item $p=2$, $\k=0$ and $X_0,X_1$ are Hilbert spaces.
\end{itemize}
For $\theta\in (0,1)$, and $p,\a$ as above let
$$
X_{\theta}:=[X_0,X_1]_{\theta}, \qquad \Xap:=(X_0,X_1)_{1-\frac{1+\a}{p},p}, \qquad \Xp:=\Xzp.
$$
\end{assumption}
The spaces $X_{\theta}$ have UMD and type $2$ (see \cite[Proposition 4.2.17]{Analysis1} and \cite[Proposition 7.1.3]{Analysis2}). The same holds for $\Xp$ but this will not be needed.

Moreover, in the case $p=2$ and $\a=0$, by \cite[Corollary C.4.2]{Analysis1} we have $X_{\frac{1}{2}}=(X_0,X_1)_{\frac{1}{2},2}=\Xzero$. This is the reason we only consider Hilbert spaces if $p=2$ and it will be used without further mentioning it.

\subsection{Stochastic maximal $L^p$-regularity}
In this subsection we collect some basic definitions.

The next assumption is solely for Section \ref{ss:SMR}, where the linear theory is treated.
\begin{assumption}
\label{ass:AB_boundedness}
Let $T\in (0,\infty]$ and set $\I_T:=(0,T)$. The maps $A:\I_T\times\Omega\to \calL(X_1,X_0)$ and $B:\I_T\times\Omega \to \calL(X_1,\g(H,X_{1/2}))$ are strongly progressively measurable. Moreover, we assume there exists $C_{A,B}>0$ such that
$$
\|A(t,\om)\|_{\calL(X_1,X_0)} + \|B(t,\om)\|_{ \calL(X_1,\g(H,X_{1/2}))}\leq C_{A,B},
$$
for a.a.\ $\om\in \O$ and all $t\in \I_T$.
\end{assumption}
Note that $A$ is a family of unbounded operators on $X_0$ and $\Do(A(t,\omega)) = X_1$, and $B$ is a family of unbounded operators on $X_{1/2}$ with domain $\Do(B(t,\omega)) = X_1$. The orders of both terms are comparable as the $A$-term is for the deterministic part, and the $B$-term for the stochastic part.

Stochastic maximal $L^p$-regularity is concerned with the optimal regularity estimate for the linear abstract stochastic Cauchy problem:
\begin{equation}
\label{eq:diffAB}
\begin{cases}
du(t) +A(t)u(t)dt=f(t) dt+ (B(t)u(t)+g(t))dW_H(t), \ \ t\in [0,T],\\
u(0)=u_0.
\end{cases}
\end{equation}

Next we give the definition of a strong solution.
\begin{definition}
\label{def:strong_linear}
Let $\tau$ be a stopping time which takes values in $[0,T]$. Let the Assumptions \ref{ass:Xtr}-\ref{ass:AB_boundedness} be satisfied. Assume that
$$u_0\in L^{0}_{\F_0}(\O;X_0),\quad f\in L^0_{\Progress}(\O;L^1(\I_{\tau};X_0)), \quad g\in L^0_{\Progress}(\O;L^2(\I_{\tau};\g(H,X_{0}))).$$
A strongly progressive process $u:\ll 0,\tau \rr\to X_1$ is a strong solution to \eqref{eq:diffAB} on $\ll 0,\tau \rr$ if a.s.\ $u\in L^2(\I_{\tau};X_1)$, and a.s.\ for all $t\in \I_{\tau}$,
\begin{equation}
\label{eq:linear_integral}
u(t)-u_{0}+\int_0^t A(s)u(s) ds= \int_0^t(B(s)u(s)+g(s))dW_H(s) +\int_0^t f(s) ds.
\end{equation}
\end{definition}
Note that a strong solution automatically satisfies $u\in L^0(\Omega;C([0,\tau];X_0))$.

We are ready to define weighted stochastic maximal $L^p$-regularity in a similar way as in \cite{VP18}.
\begin{definition}[Stochastic maximal $L^p$-regularity]
\label{def:SMRgeneralized}
Let the Assumptions \ref{ass:Xtr}-\ref{ass:AB_boundedness} be satisfied.
We write $(A,B)\in \MRtaz$ if for every $f\in L^p_{\Progress}(\I_T\times\O,w_{\a};X_0)$ and $g\in  L^p_{\Progress}(\I_T\times \O,w_{\a};\g(H,X_{1/2}))$ there exists a strong solution $u$ to \eqref{eq:diffAB} on $\ll 0,T\rr$ with $u_0=0$ such that $u\in L^p(\I_{T}\times\Omega,w_{\a};X_1)$, and moreover for all stopping times $\tau:\Omega\to [0,T]$ and any strong solution $u\in L^p(\I_{\tau}\times\Omega,w_{\a};X_1) $ the following estimate holds
\begin{align*}
\|u\|_{L^p(\I_\tau\times\Omega,w_{\a};X_1)} \leq C \|f\|_{L^p(\I_{\tau}\times \O,w_{\a};X_0)} 
+ C \|g\|_{L^p(\I_\tau \times\O,w_{\a};\g(H,X_{1/2}))},
\end{align*}
where $C$ is independent of $f$, $g$ and $\tau$.

In the unweighted case we set $\MRtz:=\MRtzeroz$. Furthermore, we write $A\in \MRtaz$ if $(A,0)\in \MRtaz$.
\end{definition}
As a consequence of the estimate in the above definition, a strong solution $u\in L^p(\I_{\tau}\times \O,w_{\a};X_1)$ on $\ll 0,\tau \rr$ to \eqref{eq:diffAB} is unique.

Often we will need the following stronger form of stochastic maximal $L^p$-regularity, where additional time-regularity is required. For technical reasons the definitions for $p>2$ and $p=2$ are different.
\begin{definition}\label{def:SMRz}
Let the Assumptions \ref{ass:Xtr}-\ref{ass:AB_boundedness} be satisfied.
\begin{enumerate}[$(1)$]
\item For $p>2$, we write $(A,B)\in \MRta$ if $(A,B)\in \MRtaz$ and for every $f\in L^p_{\Progress}(\I_T\times \O,w_{\a};X_0)$
and $g\in L^p_{\Progress}(\I_T\times\O,w_{\a};\g(H,X_{1/2}))$ the strong solution $u$ to \eqref{eq:diffAB} on $\ll 0,T\rr$ with $u_0=0$ satisfies $u\in L^p(\O;H^{\theta,p}(\I_T,w_{\a};X_{1-\theta}))$ for every $\theta\in [0,1/2)$, and
\begin{equation*}
\|u\|_{L^p(\O;H^{\theta,p}(\I_{T},w_{\a};X_{1-\theta}))}\leq C \|f\|_{L^p(\I_T\times\O,w_{\a};X_0)} 
+ C\|g\|_{L^p(\I_T\times \O,w_{\a};\g(H,X_{1/2}))},
\end{equation*}
where $C$ does not depend on $f$ and $g$.

\item We write $(A,B)\in \mathcal{SMR}^{\bullet}_{2,0}(T)$ if $(A,B)\in \mathcal{SMR}_{2,0}(T)$ and for every $f\in L^2_{\Progress}(\I_T\times\O;X_0)$
and $g\in L^2_{\Progress}(\I_T\times\O;\g(H,X_{1/2}))$ the strong solution $u$ to \eqref{eq:diffAB} with $u_0=0$ satisfies $u\in L^2(\O;C(\overline{I}_T;X_{\frac{1}{2}}))$ and
\begin{equation*}
\|u\|_{L^2(\O;C(\overline{I}_T;X_{\frac{1}{2}}))}\leq C \|f\|_{L^2(\I_T\times\O;X_0)} + C\|g\|_{L^2(\I_T\times\O;\g(H,X_{1/2}))},
\end{equation*}
where $C$ does not depend on $f$ and $g$.
\end{enumerate}
In the unweighted case we set $\MRt:=\MRtzero$. Furthermore, we write $A\in \MRta$ if $(A,0)\in \MRta$.
\end{definition}
Although we allow $\theta = \frac{1+\a}{p}$ in the above definition, later on we will omit this case since some technical difficulties arise related to Theorem \ref{t:equivalence_h_H}.

In the next section we give examples of pairs $(A,B)$ which are in $\MRta$.

\subsection{Operators with stochastic maximal $L^p$-regularity}
\label{ss:operators_with_SMR}
There exists an extensive list of examples on stochastic maximal $L^p$-regularity and in this section we review a selection. We will only consider maximal $L^p$-regularity in the Bessel-potential scale.

The case Hilbert space case for $\MRtaz$ was first studied by several different methods for $p=2$ and $\a=0$. We refer to the following papers for more detailed information.
\begin{itemize}
\item \cite[Theorem 6.14]{DPZ} the semigroup approach under restrictions on the interpolation spaces.
\item \cite{LiuRock} the monotone operators approach, where $A$ and $B$ not even need to be linear.
\item \cite{KryW22theory} $W^{k,2}$-theory on domains with weights.
\end{itemize}
In some cases one can even obtain that the operator is in $\mathcal{SMR}^{\bullet}_{2}(T)$. For instance this holds if $A$ is the generator of a $C_0$-semigroup on $X_{\frac12}$ which has a dilation to a $C_0$-group (see \cite{HS01}). In particular, this holds if the semigroup is quasi-contractive $\|e^{-tA}\|_{\calL(X_{\frac12})}\leq e^{t\omega}$ or $A$ has a bounded $H^\infty$-calculus of angle $<\pi/2$ on $X_0$ (see \cite[Theorem 11.13]{KuWe}).

In the setting $X_0 = H^{s,p}$ the stochastic maximal regularity of the form $\MRtaz$ has been obtained mostly for second order elliptic operators starting in \cite{Kry96, Kry, Kry00} in the $\R^d$-case in what is usually called Krylov's $L^p$-theory for SPDEs. It was afterwards extended to domains:
\begin{example}\label{ex:KrylovLp}
\
\begin{itemize}
\item \cite{CioicaKimLeeLindner18} and \cite{LoVer} heat equation on an angular domain with weights;
\item \cite{CioicaKimLee18} heat equation on polygonal domains with weights;
\item \cite{Du2018} $C^2$-domains no weights;
\item \cite{Kim04a,Kim04b,Kim05} $C^1$-domains with weights;
\item \cite{KryLot} half space case with weights;
\end{itemize}
and  second order systems:
\begin{itemize}
\item \cite{KimLeesystems} second order systems with $B$ of special form;
\item \cite{MiRo01} second order systems with $B$ of special form.
\end{itemize}
\end{example}

The stronger form of stochastic maximal regularity $\mathcal{SMR}^{\bullet}_{p}(T)$ was proved in \cite{MaximalLpregularity} for $B = 0$ and $A$ independent of $(t,\omega)$ using the $H^\infty$-calculus. Combined with a perturbation argument, the case $\a\in [0,\frac{p}{2}-1)$ was obtained in \cite[Section 7]{AV19}.
\begin{theorem}
\label{t:H_infinite_SMR}
Let Assumption \ref{ass:Xtr} be satisfied. Let $X_0$ be isomorphic to a closed subspace of an $L^q$-space for some $q\in [2,\infty)$ on a $\sigma$-finite measure space.
Let $A$ be a closed operator on $X_0$ such that $\Do(A)=X_1$.
Assume that there exists a $\lambda\in \R$ such that $\lambda+ A$ has a bounded $H^\infty$-calculus of angle $<\pi/2$. Then $A\in \MRta$ for all $T<\infty$. Furthermore, if $A$ is invertible and $\lambda=0$, then the result extends to $T= \infty$.
\end{theorem}

In particular, this result can be combined with the examples listed in Example \ref{ex:Hinfty}.

In \cite{NVW11eq} $\MRta$ was obtained for regular time dependent $A$ for small $B$ using perturbation arguments. By combining ideas from Krylov's $L^p$-theory and the semigroup approach of \cite{MaximalLpregularity} this was improved in \cite{VP18} to a large class of abstract operators $(A,B)$ as in Assumption \ref{ass:AB_boundedness} and where no time-regularity is assumed. In particular, it applies to second order systems with $B\neq 0$, and higher order systems with small $B\neq 0$ and in particular improves \cite{Kry96, Kry, Kry00} and \cite{KimLeesystems}. We will come back to those examples in later sections.

By definition $\MRta\subseteq \MRtaz$. The following somewhat surprising result states that $\MRta\neq \emptyset$ is a necessary and sufficient condition for the reverse inclusion to hold. Usually the non-emptyness can be checked with Theorem \ref{t:H_infinite_SMR} by showing that there is some operator $\tA$ on $X_0$ with $\Do(\tA) = X_1$ and which has a bounded $H^{\infty}$-calculus of angle $<\pi/2$.

\begin{proposition}[Transference of stochastic maximal regularity]
\label{prop:transferenceSMR}
Let the Assumptions \ref{ass:Xtr}-\ref{ass:AB_boundedness} be satisfied.
Let $(A,B)\in \MRtaz$ and assume the existence of a couple $(\tA,\tB)$ which satisfies Assumption \ref{ass:AB_boundedness} and belongs to $\MRta$. Then $(A,B)\in \MRta$.
\end{proposition}

\begin{proof}
Let us analyse the case $p>2$. The other case follows in the same way.
By Definition \ref{def:SMRz} we have to prove that for any $f\in L^p_{\Progress}(\I_T	\times\O,w_{\a};X_0)$, $g\in L^p_{\Progress}(\I_T\times\O,w_{\a};\g(H,X_{1/2}))$ and $\theta\in [0,1/2)$ the unique strong solution $u\in L^p_{\Progress}( \I_T\times \O,w_{\a};X_1)$ to \eqref{eq:diffAB} on $\ll 0,T\rr$ with $u_0 =0$ verifies
$$
u\in L^p(\O;H^{\theta,p}(I_T,w_{\a};X_{1-\theta})).
$$
To this end, note that
\begin{equation*}
\begin{cases}
du + \tA u dt= \tB udW_H+ ((\tA-A)u+f) dt+ ((B-\tB)u+g)dW_H,\quad t\in [0,T],\\
u(0)  =0.
\end{cases}
\end{equation*}
Fix $\theta\in [0,1/2)$. Since $u\in L^p_{\Progress}( \I_T\times \O,w_{\a};X_1)$ and $(\tA,\tB)\in \MRta$, one has
\begin{align*}
&\|u \|_{L^p(\O;H^{\theta,p}(\I_{T},w_{\a};X_{1-\theta}))} \\
& \lesssim \|(\tA-A)u+f\|_{L^p(\I_T\times\O,w_{\a};X_0)} + \|(B-\tB)u+g\|_{L^p(\I_T\times\O,w_{\a};\g(H,X_{1/2}))}
\\ & \stackrel{(i)}{\lesssim} \|u\|_{L^p(\I_T\times\Omega,w_{\a};X_1)} + \|f\|_{L^p(\I_T\times\O,w_{\a};X_0)}
+\|g\|_{L^p(\I_T\times\O,w_{\a};\g(H,X_{1/2}))}
\\ & \stackrel{(ii)}{\lesssim} \|f\|_{L^p(\I_T\times\O,w_{\a};X_0)} + \|g\|_{L^p(\I_T\times\O,w_{\a};\g(H,X_{1/2}))},
\end{align*}
where in $(i)$ we used Assumption \ref{ass:AB_boundedness} and in $(ii)$ we used $(A,B)\in \MRtaz$.
\end{proof}

\begin{remark}
\label{r:time_transference}
\
\begin{enumerate}[{\rm(1)}]
\item Proposition \ref{prop:transferenceSMR} is actually needed in the proof \cite[Theorem 3.18]{VP18} and it was overlooked.
The result can be used to deduce the stronger form of stochastic maximal $L^p$-regularity $\MRta$ also for some cases of the list in Example \ref{ex:KrylovLp}. In particular, this will play a role in later sections.

\item In \cite[Theorem 3.9]{VP18} there is another transference result which allows to deduce $A\in \MRta$ from maximal $L^p$-regularity for the deterministic problem (i.e.\ $g=0$, $B=0$) and $\tA\in \MRtaz$ for some family $\tA$. Moreover, in special cases it is shown that one can reduce to $B = 0$ in \cite[Theorem 3.18]{VP18}.

\item\label{it:uniformOmegaHinfty} Theorem \ref{t:H_infinite_SMR} also holds for operators $A:\Omega\to \calL(X_1, X_0)$ as long as the estimates for the $H^\infty$-calculus are uniform in $\Omega$.
\end{enumerate}
\end{remark}

To finish this subsection we mention that there are also perturbation results for $\MRta$ (see \cite[Theorem 3.15]{VP18} and \cite[Theorem 6.1]{AV19}). Other perturbation results will also be discussed in \cite{AV21_max_reg_torus} and \cite{AV19_QSEE_2}.

\subsection{Initial values and the solution operator}
The aim of this subsection is the study of the linear problem \eqref{eq:diffAB} with non-trivial initial data and to introduce some notations.
\begin{proposition}
\label{prop:initial_0}
Suppose Assumptions \ref{ass:Xtr}, and \ref{ass:AB_boundedness} hold. Let $(A,B)\in\MRtaz$. Then for any $u_0\in L^p_{\F_0}(\O;\Xap)$, $f\in L^p_{\Progress}(\I_{T}\times\Omega,w_{\a};X_0)$ and $g\in L^p_{\Progress}(\I_{T}\times\Omega,w_{\a};\g(H,X_{1/2}))$ there exists a unique strong solution $u\in L^p(\I_T\times \O, w_{\a};X_1)$ to \eqref{eq:diffAB} on $\ll 0,T\rr$ and
\begin{equation}\label{eq:SMRu0}
\begin{aligned}
\|u\|_{L^p(\I_T\times\Omega,w_{\a};X_1)} & \leq  C \|f\|_{L^p(\I_T\times\O,w_{\a};X_0)} \\ &  \qquad + C\|g\|_{L^p(\I_T\times\O,w_{\a};\g(H,X_{1/2}))} + C\|u_0\|_{L^p(\Omega;\Xap)},
\end{aligned}
\end{equation}
where $C$ is independent of $f$, $g$ and $u_0$.

If in addition $(A,B)\in \MRta$, then for all $\theta\in [0,1/2)$ the left-hand side of \eqref{eq:SMRu0} can be replaced by
$\|u\|_{L^p(\O;H^{\theta,p}(\I_{T},w_{\a};X_{1-\theta}))}$ if $p>2$ with $C$ additionally depending on $\theta$, and replaced by $\|u\|_{L^p(\O;C(\overline{\I}_{T};X_{1/2}))}$ if $p=2$.
\end{proposition}

\begin{proof}
The proof is similar to \cite[Lemma 2.2]{ACFP}. For the reader's convenience, we include the details. In steps 1-3, we assume only that $(A,B)\in \MRtaz$.

\textit{Step 1: Uniqueness}. This follows from $(A,B)\in \MRtaz$ and Definition \ref{def:SMRgeneralized}.

\textit{Step 2: $u$ exists and \eqref{eq:SMRu0} holds provided $u_0$ is simple}. Recall that (see \cite[Theorem 3.12.2]{BeLo} or \cite[Theorem 1.8.2, p.\ 44]{Tri95}) the real interpolation space $\Xap$ can be characterized as the set of all $x\in X_0+X_1$ such that there exists $h\in W^{1,p}(\R_+,w_{\a};X_0)\cap L^p(\R_+,w_{\a};X_1)$ which satisfies $x=h(0)$. Moreover,
\begin{equation}
\label{eq:real_interpolation_spaces_by_maximal_regularity}
\|x\|_{\Xap}\eqsim \inf\{\|h\|_{W^{1,p}(\R_+,w_{\a};X_0)\cap L^p(\R_+,w_{\a};X_1)}\,:\,h(0)=x\}.
\end{equation}
Let $u_0\in L^p_{\F_0}(\O;\Xap)$ be simple. By \eqref{eq:real_interpolation_spaces_by_maximal_regularity} applied pointwise w.r.t. $\om\in \O$, one can check that there exists a simple map $h\in L^p_{\F_0}(\O;W^{1,p}(\R_+,w_{\a};X_0)\cap L^p(\R_+,w_{\a};X_1))$ such that
\begin{equation}
\label{eq:estimate_h_u_0}
\|h\|_{L^p(\O;W^{1,p}(\R_+,w_{\a};X_0)\cap L^p(\R_+,w_{\a};X_1))}\lesssim \|u_0\|_{L^p(\O;\Xap)},
\end{equation}
where the implicit constant does not depend on $u_0$. Set $u:=h+v$. Then $u$ is a strong solution to \eqref{eq:diffAB} on $\ll 0,T\rr$ if and only if $v$ is a strong solution on $\ll 0,T\rr$ to
\begin{equation}
\label{eq:v_initial_u_0_proof}
\begin{cases}
dv +A(t)v dt=(f-\dot{h}-A(t)h)dt +(B(t)v+B(t)h+g)dW_{H}, \;t\in \I_T,\\
v(0)=0.
\end{cases}
\end{equation}
By \eqref{eq:estimate_h_u_0} and the fact that $(A,B)\in \MRtaz$, \eqref{eq:SMRu0} follows.

\textit{Step 3: $u$ exists and \eqref{eq:SMRu0} holds for all $u_0\in L^p_{\F_0}(\O;\Xap)$}. By \cite[Lemma 1.2.19]{Analysis1}, there exists a uniformly bounded sequence of simple maps $(u_{0,n})_{n\geq 1}\subseteq L^p_{\F_0}(\O;\Xap)$ such that $u_{0,n}\to u_0$ in $L^p_{\F_0}(\O;\Xap)$. Thus, the conclusion follows from Step 2 and the completeness of $L^p_{\Progress}(\I_T\times \O,w_{\a};X_1)$.

\textit{Step 4: The last claim holds}. Similarly to Step 3, it is enough to consider $u_0$ simple. Thus, as in Step 2, there exists $h\in L^p_{\F_0}(\O; W^{1,p}(\R_+,w_{\a};X_0)\cap L^p(\R_+,w_{\a};X_1))$ such that \eqref{eq:real_interpolation_spaces_by_maximal_regularity} holds. Then by Proposition \ref{prop:Sob_interpolation_h} and the fact that $(A,B)\in \MRta$, the claim follows by writing $u=h+v$ where $v$ solves \eqref{eq:v_initial_u_0_proof}.
\end{proof}

\begin{remark}
Under the assumption that $X_1=\Do(\tA)$, for a sectorial operator $\tA$ on $X_0$ with angle $\om(\tA)<\pi/2$, the proof of Proposition \ref{prop:initial_0} simplifies. See step 0 in \cite[Theorem 3.15]{VP18}. This type of assumption is satisfied in all the applications which will be presented in Sections \ref{s:semilinear_gradient}-\ref{s:AC_CH}.
\end{remark}

Next we will define certain solution operators which will be used in Section \ref{s:quasi}. Suppose $(A,B)\in \MRta$ and that Assumptions \ref{ass:Xtr}-\ref{ass:AB_boundedness} hold. Using Proposition \ref{prop:initial_0} for $p>2$ we can define $\Sol_{(A,B)}(u_0,f, g)=u$, where $u$ is the strong solution to \eqref{eq:diffAB} as a mapping from
\[L^p_{\F_0}(\O;\Xap)\times L^p_{\Progress}(\I_{T}\times\Omega,w_{\a};X_0)\times L^p_{\Progress}(\I_{T}\times\Omega,w_{\a};\g(H,X_{1/2}))\]
into
\[\bigcap_{\theta\in [0,1/2)} L^p(\O;H^{\theta,p}(\I_{T},w_{\a};X_{1-\theta})).\]
By linearity, we can write
\[\Sol_{(A,B)}(u_0,f, g) = \Sol_{(A,B)}(u_0,0, 0) + \Sol_{(A,B)}(0,f,0) + \Sol_{(A,B)}(0,0, g).\]

Note that
$\Sol_{(A,B)}(0, \cdot, \cdot)$ actually maps into $L^p(\O;\hz^{\theta,p}(\I_{T},w_{\a};X_{1-\theta})$ for any $\theta\in [0,\frac{1}{2})\setminus\{\frac{1+\a}{p}\}$. Indeed, this follows from $u(0) = 0$ in $X_0$, Theorem \ref{t:equivalence_h_H} and the text below it.

For later use, in the case $p>2$ and $\theta\in [0,\frac{1}{2})\setminus\{\frac{1+\a}{p}\}$, we define
\begin{equation}
\label{eq:constants_SMR_C}
\begin{aligned}
C_{(A,B)}^{\det,\theta} & = \|\Sol_{(A,B)}(0,\cdot,0)\|_{L^p_{\Progress}(\I_{T}\times\Omega,w_{\a};X_0))\to L^p(\O;\hz^{\theta,p}(\I_{T},w_{\a};X_{1-\theta}))},
\\
 C_{(A,B)}^{\stoc,\theta} & = \|\Sol_{(A,B)}(0,0,\cdot)\|_{L^p_{\Progress}( \I_{T}\times\Omega,w_{\a};\g(H,X_{1/2}))\to L^p(\O;\hz^{\theta,p}(\I_{T},w_{\a};X_{1-\theta}))}.
\end{aligned}
\end{equation}
In the case $p=2$ and $\theta\in (0,1/2)$, we replace the range space by $L^p(\O;C(\overline{\I}_{T};X_{1/2}))$ (which is constant in $\theta\in (0,1/2)$). Moreover, for $\theta\in [0,\frac{1}{2})\setminus \{\frac{1+\a}{p}\}$ we set
\begin{equation}
\label{eq:constants_SMR}
K_{(A,B)}^{\det,\theta}:=C_{(A,B)}^{\det,\theta}+C_{(A,B)}^{\det,0}, \qquad
K_{(A,B)}^{\stoc,\theta}:=C_{(A,B)}^{\stoc,\theta}+C_{(A,B)}^{\stoc,0}.
\end{equation}

In the next proposition we collect some simple properties of the solution operator $\Sol_{(A,B)}$.
\begin{proposition}
\label{prop:propertiesR}
Suppose Assumptions \ref{ass:Xtr}-\ref{ass:AB_boundedness} hold.
Let $(A,B)\in \MRtaz$ and let $\Sol:=\Sol_{(A,B)}$. Let $u_0\in L^p_{\F_0}(\O;\Xap)$, $f\in L^p_{\Progress}(\I_{T}\times\Omega,w_{\a};X_0)$, $g\in L^p_{\Progress}(\I_{T}\times\Omega,w_{\a};\g(H,X_{1/2}))$ and set $u:=\Sol(u_0,f,g)$. Then the following assertions hold
\begin{enumerate}[$(1)$]
\item\label{it:FF0} For each $F\in \F_0$,
\begin{equation*}
\one_{F}\Sol(u_0,f,g)=\Sol(\one_{F}u_0,\one_{F}f,\one_{F}g)=\one_F\Sol(\one_{F}u_0,\one_{F}f,\one_{F}g).
\end{equation*}
\item\label{it:stoppedsol} Assume that $v\in L^p_{\Progress}(\ll 0,\sigma\rr,w_{\a};X_1)$ is a strong solution to \eqref{eq:diffAB} on $\ll 0,\sigma\rr$, where $\sigma$ is a stopping time. Then
\begin{equation*}
v=u|_{\ll 0,\sigma\rr}=\Sol(u_0,\one_{\ll 0,\sigma\rr}f,\one_{\ll 0,\sigma\rr}g), \qquad \text{on }\ll 0,\sigma\rr.
\end{equation*}
\item\label{it:stoppedest} For all $T_1\leq T$, the following estimates on the maximal regularity constants hold
    \[K^{\deter,\theta}_{(A|_{\ll 0,T_1\rr},B|_{\ll 0,T_1\rr})}\leq K^{\deter,\theta}_{(A,B)}  \ \ \ \text{and} \ \ \  K^{\stoc,\theta}_{(A|_{\ll 0,T_1\rr},B|_{\ll 0,T_1\rr})}\leq K^{\stoc,\theta}_{(A,B)}.
    \]
\end{enumerate}
\end{proposition}

\begin{proof}
\eqref{it:FF0}: By Definition \ref{def:strong_linear}, $u$ verifies \eqref{eq:linear_integral}. It follows that
$v:=\one_{F}u$ satisfies
\begin{equation*}
v(t)-\one_{F}u_{0}+\int_0^t A(s)(\one_{F}u(s)) ds= \int_0^t(B(s)(v(s))+\one_{F}g(s))dW_H +\int_0^t \one_{F}f(s) ds.
\end{equation*}
By uniqueness we obtain $v=\Sol(\one_{F}u_0,\one_{F}f,\one_{F}g)$. This proves the first identity.
The second identity follows from the first identity and $\one_F^2=\one_F$.

\eqref{it:stoppedsol}: From Definition \ref{def:strong_linear} we immediately see that
$u|_{\ll 0,\sigma\rr}$ is a strong solution on $\ll 0,\sigma\rr$. By uniqueness, this implies $v = u|_{\ll 0,\sigma\rr}$. Thus, a.s.\ for all $t\in [0,\sigma]$,
\[u(t) -u_{0} +\int_0^{t} A(s) u(s) ds= \int_0^t(B(s)u(s)+g(s))dW_H(s) +\int_0^t f(s) ds.
\]
On the other hand, $\tilde{u} := \Sol(u_0,\one_{\ll 0,\sigma\rr}f,\one_{\ll 0,\sigma\rr}g)$ satisfies a.s.\ for all $t\in [0,\sigma]$,
\begin{align*}
\tilde{u}(t) -u_{0} +\int_0^{t} & A(s) \tilde{u}(s)  ds  = \int_0^t(B(s)\tilde{u}(s)+\one_{\ll 0,\sigma \rr} g(s))dW_H(s) +\int_0^t \one_{\ll 0,\sigma \rr} f(s) ds
\\ & = \int_0^t B(s)\tilde{u}(s)dW_H(s) +\int_0^{t\wedge \sigma} g(s) dW_H(s) +\int_0^{t\wedge \sigma} f(s) ds
\\ & = \int_0^t B(s)\tilde{u}(s)dW_H(s) +\int_0^{t} g(s) dW_H(s) +\int_0^{t} f(s) ds.
\end{align*}
Therefore, again by uniqueness $\tilde{u} = v$.

\eqref{it:stoppedest}: This is immediate from \eqref{it:stoppedsol} and Proposition \ref{prop:eliminate_weights}.
\end{proof}

We end this section with a lemma which can be extracted from the proof of \cite[Theorem 3.15 Step 1]{VP18}.
For the reader's convenience we sketch the proof.

\begin{lemma}
\label{l:solution_op_c_T}
Let Assumption \ref{ass:Xtr} be satisfied, and suppose that $(A,B)$ satisfies Assumption \ref{ass:AB_boundedness}. Then for each $s\in \I_T$ there exists a constant $c_s>0$ such that $\lim_{s\downarrow 0} c_s=0$ and for all $\tau$ stopping time satisfying $0\leq \tau\leq s$ a.s.\ and any $f\in L^p_{\Progress}(\I_{s}\times \O,w_{\a};X_0)$, $g\in L^p_{\Progress}(\I_{s}\times \O,w_{\a};\g(H,X_{1/2}))$ and any strong solution $u\in L^p_{\Progress}(\I_{\tau}\times \O ,w_{\a};X_1)$ to \eqref{eq:diffAB} on $\ll 0,\tau\rr$ with $u_0=0$ one has
\begin{align*}
\|u\|_{L^p(\I_{\tau}\times \O,w_{\a};X_0)}\leq
c_s\|u\|_{L^p(\I_{\tau}\times \O,w_{\a};X_1)}
&+c_s \|f\|_{L^p(\I_{\tau}\times \O,w_{\a};X_0)}\\
&+c_s\|g\|_{L^p(\I_{\tau}\times \O,w_{\a};\g(H,X_{1/2}))}.
\end{align*}
If additionally $(A,B)\in \MRtaz$, then $u=\Sol_{(A,B)}(0, f, g)$ a.e.\ on $\ll 0,\tau\rr$ and
$$
\|u\|_{L^p(\I_{\tau}\times \O,w_{\a};X_0)}\leq c_s \|f\|_{L^p(\I_{\tau}\times \O,w_{\a};X_0)}
+c_s\|g\|_{L^p(\I_{\tau}\times \O,w_{\a};\g(H,X_{1/2}))}.
$$
\end{lemma}

\begin{proof}
Let us begin by proving the first claim. Recall that $u(t)=\int_0^t (-A(r)u(r)+f(r))dr+\int_0^t (B(r)u(r)+g(r))dW_{H}(r)$ a.s.\ for each $t\in \I_{\tau}$. Let us set $v(t):=\int_0^t \one_{\ll 0,\tau\rr}(-A(r)u(r)+f(r))dr+\int_0^t \one_{\ll 0,\tau\rr} (B(r)u(r)+g(r))dW_{H}(r)$ a.s.\ for each $t\in [0,s]$. Note that $v=u$ a.e.\ on $\ll 0,\tau\rr$. By Proposition \ref{prop:Ito},
\begin{equation*}
\begin{aligned}
&\|v\|_{L^p(\O;C(\overline{\I}_s;X_0))}  \\
&\lesssim_{X,p} \|\one_{\ll 0,\tau\rr}(-A u + f)\|_{L^p(\Omega;L^1(\I_s;X_0))}
+ \|\one_{\ll 0,\tau\rr} (Bu + g)\|_{L^p(\Omega;L^2(\I_s;\gamma(H,X_0)))}
\\ &
\stackrel{(i)}{\lesssim}_{p,\a} k_s \big[ \|-Au+f\|_{L^p(\I_{\tau}\times \O,w_{\a};X_0)}
+ \|Bu+g\|_{L^p(\I_{\tau}\times \O,w_{\a};\g(H,X_0))}\big]
\\ &
\stackrel{(ii)}{\lesssim}_{A,B} k_s\big[\|u\|_{L^p(\I_{\tau}\times \O,w_{\a};X_1)}+
\|f\|_{L^p(\I_{\tau}\times \O,w_{\a};X_0)}+ \|g\|_{L^p(\I_{\tau}\times \O,w_{\a};\g(H,X_{1/2}))}\big],
\end{aligned}
\end{equation*}
where in $(i)$ we used H\"older's inequality and $\a\in [0,\frac{p}{2}-1)$, in $(ii)$ we used Assumptions \ref{ass:Xtr}-\ref{ass:AB_boundedness}. The constant $k_s$ satisfies $\lim_{s\downarrow 0} k_s =: k\in [0,\infty)$. Therefore,
\begin{align*}
\|v\|_{L^p(\I_s \times \O,w_{\a};X_0)}\leq k_s c_s
\big[\|u\|_{L^p(\I_{\tau}\times \O,w_{\a};X_1)}&+ \|f\|_{L^p(\I_{\tau}\times \O,w_{\a};X_0)}\\
&+\|g\|_{L^p(\I_{\tau}\times \O,w_{\a};\g(H,X_{1/2}))}\big],
\end{align*}
where $c_s>0$ satisfies $\lim_{s\downarrow 0} c_s=0$. Since $v=u$ a.e.\ on $\ll 0,\tau\rr$ and $\tau\leq s$ a.s., one has $\|u\|_{L^p(\I_{\tau} \times \O,w_{\a};X_0)}\leq \|v\|_{L^p(\I_s \times \O,w_{\a};X_0)}$, and thus the first estimate follows.

If $(A,B)\in \MRtaz$ and $u\in L^p_{\Progress}(\I_{\tau}\times \O,w_{\a};X_1)$ is a strong solution to \eqref{eq:diffAB} on $\ll 0,\tau\rr$, then by Proposition \ref{prop:propertiesR}\eqref{it:stoppedsol}, $u=\Sol_{(A,B)}(0,f,g)=\Sol_{(A,B)}(0,\one_{\ll 0,\tau\rr}f, \one_{\ll 0,\tau\rr} g)$ a.e.\ on $\ll 0,\tau\rr$. Thus
\begin{equation}
\label{eq:estimate_simple_u_X_1_smallness_constant}
\begin{aligned}
\|u\|_{L^p(\I_{\tau}\times \O,w_{\a};X_1)}
&\leq \|\Sol_{(A,B)}(0,\one_{\ll 0,\tau\rr}f, \one_{\ll 0,\tau\rr} g)\|_{L^p(\I_T\times \O,w_{\a},X_1)}\\
&\lesssim  \|f\|_{L^p(\I_{\tau}\times \O,w_{\a};X_1)}+\|g\|_{L^p(\I_{\tau}\times \O,w_{\a};\g(H,X_{1/2}))},
\end{aligned}
\end{equation}
and this implies the second estimate.
\end{proof}

\section{Local existence results}
\label{s:quasi}

In this section we consider the following nonlinear evolution equation
\begin{equation}
\label{eq:QSEE}
\begin{cases}
du + A(\cdot,u) u dt = (F(\cdot,u)+f)dt + (B(\cdot,u) u +G(\cdot,u)+g) dW_{H},\\
u(0)=u_0;
\end{cases}
\end{equation}
for $t\in [0,T]$ on a Banach space $X_0$ where $T<\infty$. Recall that Assumption \ref{ass:Xtr} holds throughout this section.

The equation \eqref{eq:QSEE} covers both the case of quasilinear and semilinear equations. In the quasilinear case the reader should have in mind that for each fixed $x\in \Xap$, the operators $A(t,x)$ and $B(t,x)$ satisfy the mapping properties of Assumption \ref{ass:AB_boundedness}. We refer to \ref{HAmeasur} below for the precise definitions. In the semilinear case $A(t,x)$ and $B(t,x)$ do not depend on $x$ and therefore are precisely as in Assumption \ref{ass:AB_boundedness}.

The structure of the nonlinearities $F$ and $G$ which will be assumed below is very flexible and extends many known results. Moreover, the structural conditions are satisfied by large classes of SPDE.

Compared to \cite{HornungDissertation,Hornung,NVW11eq} there are several important differences:
\begin{itemize}
\item we assume a joint condition on $(A,B)$ and therefore $B$ is not assumed to be small as one sometimes needs with the semigroup approach to SPDEs (see e.g.\ \cite{BrzVer11, Fla90});
\item the operators $A$ and $B$ are allowed to be time and $\Omega$-dependent in just a measurable way;
\item we allow weights in time, so that our initial values can be very rough;
\item we allow critical nonlinearities in the sense of \cite{PrussWeight2,CriticalQuasilinear,addendum}.
\end{itemize}

\subsection{Assumptions on the nonlinearities}
In this section we discuss the assumptions and the main results regarding \eqref{eq:QSEE}. Moreover, the definition of a strong solution to \eqref{eq:QSEE} is given in Definition \ref{def:solution1}-\ref{def:solution2} below.

Concerning the random operators $A,B$, the nonlinearities $F,G$, and the initial data, we make the following hypothesis.

\textbf{Hypothesis} $\Hip$\label{H:hip}.

\let\ALTERWERTA\theenumi
\let\ALTERWERTB\labelenumi
\def\theenumi{(HA)}
\def\labelenumi{(HA)}
\begin{enumerate}
\item\label{HAmeasur}
Assumption \ref{ass:Xtr} holds.
Let $A:[0,T]\times\Omega \times \Xap\to \calL(X_1,X_0)$ and $B:[0,T]\times\Omega\times \Xap \to \calL(X_1,\g(H,X_{1/2}))$. Assume that for all $x\in \Xap$ and $y\in X_1$, the maps $(t,\om)\mapsto A(t,\om,x)y$ and $(t,\om)\mapsto B(t,\om,x)y$ are strongly progressively measurable.

Moreover, for all $n\geq 1$, there exists $C_n,L_n\in \R_+$ such that for all $x,y\in \Xap$ with $\|x\|_{\Xap},\|y\|_{\Xap}\leq n$, $t\in [0,T]$, and a.a.\ $\om\in \O$,
\begin{align*}
\|A(t,\om,x)\|_{\calL(X_1,X_0)}&\leq C_n (1+\|x\|_{\Xap}),\\
\|B(t,\om,x)\|_{\calL(X_1,\g(H,X_{1/2}))}&\leq C_n(1+ \|x\|_{\Xap}),\\
\|A(t,\om,x)-A(t,\om,y)\|_{\calL(X_1,X_0)}&\leq L_n \|x-y\|_{\Xap},\\
\|B(t,\om,x)-B(t,\om,y)\|_{\calL(X_1,\g(H,X_{1/2}))}&\leq L_n \|x-y\|_{\Xap}.
\end{align*}
\end{enumerate}
\let\theenumi\ALTERWERTA
\let\labelenumi\ALTERWERTB

\let\ALTERWERTA\theenumi
\let\ALTERWERTB\labelenumi
\def\theenumi{(HF)}
\def\labelenumi{(HF)}
\begin{enumerate}
\item\label{HFcritical}  The map $F:[0,T]\times\Omega\times X_1 \to X_0$ decomposes as $F:=F_{L}+F_{c}+F_{\Tr}$ where for all $x\in X_1$ the map $(t,\om)\mapsto F_{\ell}(t,\om,x)$ is strongly progressively measurable for $\ell\in \{L,c,\Tr\}$. Moreover, $F_L,F_c,F_{\Tr}$ satisfy the following estimates.
\begin{enumerate}[(i)]
\item\label{it:F_L} There exist constants $L_{F},\tilde{L}_{F}, C_{F}>0$, such that for all $x,y\in X_1$, $t\in [0,T]$ and a.a.\ $\om \in \O$,
\begin{align*}
\|F_L(t,\om,x)\|_{X_0}&\leq C_{F}(1+ \|x\|_{X_1}),\\
\|F_L(t,\om,x)-F_L(t,\om,y)\|_{X_0}&\leq L_{F} \|x-y\|_{X_1}+ \tilde{L}_{F} \|x-y\|_{X_0}.
\end{align*}
\item\label{it:F_c} There exist $\mf\geq 1$, $\varphi_j\in (1-(1+\k)/p,1)$, $\beta_j\in (1-(1+\k)/p,\varphi_j]$, $\rho_j\geq 0$ for $j\in \{1,\dots,\mf\}$ such that
$F_c:[0,T]\times \O\times X_{1}\to X_0$ and for each $n\geq 1$ there exist $C_{c,n},L_{c,n}\in \R_+$ for which
\begin{align*}
\|F_c(t,\om,x)\|_{X_{0}}&\leq C_{c,n} \sum_{j=1}^{\mf}(1+\|x\|_{X_{\varphi_j}}^{\rho_j})\|x\|_{X_{\beta_j}}+C_{c,n},\\
\|F_c(t,\om,x)-F_c(t,\om,y)\|_{X_{0}}&\leq L_{c,n} \sum_{j=1}^{m_F}(1+\|x\|_{X_{\varphi_j}}^{\rho_j}
+\|y\|_{X_{\varphi_j}}^{\rho_j})\|x-y\|_{X_{\beta_j}},
\end{align*}
a.s.\ for all $x,y\in X_{1}$, $t\in [0,T]$ such that $\|x\|_{\Xap},\|y\|_{\Xap}\leq n$. Moreover, $\rho_j,\varphi_j,\beta_j,\k$ satisfy
\begin{equation}
\label{eq:HypCritical}
\rho_j \Big(\varphi_j-1 +\frac{1+\k}{p}\Big) + \beta_j \leq 1, \qquad j\in \{1,\dots,m_{F}\}.
\end{equation}
\item For each $n\geq 1$ there exist $L_{\Tr,n},C_{\Tr,n}\in \R_+$ such that the mapping $F_{\Tr}:[0,T]\times\Omega\times\Xap\to X_0$ satisfies
\begin{align*}
\|F_{\Tr}(t,\om,x)\|_{X_0}&\leq C_{\Tr,n}(1+\|x\|_{\Xap}),\\
\|F_{\Tr}(t,\om,x)-F_{\Tr}(t,\om,y)\|_{X_0}&\leq L_{\Tr,n}\|x-y\|_{\Xap},
\end{align*}
for a.a.\ $\omega\in \O$, for all $t\in [0,T]$ and $\|x\|_{\Xap},\|y\|_{\Xap}\leq n$.
\end{enumerate}

\end{enumerate}
\let\theenumi\ALTERWERTA
\let\labelenumi\ALTERWERTB

\let\ALTERWERTA\theenumi
\let\ALTERWERTB\labelenumi
\def\theenumi{(HG)}
\def\labelenumi{(HG)}
\begin{enumerate}
\item\label{HGcritical} The map $G:[0,T]\times\Omega\times X_1 \to \g(H,X_{1/2})$ decomposes as $G:=G_{L}+G_{c}+G_{\Tr}$ where for all $x\in X_1$ the map $(t,\om)\mapsto G_{\ell}(t,\om,x)$ is strongly progressively measurable for $\ell\in \{L,c,\Tr\}$. Moreover, $G_L,G_c,G_{\Tr}$ satisfy the following estimates.
\begin{enumerate}[(i)]
\item\label{it:G_L} There exist constants $L_{G},\tilde{L}_{G}, C_{G}$, such that for all $x,y\in X_1$, $t\in [0,T]$ and a.a.\ $\om \in \O$,
\begin{align*}
\|G_L(t,\om,x)\|_{\g(H,X_{1/2})}&\leq C_{G}(1+ \|x\|_{X_1}),\\
\|G_L(t,\om,x)-G_L(t,\om,y)\|_{\g(H,X_{1/2})}&\leq L_{G} \|x-y\|_{X_1}+ \tilde{L}_{G} \|x-y\|_{X_0}.
\end{align*}
\item\label{it:G_c} There exist $\mg\geq 1$, $\varphi_j\in (1-(1+\k)/p,1)$, $\beta_j\in (1-(1+\k)/p,\varphi_j]$, $\rho_j\geq  0$  for $j\in \{m_F+1,\dots,m_F+m_G\}$  such that
$G_c:[0,T]\times \O\times X_{1}\to X_0$ and for each $n\geq 1$ there exist $C_{c,n},L_{L_c,n}\in \R_+$ for which
\begin{align*}
\|G_c(t,\om,x)\|_{\g(H,X_{1/2})}&\leq C_{c,n}
\sum_{j=\mf+1}^{\mf+\mg}(1+\|x\|_{X_{\varphi_j}}^{\rho_j})\|x\|_{X_{\beta_j}}+C_{c,n},\\
\|G_c(t,\om,x)-G_c(t,\om,y)\|_{\g(H,X_{1/2})}&\leq
L_{c,n}\sum_{j=\mf+1}^{\mf+\mg}
 (1+\|x\|_{X_{\varphi_j}}^{\rho_j}+\|y\|_{X_{\varphi_j}}^{\rho_j})\|x-y\|_{X_{\beta_j}},
\end{align*}
a.s.\ for all $x,y\in X_{1}$, $t\in [0,T]$ such that $\|x\|_{\Xap},\|y\|_{\Xap}\leq n$.
Moreover, $\varphi_j,\beta_j,\k$ satisfy
\begin{equation}
\label{eq:HypCriticalG}
\rho_j \Big(\varphi_j-1 +\frac{1+\k}{p}\Big) + \beta_j \leq 1, \qquad j\in \{\mf+1,\dots,\mf+\mg\}.
\end{equation}
\item For each $n\geq 1$ there exist constants $L_{\Tr,n},C_{\Tr,n}$ such that mapping $G_{\Tr}:[0,T]\times\Omega\times\Xap\to X_0$ satisfies
\begin{align*}
\|G_{\Tr}(t,\om,x)\|_{\g(H,X_{1/2})}&\leq C_{\Tr,n}(1+\|x\|_{\Xap}),\\
\|G_{\Tr}(t,\om,x)-G_{\Tr}(t,\om,y)\|_{\g(H,X_{1/2})}&\leq L_{\Tr,n}\|x-y\|_{\Xap},
\end{align*}
for a.a.\ $\omega\in \O$, for all $t\in [0,T]$ and $\|x\|_{\Xap},\|y\|_{\Xap}\leq n$.
\end{enumerate}

\end{enumerate}
\let\theenumi\ALTERWERTA
\let\labelenumi\ALTERWERTB

\let\ALTERWERTA\theenumi
\let\ALTERWERTB\labelenumi
\def\theenumi{(Hf)}
\def\labelenumi{(Hf)}
\begin{enumerate}
\item\label{Hf}
$f\in L^p_{\Progress}(\I_T\times\Omega,w_{\a};X_0)$ and $g\in L^p_{\Progress}(\I_T\times\Omega,w_{\a};\g(H,X_{1/2}))$.
\end{enumerate}
\let\theenumi\ALTERWERTA
\let\labelenumi\ALTERWERTB
The relations \eqref{eq:HypCritical}-\eqref{eq:HypCriticalG} will play an important role in the analysis of \eqref{eq:QSEE}. As announed in Subsection \ref{ss:critical_spaces_introduction}, following \cite{CriticalQuasilinear}, we may give an abstract definition of critical space for \eqref{eq:QSEE}.

The space $\Xap$ will be called a {\em critical space} for \eqref{eq:QSEE} if for some $j\in \{1,\dots,m_F+m_G\}$ equality in \eqref{eq:HypCritical} or \eqref{eq:HypCriticalG} holds. Moreover, the value of $\a$ for which equality in \eqref{eq:HypCritical} or \eqref{eq:HypCriticalG} holds, will be called the {\em critical weight} and it will be denoted by $\a_{\crit}$.
Some remarks may be in order.
\begin{remark}
Let us note that in Theorem \ref{t:local} and \ref{t:local_Extended} below, only the constants $L_{F},L_{G}$ are assumed to be small. The other constants are arbitrary.
At first sight the splitting of the nonlinearities $F$ and $G$ in several parts seems quite complicated. Let us emphasise that the most important part is $F_c$ and $G_c$ as these will usually determine {\em critical spaces} as defined above. The flexibility in the form we choose the nonlinearities is quite important in application to SPDEs. It will allow us in many cases to find a broad class of initial value spaces for which the SPDE can be solved. Let us note that usually it is enough to take $m_F = m_G = 1$.
\end{remark}

\begin{remark}
\label{r:non_linearities}
Below we collect some observations which will be used later on to check \ref{HFcritical} or \ref{HGcritical}. We discuss this only for $F$ since the same arguments apply to $G$.
\begin{enumerate}[{\rm(1)}]
\item\label{it:non_linearities_continuous_trace} If $F:\I_T\times \O\times X_{\theta}\to X_0$, for some $\theta<1-(1+\a)/p$, is locally Lipschitz on $X_{\theta}$ uniformly w.r.t. $(t,\om)\in \I_T\times \O$, then $F$ verifies \ref{HFcritical}. Indeed, it is enough to recall that
$$
\Xap= (X_0,X_1)_{1-\frac{1+\a}{p},p} \hookrightarrow (X_0,X_1)_{\theta,1} \hookrightarrow X_{\theta},
$$
where in the first inclusion we used \cite[Theorem 3.4.1]{BeLo} since $X_1\hookrightarrow X_0$, and in the last inclusion \cite[Theorem 3.9.1]{BeLo}. Then the conclusion follows by setting $F_{\Tr}:=F$, $F_{c}=F_{L}=0$. As soon as $\theta$ is larger, then we need a nonzero $F_c$ as the situation is more sophisticated.
\item\label{it:non_linearities_varphi_equal_to_beta} We can additionally allow the case $\beta_j=\varphi_j = 1-(1+\a)/p$ for all $j\in \{1,\dots,m_F\}$ in \ref{HFcritical}. Indeed, $\rho_j (\varphi_j+\varepsilon-1 +\frac{1+\k}{p}) + \beta_j = \rho_j \varepsilon + \beta_j$. Thus, there exists $\varepsilon>0$ such that $\rho_j\varepsilon+\beta_j<1$ for all $j\in \{1,\dots,m_F+m_G\}$. Since $X_{1-(1+\a)/p+\varepsilon}\hookrightarrow X_{1-(1+\a)/p}$, we may replace $\varphi_j=\beta_j$ by $1-(1+\a)/p+\varepsilon$ obtaining that $F_c$ satisfies \ref{HFcritical} and \eqref{eq:HypCritical} holds with strict inequality.
\item\label{it:critical_beta_grater_than} Assume that $\beta_j=\varphi_j<1$ for some $j\in \{1,\dots,m_F\}$ and that equality in \eqref{eq:HypCritical} holds. Then $\rho_j>0$ and thus $\varphi_j>1-(1+\a)/p$ holds since $\varphi_j-1+(1+\a)/p = (1-\beta_j)/\rho_j>0$. Therefore, in applications we do not need to check $\beta_j>1-(1+\a)/p$ if equality in \eqref{eq:HypCritical} holds (e.g.\ in the critical case).
\end{enumerate}
\end{remark}

Next we define $L^p_{\a}$-strong solutions to \eqref{eq:QSEE}.
Here we add the prefix $L^p_{\a}$ since the definition depends on $(p,\a)$.

\begin{definition}[$L^{p}_{\a}$-strong solutions]
\label{def:solution1}
Let the Hypothesis \hyperref[H:hip]{$\Hip$} be satisfied and let $\sigma$ be a stopping time with $0\leq \sigma\leq T$. A strongly progressively measurable process $u$ on $\ll 0,\sigma\rr$
satisfying
$$u\in L^p(\I_{\sigma},w_{\a};X_1)\cap C(\overline{\I}_{\sigma};\Xap) \ \ a.s.$$
is called an $L^p_{\a}$-\emph{strong solution} to \eqref{eq:QSEE} on $\ll 0,\sigma\rr$ if $F(\cdot,u)\in L^p(\I_{\sigma},w_{\a};X_0)$, $G(\cdot,u)\in L^p(\I_{\sigma},w_{\a};\g(H,X_{1/2}))$ a.s., and the following identity holds a.s.\ and for all $t\in[0,\sigma]$,
\begin{equation}
\label{eq:identity_sol}
\begin{aligned}
u(t)-u_0 &+\int_{0}^{t} A(s,u(s))u(s)ds=\int_{0}^{t} F(s,u(s))+f(s)ds \\
&+ \int_{0}^t\one_{\ll 0,\sigma\rr} (B(s,u(s))u(s)+G(s,u(s))+g(s))\,dW_H(s).
\end{aligned}
\end{equation}
\end{definition}
Note that if $u$ is an $L^p_{\a}$-strong solution, then the integrals appearing in \eqref{eq:identity_sol} are well-defined. To see this, note that $s\mapsto A(s,u(s))u(s)$ and $s\mapsto B(s,u(s))u(s)$  are strongly progressively measurable by the conditions on $u$ and \ref{HAmeasur} (see \cite[Lemma 4.51]{AliBor}). Moreover, pointwise in $\Omega$ we can take $\N\ni n\geq  \|u\|_{C(\overline{\I}_{\sigma};\Xap)}$ and write
\[\|A(s,u(s))u(s)\|_{X_0}\leq C_n (1+\|u(s)\|_{\Xap}) \|u(s)\|_{X_1}\leq C_n(1+n) \|u(s)\|_{X_1}.\]
Integrating over $s\in [0,\sigma]$ we obtain
\begin{align*}
\|s\mapsto A(s,u(s))u(s)\|_{L^1(0,\sigma;X_0)}  \leq C_n(1+n) \|u\|_{L^1(0,\sigma;X_1)},
\end{align*}
and the latter is finite since $u\in L^p(\I_{\sigma},w_{\a};X_1)$ a.s. Thus, the integral on the left-hand side of \eqref{eq:identity_sol} is well-defined. In the same way one can check that $s\mapsto B(s,u(s))u(s)$ is in $L^2(0,\sigma;X_{1/2})$ and the first stochastic integral on the right-hand side of \eqref{eq:identity_sol} is well-defined by Proposition \ref{prop:Ito}. Using the above argument and that $f,F(\cdot,u)\in L^p(\I_{\sigma},w_{\a};X_0)$ a.s.\ and $g,G(\cdot,u)\in L^p(\I_{\sigma},w_{\a};\g(H,X_{1/2}))$ a.s., one can check that the remaining integrals are well-defined.

Next we define $L^p_{\a}$-local and $L^p_{\a}$-maximal local solutions to \eqref{eq:QSEE}.

\begin{definition}[$L^p_{\a}$-local and $L^p_{\a}$-maximal solution]
\label{def:solution2}
Let $\sigma$ be a stopping time with $0\leq \sigma\leq T$. Let $u:\ll 0,\sigma\rro\rightarrow X_1$ be strongly progressively measurable.
\begin{itemize}
\item $(u,\sigma)$ is called an $L^p_{\a}$-{\em local solution} to \eqref{eq:QSEE} on $[0,T]$, if there exists an increasing sequence $(\sigma_n)_{n\geq 1}$ of stopping times such that $\lim_{n\uparrow \infty} \sigma_n =\sigma$ a.s.\ and $u|_{\ll 0,\sigma_n\rr}$ is an $L^p_{\a}$-strong solution to \eqref{eq:QSEE} on $\ll 0,\sigma_n\rr$. In this case, $(\sigma_n)_{n\geq 1}$ is called a {\em localizing sequence} for the local solution $(u,\sigma)$.
\item An $L^p_{\a}$-local solution $(u,\sigma)$ to \eqref{eq:QSEE} on $[0,T]$ is called {\em unique}, if for every $L^p_{\a}$-local solution $(v,\nu)$ to \eqref{eq:QSEE} on $[0,T]$ for a.a.\ $\om\in \Omega$ and for all $t\in [0,\nu(\om)\wedge \sigma(\omega))$  one has $v(t,\omega)=u(t,\omega)$.
\item A unique $L^p_{\a}$-local solution $(u,\sigma)$ to \eqref{eq:QSEE} on $[ 0,T]$ is called an {\em $L^p_{\a}$-maximal local solution}, if for any other unique $L^p_{\a}$-local solution $(v,\varrho)$ to \eqref{eq:QSEE} on $[0,T]$, we have a.s.\ $\varrho\leq \sigma$ and  for a.a.\ $\om \in \Omega$ and all $t\in [0,\varrho(\om))$, $u(t,\omega)=v(t,\omega)$.
\end{itemize}
\end{definition}

Note that $L^p_{\a}$-maximal local solutions are unique by definition. In addition, an (unique) $L^p_{\a}$-strong solution $u$ on $\ll 0,\sigma\rr$ gives an (unique) $L^{p}_{\a}$-local solution $(u,\sigma)$ to \eqref{eq:QSEE}. In the following, we omit the prefix $L^p_{\a}$ and the ``on $[0,T]$" if no confusion seems likely.

\subsection{Statement of the main results}
\label{ss:statement_main_results_local_existence}
Our first result on \eqref{eq:QSEE} reads as follows.

\begin{theorem}[Quasilinear I]
\label{t:local}
Let Hypothesis \hyperref[H:hip]{$\Hip$} be satisfied. Assume that $u_0\in L^{\infty}_{\F_0}(\O;\Xap)$ and
$(A(\cdot,u_0),B(\cdot,u_0))\in \MRta$.
Then there exists an $\varepsilon>0$ such that if
\begin{equation}
\label{eq:smallness_condition_nonlinearities_QSEE}
\max\{L_{F},L_{G}\}< \varepsilon,
\end{equation}
then the following assertions hold:
\begin{enumerate}[{\rm (1)}]
\item\label{it:existence_maximal_local_Linfty}{\rm (Existence and uniqueness)}
There exists an $L^p_{\a}$-maximal local solution $(u,\sigma)$ to \eqref{eq:QSEE} such that $\sigma>0$ a.s.
\item\label{it:regularity_data_Linfty} {\rm (Regularity)} There exists a localizing sequence $(\sigma_n)_{n\geq 1}$ for $(u,\sigma)$ such that $\sigma_n>0$ a.s.\ and
\begin{itemize}
\item If $p>2$ and $\a\in [0,\frac{p}{2}-1)$, then for all $n\geq 1$, $\theta\in [0,\frac{1}{2})$,
 \[u\in L^p(\O;H^{\theta,p}(\I_{\sigma_n},w_{\a};X_{1-\theta})) \cap L^p(\Omega; C(\overline{\I}_{\sigma_n};\Xap)).\]
Moreover, $(u,\sigma)$ instantaneously regularizes to $u\in C((0,\sigma);\Xp)$ a.s.
\item If $p=2$ and $\a=0$, then for all $n\geq 1$,
 \[u\in L^2(\O;L^2(\I_{\sigma_n};X_{1})) \cap L^2(\Omega; C(\overline{\I}_{\sigma_n};X_{1/2})).\]
\end{itemize}
\item\label{it:continuity_initial_data_Linfty} {\rm (Continuous dependence on the initial data)}
There exist $\eta,C>0$ depending on $u_0$ such that if $v_0\in B_{L^{\infty}_{\F_0}(\O;\Xap)}(u_0,\eta)$, then the following hold:
\begin{itemize}
\item there exists an $L^p_{\a}$-maximal local solution $(v,\tau)$ to \eqref{eq:QSEE} with $\tau>0$ a.s.\ and initial data $v_0$;
\item For each stopping time $\nu$ with $\nu\in (0,\tau\wedge \sigma]$ a.s.\ one has
$$
\|u-v\|_{L^p(\O;E)}\leq C\|u_0-v_0\|_{L^{p}(\O;\Xap)},
$$
where either $E\in \{H^{\theta,p}(\I_{\nu},w_{\a};X_{1-\theta}), C(\overline{\I}_{\nu};\Xap)\}$ with $p>2$, $\a\in [0,\frac{p}{2}-1)$, $\theta\in [0,\frac{1}{2})$ or
$E\in \{L^2(\I_{\nu};X_1),C(\overline{\I}_{\nu};X_{1/2})\}$ and $p=2$ and $\a=0$.
\end{itemize}
\item\label{it:localization_Linfty} {\rm (Localization)} If $(v,\tau)$ is an $L^p_{\a}$-maximal local solution to \eqref{eq:QSEE} with initial data $v_0\in L^{\infty}_{\F_0}(\O;\Xap)$, then setting $\Gamma:=\{v_0=u_0\}$ one has
$$
\tau|_{\Gamma}=\sigma|_{\Gamma},\qquad  v|_{\Gamma\times [0,\tau)}=u|_{\Gamma\times [0,\sigma)}.
$$
\end{enumerate}
\end{theorem}
A more explicit bound for the number $\varepsilon$ in \eqref{eq:smallness_condition_nonlinearities_QSEE} will be provided in Remark \ref{r:smallness}.

In \eqref{it:regularity_data_Linfty} we see that the paths of the solution are in $C([0,\sigma);\Xap)$. However, if $\a>0$, after $t=0$ the regularity immediately improves to $C((0,\sigma);\Xp)$, where we recall $\Xap = (X_0,X_1)_{1-\frac{1+\a}{p},p}$ and $\Xp = (X_0,X_1)_{1-\frac{1}{p},p}$
This phenomena will play a crucial role in \cite{AV19_QSEE_2}.
Note that the $L^p(\O)$-norms in \eqref{it:continuity_initial_data_Linfty} are well-defined due to Lemma \ref{lem:XYtmeas} and the text below it. Furthermore, in Step 4 in the proof of Theorem \ref{t:local} we show that the estimates in \eqref{it:continuity_initial_data} also hold for the choice $E=\X(\nu)$ where the space $\X$ is defined in \eqref{eq:X_space_c} below.

\begin{remark}
\label{r:W_H_filtration}
In applications to SPDEs, one does not always have $u_0\in L^{\infty}_{\F_0}(\O;\Xap)$. To weaken this condition we make a further extension of Theorem \ref{t:local} at the expense of a stronger hypothesis on $F_L,G_L$, see Theorem \ref{t:local_Extended} below.

On the other hand, if the filtration $\Filtr=(\F_t)_{t\geq 0}$ is generated by the cylindrical Brownian motion $W_H$, then $L^0_{\F_0}(\O;\Xap)=\Xap$. Thus, Theorem \ref{t:local} can be applied without any restriction.
\end{remark}

We would like to present an additional result on the quasilinear case, where we weaken the integrability hypothesis on the initial data, at the cost of more restrictions on the nonlinearities $F_L,G_L$. More specifically, we need a local version of the assumptions \ref{HFcritical}-\ref{HGcritical} and \ref{Hf}.

\let\ALTERWERTA\theenumi
\let\ALTERWERTB\labelenumi
\def\theenumi{(HF$^\prime$)}
\def\labelenumi{(HF$^\prime$)}
\begin{enumerate}
\item\label{HFcritical_weak} The map $F:[0,T]\times\Omega\times X_1 \to X_0$ has the same measurability properties in \ref{HFcritical} and it can be decomposed as $F:=F_L+F_c+F_{\Tr}$, where $F_c,F_{\Tr}$ are as in \ref{HFcritical}. Assume that for each $n\geq 1$ there exist constants $L_{F,n},\tilde{L}_{F,n}, C_{F,n}$, such that for all $x,y\in X_1$, $t\in [0,T]$ and a.a.\ $\om \in \O$, and $\|x\|_{\Xap},\|y\|_{\Xap}\leq n$ one has
\begin{align*}
\|F_L(t,\om,x)\|_{X_0}&\leq C_{F,n}(1+ \|x\|_{X_1}),\\
\|F_L(t,\om,x)-F_L(t,\om,y)\|_{X_0}&\leq L_{F,n} \|x-y\|_{X_1}+ \tilde{L}_{F,n} \|x-y\|_{X_0}.
\end{align*}
\end{enumerate}
\let\theenumi\ALTERWERTA
\let\labelenumi\ALTERWERTB

\let\ALTERWERTA\theenumi
\let\ALTERWERTB\labelenumi
\def\theenumi{(HG$^\prime$)}
\def\labelenumi{(HG$^\prime$)}
\begin{enumerate}
\item\label{HGcritical_weak} The map $G: [0,T]\times \O \times X_1 \to X_0$ has the same measurability properties in \ref{HFcritical} and it can be decomposed as $G:=G_L+G_c+G_{\Tr}$ where $G_c,G_{\Tr}$ are as in \ref{HGcritical}. Assume that for each $n\geq 1$ there exist constants $L_{G,n},\tilde{L}_{G,n}, C_{G,n}$, such that for all $x,y\in X_1$, $t\in [0,T]$ and a.a.\ $\om \in \O$, and $\|x\|_{\Xap},\|y\|_{\Xap}\leq n$ one has
\begin{align*}
\|G_L(t,\om,x)\|_{\g(H,X_{1/2})}&\leq C_{G,n}(1+ \|x\|_{X_1}),\\
\|G_L(t,\om,x)-G_L(t,\om,y)\|_{\g(H,X_{1/2})}&\leq L_{G,n} \|x-y\|_{X_1}+ \tilde{L}_{G,n} \|x-y\|_{X_0}.
\end{align*}
\end{enumerate}
\let\theenumi\ALTERWERTA
\let\labelenumi\ALTERWERTB

\let\ALTERWERTA\theenumi
\let\ALTERWERTB\labelenumi
\def\theenumi{(Hf\,$^\prime$)}
\def\labelenumi{(Hf\,$^\prime$)}
\begin{enumerate}
\item\label{Hf'}
$f\in L^0_{\Progress}(\O;L^p( \I_T,w_{\a};X_0))$ and $g\in L^0_{\Progress}(\O;L^p( \I_T,w_{\a};\g(H,X_{1/2})))$.
\end{enumerate}
\let\theenumi\ALTERWERTA
\let\labelenumi\ALTERWERTB

We say that the \textbf{Hypothesis} $\Hipprime$\label{H:hipprime} holds if \ref{HAmeasur}, \ref{HFcritical_weak}, \ref{HGcritical_weak}, and \ref{Hf'} are satisfied. Definitions \ref{def:solution1}  and \ref{def:solution2} clearly extend to the setting of Hypothesis \hyperref[H:hipprime]{$\Hipprime$}.

To extend Theorem \ref{t:local} in case of $L^0$-data we employ a cut-off argument. To this end, given $u_0\in L^0_{\F_0}(\O;\Xap)$ we denote by $(u_{0,n})_{n\geq 1}$ a sequence such that
\begin{equation}
\label{eq:approximating_sequence_initial_data}
u_{0,n}\in L^{\infty}_{\F_0}(\O;\Xap),  \quad  \text{ and }\quad  u_{0,n}=u_0  \   \text{ on  }  \ \{\|u_0\|_{\Xap}\leq n\}.
\end{equation}
For possible choices of $(u_{0,n})_{n\geq 1}$ see \eqref{eq:truncation_operator} and the text below it.

\begin{theorem}[Quasilinear II]
\label{t:local_Extended}
Let Hypothesis \hyperref[H:hipprime]{$\Hipprime$} be satisfied. Let $u_0\in L^0_{\F_0}(\O;\Xap)$. Assume that there exists $(u_{0,n})_{n\geq 1}$ satisfying \eqref{eq:approximating_sequence_initial_data} and for all $n\geq 1$
\begin{equation}
\label{eq:stochastic_maximal_regularity_assumption_local_extended}
(A(\cdot,u_{0,n}),B(\cdot,u_{0,n}))\in \MRta.
\end{equation}
There exists a decreasing sequence $(\varepsilon_n)_{n\geq 1}$ in $(0,\infty)$ such that if
\begin{equation}
\label{eq:smallness_condition_nonlinearities_QSEE_extended}
\max\{L_{F,n},L_{G,n}\}<\varepsilon_n, \ \ \ \ \text{for all $n\geq 1$},
\end{equation}
then the following assertions hold:
\begin{enumerate}[{\rm (1)}]
\item\label{it:existence_extended_maximal}{\rm (Existence and uniqueness)} There exists an $L^p_{\a}$-maximal local solution $(u,\sigma)$ to \eqref{eq:QSEE} such that $\sigma>0$ a.s.

\item\label{it:regularity_data_L0} {\rm (Regularity)} For each localizing sequence $(\sigma_n)_{n\geq 1}$ for $(u,\sigma)$, one has
\begin{itemize}
\item If $p>2$ and $\a\in [0,\frac{p}{2}-1)$, then for all $n\geq 1$, $\theta\in [0,\frac{1}{2})$,
 \[u\in H^{\theta,p}(\I_{\sigma_n},w_{\a};X_{1-\theta})\cap C(\overline{\I}_{\sigma_n};\Xap)\ \ \ a.s.\]
Moreover, $(u,\sigma)$ instantaneously regularizes to $u\in C((0,\sigma);\Xp)$ a.s.
\item If $p=2$ and $\a=0$, then for all $n\geq 1$,
 \[u\in L^2(\I_{\sigma_n};X_{1}) \cap C(\overline{\I}_{\sigma_n};X_{1/2}) \ \ \ a.s.\]
\end{itemize}
\item\label{it:continuity_initial_data} {\rm (Local existence and continuous dependence on the initial data)} Let $n\geq 1$ and $\Gamma_n:=\{\|u_0\|_{\Xap}\leq n\}$. Then Theorem \ref{t:local}\eqref{it:continuity_initial_data_Linfty} holds with $u_0,v_0$ and $\O$ replaced by $\one_{\Gamma_n}u_0,\one_{\Gamma_n}v_0$ and $\Gamma_n$, respectively.

\item\label{it:localization_L0} {\rm (Localization)}
Theorem \ref{t:local}\eqref{it:localization_Linfty} holds, where the assumptions on $u_0, v_0$ are replaced by $u_0,v_0\in L^{0}_{\F_0}(\O;\Xap)$.
\end{enumerate}
\end{theorem}

For the more precise estimates on the sequence $(\varepsilon_n)_{n\geq 1}$ we refer to Remark \ref{r:smallness_n}.

Let $u_0\in L^0_{\F_0}(\O;\Xap)$ and set $\Gamma_n:=\{\|u_0\|_{\Xap}\leq n\}\in \F_0$. A typical choice of the sequence $(u_{0,n})_{n\geq 1}$ in Theorem \ref{t:local_Extended} is given by
\begin{equation}
\label{eq:truncation_operator}
u_{0,n}:= \one_{\Gamma_n} u_0 +\one_{\O\setminus \Gamma_n} \frac{n u_0} {\|u_0\|_{\Xap}}.
\end{equation}
However, the condition \eqref{eq:approximating_sequence_initial_data} allows us to choose $u_{0,n}|_{\O\setminus\Gamma_n}$ differently, and we will exploit this fact in applications. More precisely, instead of \eqref{eq:truncation_operator} one can use $u_{0,n}=\one_{\Gamma_n} u_0+ \one_{\O\setminus \Gamma_n} x$ where $x\in \Xap$ can be chosen such that \eqref{eq:stochastic_maximal_regularity_assumption_local_extended} holds. Throughout Sections \ref{s:semilinear_gradient}-\ref{s:AC_CH} we will use the choice \eqref{eq:truncation_operator}, but in Subsection \ref{ss:porous_media_equations} we need a different choice (see \eqref{eq:choice_approximating_sequence_initial_data_porous_media_equations}).

If \eqref{eq:QSEE} is of semilinear type, (see Assumption \ref{ass:AB_boundedness}), the condition $u_0\in L^{\infty}_{\F_0}(\O;\Xap)$ can be weakened and we still get $L^p$-integrability with respect to $\om\in\O$. More precisely, we have the following.

\begin{theorem}[Semilinear]
\label{t:semilinear}
Let the Hypothesis \hyperref[H:hip]{$\Hip$} be satisfied, where $A$ and $B$ are of semilinear type as in Assumption \ref{ass:AB_boundedness} and satisfy $(A,B)\in \MRta$. Then there exists an $\varepsilon>0$ such that if
\begin{equation}
\label{eq:smallness_semilinear}
\max\{L_{G},L_F\}< \varepsilon,
\end{equation}
then the following assertions hold:
\begin{enumerate}[{\rm(1)}]
\item\label{it:semilinear_u_L_infty} If $u_0\in L^\infty_{\F_0}(\O;\Xap)$, then the statements in Theorem \ref{t:local}\eqref{it:existence_maximal_local_Linfty}--\eqref{it:localization_Linfty} hold.
\item\label{it:semilinear_u_L_p} If $u_0\in L^p_{\F_0}(\O;\Xap)$ and
the constants $C_{c,n},L_{c,n},C_{\Tr,n},L_{\Tr,n}$ in \emph{\ref{HFcritical}-\ref{HGcritical}} do not depend on $n\geq 1$, then the statements in Theorem \ref{t:local}\eqref{it:existence_maximal_local_Linfty}--\eqref{it:localization_Linfty} hold.
\item\label{it:semilinear_u_L_0} If $u_0\in L^0_{\F_0}(\O;\Xap)$, then the statements in Theorem \ref{t:local_Extended}\eqref{it:existence_extended_maximal}--\eqref{it:localization_L0} hold.
\end{enumerate}
\end{theorem}
Assertion \eqref{it:semilinear_u_L_infty} is immediate from Theorem \ref{t:local}. Under additional growth conditions one can often derive $L^p$-estimates as well. Assertion \eqref{it:semilinear_u_L_p} shows that in the semilinear case the condition $u_0\in L^{\infty}_{\F_0}(\O;\Xap)$ in Theorem \ref{t:local} can be weakened. Assertion \eqref{it:semilinear_u_L_0} will be immediate from the proof of Theorem \ref{t:local_Extended}.

\subsection{The role of the space $\X(T)$}
In the proofs of the results stated in Subsection \ref{ss:statement_main_results_local_existence} in the case $p>2$ we need a family of function spaces $(\X(t))_{t\in [0,T]}$ having the following three properties: the nonlinearities $F_c(\cdot,u),G_c(\cdot,u)$ can be controlled by $\|u\|_{\X(t)}$, the map $[0,T]\ni t\mapsto \|f|_{(0,t)}\|_{\X(t)}$ is continuous for all $f\in \X(T)$, and
\begin{equation}
\label{eq:max_regularity_space_E_X_usefulness}
H^{\d,p}(\I_T,w_{\a};X_{1-\d})\cap L^p(\I_T,w_{\a};X_1)
\hookrightarrow \X(T)
\end{equation}
for some $\delta\in (\frac{1+\a}{p},\frac{1}{2})$.
Note that the left-hand side in \eqref{eq:max_regularity_space_E_X_usefulness} is part of our usual maximal regularity space (see Definition \ref{def:SMRz}).
As mentioned below Lemma \ref{lem:XYtmeas}, it is not obvious whether the $H^{\delta,p}(\I_t,w_{\a};X_{1-\delta})$-norm is continuous in $t$ and therefore we do not define $\X(T)$ as the left-hand side of \eqref{eq:max_regularity_space_E_X_usefulness}.

Recall that the numbers $(\rho_j)_{j=1}^{\mf+\mg},(\beta_j)_{j=1}^{\mf+\mg}$ and $(\varphi_j)_{j=1}^{\mf+\mg}$ are defined in \ref{HFcritical} and \ref{HGcritical}. In the case that for some $j\in \{1,\dots,\mf+\mg\}$, \eqref{eq:HypCritical} or \eqref{eq:HypCriticalG} holds with \textit{strict} inequality, we may increase $\rho_j$ in order to obtain equality. More precisely, we set
\begin{equation}\label{eq:rhostar}
\rhos_j:=\frac{1-\beta_j}{\varphi_j-1+(1+\a)/p},\qquad j\in \{1,\dots,m_F+m_G\}.
\end{equation}
Since $\beta_j<1$ and $\varphi_j>1-(1+\a)/p$, one has $\rhos_j>0$ for all $j\in \{1,\dots,m_F+m_G\}$.

To define a space $\X(T)$ which satisfies the previous requirements, suppose
\ref{HFcritical}-\ref{HGcritical} are satisfied and let $\rhos_j$ be as in \eqref{eq:rhostar}. For $j\in \{1,\dots,m_F+m_G\}$ we let
\begin{equation}
\label{eq:rr'}
\frac{1}{r_{j}'}:=\frac{\rhos_j(\varphi_j-1+(1+\a)/p)}{(1+\a)/p}<1,
\qquad
\frac{1}{r_j}:=\frac{\beta_j -1 + (1+\a)/p}{(1+\a)/p}<1
\end{equation}
and, for $T>0$,
\begin{equation}
\label{eq:X_space_c}
\X(T):=\Big(\bigcap_{j=1}^{\mf+\mg} L^{pr_j}(\I_T,w_{\a};X_{\beta_j})\Big) \cap
\Big(\bigcap_{j=1}^{\mf+\mg} L^{\rhos_j p r_j'}(\I_T,w_{\k};X_{\varphi_j})\Big).
\end{equation}
We will see that $\X(T)$ is the natural space to control the non-linearities $F_c,G_c$; see Lemmas \ref{l:F_G_bound_N} and \ref{l:truncationFG} below. Moreover, if \eqref{eq:max_regularity_space_E_X_usefulness} holds, then the solution paths will automatically be in \eqref{eq:X_space_c} as soon as we have maximal regularity. Finally, we note that the continuity of the $\X$-norm in $t\in [0,T]$ follows from the Lebesgue dominated convergence theorem.

Next we prove \eqref{eq:max_regularity_space_E_X_usefulness} with the above defined $\X$ under a trace condition. The general case is discussed in Remark \ref{r:embedding_with_T_dependence}. Here we follow \cite[Section 2]{addendum}.

\begin{lemma}
\label{l:embeddings}
Let \emph{\ref{HFcritical}}-\emph{\ref{HGcritical}} be satisfied. Let $T\in (0,\infty]$ and let $r_j,r'_j$ and $\X(T)$ be as in \eqref{eq:rr'} and \eqref{eq:X_space_c} respectively.
If $p>2$ and $\a\in [0,\frac{p}{2}-1)$, then for any $\delta\in (\frac{1+\a}{p},\frac{1}{2})$
\begin{equation*}
\hz^{\d,p}(0,T,w_{\a};X_{1-\d})\cap L^p(0,T,w_{\a};X_1)
\hookrightarrow \X(T),
\end{equation*}
where the embedding constant can be chosen to be independent of $T$.

Furthermore, if $p=2$ and $\a=0$, the same holds with $\hz^{\d,p}(0,T,w_{\a};X_{1-\d})$ replaced by $C([0,T];X_{1/2})$.
\end{lemma}

\begin{proof}
Recall that in \eqref{eq:rhostar} we have defined $\rhos_j$ such that \eqref{eq:HypCritical}-\eqref{eq:HypCriticalG} hold with equality for each $j\in\{1,\dots,m_F+m_G\}$. Due to \eqref{eq:rr'}, this implies that
$$
\frac{1}{r_j}+\frac{1}{r'_j}=1, \text{ for all } j\in\{1,\dots,m_F+m_G\}.
$$

\textit{Step 1: Case $p=2$, $\a=0$}. Let $\vartheta\in (0,1)$ be arbitrary. By interpolation one has
$$\|x\|_{X_{\frac{1}{2}+\frac{\vartheta}{2}}}\leq  \|x\|_{X_{\frac{1}{2}}}^{1-\vartheta}\|x\|_{X_{1}}^{\vartheta},$$
for $x\in X_1$. Thus, by Young's inequality
\begin{align*}
\|u\|_{L^{\frac{2}{\vartheta}}(0,T;X_{\frac{1}{2}+\frac{\vartheta}{2}})} & \leq \|u\|_{C([0,T];X_{1/2})}^{1-\vartheta} \|u\|_{L^2(I_T;X_{1})}^{\vartheta}
\\ & \leq (1-\vartheta) \|u\|_{C([0,T];X_{1/2})} + \vartheta \|u\|_{L^2(I_T;X_{1})}.
\end{align*}
Therefore, we have the following contractive embedding
\begin{equation}
\label{eq:interpolation_p_2}
C([0,T];X_{1/2})\cap L^2(0,T;X_1)
\hookrightarrow L^{\frac{2}{\vartheta}}(0,T;X_{\frac{1}{2}+\frac{\vartheta}{2}}).
\end{equation}
By \eqref{eq:interpolation_p_2} with $\vartheta=1/r_j=2(\beta_j-1/2)$ and $\vartheta=1/({\rhos_j r_j'})=2(\varphi_j-1/2)$ one obtains
\begin{equation*}
C([0,T];X_{1/2})\cap L^2(0,T;X_1)\hookrightarrow L^{2r_j}(0,T;X_{\beta_j})\cap L^{2\rhos_j r_j'}(0,T;X_{\varphi_j}).
\end{equation*}

\textit{Step 2: case $p>2$ and $\a\in [0,\frac{p}{2}-1)$}. By Proposition \ref{prop:embeddingSobolevWeights} for each $j\in \{1,\dots,m_F+m_G\}$
\begin{align}
\label{eq:Sob1}
\hz^{1-\beta_j,p}(0,T,w_{\a};X_{\lambda})&\hookrightarrow L^{pr_j}(0,T,w_{\k};X_{\lambda}),\quad\\
\label{eq:Sob2}
\hz^{1-\varphi_j,p}(0,T,w_{\a};X_{\lambda})&\hookrightarrow L^{\rhos_j p r_j'}(0,T,w_{\k};X_{\lambda}),
\end{align}
for each $\lambda\in [0,1]$ and where the embedding constants do not depend on $T$.

Let $0<\eta<\zeta<1$ and assume $\eta\neq (1+\a)/p$. Using Proposition \ref{prop:Sob_interpolation_h} with $\theta:=\frac{\zeta-\eta}{\zeta}\in (0,1)$, one obtains
\begin{equation}
\label{eq:interp_H_X1_Xsigma}
\hz^{\zeta,p}(0,T;w_{\a};X_{1-\zeta})\cap L^p(0,T,w_{\a};X_1)\hookrightarrow
\hz^{\eta,p}(0,T,w_{\a};X_{1-\eta}),
\end{equation}
where we used that $[X_{1-\zeta},X_1]_{\theta}=X_{\eta}$, which follows immediately from the reiteration theorem for complex interpolation and Assumption \ref{ass:Xtr}.

Let $\d\in (\frac{1+\a}{p},\frac{1}{2})$ be arbitrary. Since $\b_j\in (1-\frac{1+\a}{p},1)$ one has $\delta>1-\beta_j \in (0,\frac{1+\a}{p})$ for each $j\in \{1,\dots,m_F+m_G\}$, and hence it follows that
\begin{align*}
\hz^{\d,p}(0,T,w_{\a};X_{1-\d})\cap L^p(0,T,w_{\a};X_1)
&\hookrightarrow \hz^{1-\beta_j,p}(0,T,w_{\a};X_{\beta_j})\\
&\hookrightarrow L^{pr_j}(0,T,w_{\a};X_{\beta_j}),
\end{align*}
where in the first embedding we used $\delta>1-(1+\a)/p>1-\beta_j$ and \eqref{eq:interp_H_X1_Xsigma}, and the second one follows from \eqref{eq:Sob1}. Analogously, for $j\in \{1,\dots,m_F+m_G\}$, using $\delta>\frac{1+\a}{p}>1-\varphi_j$, \eqref{eq:Sob2}, and \eqref{eq:interp_H_X1_Xsigma}, one obtains
\begin{align*}
\hz^{\d,p}(0,T;w_{\a};X_{1-\d})\cap L^p(0,T,w_{\a};X_1)
&\hookrightarrow \hz^{1-\varphi_j,p}(0,T,w_{\a};X_{\varphi_j})\\
&\hookrightarrow L^{\rhos_j p r_j'}(0,T,w_{\k};X_{\varphi_j}).
\end{align*}
Putting together the above inclusions the result follows.
\end{proof}

\begin{remark}
\label{r:embedding_with_T_dependence}
Let $p>2$. The embedding in Lemma \ref{l:embeddings} also holds in the case where $\hz^{\d,p}$ is replaced by $H^{\d,p}$, but with an embedding constant which depends on $T>0$.
\end{remark}

Let us show that the space $\X(T)$ defined in \eqref{eq:X_space_c} is well suited to bound the nonlinearities $F_c,G_c$. Actually, we prove a more refined result since this will be needed in our paper \cite{AV19_QSEE_2} on blow-up criteria and regularization.

\begin{lemma}
\label{l:F_G_bound_N}
Let the hypothesis \emph{\ref{HFcritical}-\ref{HGcritical}} be satisfied. Let $0<T<\infty$ and $N\geq 1$ be fixed. Then there exists $C_T>0$ and $\zeta>1$ such that for all $u\in C(\overline{\I}_T;\Xap)\cap \X(T)$ which verifies $\|u\|_{C(\overline{\I}_T;\Xap)}\leq N$, one has a.s.
\begin{multline*}
\|F_c(\cdot,u)-F_c(\cdot,0)\|_{L^p(\I_T,w_{\a};X_0)}+\|G_c(\cdot,u)-G_c(\cdot,0)\|_{L^p(\I_T,w_{\a};\g(H,X_{1/2}))}
\\
\leq C_{T}(\|u\|_{\X(T)}+\|u\|_{\X(T)}^{\zeta}).
\end{multline*}
Moreover, if $\Xap$ is not critical for \eqref{eq:QSEE}, then $\lim_{T\downarrow 0}C_{T}= 0$.
\end{lemma}

\begin{proof}
For notational simplicity we only consider the case $m_F=1$. Thus, we set $\rhos:=\rhos_1$, $\rho:=\rho_1$, $\varphi:=\varphi_1$ and $\beta:=\beta_1$. In this case,
\begin{equation}
\label{eq:X_spaces_special_Case_lemma}
\X(T)=L^{pr}(\I_T,w_{\a};X_{\beta}) \cap L^{\rho p r'}(\I_T,w_{\a};X_{\varphi}),
\end{equation}
where $r:=r_1$ and $r':=r'_1$ are defined in \eqref{eq:rr'}. Thus, by \ref{HFcritical}, for $x\in X_{\varphi}$,
\begin{equation*}
\|F_c(t,x)-F_c(t,0)\|_{X_0}\leq L_{c,N} (1+\|x\|^{\rho}_{X_{\varphi}})\|x\|_{X_{\beta}}.
\end{equation*}
This implies
\begin{equation}
\label{eq:F_C_a_b}
\begin{aligned}
&\|F_c(\cdot,u)-F_c(\cdot,0)\|_{L^p(0,T,w_{\a};X_0)}\\
&\leq L_{c,N} \|u\|_{L^p(0,T,w_{\a};X_{\beta})}+ \big\|\|u\|^{\rho}_{X_{\varphi}}\|u\|_{X_{\beta}}\big\|_{L^p(0,T,w_{\a})}\\
&\leq L_{c,N}(C_{T} \|u\|_{L^{pr}(0,T,w_{\a};X_{\beta})}+\|u\|_{L^{\rho p r'}(0,T,w_{\a};X_{\varphi})}^{\rho}\|u\|_{L^{pr}(0,T,w_{\a};X_{\beta})});
\end{aligned}
\end{equation}
where $\lim_{T\downarrow 0}C_T= 0$. For simplicity, let us distinguish two cases:

\textit{Case $\rhos=\rho:=\rho_1$}. In other words $\Xap$ is critical for \eqref{eq:QSEE}. The claimed inequality follows from \eqref{eq:X_spaces_special_Case_lemma}-\eqref{eq:F_C_a_b} by setting $\zeta=1+\rho$.

\textit{Case $\rhos>\rho:=\rho_1$}. By the H\"{o}lder inequality, one has
$$
\|u\|_{L^{\rho p r'}(0,T,w_{\a})}^{\rho}\leq C_T \|u\|_{L^{\rhos p r'}(0,T,w_{\a};X_{\varphi})}^{\rho}\leq C_T \|u\|_{\X(T)}^{\rho},
$$
where $\lim_{T\downarrow 0}C_T=0$.

The assertion for $G_c$ is proved in the same way.
\end{proof}

\begin{remark}
If the constants $L_{c,n}$ in \ref{HFcritical}$(ii)$ and \ref{HGcritical}$(ii)$ do not depend on $n\geq 1$, then the constant $C_{T}$ can be chosen independent of $N$ and the above proof extends to any $u \in \X(T)$.
\end{remark}

\subsection{Truncation lemmas}
\label{s:proofs_quasi_local}
In this subsection we collect several truncation lemmas which are needed in the proofs of Theorems \ref{t:local}, \ref{t:local_Extended} and \ref{t:semilinear}.

First we define suitable truncations of $F_c,G_c$. To this end let $\xi\in W^{1,\infty}([0,\infty))$ be such that $\xi=1$ on $[0,1]$ and $\xi=0$ on $[2,\infty)$ and $\xi$ is linear on $[1,2]$. For each $\lambda>0$, set $\xi_{\lambda}(x):=\xi(x/\lambda)$ for $x\in \R_+$. Then $\supp \xi_{\lambda}\subseteq [0,2\lambda]$, $\xi_{\lambda}|_{(0,\lambda)}=1$ and $\| \xi_{\lambda}'\|_{L^{\infty}(\R_+)}\leq 1/\lambda$. For $t\in [0,T]$,  $x\in \Xap$, and $u\in \X(T)\cap C(\overline{\I}_T;\Xap)$ we set
\begin{equation}
\label{eq:phiLambda}
\Theta_{\lambda}(t,x,u):=\xi_{\lambda}
\Big(\|u\|_{\X(t)}+\sup_{s\in [0,t]}\|u(s)-x\|_{\Xap}\Big).
\end{equation}

In the next lemma we fix $\omega\in \O$, but omit it from our notation.
\begin{lemma}
\label{l:truncationFG}
Let \emph{\ref{HFcritical}}-\emph{\ref{HGcritical}} be satisfied. Let $T>0$ and let $\sigma\in [0,T]$. Let $\Theta_{\lambda}$ be defined in \eqref{eq:phiLambda}.
For any $\lambda\in (0,1)$, let the maps
\begin{align*}
F_{c,\lambda}:\Xap\times \X({\sigma})\cap C(\overline{\I}_{\sigma};\Xap)&\to L^p(\I_{\sigma},w_{\a};X_0)),\\
G_{c,\lambda}:\Xap\times \X({\sigma})\cap C(\overline{\I}_{\sigma};\Xap)&\to L^p(\I_{\sigma},w_{\a};\g(H,X_{1/2})),
\end{align*}
be given by
\begin{align*}
F_{c,\lambda}(x,u)&:= \Theta_{\lambda}(\cdot,x,u)(F_c(\cdot,u)-F_c(\cdot,0)),\\
G_{c,\lambda}(x,u)&:=\Theta_{c,\lambda}(\cdot,x,u) (G_c(\cdot,u)-G_c(\cdot,0)).
\end{align*}
Then for any $N\geq 1$ there exist constants $C_{\lambda},L_{T,\lambda}$ such that if $\|x\|_{\Xap}\leq N$,
\begin{align*}
\|F_{c,\lambda}(x,u)\|_{L^p(\I_{\sigma},w_{\k};X_0)}&\leq C_{\lambda}\\
\|G_{c,\lambda}(x,u)\|_{L^p(\I_{\sigma},w_{\k};\g(H,X_{1/2}))}&\leq C_{\lambda}\\
\|F_{c,\lambda}(x,u)-F_{c,\lambda}(\cdot,x,v)\|_{L^p( \I_{\sigma},w_{\k};X_0)}
&\leq L_{\lambda,T}(\|u-v\|_{\X(\sigma)}+\|u-v\|_{C(\overline{\I}_{\sigma};\Xap)}),\\
\|G_{c,\lambda}(\cdot,x,u)-G_{c,\lambda}(\cdot,x,v)\|_{L^p(\I_{\sigma},w_{\k};\g(H,X_{1/2}))}
&\leq L_{\lambda,T}(\|u-v\|_{\X(\sigma)}+\|u-v\|_{C(\overline{\I}_{\sigma};\Xap)});
\end{align*}
a.s. Moreover, for each $\varepsilon>0$ there exist $\bar{T}=\bar{T}(\varepsilon)>0$ and $\bar{\lambda}=\bar{\lambda}(\varepsilon)>0$ such that for all $T\in (0,\bar{T})$, $\lambda\in(0,\bar{\lambda})$ one has $L_{\lambda,T}<\varepsilon$.
\end{lemma}

\begin{proof}
We only consider the estimates for $F_{c,\lambda}$ since the other case is similar.
Recall that in \eqref{eq:rhostar} we have defined $\rhos_j$ such that \eqref{eq:HypCritical}-\eqref{eq:HypCriticalG} hold with equality for each $j\in \{1,\dots,m_F+m_G\}$.
For notational convenience, we assume that $m_F=1$ and we set $\rho:=\rhos_1$, $\varphi:=\varphi_1$, $\beta:=\beta_1$ and $r:=r_1$, $r':=r'_1$ (see \eqref{eq:rr'}). The general case can be proven with the same considerations. Moreover, it is enough to consider the case $\sigma=T$.
In the proof $C_T$ denotes a suitable constant, which can be different from line to line, and verifies $\lim_{T\downarrow 0 }C_T=0$.

Set $\tilde{F}_{c}(t,u):=F_c(t,u)-F_c(t,0)$. Thus, $\tilde{F}_c(t,0)=0$, and by \ref{HFcritical} it follows that for $u,v\in X_{\varphi}$,
\begin{equation}
\label{eq:F_c_estimate_u_N_2}
\|\tilde{F}_{c}(t,u)-\tilde{F}_{c}(t,v)\|_{X_0}\leq L_{c,N+2} (1+\|u\|_{X_{\varphi}}^{\rho}+\|v\|_{X_{\varphi}}^{\rho})\|u-v\|_{X_{\beta}},
\end{equation}
provided $\|u\|_{\Xap},\|v\|_{\Xap}\leq N+2$. For convenience we set $L_{c,N+2}=:C_F$.

Let us set
\begin{equation}
\label{eq:tau_truncation_lambda_Y_t}
\tau_{u}:=\inf\Big\{t\in [0,T]\,:\,\|u\|_{\X(t)}+\sup_{s\in [0,t]}\|u(s)-x\|_{\Xap} \geq 2\lambda\Big\}\wedge T.
\end{equation}
Then since $\Theta_{\lambda}(t,x,u)=0$ if $t\geq \tau_u$ we can write
\begin{align*}
\|F_{c,\lambda}&(x,u)\|_{L^p(\I_T,w_{\k};X_0)}
{=}\|F_{c,\lambda}(x,u)\|_{L^p(0,\tau_u,w_{\k};X_0)}	\\
&\stackrel{(i)}{\leq}C_F  \Big( \int_0^{\tau_u}(1+\|u\|_{X_{\varphi}}^{\rho})^p\|u\|_{X_{\beta}}^p t^{\a}dt\Big)^{1/p}\\
&\stackrel{(ii)}{\leq} C_F( \|u\|_{L^p(\I_{\tau_u},w_{\a};X_{\beta})}+\|u\|_{L^{\rho p r'}(\I_{\tau_u},w_{\a};X_{\varphi})}^{\rho}
\|u\|_{L^{pr}(\I_{\tau_u},w_{\a};X_{\beta})})\\
&\stackrel{(iii)}{\leq} C_F (C_T \|u\|_{L^{pr}(\I_{\tau_u},w_{\a};X_{\beta})}+\|u\|_{L^{\rho p r'}(\I_{\tau_u},w_{\a};X_{\varphi})}^{\rho}
\|u\|_{L^{pr}(\I_{\tau_u},w_{\a};X_{\beta})})\\
&\stackrel{(iv)}{\leq} C_F( 2 C_T \lambda + (2\lambda)^{\rho}2\lambda)=:C_{\lambda}.
\end{align*}
In $(i)$ we used \eqref{eq:F_c_estimate_u_N_2} and the fact that $\|u\|_{C(\overline{\I}_{\tau_u};\Xap)}\leq N+2$, $\|x\|_{L^{\infty}(\O;\Xap)}\leq N$, and $\lambda\in (0,1)$. In $(ii)$ and $(iii)$ we used H\"{o}lder's inequality with exponent $r,r'$ defined in \eqref{eq:rr'}. In $(iv)$ we used \eqref{eq:tau_truncation_lambda_Y_t}.

Next we estimate $\Delta F:=F_{c,\lambda}(x,u)-F_{c,\lambda}(x,v)$. Without loss of generality, we may assume that $\tau_u\leq \tau_v$.
Clearly, we can estimate
\begin{equation*}
\begin{aligned}
\|\Delta F\|_{L^p(\I_T,w_{\k};X_0)}
\leq &\|\Theta_{\lambda}(\cdot,x,u)(\tilde{F}_c(\cdot,u)-\tilde{F}_c(\cdot,v))\|_{L^p(\I_T,w_{\k};X_0)}\\
& \ \ +\|(\Theta_{\lambda}(\cdot,x,u)-\Theta_{\lambda}(\cdot,x,v))\tilde{F}_c(\cdot,v)\|_{L^p(\I_T,w_{\k};X_0)} =: R_1+R_2.
\end{aligned}
\end{equation*}
Since $\Theta_{\lambda}(t,x,u)=0$ if $t\geq \tau_u$, the term $R_1$ can be estimated as
\begin{align*}
R_1 & = \|\Theta_{\lambda}(\cdot,x,u)(\tilde{F}_c(u)-\tilde{F}_c(v))\|_{L^p( 0,\tau_u,w_{\k};X_0)}\\
&\stackrel{(i)}{\leq}C_{F}\Big(\int_0^{\tau_u}
(1+\|u(t)\|_{X_{\varphi}}^{\rho}+\|v(t)\|_{X_{\varphi}}^{\rho})^p \|u(t)-v(t)\|_{X_{\beta}}^p t^{\k}dt\Big)^{\frac{1}{p}}\\
&\stackrel{(ii)}{\leq}C_{F}\Big( C_T+\|u\|_{L^{\rho p r'}(0,\tau_u,w_{\k};X_{\varphi})}^{\rho}
+\|v\|_{L^{\rho p r'}(0,\tau_u,w_{\k};X_{\varphi})}^{\rho }\Big) \|u-v\|_{L^{pr}(\I_T,w_{\k};X_{\beta})}\\
&\stackrel{(iii)}{\leq}  C_{F}\Big( C_T + 2^{1+\rho}\lambda^{\rho }\Big) \|u-v\|_{\X(T)}.
\end{align*}
In $(i)$ we used \eqref{eq:F_c_estimate_u_N_2}, in $(ii)$ we used H\"older's inequality with exponent $r,r'$, and $(iii)$ follows from $\tau_u\leq \tau_v$. For $R_2$ note that since $\|\xi_{\lambda}'\|_{L^{\infty}(\R_+)}\leq 1/\lambda$, for all $t\in [0,T]$, one has
\begin{align*}
|\Theta_{\lambda}& (t,x,u)- \Theta_{\lambda}(t,x,v)|
\\  & \leq \frac{1}{\lambda}\Big| \|u\|_{\X(t)}-\|v\|_{\X(t)}+
\|u-x\|_{C(\overline{\I}_t;\Xap)}-\|v-x\|_{C(\overline{\I}_t;\Xap)} \Big|
\\ & \leq \frac{1}{\lambda} \big[ \|u-v\|_{\X(T)}+\|u-v\|_{C(\overline{\I}_T;\Xap)}\big].
\end{align*}
Therefore, using that $\Theta_{\lambda}(t,x,u) = \Theta_{\lambda}(t,x,v) = 0$ if $t\geq\tau_v$, we obtain
\begin{align*}
R_2 & = \|(\Theta_{\lambda}(\cdot,x,u)-\Theta_{\lambda}(\cdot,x,v))\tilde{F}_c(\cdot,v)\|_{L^p(\I_{\tau_v},w_{\k};X_0)} \\ & \leq \frac{1}{\lambda} \big[ \|u-v\|_{\X(T)}+\|u-v\|_{C(\overline{\I}_T;\Xap)}\big] \|\tilde{F}_c(\cdot,v)\|_{L^p(0,\tau_v,w_{\a};X_0)}.
\end{align*}
By H\"older's inequality, and $\|v\|_{\X(\tau_v)}\leq 2\lambda$  (see \eqref{eq:tau_truncation_lambda_Y_t}), we obtain
\begin{align*}
 \|\tilde{F}_c(\cdot,v)\|_{L^p(0,\tau_v,w_{\a};X_0)}& \leq C_F \Big(\int_0^{\tau_v}(1+ \|v\|_{X_{\varphi}}^{\rho})^p\|v\|_{X_{\beta}}^p t^{\k}dt\Big)^{\frac{1}{p}}
\\ & \leq C_F[C_T+\|v\|_{L^{p\rho r'}(0,\tau_v,w_{\a};X_{\varphi})}^{\rho}]\|v\|_{L^{pr}(0,\tau_v,w_{\a};X_{\beta})}\\
&\leq  2 C_F(C_T+(2\lambda)^{\rho})\lambda
\end{align*}
It follows that
\[R_2 \leq 2 C_F(C_T+(2\lambda)^{\rho}) (\|u-v\|_{\X(T)}+\|u-v\|_{C(\overline{\I}_T;\Xap)}).\]
\end{proof}

\begin{remark}
\label{r:truncationFG_n}
In the setting of Lemma \ref{l:truncationFG}, if the constants $L_{n,c},C_{n,c}$ in \ref{HFcritical}$(ii)$-\ref{HGcritical}$(ii)$ do not depend on $n\geq 1$, then Lemma \ref{l:truncationFG} also holds with $\Theta_{\lambda}(t,u,x)$ replaced by $\wt{\Theta}_{\lambda}(t,x,u):=\xi_{\lambda}
\big(\|u\|_{\X(t)}\big)$.
\end{remark}

The last ingredient we need for the proof of Theorem \ref{t:local} is a suitable truncation of the remaining non-linearities $A,B,F_{\Tr},G_{\Tr}$. Here the proof in \cite[Lemma 4.4]{Hornung} extends to our setting. Let $\xi_{\lambda}$ be the truncation defined before Lemma \ref{l:truncationFG}. For $t\in [0,T]$,  $x\in \Xap$, and $u\in C(\overline{\I}_T;\Xap)\cap L^p(I_T,w_{\a};X_1)$ we set
\begin{equation}
\label{eq:Psi_truncation}
\Psi_{\lambda}(t,x,u):=\xi_{\lambda}\Big( \sup_{s\in [0,t]}\|u(s)-x\|_{\Xap}+ \|u\|_{L^p(\I_t,w_{\a};X_1)} \Big).
\end{equation}
Similar to Lemma \ref{l:truncationFG}, we have the following.

\begin{lemma}
\label{l:truncationAB}
Let \emph{\ref{HAmeasur}}, \emph{\ref{HFcritical}-\ref{HGcritical}} be satisfied. Let $T>0$, $\lambda\in(0,1)$, and let $\sigma$ be a stopping time with values in $[0,T]$. Moreover, let the maps
\begin{align*}
F_{A,\lambda}(x,\cdot):\Xap\times L^p(\I_{\sigma},w_{\a};X_1)\cap C(\overline{\I}_{\sigma};\Xap)&\to L^p(\I_{\sigma},w_{\a};X_0),\\
G_{B,\lambda}(x,\cdot):\Xap\times L^p(\I_{\sigma},w_{\a};X_1)\cap C(\overline{\I}_{\sigma};\Xap)&\to L^p(\I_{\sigma},w_{\a};\g(H,X_{1/2})),
\end{align*}
be given by
\begin{align*}
F_{A,\lambda}(x,u)&:=\Psi_{\lambda}(\cdot,x,u)[(A(\cdot,x)-A(\cdot,u))u+F_{\Tr}(\cdot,u)-F_{\Tr}(\cdot,x)],\\
G_{B,\lambda}(x,u)&:=\Psi_{\lambda}(\cdot,x,u)[-(B(\cdot,x)-B(\cdot,u))u+G_{\Tr}(\cdot,u)-G_{\Tr}(\cdot,x)].
\end{align*}
Then for any $N\geq 1$ there exist constants $\tilde{C}_{\lambda},L_{T,\lambda}$ such that for all $\|x\|_{\Xap}<N$,
\begin{align*}
\|F_{A,\lambda}(x,u)\|_{L^p(\I_{\sigma},w_{\k};X_0)}&\leq \tilde{C}_{\lambda}\,,\\
\|G_{B,\lambda}(x,u)\|_{L^p(\I_{\sigma},w_{\k};\g(H,X_{1/2}))}&\leq \tilde{C}_{\lambda},\\
\|F_{A,\lambda}(x,u)-F_{A,\lambda}(x,v)\|_{L^p( \I_{\sigma},w_{\k};X_0)}
&\leq \tilde{L}_{\lambda,T}( \|u-v\|_{L^p(\I_{\sigma},w_{\a};X_1)}+\|u-v\|_{C(\overline{\I}_{\sigma};\Xap)}),\\
\|G_{B,\lambda}(x,u)-G_{B,\lambda}(x,v)\|_{L^p(\I_{\sigma},w_{\k};\g(H,X_{1/2}))}
&\leq \tilde{L}_{\lambda,T}( \|u-v\|_{L^p(\I_{\sigma},w_{\a};X_1)}+\|u-v\|_{C(\overline{\I}_{\sigma};\Xap)}),
\end{align*}
a.s. Moreover, for each $\varepsilon>0$ there exist $\bar{T}=\bar{T}(\varepsilon)>0$ and $\bar{\lambda}=\bar{\lambda}(\varepsilon)>0$ such that
$$\tilde{L}_{\lambda,T}<\varepsilon,$$
 for any $T<\bar{T}$, $\lambda<\bar{\lambda}$.
\end{lemma}

\begin{proof}
Recall that $L_{\Tr,n},L_{A,n},L_F,\tilde{L}_F$ are the constants defined in \ref{HAmeasur}, \ref{HFcritical}-\ref{HGcritical}. For simplicity we set $L:=L_{N+2}:=\max\{L_{\Tr,N+2},L_{A,N+2},\tilde{L}_F\}$, where $N$ is as in the statement. Moreover, as before $C_T>0$ denotes a constant which may change from line to line and satisfies $\lim_{T\downarrow 0} C_T=0$. We proof only the estimates for $F_{A,\lambda}$, since the other follows similarly. Again, as in Lemma \ref{l:truncationFG}, the above claimed estimates are pointwise with respect to $\om\in \O$. Thus, it is enough to consider the case $\sigma=T$.

To begin, we set
\begin{equation}
\label{eq:def_zeta_u}
\zeta_u:=\inf\{t\in [0,T]\,:\, \|u\|_{L^p(\I_{t},w_{\a};X_1)}+\sup_{s\in [0,t]}\|u(s)-x\|_{\Xap}>2\lambda \}\wedge T.
\end{equation}
Without loss of generality we can assume $\zeta_u\geq \zeta_{v}$. Firstly,
\begin{align*}
&\|F_{A,\lambda}(x,u)\|_{L^p(\I_T,w_{\k};X_0)}\\
&\stackrel{(i)}{\leq} \|A(\cdot,x)u-A(\cdot,u)u\|_{L^p(\I_{\zeta_u},w_{\k};X_0)}+\|F_{\Tr}(\cdot,u)-F_{\Tr}(\cdot,x)\|_{L^p(\I_{\zeta_u},w_{\a};X_0)}		\\
&\stackrel{(ii)}{\leq} (N+2) L\|u\|_{L^p(\I_{\zeta_u},w_{\k};X_1)}+
L\|u-x\|_{L^p(\I_{\zeta_u},w_{\k};\Xap)}\\
&\stackrel{(iii)}{\leq} 2 L\lambda (N+2+C_T)=: C_{\lambda,T},
\end{align*}
where in $(i)$ we used \eqref{eq:def_zeta_u} and $\Psi_{\lambda}(t,x,u)=0$ if $t\geq \zeta_u$. In $(ii)$ we used the assumption \ref{HAmeasur}, \ref{HFcritical} and $\sup_{t\in [0,\zeta_u]}\|u(t)-x\|_{\Xap}\leq N+2$ by \eqref{eq:def_zeta_u}. In $(iii)$ we used that $\|u\|_{L^p(\I_{\zeta_u},w_{\k};X_1)}+\sup_{s\in [0,\zeta_u]}\|u(s)-x\|_{\Xap}\leq 2\lambda$.

To prove the Lipschitz estimate we split the proof into two steps.

\textit{Step 1: Lipschitz estimate for $t\mapsto \Psi(t,x,u)( F_{\Tr}(t,u)-F_{\Tr}(t,0))$}. For simplicity, let us set $\tilde{F}_{\Tr}(u):=F_{\Tr}(\cdot,u)-F_{\Tr}(\cdot,0)$. As in the proof of Lemma \ref{l:truncationFG}, one has
\begin{align*}
&\|\Psi(\cdot,x,u) \tilde{F}_{\Tr}(u)- \Psi(\cdot,x,v)\tilde{F}_{\Tr}(v)\|_{L^p(\I_{T},w_{\a};X_0)}\\
&\leq \|(\Psi(\cdot,x,u)- \Psi(\cdot,x,v)) \tilde{F}_{\Tr}(u)\|_{L^p(\I_{T},w_{\a};X_0)}
+
 \|\Psi(\cdot,x,v)( \tilde{F}_{\Tr}(u)- \tilde{F}_{\Tr}(v))\|_{L^p(\I_{T},w_{\a};X_0)}\\
 &\leq \|(\Psi(\cdot,x,u)- \Psi(\cdot,x,v))\tilde{F}_{\Tr}(u)\|_{L^p(\I_{\zeta_u},w_{\a};X_0)}
+
 \|\tilde{F}_{\Tr}(u)- \tilde{F}_{\Tr}(v)\|_{L^p(\I_{\zeta_v},w_{\a};X_0)}.
\end{align*}
Note that
\begin{align*}
\| \tilde{F}_{\Tr}(u)- \tilde{F}_{\Tr}(v)\|_{L^p(\I_{\zeta_v},w_{\a};X_0)}
&\leq L\|u-v\|_{L^p(\I_{\zeta_v},w_{\a};\Xap)} \\ & \leq  L C_T \|u-v\|_{C(\overline{\I}_T;\Xap)} ,
\end{align*}
and
\begin{align*}
&\|(\Psi(\cdot,x,u)- \Psi(\cdot,x,v)) \tilde{F}_{\Tr}(u)\|_{L^p(\I_{\zeta_u},w_{\a};X_0)}\\
&\leq \sup_{t\in [0,\zeta_u]}|\Psi(t,x,u)- \Psi(t,x,v)|\|\tilde{F}_{\Tr}(u)\|_{L^p(\I_{\zeta_u},w_{\a};X_0)}\\
&\leq L \frac{1}{\lambda}
(\|u-v\|_{C(\overline{\I}_T;\Xap)}+\|u-v\|_{L^p(\I_{\zeta_u},w_{\a};X_1)})
\|u-x\|_{L^p(\I_{\zeta_u},w_{\a};\Xap)}
\\
&\leq 2 C_T L
(\|u-v\|_{C(\overline{\I}_T;\Xap)}+\|u-v\|_{L^p(\I_{\zeta_u},w_{\a};X_1)});
\end{align*}
where in the last inequality we used that $\|u-x\|_{L^p(\I_{\zeta_u},w_{\a};\Xap)}\leq 2 C_T \lambda$ by \eqref{eq:def_zeta_u}.

\textit{Step 2: Lipschitz estimate for $t\mapsto \Psi_{\lambda}(t,x,u)(A(t,x)u-A(t,u)u)$}. Writing
\begin{align*}
&\|\Psi_{\lambda}(\cdot,x,u)(A(\cdot,x)u-A(\cdot,u)u)-\Psi(\cdot,x,v)(A(\cdot,x)v-A(\cdot,v)v)\|_{L^p(\I_{T},w_{\a};X_0)}
\\ &\leq  \|(\Psi_{\lambda}(\cdot,x,u)-\Psi_{\lambda}(\cdot,x,v))((A(\cdot,x)-A(\cdot,u))u)\|_{L^p(\I_{T},w_{\a};X_0)}
\\ &\ \ +\|\Psi_{\lambda}(\cdot,x,v)((A(\cdot,v)-A(\cdot,x))(u-v))\|_{L^p(\I_{T},w_{\a};X_0)}
\\ & \ \ +\|\Psi(\cdot,x,v)((A(\cdot,v)-A(\cdot,u))u)\|_{L^p(\I_{T},w_{\a};X_0)} =: R_1+R_2+R_3.
\end{align*}
For $R_1$ note that
\begin{align*}
R_1&=\|(\Psi_{\lambda}(\cdot,x,u)-\Psi_{\lambda}(\cdot,x,v))((A(\cdot,x)-A(\cdot,u))u)\|_{L^p(\I_{\zeta_u},w_{\a};X_0)}\\
&\leq \sup_{t\in [0,\zeta_u]}|\Psi_{\lambda}(t,x,u)-\Psi_{\lambda}(t,x,v)|
\|(A(\cdot,x)-A(\cdot,u))u\|_{L^p(\I_{\zeta_u},w_{\a};X_0)}\\
\end{align*}
As before, for all $t\in [0,\zeta_u]$,
\[|\Psi_{\lambda}(t,x,u)-\Psi_{\lambda}(t,x,v)|\leq \frac{1}{\lambda}( \|u-v\|_{C(\overline{\I}_{T};\Xap)}+ \|u-v\|_{L^p(\I_{T},w_{\a};X_1)}).\]
Moreover,
\begin{align*}
\|(A(\cdot,x)-A(\cdot,u))u\|_{L^p(\I_{\zeta_u},w_{\a};X_0)} &\leq L\big(\int_0^{\zeta_u}\|u(t)-x\|_{\Xap}^p\|u(t)\|_{X_1}^p t^{\a}dt\big)^{\frac{1}{p}}\leq 4\lambda^2 L.
\end{align*}
Therefore,
\[R_1\leq 4CL \lambda ( \|u-v\|_{C(\overline{\I}_{T};\Xap)}+ \|u-v\|_{L^p(\I_{T},w_{\a};X_1)}).\]
Similarly, one gets
\begin{align*}
R_2 +R_3  & \leq L \lambda \|u-v\|_{L^p(\I_{T},w_{\a};X_1)} + L \lambda \|u-v\|_{C(\overline{\I}_{T};\Xap)}.
\end{align*}
Putting together the estimates in Step 1-2 the conclusion follows.
\end{proof}

In the proof of Theorem \ref{t:local_Extended} we need a further truncation. To this end, let $\xi_{\lambda}$ be as above. Then for $u\in C(\overline{\I}_T;\Xap)\cap L^p(\I_T,w_{\a};X_1)$, $n\geq 1$ and $t\in \I_T$ we set
\begin{equation}
\label{eq:Phin_def}
\Phi_n(t,u):=\xi_{n}\Big(\|u\|_{L^p(\I_t,w_{\a};X_1)}+ \sup_{s\in [0,t]}\|u(s)\|_{\Xap}\Big).
\end{equation}
As before, we fix $\om\in\O$, but we omit it from the notation.

\begin{lemma}
\label{l:truncation_FG_L}
Let \emph{\ref{HFcritical_weak}}-\emph{\ref{HGcritical_weak}} be satisfied. Let $T>0$ and let $\sigma$ be a stopping time with value in $[0,T]$. Let $\Phi_{n}$ be as in \eqref{eq:Phin_def}.
For any $n\geq 1$, let the maps
\begin{align*}
F_{L,n}: L^p(\I_{\sigma},w_{\a};X_1)\cap C(\overline{\I}_{\sigma};\Xap)&\to L^p(\I_{\sigma},w_{\a};X_0),\\
G_{L,n}: L^p(\I_{\sigma},w_{\a};X_1)\cap C(\overline{\I}_{\sigma};\Xap)&\to L^p(\I_{\sigma},w_{\a};\g(H,X_{1/2})),
\end{align*}
be given by
\begin{align*}
F_{L,n}(u)&:= \Phi_n(\cdot,u)(F_L(\cdot,u)-F_L(\cdot,0)),\\
 G_{L,n}(u)&:=\Phi_n(\cdot,u) (G_L(\cdot,u)-G_L(\cdot,0)).
\end{align*}
Then there exist constants $C_{n},C_T>0$ such that a.s.
\begin{align*}
\|F_{L,n}(\cdot,u)\|_{L^p(\I_{\sigma},w_{\k};X_0)}&\leq C_{n}\\
\|G_{L,n}(\cdot,u)\|_{L^p(\I_{\sigma},w_{\k};\g(H,X_{1/2}))}&\leq C_{n}\\
\|F_{L,n}(\cdot,u)-F_{F,n}(\cdot,v)\|_{L^p( \I_{\sigma},w_{\k};X_0)}
&\leq L_{F,n}'(\|u-v\|_{L^p(\I_{\sigma},w_{\a};X_1)}+\|u-v\|_{C(\overline{\I}_{\sigma};\Xap)}),\\
\|G_{L,n}(\cdot,u)-G_{F,n}(\cdot,v)\|_{L^p(\I_{\sigma},w_{\k};\g(H,X_{1/2}))}
&\leq  L_{G,n}'(\|u-v\|_{L^p(\I_{\sigma},w_{\a};X_1)}+\|u-v\|_{C(\overline{\I}_{\sigma};\Xap)});
\end{align*}
where $L_{F,n}':=3 L_{F,2n}+C_T\tilde{L}_{F,2n}$, $L_{G,n}':=3 L_{G,2n}+C_T\tilde{L}_{F,2n}$ and $\lim_{T\downarrow 0}C_T=0$.
\end{lemma}

\begin{proof}
The proof is similar to the one given in Lemmas \ref{l:truncationFG} and \ref{l:truncationAB}. For the sake of completeness we sketch the proof of the Lipschitz continuity of $F_{L,n}$. Since the estimates are pointwise with respect to $\om\in\O$, we may assume $\sigma=T$. Let $u,v\in C(\overline{\I}_T;\Xap)\cap L^p(\I_T,w_{\a};X_1)$ and set
\begin{equation}
\label{eq:def_truncation_F_L_lambda}
\lambda_u:=\inf\Big\{t\in [0,T]\,:\,\|u\|_{L^p(\I_t,w_{\a};X_1)}+\sup_{s\in [0,t]}\|u(s)\|_{\Xap}\geq 2n\Big\}\wedge T.
\end{equation}
A similar definition holds for $\lambda_v$. As usual, we assume $\lambda_u\geq  \lambda_v$. Therefore
\begin{align*}
\|F_{L,n}(\cdot,u)-F_{L,n}(\cdot,v)\|_{L^p(\I_T,w_{\a};X_0)}
&=
\|F_{L,n}(\cdot,u)-F_{L,n}(\cdot,v)\|_{L^p(\I_{\lambda_u},w_{\a};X_0)}\\
&\leq \|(\Phi_n(\cdot,u)-\Phi_n(\cdot,v))\wt{F}_L(\cdot,u)\|_{L^p(\I_{\lambda_u},w_{\a};X_0)}\\
&+\|\Phi_n(\cdot,v)(F_L(\cdot,u)-F_L(\cdot,v))\|_{L^p(\I_{\lambda_v},w_{\a};X_0)};
\end{align*}
where we have set $\wt{F}_L(\cdot,u):=F_L(\cdot,u)-F_L(\cdot,0)$. Since $\|\xi'\|_{L^{\infty}(\R_+)}\leq 1$, one has
\begin{align*}
&\|(\Phi_n(\cdot,u)-\Phi_n(\cdot,v))\wt{F}_L(\cdot,u)\|_{L^p(\I_{\lambda_u},w_{\a};X_0)}\\
&\leq \frac{1}{n} (\|u-v\|_{C(\overline{\I}_T;\Xap)} +\|u-v\|_{L^p(\I_T,w_{\a};X_1)})\|L_{F,2n} \|u\|_{X_1}+\tilde{L}_{F,2n}\|u\|_{X_0}\|_{L^p(\I_{\lambda_u},w_{\a})}\\
&\leq 2(\|u-v\|_{C(\overline{\I}_T;\Xap)} +\|u-v\|_{L^p(\I_T,w_{\a};X_1)})(L_{F,2n}+C_T\tilde{L}_{F,2n});
\end{align*}
where in the last inequality we used \eqref{eq:def_truncation_F_L_lambda}. Finally, since $\lambda_v\leq \lambda_u$,
\begin{align*}
\|\Phi_n(\cdot,v)&(F_L(\cdot,u)-F_L(\cdot,v))\|_{L^p(\I_{\lambda_v},w_{\a};X_0)}\leq  \|F_L(\cdot,u)-F_L(\cdot,v)\|_{L^p(\I_{\lambda_v},w_{\a};X_0)}\\
& \leq  L_{F,2n}\|u-v\|_{L^p(\I_{\lambda_v},w_{\a};X_0)}+ C_T\tilde{L}_{F,2n} \|u-v\|_{C(\overline{\I}_{\lambda_v};\Xap)}.
\end{align*}
The above estimates readily imply the claim.
\end{proof}

\subsection{Proofs of Theorems \ref{t:local} and \ref{t:local_Extended}-\ref{t:semilinear}}

With this preparation, we are ready to prove our first result concerning \eqref{eq:QSEE}.

\begin{proof}[Proof of Theorem \ref{t:local}]
To begin, we look to a suitable modification of \eqref{eq:QSEE}. More specifically, fix $w_0\in L^{p}_{\F_0}(\O;\Xap)$ and let us consider the following semilinear equation:
\begin{equation}
\label{eq:QSEE_tilde}
\begin{cases}
du+A(\cdot,u_0)u dt = (\tilde{F}_{\lambda}(u)+\tilde{f})dt + (B(\cdot,u_0)u+\tilde{G}_{\lambda}(u)+\tilde{g})dW_{H},	\\
u(0)=w_0;
\end{cases}
\end{equation}
on $[0,T]$, where
\begin{equation}
\label{eq:def_F_G_tilde}
\begin{aligned}
\tilde{F}_{\lambda}(u)&:=F_{c,\lambda}(u_0,u)+F_{A,\lambda}(u_0,u)+F_L(\cdot,u),\\
\tilde{G}_{\lambda}(u)&:=G_{c,\lambda}(u_0,u)+G_{A,\lambda}(u_0,u)+G_L(\cdot,u),\\
\tilde{f}&:=f+ F_c(\cdot,0)+F_{\Tr}(\cdot,u_0),\\
\tilde{g}&:=g+ G_c(\cdot,0)+G_{\Tr}(\cdot,u_0),
\end{aligned}
\end{equation}
where $F_{c,\lambda},G_{c,\lambda},F_{A,\lambda}$ and $G_{A,\lambda}$ are defined in Lemmas \ref{l:truncationFG} and \ref{l:truncationAB}. By \ref{HFcritical}-\ref{HGcritical} and the fact that $T<\infty$, it follows that
$\tilde{f}\in L^p_{\Progress}(\I_T\times \O,w_{\a};X_0)$ and $\tilde{g}\in L^p_{\Progress}(\I_T\times \O,w_{\a};\g(H,X_{1/2}))$. Let $\Sol:=\Sol_{(A(\cdot,u_0),B(\cdot,u_0))}$ be the solution operator associated to the couple $(A(\cdot,u_0),B(\cdot,u_0))\in \MRta$.

To study existence of strong solutions to \eqref{eq:QSEE_tilde}
let $\sigma$ be a stopping time with values in $[0,T]$ and consider
\begin{equation}
\label{eq:def_z_phi}
\z_{\sigma}:=L^p_{\Progress}(\O;\X({\sigma}))\cap L^p_{\Progress}( \I_{\sigma}\times \Omega,w_{\a};X_1)\cap L^p_{\Progress}(\Omega;C(\overline{\I}_{\sigma};\Xap)),
\end{equation}
equipped with the sum of the three norms. Note that the stopped space and norm were defined in Definition \ref{def:phi_spaces}. Recall that $\X(\sigma)$ was defined in \eqref{eq:X_space_c}. On $\z_{\sigma}$ we define an equivalent norm by
\begin{equation*}
\nnn\cdot\nnn_{\z_{\sigma}}=\|\cdot\|_{\z_{\sigma}} +M\|\cdot\|_{L^p(\O;L^p(\I_{\sigma},w_{\a};X_0))},
\end{equation*}
here $M\geq 0$ will be specified below. We shall study the map $\T_{w_0}$ defined on $\z_{\sigma}$ by
\begin{equation}
\label{eq:T_map_w_0}
\T_{w_0}(v):=\Sol(w_0,\tilde{F}_{\lambda}(v)+\tilde{f},\tilde{G}_{\lambda}(v)+\tilde{g}).
\end{equation}
For the sake of clarity, we divide the proof into several steps.

\textit{Step 1: There exist $M>0$, $\lambda^*>0$, $ T^*\in (0,T]$, $\varepsilon>0$ and $\alpha<1$ such that if $\max\{L_F,L_G\}\leq \varepsilon$, then for any stopping time $\sigma:\O\to [0,T^*]$ and any $w_0\in L^p_{\F_0}(\O;\Xap)$ one has $\T_{w_0}:\z_{\sigma}\to \z_{\sigma}$ and for all $v,w\in \z_{\sigma}$,}
\begin{equation}
\label{eq:estimate_T_w_0_epsilon}
\nnn \T_{w_0}(v)-\T_{w_0}(w)\nnn_{\z_{\sigma}}\leq \alpha\nnn u-w\nnn_{\z_{\sigma}}.
\end{equation}

In the following, we consider the case $p>2$, the case $p=2$ follows by replacing $\hz^{\d,p}(\I_{\sigma},w_{\a};X_{1-\d})$ by $C(\overline{\I}_{\sigma},X_{1/2})$ below.

Let $p>2$, and fix a stopping time $\sigma$ with values in $[0,T]$.
Fix $\d\in ((1+\a)/{p},1/2)$. Note that for $z\in L^p_{\Progress}(\I_{T}\times \O,w_{\a};X_1)\cap L^p_{\Progress}(\Omega;H^{\d,p}(\I_T,w_{\a};X_{1-\d}))$
\begin{align}\label{eq:embedZTH1}
\|z\|_{\z_T} \leq k_T (\|z\|_{L^p(\I_{T}\times \O,w_{\a};X_1)} + \|z\|_{L^p(\Omega;H^{\d,p}(\I_T,w_{\a};X_{1-\d}))}),
\end{align}
where $k_T$ is a constant which depends on $T$. Moreover, if $z\in L^p_{\Progress}(\I_{\sigma}\times \Omega,w_{\a};X_1)\cap L^p_{\Progress}(\Omega;\hz^{\d,p}(\I_{\sigma},w_{\a};X_{1-\d}))$, then
\begin{align}\label{eq:embedZTH2}
\|z\|_{\z_\sigma} \leq C_1 (\|z\|_{L^p(\I_{\sigma}\times \O,w_{\a};X_1)} + \|z\|_{L^p(\Omega;\hz^{\d,p}(\I_{\sigma},w_{\a};X_{1-\d}))}),
\end{align}
where the constant $C_1$ is independent of $T$. Both estimates \eqref{eq:embedZTH1} and \eqref{eq:embedZTH2} follow from Proposition \ref{prop:continuousTrace}, Lemma \ref{l:embeddings} and Remark \ref{r:embedding_with_T_dependence}.

By Proposition \ref{prop:initial_0} and \eqref{eq:embedZTH1} one has
\begin{align}\label{eq:Solw0}
\|\Sol(w_0,0,0)\|_{\z_{T}} \leq k_T \|w_0\|_{L^p(\Omega;\Xap)}.
\end{align}

Since $(A(\cdot,u_0),B(\cdot,u_0))\in \MRta$, Definition \ref{def:SMRz}, \eqref{eq:constants_SMR}, Proposition \ref{prop:propertiesR} and \eqref{eq:embedZTH2} give that for all
$\phi\in L^p_{\Progress}(\I_T\times\Omega,w_{\a};X_0)$ and $\psi\in L^p_{\Progress}(\I_T\times\Omega,w_{\a};\g(H,X_{1/2}))$,
\begin{equation}\label{eq:Solphipsi}
\begin{aligned}
&\|\Sol(0,\phi,\psi)\|_{\z_{\sigma}} \leq  \|\Sol(0,\one_{\ll 0,\sigma\rr }\phi,\one_{[0,\sigma]}\psi)\|_{\z_{T}}
\\ & \leq C_1K^{\det,\delta}\|\phi\|_{L^p(\Omega;L^p(\I_\sigma,w_{\a};X_0))} + C_1
K^{\stoc,\delta} \|\psi\|_{L^p(\Omega;L^p(\Omega;\I_\sigma,w_{\a};\g(H,X_{1/2})))},
\end{aligned}
\end{equation}
where $K^{\det,\delta} := K_{(A(\cdot,u_0),B(\cdot,u_0))}^{\det,\delta}$, $K^{\det,\delta} := K_{(A(\cdot,u_0),B(\cdot,u_0))}^{\stoc,\delta}$ and $C_1$ is as in \eqref{eq:embedZTH2}.

Next we show that $\T_{w_0}$ maps $\z_{\sigma}$ into itself. Let $v\in \z_{\sigma}$. By \eqref{eq:Solw0} and \eqref{eq:Solphipsi} we can write
\begin{align*}
\|\T_{w_0}(v)\|_{\z_{\sigma}}
&\leq \|\Sol(w_0,0,0)\|_{\z_{T}}+\|\Sol(0,\tilde{F}_{\lambda}(v)+\tilde{f},\tilde{G}_{\lambda}(v)+\tilde{g}\|_{\z_{\sigma}}\\
 & \leq k_T \|w_0\|_{L^p(\Omega;\Xap)} +  C_1 K^{\det,\delta} \|\tilde{F}_{\lambda}(v)+\tilde{f}\|_{L^p(\Omega;L^p(\I_{\sigma},w_{\a};X_0))} \\ & \qquad + C_1K^{\stoc,\delta}\|\tilde{G}_{\lambda}(v)+\tilde{g}\|_{L^p(\Omega;L^p(\I_\sigma,w_{\a};\g(H,X_{1/2})))}
\end{align*}
and the latter is finite by Lemmas \ref{l:truncationFG} and \ref{l:truncationAB}.

Moreover, for $v,w\in \z_{\sigma}$ by Proposition \ref{prop:propertiesR} we can write
\begin{align}\label{eq:Tdiffw0vw}
\T_{w_0}(v)-\T_{w_0}(w)=\Sol(0,\one_{\ll 0,\sigma\rr}(\tilde{F}_{\lambda}(v)-\tilde{F}_{\lambda}(w)),\one_{\ll 0,\sigma\rr}(\tilde{G}_{\lambda}(v)-\tilde{G}_{\lambda}(w)))
\end{align}
on $\ll 0,\sigma\rr$. The previous identity and \eqref{eq:Solphipsi} gives
\begin{equation}
\label{eq:estimate_step_1_local}
\begin{aligned}
&\|\T_{w_0}(v)-\T_{w_0}(w)\|_{\z_{\sigma}}
\\ & = \|\Sol(0,\one_{\ll 0,\sigma\rr}(\tilde{F}_{\lambda}(v)-\tilde{F}_{\lambda}(w)),\one_{\ll 0,\sigma\rr}(\tilde{G}_{\lambda}(v)-\tilde{G}_{\lambda}(w)))\|_{\z_{\sigma}}
\\ & \leq C_1K^{\det,\delta} \|\tilde{F}_{\lambda}(v)-\tilde{F}_{\lambda}(w)\|_{L^p(\Omega;L^p(\I_{\sigma},w_{\a};X_0))} \\ & \qquad + C_1K^{\stoc,\delta}\|\tilde{G}_{\lambda}(v)-\tilde{G}_{\lambda}(w)\|_{L^p(\Omega;L^p(\I_\sigma,w_{\a};\g(H,X_{1/2})))}
\\ & \leq C_1 [K^{\deter,\delta}(L_{\lambda,T}' +L_F)+K^{\stoc}(L_{\lambda,T}' +L_G)]\|v-w\|_{\z_{\sigma}}
\\ & \qquad + C_1(K^{\deter,\d}\tilde{L}_F+K^{\stoc,\d}\tilde{L}_G)\|v-w\|_{L^p(\Omega;L^p(I_{\sigma},w_{\a};X_0))},
\end{aligned}
\end{equation}
where the last estimate follows from Lemmas \ref{l:truncationFG} and \ref{l:truncationAB} and where we have set $L_{\lambda,T}'=L_{\lambda,T}+\tilde{L}_{\lambda,T}$.

Let $\varepsilon>0$ be such that if \eqref{eq:smallness_condition_nonlinearities_QSEE} holds, then
\begin{equation}
\label{eq:smallness_proof_local}
C_1  [K^{\deter,\d}L_F+K^{\stoc,\d}L_G]<1.
\end{equation}
By Lemmas \ref{l:truncationFG} and \ref{l:truncationAB} one can find $\wT$ and $\wl$ such that
\begin{equation}
\label{eq:alpha_primus}
C_1  [K^{\deter,\d}(L_F+L_{\lambda,T}')+K^{\stoc,\d}(L_G+L_{\lambda,T}')]:=\alpha'<1;
\end{equation}
for all $T\leq \wT$ and $\lambda\leq \wl$. To complete the proof we extend the argument in \cite[Theorem 4.5]{NVW11eq} to our setting. Set
$$M:=\frac{K^{\deter,\delta}\tilde{L}_F+K^{\stoc,\delta}\tilde{L}_G}{K^{\deter,\delta}L_F+K^{\stoc,\delta}L_G}.$$
With such a choice the inequality \eqref{eq:estimate_step_1_local} implies that
$$
\|\T_{w_0}(v)-\T_{w_0}(w)\|_{\z_{\sigma}}\leq\alpha' \nnn v-w\nnn_{\z_{\sigma}}.
$$
Applying Lemma \ref{l:solution_op_c_T} with $u$ given by \eqref{eq:Tdiffw0vw} we find
\begin{equation}
\label{eq:estimate_step_1_local_simple}
\begin{aligned}
&\|\T_{w_0}(v)-\T_{w_0}(w)\|_{L^p(\Omega;L^p(\I_\sigma,w_{\a};X_0))}\\
&\leq c_T \big[ \| \tilde{F}_{\lambda}(v)-\tilde{F}_{\lambda}(w)\|_{L^p(\I_{\sigma}\times \O,w_{\a};X_0)} +\|\tilde{G}_{\lambda}(v)-\tilde{G}_{\lambda}(w)\|_{L^p(\I_{\sigma}\times \O,w_{\a};X_0)}\big]\\
&\leq \tilde{c}_T \nnn v-w\nnn_{\z_{\sigma}};
\end{aligned}
\end{equation}
where the last step follows from Lemmas \ref{l:truncationFG} and \ref{l:truncationAB}, and
where $c_T,\tilde{c}_T>0$ and both tend to zero as $T\downarrow 0$. The claim follows from \eqref{eq:estimate_step_1_local} and \eqref{eq:estimate_step_1_local_simple} by choosing $T^*>0$ such that $M\tilde{c}_{T^*}<1-\alpha'$, $\lambda^*=\wl$ and $\alpha:=\alpha'+Mc_{T^*}<1$.

\textit{Step 2: Let $\lambda^*,T^*$ be as in Step 1. Then for each $w_0\in L^p_{\F_0}(\O;\Xap)$ the problem \eqref{eq:QSEE_tilde} has a unique strong solution $u_{w_0}\in \z_{T^*}$ on $\ll 0,T^*\rr$. Moreover, there exists a constant $C=C(T^*,\lambda^*)>0$ such that for all $w_0,w_1\in L^p_{\F_0}(\O;\Xap)$, one has}
\begin{equation}
\label{eq:continuity_QSEE_truncated}
\|u_{w_0}-u_{w_1}\|_{\z_{T^*}}
\leq C \|w_0-w_1\|_{L^p_{\F_0}(\O;\Xap)}.
\end{equation}
Applying Step 1 to $\sigma\equiv T^*$, we obtain that $\T_{w_0}:\z_{T^*}\to \z_{T^*}$ is a contraction. Therefore, by the Banach fixed point theorem there exists a unique $u_{w_0}\in \z_{T^*}$ such that $\T_{w_0}(u_{w_0}) = u_{w_0}$. From this we can conclude that $u_{w_0}$ is a strong solution to \eqref{eq:QSEE_tilde} on $\ll 0,T^*\rr$ (see Definition \ref{def:solution1} and \eqref{eq:T_map_w_0}).

It remains to prove \eqref{eq:continuity_QSEE_truncated}. The linearity of $\Sol$ shows that
$$
u_{w_0}-u_{w_1}=\T_{w_0}(u_{w_0})-\T_{w_1}(u_{w_1})
= \Sol(w_0-w_1,0,0)+ \T_0 (u_{w_0})-\T_0(u_{w_1}).
$$
Therefore, by \eqref{eq:Solw0} and \eqref{eq:estimate_T_w_0_epsilon},
\begin{align*}
\nnn u_{w_0}-u_{w_1}\nnn_{\z_{T^*}}&\leq \nnn\Sol(w_0-w_1,0,0)\nnn_{\z_{T^*}}+ \nnn \T_0 (u_{w_0})-\T_0(u_{w_1})\nnn_{\z_{T^*}}\\
&\leq \tilde{k}_{T^*} \|w_0-w_1\|_{L^p_{\F_0}(\O;\Xap)}+ \alpha \nnn u_{w_0}-u_{w_1}\nnn_{\z_{T^*}}.
\end{align*}
Since $\alpha<1$, the latter implies \eqref{eq:continuity_QSEE_truncated}.

\textit{Step 3: Let $(v,\tau)$ be a local solution to \eqref{eq:QSEE_tilde} with initial data $w_0 \in L^p_{\F_0}(\O;\Xap)$. Then $v=u_{w_0}$ on $\ll 0,\tau\wedge T^*\rro$}. Without loss of generality, we can assume that $\tau<T^*$. For $n\geq 1$ let
$$
\tau_n:=\inf\{t\in [0,\tau)\,:\,\|v\|_{\X(t)}+\|v-w_0\|_{C(\overline{\I}_t;\Xap)}+\|v\|_{L^p(\I_t,w_{\a};X_1)}\geq n\}
$$
and $\tau_n:=\tau$ if the set is empty. Then $(\tau_n)_{n\geq 1}$ is a localizing sequence for $(v,\tau)$.

Fix $n\geq 1$. Lemmas \ref{l:truncationFG} and \ref{l:truncationAB} ensure that $\one_{\ll 0,\tau_n\rr}(\tilde{F}_{\lambda}(v)+\tilde{f})\in L^p_{\Progress}(\I_T\times \O,w_{\a};X_0)$ and $\one_{\ll 0,\tau_n\rr}(\tilde{G}_{\lambda}(v)+\tilde{f})\in L^p_{\Progress}(\I_T\times \O,w_{\a};g(H,X_{1/2}))$. Moreover, by Proposition \ref{prop:propertiesR} one obtains
\begin{align*}
v&= \Sol(w_0,\one_{\ll 0,\tau_n\rr}(\tilde{f}+ \tilde{F}_{\lambda}(v)),\one_{\ll 0,\tau_n\rr}(\tilde{G}_{\lambda}(v)+\tilde{g})),\\
u_{w_0}&=\Sol(w_0,\one_{\ll 0,\tau_n\rr}(\tilde{f} + \tilde{F}_{\lambda}(u_{w_0})),\one_{\ll 0,\tau_n\rr}(\tilde{G}_{\lambda}(u_{w_0})+\tilde{g}));
\end{align*}
on $\ll 0,\tau_n\rr$. Using \eqref{eq:Tdiffw0vw} this implies that
\begin{align*}
\nnn u_{w_0}-v\nnn_{\z_{\tau_n}}
&=\nnn \Sol(0,\one_{\ll 0,\tau_n\rr}(\tilde{F}_{\lambda}(v)-\tilde{F}_{\lambda}(u_{w_0})),\one_{\ll 0,\tau_n\rr}(\tilde{G}_{\lambda}(v)-\tilde{G}_{\lambda}(u_{w_0})))\nnn_{\z_{\tau_n}}\\
 & = \nnn \T_{0}(u_{w_0})-\T_{0}(v)  \nnn_{\z_{\tau_n}}
\leq \alpha \nnn u_{w_0}-v\nnn_{\z_{\tau_n}};
\end{align*}
where in the last step we used \eqref{eq:estimate_T_w_0_epsilon}.  Since $\alpha<1$, we obtain that
$u_{w_0}=v$ on $\ll 0,\tau_n\wedge T^*\rr$. Since $n\geq 1$ was arbitrary, it follows that $u_{w_0}=v$ on $\ll 0,\tau\wedge T\rro$.

Steps 1-3 complete our treatment of \eqref{eq:QSEE_tilde}. Below we apply these results to study \eqref{eq:QSEE}.

\textit{Step 4: Let $\eta:=\lambda^*/2$. Then \eqref{eq:QSEE} has a strong solution $(v,\tau)$ with initial data $v_0\in L^{\infty}(\O;\Xap)$ and $\tau>0$ a.s.\, provided $v_0\in \B_{L^{\infty}_{\F_0}(\O;\Xap)}(u_0,\eta)$. In particular, this gives a strong solution $(u,\sigma)$ to \eqref{eq:QSEE} with $\sigma>0$ a.s.}

Step 1 ensures that \eqref{eq:QSEE_tilde} with initial data $v_0$ has a unique strongly progressively measurable solution $u_{v_0}$ if $\lambda= \lambda^*$ and $T= T^*$. Set
\begin{align*}
\label{eq:def_tau_local_1}
\tau&:=\inf\big\{t\in[0,T]\,:\, \|u_{v_0}\|_{\X(t)}+ \|u_{v_0}-u_0\|_{C(\overline{\I}_t;\Xap)} + \|u_{v_0}\|_{L^p(\I_t,w_{\a};X_1)}>{\lambda^*}/{2} \big\}.
\end{align*}
Since the maps $t\mapsto \|u_{u_0}\|_{\X(t)}$, $t\mapsto\sup_{s\in [0,t]}\|u_{u_0}(s)-v_0\|_{\Xap}$ are continuous and adapted, $\tau$ is a stopping time. Note that if $v_0\in \B_{L^{\infty}_{\F_0}(\O;\Xap)}(u_0,\eta)$, then $0<\tau$ a.s.

Setting $v:=u_{v_0}|_{\ll 0,\tau\rr}$, then a.s.\ for $t\in [0,\tau]$, one has
$$
\Theta_{\lambda^*}(t,u_0,v)=1, \qquad \Psi_{\lambda^*}(t,u_0,v)=1.
$$
Using the latter, by \eqref{eq:def_F_G_tilde} a.s.\ on $\ll 0,\tau\rr$
\begin{align*}
\tilde{F}_{\lambda^*}(v)&=A(\cdot,u_0)v-A(\cdot,v)v + F_c(\cdot,v)-F_c(\cdot,0)+F_{\Tr}(\cdot,v)-F_{\Tr}(\cdot,u_0)+F_L(\cdot,v),\\
\tilde{G}_{\lambda^*}(v)&=B(\cdot,u_0)v-B(\cdot,v)v + G_c(\cdot,v)-G_c(\cdot,0)+G_{\Tr}(\cdot,v)-G_{\Tr}(\cdot,u_0)+G_L(\cdot,v).
\end{align*}
Using this and \eqref{eq:QSEE_tilde}, it follows that $v$ is a strong solution to \eqref{eq:QSEE} on $\ll 0,\tau\rr$ with initial data $v_0$.

Next, we prove the continuity estimate claimed in \eqref{it:continuity_initial_data_Linfty} for the solutions just constructed. Let $(u,\sigma)$, $(v,\tau)$ be solutions of \eqref{eq:QSEE} constructed above with initial value $u_0,v_0$ respectively. Therefore, $u=u_{u_0}|_{\ll 0,\sigma\rr}$ and $v=u_{v_0}|_{\ll 0,\tau\rr}$.

Let $\nu:=\sigma\wedge \tau$, this implies that $u=u_{u_0}|_{\ll 0,\nu\rr}$, $v=u_{v_0}|_{\ll 0,\nu\rr}$ and
\begin{equation}
\label{eq:continuity_step_proof_local}
\begin{aligned}
\|u-v\|_{\z_\nu} & \leq \|u_{u_0}-u_{v_0}\|_{\z_{T^*}}\leq C \|u_0-v_0\|_{L^p_{\F_0}(\O;\Xap)};
\end{aligned}
\end{equation}
where in the last step we used \eqref{eq:continuity_QSEE_truncated}.

\textit{Step 5: \eqref{it:existence_maximal_local_Linfty} and the first part of \eqref{it:continuity_initial_data_Linfty} hold}. For the sake of clarity, we divide the proof of this step into two parts.

\textit{Step 5a: Uniqueness of the strong solution $(v,\tau)$ constructed in Step 4}. Recall that $(v,\tau)$ is a strong solution to \eqref{eq:QSEE} with initial data $v_0$ and satisfies $v = u_{v_0}$ on $\ll 0,\tau\rr$. Let $(w,\mu)$ be a local solution to \eqref{eq:QSEE} with initial data $v_0$. By Definition \ref{def:solution2}, it is enough to prove that $v=w$ on $\ll 0,\tau \wedge \mu\rro$.
We claim that
\begin{equation}
\label{eq:integrability_w_proof_uniqueness}
w\in \X(t)  \  \text{ a.s.\ for all }t\in [0,\mu).
\end{equation}
Let us first show that \eqref{eq:integrability_w_proof_uniqueness} implies the claim of Step 5a. Thus, suppose that \eqref{eq:integrability_w_proof_uniqueness} holds. Let $(\mu_n)_{n\geq 1}$ be a localizing sequence for $(w,\mu)$, and define the following stopping times
\begin{align*}
\mu_n^*&:=\inf\big\{t\in [0,\mu_n)\,:\,\|w\|_{\X(t)}+\|w-u_0\|_{C(\overline{\I}_t;\Xap)}+\|w\|_{L^p(\I_t,w_{\a};X_1)}>{\lambda^*}/{2}\big\},\\
\mu^*&:=\inf\big\{t\in [0,\mu)\,:\,\|w\|_{\X(t)}+ \|w-u_0\|_{C(\overline{\I}_t;\Xap)}+\|w\|_{L^p(\I_t,w_{\a};X_1)}>{\lambda^*}/{2}\big\},
\end{align*}
where $\lambda^*>0$ is as in Step 1 and where we set $\mu_n^* = \mu_n$ and $\mu^* = \mu$ if the set is empty. Let $n\geq 1$ be fixed. The argument used in Step 4 shows that $(w,\mu^*_n)$ is a local solution to \eqref{eq:QSEE_tilde} with initial data $v_0\in L^{\infty}_{\F_0}(\O;\Xap)$. Therefore, by Step 3
$w=u_{v_0}$ on $\ll 0,\mu_n^*\wedge \tau\rr$. Letting $n\uparrow \infty$ we find $v=w$ on $\ll 0,\mu^*\wedge \tau\rro$. From the latter equality, it follows that $\mu\wedge \tau=\mu^*\wedge \tau$ a.s. This proves the uniqueness of $(v,\tau)$.

Now we turn to the proof of \eqref{eq:integrability_w_proof_uniqueness}. To this end we set, for a.a.\ $(\om,t)\in [0,\mu) \times\O$,
$$
\mathcal{N}_w(t,\om):=\|F(\cdot,\om,w(\cdot,\om))\|_{L^{p}(0,t,w_{\a};X_0)}+\|G(\cdot,\om,w(\cdot,\om))\|_{L^{p}(0,t,w_{\a};\g(H,X_{1/2}))}.
$$
By Definitions \ref{def:solution1}-\ref{def:solution2} we have $\mathcal{N}_w(t)<\infty$ a.s.\ for all $t\in [0,\mu)$. Define a sequence of stopping times by
\begin{align*}
\nu_n&:=\inf\big\{t\in [0,\mu)\,:\,\mathcal{N}_w(t)+\|w-u_0\|_{C(\overline{\I}_t;\Xap)}+\|w\|_{L^p(\I_t,w_{\a};X_1)}>n\big\},
\end{align*}
where $\inf\emptyset:=\mu$. Then $\lim_{n\uparrow\infty}\nu_n=\mu$ a.s., and  therefore to prove \eqref{eq:integrability_w_proof_uniqueness} it is enough to show $w\in \X(\nu_n)$ a.s.\ for all $n\geq 1$. Note that for any $n\geq 1$,
\begin{equation}
\label{eq:integrability_w_nu_n}
w|_{\ll 0,\nu_n\rr}\in L^{\infty}(\O;C(\overline{\I}_{\nu_n};\Xap)\cap L^p(\I_{\nu_n}\times \O,w_{\a};X_1)),
\end{equation}
where we used that $u_0\in L^{\infty}(\O;\Xap)$ by assumption, and 
\begin{equation}
\begin{aligned}
\label{eq:integrability_F_w_composed_nu_n}
\one_{\ll 0,\nu_n \rr} F(\cdot,w)&\in L^p(\I_{T}\times \O,w_{\a};X_0),\\
\one_{\ll 0,\nu_n \rr} G(\cdot,w)&\in L^p(\I_{T}\times \O,w_{\a};\g(H,X_{1/2})).
\end{aligned}
\end{equation}
Since $(w,\mu)$ is a local solution to \eqref{eq:QSEE} and $\nu_n\leq \mu$ a.s.\ we have that $w|_{\ll 0,\nu_n\rro}$ is a strong solution to \eqref{eq:QSEE} on $\ll 0,\nu_n\rr$. Writing $A(\cdot,w)=A(\cdot,u_0)+ (A(\cdot,w)-A(\cdot,u_0))$ and $B(\cdot,w)=B(\cdot,u_0)+ (B(\cdot,w)-B(\cdot,u_0))$, one sees that $(w,\nu_n)$ is a strong solution to \eqref{eq:diffAB} on $\ll 0,\nu_n\rr$ with $(A,B)$ and $(f,g)$ replaced by $(A(\cdot,u_0),B(\cdot,u_0))$ and $(f_n^w,g_n^w)$, where
\begin{align*}
f_n^w
&:=\one_{\ll 0,\nu_n\rr}[(A(\cdot,u_0)-A(\cdot,w))+F(\cdot,w)+f],\\
g_n^w&:= \one_{\ll 0,\nu_n\rr} [(B(\cdot,w)-B(\cdot,u_0))+G(\cdot,w)+g],
\end{align*}
respectively. By \eqref{eq:integrability_w_nu_n}-\eqref{eq:integrability_F_w_composed_nu_n} and \ref{HAmeasur}, $f_n^w\in L^p_{\Progress}(\I_T\times \O,w_{\a};X_0)$ and $g_n^w\in L^p_{\Progress}(\I_T\times \O,w_{\a};\g(H,X_{1/2}))$. Since $(A,B)\in \MRta$, $w=\Sol_{(A(\cdot,u_0),B(\cdot,u_0))}(u_0,f_n^w,g_n^w)$ on $\ll 0,\nu_n\rr$ by Proposition \ref{prop:propertiesR}\eqref{it:stoppedsol}. Therefore, the last statement in Proposition \ref{prop:initial_0} ensures that for all $\delta\in (\frac{1+\a}{p},\frac{1}{2})$ and $n\geq 1$,
$$
w|_{\ll 0,\nu_n\rr} \in H^{\delta,p}(\I_{\nu_n},w_{\a};X_{1-\delta})\cap L^p(\I_{\nu_n},w_{\a};X_1)\hookrightarrow \X(\nu_n) \ \ \text{ a.s.,}
$$
where we used Lemma \ref{l:embeddings} for the embedding (see Remark \ref{r:embedding_with_T_dependence}).

\emph{Step 5b: Proof of the claim in Step 5}.
It remains to prove the existence of a maximal solution $(v,\tau)$ of \eqref{eq:QSEE} with initial data $v_0$ as in \eqref{it:continuity_initial_data_Linfty}. Let $\Xi$ be the set of all stopping time $\tau$ such that \eqref{eq:QSEE} admits a unique local solution on $[0,\tau)$ in the sense of Definitions \ref{def:solution1}-\ref{def:solution2} with initial value $v_0$. Then the above ensures that $\Xi$ is not empty. We claim that $\Xi$ is closed under pairwise maximization, i.e.\ if $\tau_0,\tau_1\in \Xi$, then $\tau_0\vee \tau_1\in \Xi$. A similar argument appears in \cite[Lemma 4.6]{Hornung}, but our setting is different. Let $(v_i,\tau_i)$ be the unique local solution to \eqref{eq:QSEE} with the same initial data and localizing sequences $(\tau_i^n)_{n\geq 1}$ for $i=0,1$. The uniqueness ensures that $v_0=v_1$ on $\ll 0,\tau_0\wedge \tau_1\rro$. Define the process $u^n:\ll 0,\tau_0^n\vee \tau_1^n\rr\to X_0$ given by
\begin{equation*}
u^n(t)=v_0(t\wedge \tau_0^n)+v_1(t\wedge \tau_1^n)- v_0(t\wedge \tau_0^n\wedge \tau_1^n).
\end{equation*}
Note that $u^n(t)=v_1(t)$ on $\{\tau_0^n\leq t\leq \tau_1^n\}$ and $u^n(t)=v_0(t)+v_1(\tau_1^n)-v_0(\tau_1^n)=v_0(t)$ on $\{\tau_1^n\leq t\leq \tau_0^n\}$. By definition $u^n$ is strongly progressively measurable and has the same regularity properties of $v_0$ and $v_1$ on $\ll 0,\tau_0^n\vee \tau_1^n\rr$. Letting $n\uparrow \infty$ we obtain a unique local solution $(v, \tau_0\vee \tau_1)$ and thus $\tau_0\vee \tau_1\in \Xi$.

By \cite[Theorem A.3]{KS98}, $\sigma:=\esssup \Xi$ exists, and there exists a sequence of stopping times $(\tau_n)_{n\geq 1}\subseteq\Xi$ such that $\tau_n\leq {\sigma}$, $\lim_{n\uparrow \infty}\tau_n={\sigma}$ a.s.\ and by the above uniqueness there exists a process $v:[0,\tau]\times \O \to X_0$ such that $u$ is a local solution to \eqref{eq:QSEE} on $\ll 0,\tau_n\rro$. In addition, $\tau>0$ a.s.\ by Step 4. This implies, the existence of a maximal local solution $(v,\tau)$ to \eqref{eq:QSEE} with initial value $v_0$ and localizing sequence $(\tau_n)_{n\geq 1}$. This finishes the proof of the first part of  \eqref{it:continuity_initial_data_Linfty} and in particular \eqref{it:existence_maximal_local_Linfty}.

\textit{Step 6: \eqref{it:regularity_data_Linfty}}. Let $(v,\tau^v)$ be the maximal solution to \eqref{eq:QSEE} with initial value $v_0$, where $v_0$ is as in \eqref{it:continuity_initial_data_Linfty}. Let $(\tau_n^v)_{n\geq 1}$ be a localizing sequence for $(v,\tau^v)$ with $\tau_n^v>0$ a.s. For each $n\geq 1$, set
\begin{equation}
\label{eq:def_stopping_times_regularity_estimates}
\tilde{\tau}_n^v:=\inf\{t\in [0,\tau_n^v)\,:\,\|v\|_{\X({t})} + \|v-v_0\|_{C(\overline{\I}_{t};\Xap)}+ \|v\|_{L^p(\I_t,w_{\a};X_1)}\geq n\},
\end{equation}
where we set $\tilde{\tau}_n^v = \tau_n^v$ if the set is empty. Thus, each $\tilde{\tau}_n^v$ is a stopping time and $\lim_{n\uparrow \infty}\tilde{\tau}_n^v=\tau^v$. Moreover, $\tau_n^v>0$ a.s. Let $\nu_n = \min\{\tilde{\tau}_n^u,\tilde{\tau}_n^v\}$.

Hypothesis \ref{HAmeasur} and \ref{HFcritical}--\ref{Hf} and Lemma \ref{l:F_G_bound_N} show that
\begin{align*}
f_n^v:=\one_{\ll 0,\nu_n\rr} [(A(\cdot,u_0)-A(\cdot,v))v+F(\cdot,v)+f]&\in L^p_{\Progress}(\I_T\times \Omega,w_{\a};X_0),\\
g_n^v:=\one_{\ll 0,\nu_n\rr} [(B(\cdot,v)-B(\cdot,v_0))u+G(\cdot,v)+g]&\in L^p_{\Progress}(\I_T\times \Omega,w_{\a};\g(H,X_{1/2})),
\end{align*}
for all $n\geq 1$. As in Step 5a, since $u$ and $v$ are strong solution to \eqref{eq:QSEE}, by Proposition \ref{prop:propertiesR}\eqref{it:stoppedsol} we have
\begin{equation*}
v=\Sol(v_0,f_n^v,g_n^v), \qquad\text{on}\;\;\ll 0,\nu_n\rr,
\end{equation*}
where $\Sol:=\Sol_{(A(\cdot,u_0),B(\cdot,u_0))}$. Since $(A(\cdot,u_0),B(\cdot,u_0))\in \MRta$, it follows from Proposition \ref{prop:initial_0} that
\begin{equation}
\label{eq:regularity_u_Step_6}
v\in \bigcap_{\theta\in [0,1/2)} L^p_{\Progress}(\O;H^{\theta,p}(\I_{\nu_n},w_{\a};X_{1-\theta})), \qquad
\forall\, n\geq 1.
\end{equation}
In particular, by Proposition \ref{prop:continuousTrace}\eqref{it:trace_with_weights_Xap}
\[v\in L^p(\Omega;C(\overline{I}_{\nu_n};\Xap)).\]
It remains to prove the instantaneously regularization effect. Let $\a>0$, by \eqref{eq:regularity_u_Step_6} and Definition \ref{def:phi_spaces}, for each $n\geq 1$ there exists $\tilde{v}_n \in L^p_{\Progress}(\O;H^{\delta,p}(\I_{T^*},w_{\a};X_{1-\delta})\cap L^p(\I_{T^*},w_{\a};X_1))$ such that $v|_{\ll 0,\tilde{\sigma}_n\rr}=\tilde{v}_n|_{\ll 0,\nu_n\rr}$ and for any $\varepsilon>0$,
\begin{align*}
\tilde{v}_n&\in  L^p_{\Progress}(\O;H^{\delta,p}(\I_{T^*},w_{\a};X_{1-\delta})\cap L^p(\I_{T^*},w_{\a};X_1))\hookrightarrow L^p_{\Progress}(\O;C([\varepsilon,T^*];\Xp)),
\end{align*}
where in the last inclusion we used Proposition \ref{prop:continuousTrace}\eqref{it:trace_without_weights_Xp} and the fact that $\delta>\frac{1+\a}{p}\geq \frac{1}{p}$ since $\a\geq 0$. The claim follows from the arbitrariness of $n\geq 1$ and $\varepsilon>0$. By taking $v=u$ this completes the proof of \eqref{it:regularity_data_Linfty}

\textit{Step 7: the second part of \eqref{it:continuity_initial_data_Linfty}}.
The cases $E\in \{L^p(\I_{\nu},w_{\a};X_{1}), C(\overline{\I}_{\nu};\Xap), \X(\nu)\}$ have already been considered in  \eqref{eq:continuity_step_proof_local}. It remains to consider $E = H^{\theta,p}(\I_{\nu},w_{\a};X_{1-\theta})$. Carefully checking the proofs of \eqref{eq:estimate_T_w_0_epsilon} and \eqref{eq:Solw0} one also obtains the latter case.

\textit{Step 8: \eqref{it:localization_Linfty} holds}. Let $(u,\sigma)$ and $(v,\tau)$ be as in the statement. Recall that $\Gamma:=\{u_0=v_0\}$. Without loss of generality we assume $\P(\Gamma)>0$.

Set $\tilde{\sigma}:=\one_{\Gamma}\sigma + \one_{\O\setminus \Gamma} \tau$ and $\tilde{u}:=\one_{\Gamma\times [0,\tau)} v+ \one_{(\O\setminus \Gamma)\times [0,\sigma)}u$. Then with the same argument used in the proof of Proposition \ref{prop:propertiesR}, one can check that $(\tilde{u},\tilde{\sigma})$ is a unique local solution to \eqref{eq:QSEE} since $u_0=v_0$ on $\Gamma$.

The maximality of $(u,\sigma)$ implies $\tau\leq \sigma$ on $\Gamma$ and
$$
u=\tilde{u}=v, \qquad \Gamma\times [0,\tau).
$$
Exchanging the role of $(u,\sigma)$ and $(v,\tau)$, one obtains also $\sigma\leq \tau$ on $\Gamma$ and $u=v$ on $\Gamma\times [0,\sigma)$. This implies the claim.
\end{proof}

Some remark may be in order.

\begin{remark}
\label{r:smallness}
Due to \eqref{eq:smallness_proof_local}, the argument used in Step 1 in the proof of Theorem \ref{t:local} ensures that instead of \eqref{eq:smallness_condition_nonlinearities_QSEE} we can assume
$$
C_1 (L_F K^{\deter,\d}_{(A(\cdot,u_0),B(\cdot,u_0))} +L_B K^{\stoc,\d}_{(A(\cdot,u_0),B(\cdot,u_0))})<1.
$$
Here $C_1$ is the constant in \eqref{eq:embedZTH2} and $\delta\in ((1+\a)/{p},1/2)$. Typically the above constants are difficult to compute. See \cite[Section 5]{NVW11eq} for examples in which explicit computations can be worked out.
\end{remark}

\begin{remark}
\label{r:general_version_F_L_G_L}
By analysing the argument in the above proof one can readily check that Theorem \ref{t:local} holds in case that the assumptions \ref{HFcritical}(i) and \ref{HGcritical}(i) are replaced by:
\begin{enumerate}[{\rm(1)}]
\item\label{it:F_G_L_generalized_mapping} For any stopping time $\mu:\O\to[0,T]$, one has
\begin{align*}
F_L:L^0_{\Progress}(\O;L^p(\I_{\mu},w_{\a};X_1)\cap C(\overline{\I}_{\mu};\Xap))
&\to L^0_{\Progress}(\O;L^p(\I_{\mu},w_{\a};X_0)), \\
G_L:L^0_{\Progress}(\O;L^p(\I_{\mu},w_{\a};X_1)\cap C(\overline{\I}_{\mu};\Xap))
&\to L^0_{\Progress}(\O;L^p(\I_{\mu},w_{\a};\g(H,X_{1/2}))).
\end{align*}
Moreover, there exist $\tilde{C},L_F,L_G,\tilde{L}_F,\tilde{L}_G>0$ such that for a.a.\ $\om\in\O$ and for all $u,v \in L^p(\I_{\mu},w_{\a};X_1)\cap C(\overline{\I}_{\mu};\Xap)$
\begin{align*}
\|F_L(\cdot,\om,u)\|_{L^p(\I_{\mu},w_{\a},X_0)}&\leq \tilde{C}(1+ \|u\|_{L^p(\I_{\mu},w_{\a};X_1)}+\|u\|_{ C(\overline{\I}_{\mu};\Xap)}) ,\\
\|G_L(\cdot,\om,u)\|_{L^p(\I_{\mu},w_{\a},\g(H,X_{1/2}))}&\leq \tilde{C}(1+ \|u\|_{L^p(\I_{\mu},w_{\a};X_1)}+\|u\|_{  C(\overline{\I}_{\mu};\Xap)}) ,\\
\|F_L(\cdot,\om,u)-F_L(\cdot,\om,v)\|_{L^p(\I_{\mu},w_{\a},X_0)}&\leq L_{F}  (\|u-v\|_{L^p(\I_{\mu},w_{\a};X_1)}+\|u-v\|_{ C(\overline{\I}_{\mu};\Xap)})\\
 &+ \tilde{L}_{F} \|u-v\|_{L^p(\I_{\mu},w_{\a};X_0)},\\
 \|G_L(\cdot,\om,u)-G_L(\cdot,\om,v)\|_{L^p(\I_{\mu},w_{\a},\g(H,X_{1/2}))}&\leq L_{G} (\|u-v\|_{L^p(\I_{\mu},w_{\a};X_1)}+\|u-v\|_{ C(\overline{\I}_{\mu};\Xap)})\\
 &+ \tilde{L}_{G} \|u-v\|_{L^p(\I_{\mu},w_{\a};X_0)}.
\end{align*}
\item\label{it:F_G_L_generalized_causality} For $\Gamma\in \{F_L,G_L\}$ and all stopping times $\nu\in [0,\mu]$ a.s., $\one_{[0,\nu]}\Gamma(\cdot,u)=\one_{[0,\nu]}\Gamma(\cdot,v)$ provided $\one_{[0,\nu]}u=\one_{[0,\nu]}v$ and $u,v\in L^0_{\Progress}(\O;L^p(\I_{\mu},w_{\a};X_1)\cap C(\overline{\I}_{\mu};\Xap))$.
\end{enumerate}
To see that \eqref{it:F_G_L_generalized_mapping}-\eqref{it:F_G_L_generalized_causality} are sufficient to prove Theorem \ref{t:local} it is enough to note that only \eqref{it:F_G_L_generalized_mapping} and \eqref{it:F_G_L_generalized_causality} are needed in Step 1 (resp.\ 5) to prove existence (resp.\ uniqueness). The other steps hold without any changes.
\end{remark}

Next, we prove Theorem \ref{t:local_Extended}.

\begin{proof}[Proof of Theorem \ref{t:local_Extended}]
We start by collecting some useful facts. To begin, let
\begin{align*}
\xi:=\inf\{t\in [0,T]\,:\,\|f\|_{L^p(\I_t,w_{\a};X_0)}+\|g\|_{L^p(\I_t,w_{\a};\g(H,X_{1/2}))}\geq 1\}.
\end{align*}
By \ref{Hf'}, $\xi$ is an stopping time, $\xi>0$ a.s.\ and
$$
\one_{\ll 0,\xi\rr }f\in L^p_{\Progress}(\I_T\times \Omega,w_{\a};X_0), \qquad
 \one_{\ll 0,\xi\rr }g\in L^p_{\Progress}(\I_T\times \Omega,w_{\a};\g(H,X_{1/2})).
$$
Moreover, let $n\geq 1$ be fixed and define $\Gamma_n:=\{\|u_0\|_{\Xap}\leq n\}\in \F_0$.
Recall that $(u_{0,n})_{n\geq 1}$ satisfies \eqref{eq:approximating_sequence_initial_data}. Finally, let $F_{L,n},G_{L,n}$ be as in Lemma \ref{l:truncation_FG_L}. The same lemma implies that $F_{L,n}$ and $G_{L,n}$ verify the condition in Remark \ref{r:general_version_F_L_G_L} for
\[L_F= 3 L_{F,2n}+C_T\tilde{L}_{F,2n},\quad \tilde{L}_F=0,\quad  L_G= 3 L_{G,2n}+C_T\tilde{L}_{G,2n},\quad \tilde{L}_G=0,\]
where $\lim_{T\downarrow 0}C_T=0$. For $n\geq 1$, set $F_n = F_{L,n}+F_{c}+F_{\Tr}$, $G_n = G_{L,n}+G_{c}+G_{\Tr}$.
By \eqref{eq:approximating_sequence_initial_data}, $\sup_{\O}\|u_{0,n}\|_{\Xap}<\infty$. Let $R_n\geq 1$ be the smallest integer satisfying
\begin{equation}
\label{eq:def_C_n_revision_stage}
R_n\geq \sup_{\O}\|u_{0,n}\|_{\Xap}.
\end{equation}
Theorem \ref{t:local} and Remarks \ref{r:smallness}-\ref{r:general_version_F_L_G_L} ensure the existence of a maximal local solution $(u_n,\sigma_n)$
to \eqref{eq:QSEE} with $(u_0, f, g, F, G)$ replaced by
\[(u_{0,n}, \one_{\ll 0,\xi\rr }f+F_L(t,0),\one_{\ll 0,\xi\rr }g+F_L(t,0),F_{R_n},G_{R_n})\]
provided
\begin{equation}
\label{eq:smallness_n_proof}
3C_1 (L_{F,2 R_n} K^{\deter,\d}_{(A(\cdot,u_{0,n}),B(\cdot,u_{0,n}))}+L_{B,2R_n} K^{\stoc,\d}_{(A(\cdot,u_{0,n}),B(\cdot,u_{0,n}))})<1,\ \forall\,n\geq 1,
\end{equation}
where $C_1>0$ is the constant in the embedding of Lemma \ref{l:embeddings} and does not depend on $T>0$. Note that choosing $\varepsilon_n>0$ suitably we obtain \eqref{eq:smallness_n_proof}. Recall that the constants $K^{\deter,\d}_{(A(\cdot,u_{0,n}),B(\cdot,u_{0,n}))},K^{\stoc,\d}_{(A(\cdot,u_{0,n}),B(\cdot,u_{0,n}))}$ are defined in \eqref{eq:constants_SMR} and $\delta\in ((1+\a)/p,1/2)$ is arbitrary.

For the sake of clarity, we split the proof into several steps.

\textit{Step 1: Existence of a local solution to \eqref{eq:QSEE} if $u_0\in L^0_{\F_0}(\O;\Xap)$}. Let $(u_n,\sigma_n)$ as above. Then let us define the following stopping time
$$
\tau_n:=\inf\Big\{t\in [0,\sigma_n)\,:\,\|u_n\|_{L^p(\I_t,w_{\a};X_1)}+\sup_{s\in [0,t]}\|u_n\|_{\Xap}\geq  2R_n\Big\}
$$
and $\tau_n:=\sigma_n$ if the set is empty. Then reasoning as in Step 4 in the proof of Theorem \ref{t:local} one immediately sees that $(u_n,\sigma_n\wedge \tau_n)$ verifies \eqref{eq:QSEE} with initial data $u_{0,n}$. Note that $u_{0,n}$ has norm less than $R_n$ (see \eqref{eq:def_C_n_revision_stage}), and therefore $\tau_n>0$ a.s. Thus, $\sigma_n\wedge \tau_n>0$ a.s.

Set $\sigma_n':=\sigma_n\wedge \tau_n$. Let $(\Lambda_n)_{n\geq 1}\subseteq\F_0$ be defined as $\Lambda_1:=\Gamma_1$ and $\Lambda_n:=\Gamma_{n+1}\setminus \Gamma_n$ for each $n>1$. Define $(u,\sigma)$ as $\sigma:=\sigma_n'$ on $\Lambda_n$, and $u=u_n$ on $\Lambda_n\times [0,\sigma_n')$. Since $(u_n,\sigma_n')$ is a local solution to \eqref{eq:QSEE} with initial data $u_{0,n}$, one can check that $(u,\sigma)$ is a local solution to \eqref{eq:QSEE}.

\textit{Step 2: Uniqueness of $(u,\sigma)$}.
Let $(v,\mu)$ be another local solution to \eqref{eq:QSEE}. Set
$$
\mu_n:=\inf\Big\{t\in [0,\mu)\,:\,\|v\|_{L^p(\I_t,w_{\a};X_1)}+\sup_{s\in [0,t]}\|v\|_{\Xap}\geq 2R_n\Big\},
$$
and $\tau_n=\mu$ if the set is empty. Then $(\one_{\Lambda_n}v,\one_{\Lambda_n}\mu_n)$ is a local solution to \eqref{eq:QSEE} with data $(\one_{\Lambda_n}u_{0,n},\one_{\Lambda_n}(\one_{\ll 0,\xi\rr }f+F_L(t,0)),\one_{\Lambda_n}(\one_{\ll 0,\xi\rr }g+G_L(t,0)))$ and $F=F_{R_n},G=G_{R_n}$.
At this stage, the conclusion follows as in Step 5 in the proof of Theorem \ref{t:local}.

\textit{Step 2: Existence of a maximal local solution}. Similarly as in Step 6 in the proof of Theorem \ref{t:local}, consider the set $\Xi$ of all stopping time $\tau$ such that \eqref{eq:QSEE} admits a unique local solution. Steps 1-2 ensure that $\Xi$ is not empty, and that there exists $\tau\in \Xi$ such that $\tau>0$ a.s. The rest of the proof follows as Step 5 in the proof of Theorem \ref{t:local}.

\textit{Step 3: Regularity}. The claimed regularity follows as in Step 6 in the proof of Theorem \ref{t:local} by replacing $\tilde{\tau}_n^v$ in \eqref{eq:def_stopping_times_regularity_estimates} by $\one_{\Gamma_n}\tilde{\tau}_n^v$.
\end{proof}

\begin{remark}
\label{r:smallness_n}
As in Remark \ref{r:smallness} the proof of Theorem \ref{t:local_Extended} shows that the condition \eqref{eq:smallness_condition_nonlinearities_QSEE_extended} can be replaced by \eqref{eq:smallness_n_proof}.
\end{remark}

\begin{proof}[Proof of Theorem \ref{t:semilinear}]
\eqref{it:semilinear_u_L_infty}: Follows by Theorem \ref{t:local}.

\eqref{it:semilinear_u_L_p}: The proof is similar to the one proposed for Theorem \ref{t:local}. Indeed, we may replace the truncations in Step 1 by
\begin{equation*}
\begin{aligned}
\tilde{F}_{\lambda}(u)&:=F_{c,\lambda}(u_0,u)+F_L(\cdot,u)+F_{\Tr}(\cdot,u),\\
\tilde{G}_{\lambda}(u)&:=G_{c,\lambda}(u_0,u)+G_L(\cdot,u)+G_{\Tr}(\cdot,u),\\
\tilde{f}&:=f+ F_c(\cdot,0)+F_{\Tr}(\cdot,u_0),\\
\tilde{g}&:=g+ G_c(\cdot,0)+G_{\Tr}(\cdot,u_0).
\end{aligned}
\end{equation*}
Due to Remark \ref{r:truncationFG_n} and the assumptions, the assertion of Lemma \ref{l:truncationFG} still holds. Now one can repeat the proof of Theorem \ref{t:local} literally.

\eqref{it:semilinear_u_L_0}: This follows from Theorem \ref{t:local} and the fact that the constants $\varepsilon_n$ do not depend on $n\geq 1$ (see  Remark \ref{r:smallness_n}).
\end{proof}

\section{Applications to semilinear SPDEs with gradient noise}
\label{s:semilinear_gradient}
In this section we will consider semilinear SPDEs on $X_0 = H^{s,q}$ which can be written in the form
\begin{equation}
\label{eq:semilinearabstract}
\begin{cases}
du +A(\cdot) u dt= F(\cdot,u)dt + (G(\cdot,u)+ B(\cdot) u)dW_{H},\qquad t\in \I_T,\\
u(0)=u_0,
\end{cases}
\end{equation}
which is a special case of the setting considered in Theorem \ref{t:semilinear}. In Subsections \ref{ss:conservative_RD}-\ref{ss:reaction_diffusion_gradient_nonlinearities} we take $H=\ell^2$ and in Subsection \ref{ss:Burgers_semilinear} $H=L^2(\Tor)$.

In the next section we motivate this setting and explain which class of operator pairs $(A,B)$ we will be considering.

\subsection{Introduction and motivations}
\label{ss:introduction_motivation_semilinear}
In this section we study a large class of nonlinear second order equations with gradient noise. Such equations are commonly known as stochastic--reaction diffusion equations, but they also include the filtering equation see \cite[Section 8]{Kry} and Allen-Cahn equations  \cite{BBP17_2, BBP17,FS19,RW13}. Allen-Cahn equations will be further investigated in Subsection \ref{ss:Allen_Cahn_stochastic_potentials}.

Stochastic reaction--diffusion equations have been extensively studied in the last decades. Nonlinear reaction--diffusion models arise in many scientific areas such as chemical reactions, pattern-formation, population dynamics. Stochastic perturbations of such models can model thermal fluctuations, uncertain determinations of the parameters and non-predictable forces acting on the system. For the sake of completeness let us mention some works on the deterministic case \cite{CDW09,F66,QS19,W86} and for the stochastic case one may consult \cite{CCLR07,C03,CR05,DHI13,EKHL18,FC13,F91,G19,HJT18,W19,WX18} and the references therein.

To the best of our knowledge, the results presented below are new. The reader can compare our results with the results in \cite[Section 3]{CriticalQuasilinear} in the deterministic framework.

In this section we analyse second order stochastic PDEs in non-divergence form with gradient noise:
\begin{equation}
\label{eq:semilinear_reaction_diffusion_prototype}
\begin{cases}
du +\A u dt= f(u,\nabla u) dt + \sum_{n\geq 1}\big(\bb_n u +g_n(u)\big) d{w}_{t}^n,  &\text{ on } \Dom,\\
u(0)=u_0, &\text{ on } \Dom.
\end{cases}
\end{equation}
here $(w_t^n:t\geq 0)_{n\geq 1}$ denotes a sequence of independent standard Brownian motions and $u:\I_T\times \O\times \Dom\to \R$ is the unknown process. Moreover, the differential operators $\A,\bb_n$ for each $x\in \Dom$, $\om\in\O$, $t\in(0,T)$ are given by
\begin{equation}
\label{eq:SND_quasi_AB_n_Def}
\begin{aligned}
(\A(t,\om)u)(t,\om,x)&:=-\sum_{i,j=1}^d a_{ij}(t,\om)\partial_{ij}^2 u(x),\\
(\bb_n(t,\om)u)(t,\om,x)&:=\sum_{j=1}^d b_{jn}(t,\om)\partial_j u(x).
\end{aligned}
\end{equation}
Lower order terms in the previous differential operators can be added (see Subsection \ref{sss:lower_order_terms}).  The assumptions on $f, g_n$ will be specified below.

In the applications of Theorem \ref{t:semilinear}, the following splitting arises naturally:
\begin{itemize}
\item $\Dom=\R^d$ or $\Dom=\Tor^d$;
\item $\Dom$ is a smooth domain in $\R^d$.
\end{itemize}
We will only consider $\R^d$ in detail since $\Tor^d$ can be treated by the same arguments. This will be done in Sections \ref{ss:conservative_RD}, \ref{ss:semilinear_reaction_diffusion}, and \ref{ss:reaction_diffusion_gradient_nonlinearities} using the maximal regularity result of Lemma \ref{l:SMR_semilinear_PDEs} below. In Section \ref{ss:stochastic_reaction_diffusion_domains} we will comment on domains and boundary conditions of Dirichlet and Neumann type. However, these results will only be formulated under suboptimal smallness assumptions on the $b_{jn}$.

To avoid the need for too many subcases, we will only consider $d\geq 2$. However, under suitable conditions on the parameters the case $d=1$ could also be included in most examples.

Next we introduce the function spaces which will be needed below. As usual, for $q\in (1,\infty)$ and $k\geq 1$, we denote by $W^{k,q}(\R^d)$ the set of all $f\in L^q(\R^d)$ such that $\partial^{\alpha} f\in L^q(\R^d)$ for any $\alpha\in \N^d_0$ such that $|\alpha|\leq k$ endowed with the natural norm. Let $\Four$ be the Fourier transform on $\R^d$. Then for any $s\in \R$ and $q\in (1,\infty)$ we set $H^{s,q}(\R^d)=\{f\in \Sch'(\R^d)\,:\,\Four^{-1}((1+|\cdot|^2)^{s/2}\Four(f))\in L^q(\R^d)\}$ with its natural norm. For $s\in \R$, $q\in (1,\infty)$ and $p\in [1,\infty]$, we define Besov spaces through real interpolation:
$$
B^{s}_{q,p}(\RR)=(H^{s_0,q}(\RR),H^{s_1,q}(\RR))_{\theta,p},
$$
where $s_0<s<s_1$ and $\theta\in (0,1)$ are chosen in such a way that $s=s_0(1-\theta)+s_1\theta$. We refer to \cite[Chapter 6]{BeLo} for alternative descriptions of the Besov spaces $B^{s}_{q,p}(\RR)$. For $s\in \R$ and $q\in (1,\infty)$, we denote the Sobolev-Slobodeckij spaces by $W^{s,q}(\RR):=B^{s}_{q,q}(\RR)$.

Recall from \cite[Theorem 6.4.5]{BeLo} that
\begin{equation}
\label{eq:H_complex_interpolation}
[H^{s_0,q}(\RR),H^{s_1,q}(\RR)]_{\theta}=H^{s,q}(\RR),\qquad s:=(1-\theta)s_0+\theta s_1.
\end{equation}
For the sake of simplicity, sometimes, we write $H^{s,q}$ instead of $H^{s,q}(\RR)$ (and analogously for other spaces) if no confusion seems possible.

The following will be a standing assumption in this section:
\begin{assumption}\label{ass:SND}
Suppose that one of the two conditions hold:
\begin{itemize}
\item $q\in [2,\infty)$, $p\in(2,\infty)$ and $\a\in [0,\frac{p}{2}-1)$;
\item $q=p=2$ and $\a=0$.
\end{itemize}
Assume the following two conditions on $a_{ij}$ and $b_{in}$:
\begin{enumerate}[{\rm(1)}]
\item \label{it:SND_a_b}
The functions $a_{ij}:(0,T)\times \O \to	\R$ and
$b_{jn}:(0,T)\times \O \to\R$ are progressively measurable. Moreover, there exists $K>0$ such that
$$|a_{ij}(t,\om)|+\|(b_{jn}(t,\om))_{n\geq 1}\|_{\ell^2}\leq K,\quad \text{ a.a.\ }\om\in \O,\text{ for all } t\in \I_T.$$
\item \label{it:SND_ellipticity}
There exists $\epsilon>0$ such that  a.s.\ for all $\xi\in \R^d$, $t\in \I_T$,
$$
\sum_{i,j=1}^d \Big( a_{ij}(t)- \frac{1}{2}\sum_{n\geq 1} b_{in}(t)b_{jn}(t)\Big) \xi_i\xi_j  \geq \epsilon |\xi|^2.
$$
\end{enumerate}
\end{assumption}

The following result will be employed several times.
\begin{lemma}
\label{l:SMR_semilinear_PDEs}
Let the Assumption \ref{ass:SND} be satisfied. Let $X_0 = H^{s,q}(\R^d)$ and $X_1 = H^{s+2,q}(\R^d)$ with $s\in \R$. Let
$A:\I_T\times \O\to \calL(X_1;X_0)$ and $B:\I_T\times \O\to \calL(X_1,\g(\ell^2,X_{\frac12}))$ be given by
$$
A(t)u:=\A(t)u,\quad (B(t)u)_n:=\bb_n(t)u, \ \ \ \ n\geq 1,
$$
where $\A,\bb_n$ are as in \eqref{eq:SND_quasi_AB_n_Def}. Then $(A,B)\in \MRta$ (see Definition \ref{def:SMRz}).
\end{lemma}

\begin{proof}
Since the coefficients $a_{ij},b_{jn}$ are $x$-independent by applying $(1-\Delta)^{s/2}$ to the equation, one can reduce to the case $s=0$. Now the result follows from \cite[Theorem 5.3]{VP18}.
\end{proof}

\subsection{Conservative stochastic reaction diffusion equations}
\label{ss:conservative_RD}
In this subsection we study the following differential problem for the unknown process $u:[0,T]\times \Omega\times\RR\to \R$,
\begin{equation}
\label{eq:semilinear_reaction_diffusion}
\begin{cases}
du -\A u dt= \div( f(\cdot,u)) dt+\sum_{n\geq 1} (\bb_n u+ g_{n}(\cdot,u))
dw_t^{n},  & \text{on } \RR,\\
u(0)=u_0, & \text{on } \RR;
\end{cases}
\end{equation}
for $t\in \I_T$. Here $\A,\bb_n$ are as in \eqref{eq:SND_quasi_AB_n_Def}.

A formal integration of \eqref{eq:semilinear_reaction_diffusion} shows that the system preserves mass under the flow, i.e.\ $\E\int_{\R^d}u(x,t)dx=\E\int_{\R^d}u_0(x)dx$. This feature is very important from a modelling point of view, since $u$ (typically) represents the mass of chemical reactants. This motivates the name `conservative reaction-diffusion equations'.

We study \eqref{eq:semilinear_reaction_diffusion} under the following assumption:
\begin{assumption}\label{ass:RDC}
The maps $f:\I_T\times\O\times \RR\times \R\to \R^d$, $g:=(g_n)_{n\geq 1}:\I_T\times\O\times\RR\times  \R\to \ell^2$ are $\Progress\otimes \Borel(\R^d)\otimes \Borel(\R)$-measurable with $f(\cdot, 0) = 0$ and $g(\cdot, 0) = 0$. Moreover, there exist $h>1$ and $C>0$ such that a.s.\ for all $t\in \I_T$, $z,z'\in \R$ and $x\in \RR$,
$$
|f(t,x,z)-f(t,x,z')|+\|g(t,x,z)-g(t,x,z')\|_{\ell^2}\leq C (|z|^{h-1}+|z'|^{h-1})|z-z'|.
$$
\end{assumption}

Typical examples of $f$ and $g$ which satisfies Assumption \ref{ass:RDC} are:
\begin{align}
\label{eq:consevation_choice_nonlinearities}
f(x,u)&=\tilde{f}(x)|u|^{h-1}u,  & g(x,u) &=\tilde{g}(x)|u|^{h-1}u,  \qquad & h\in (1, \infty),
\end{align}
where $\tilde{f}\in  L^{\infty}_{\Progress}((0,T)\times\Omega\times \RR;\R^d)$ and $\tilde{g}\in L^{\infty}_{\Progress}((0,T)\times\Omega\times \RR;\ell^2)$. The condition $f(\cdot, 0) = 0$ and $g(\cdot, 0) = 0$ can be weakened to a decay condition in the $x$-variable.

We study \eqref{eq:semilinear_reaction_diffusion} directly in `the almost very weak setting', i.e.\ in $X_0:=H^{-1-s,q}$ with $s\in [0,1)$ (cf. \cite[Subsection 4.5]{CriticalQuasilinear}). This will give us additional flexibility in the treatment of \eqref{eq:semilinear_reaction_diffusion}. The weak setting can be derived by setting $s=0$.

\subsubsection{Almost very weak setting}
\label{ss:reaction_diffusion_divergence_semilinear_almost_weak}
Let $s\in [0,1)$ and let $q\in [2,\infty)$. The differential problem \eqref{eq:semilinear_reaction_diffusion} can be rephrased as a stochastic evolution equation of the form \eqref{eq:semilinearabstract} with $X_0:=H^{-1-s,q}$ and $X_1:=H^{1-s,q}$.
Here
\begin{align*}
A(t)u&=\A(t)u, &  B(t)u &=(\bb_n (t)u)_{n\geq 1},
\\ F(t,u)&=\div(f(t,\cdot,u)), & G(t,u)&=(g_n(t,\cdot,u))_{n\geq 1}
\end{align*}
for $u\in H^{1-s,q}$. We say that $(u,\sigma)$ is a maximal local solution to \eqref{eq:semilinear_reaction_diffusion} if $(u,\sigma)$ is a maximal local solution to \eqref{eq:semilinearabstract} in the sense of Definition \ref{def:solution2}.

To show local existence for \eqref{eq:semilinear_reaction_diffusion} we employ Theorem \ref{t:semilinear}. By Lemma \ref{l:SMR_semilinear_PDEs} it is enough to look at suitable bounds for the non-linearities $F,G$. To this end, let us start by looking at $F$. By Assumption \ref{ass:RDC}, it follows that
\begin{equation}
\label{eq:reaction_diffusion_estimate_F}
\begin{aligned}
\|F(\cdot, u)-F(\cdot, v)\|_{H^{-1-s,q}}	&\stackrel{(i)}{\lesssim} \|F(\cdot, u)-F(\cdot, v)\|_{H^{-1,r}}\\
& \lesssim\|f(\cdot, u)-f(\cdot, v)\|_{L^{r}}\\
&\lesssim \Big\|(|u|^{h-1}+|v|^{h-1})|u-v|\Big\|_{L^r}\\
&\stackrel{(ii)}{\lesssim} (\|u\|_{L^{h r}}^{h-1}+\|v\|_{L^{h r}}^{h-1})\|u-v\|_{L^{h r}}\\
&\stackrel{(iii)}{\lesssim} (\|u\|_{H^{\theta,q}}^{h-1}+\|v\|_{H^{\theta,q}}^{h-1})\|u-v\|_{H^{\theta,q}};
\end{aligned}
\end{equation}
where in $(i)$ we used the Sobolev embedding with $r$ defined by $-1-\frac{d}{r} = -1-s -\frac{d}{q}$, in $(ii)$ the H\"{o}lder inequality with exponent $h,\frac{h}{h-1}$ and in $(iii)$ the Sobolev embedding \eqref{eq:HD_embedding} and $\theta -\frac{d}{q} = -\frac{d}{hr}$. Note that $r\in (1,\infty)$ since $q\geq 2$, $d\geq 2$ and $s\in [0,1)$ by assumption. Note that $\theta$ has to satisfy $\theta\in (0,1-s)$ in order to obtain a space in between $X_0$ and $X_1$. Combining the identities we obtain
$$
\frac{d}{q}-\theta= \frac{d}{h r}=\frac{1}{h}\Big(\frac{d}{q}+s\Big)
\Rightarrow \theta= \frac{d}{q}\Big(1-\frac{1}{h}\Big)-\frac{s}{h}.
$$
Therefore, to ensure that $\theta\in (0,1-s)$ we assume\footnote{Here we have set $1/0:=\infty$.}
\begin{equation}
\label{eq:reaction_diffusion_1_limitation_q}
 \frac{d(h-1)}{h-s(h-1)}<q<\frac{d(h-1)}{s}.
\end{equation}
Since $s\neq 1$ and $h>1$ the set of $q$ which satisfies \eqref{eq:reaction_diffusion_1_limitation_q}  is not-empty. If \eqref{eq:reaction_diffusion_1_limitation_q} holds, due to \eqref{eq:H_complex_interpolation} one has $H^{\theta,q}=[H^{-1-s,q},H^{1-s,q}]_{\beta_1}$ where
\begin{equation}
\label{eq:beta_1_conservative_reaction_diffusion}
\beta_1=\frac{1+\theta+s}{2}=\frac{1}{2}\Big[\Big(\frac{d}{q}+s\Big)\Big(1-\frac{1}{h}\Big)+1\Big]\in (0,1).
\end{equation}
To check the condition \ref{HFcritical} we may split the discussion into three cases:
\begin{enumerate}[{\rm(1)}]
\item If $1-\frac{1+\a}{p}>\beta_1$, by Remark \ref{r:non_linearities}\eqref{it:non_linearities_continuous_trace}, \ref{HFcritical} follows by setting $F_{\Tr}(t,u):=\div(f(t,\cdot,u))$ and $F_{L}\equiv F_c\equiv 0$.
\item If $1-\frac{1+\a}{p}= \beta_1$, by \eqref{eq:reaction_diffusion_estimate_F} and Remark \ref{r:non_linearities}\eqref{it:non_linearities_varphi_equal_to_beta}, \ref{HFcritical} follows by setting $F_L\equiv F_{\Tr}\equiv 0$, $F_{c}(t,u):=\div(f(t,\cdot,u))$, $m_F=1$, $\rho_1=h-1$ and $\varphi_1=\beta_1$.
\item If $1-\frac{1+\a}{p}<\beta_1$ we set $F_{c}(t,u):=\div(f(t,\cdot,u))$ and $F_{L}\equiv F_{\Tr}\equiv 0$. As in the previous item we set $m_F=1$, $\rho_1=h-1$ and $\varphi_1=\beta_1$. By \eqref{eq:reaction_diffusion_estimate_F} it remains to check the condition \eqref{eq:HypCritical}. In this situation, \eqref{eq:HypCritical} becomes,
\begin{equation}
\label{eq:reaction_diffusion_div_critical_weights}
\frac{1+\a}{p}\leq \frac{\rho_1+1}{\rho_1}(1-\beta_1)
=\frac{1}{2}\frac{h}{h-1}-\frac{1}{2}\Big(\frac{d}{q}+s\Big).
\end{equation}
Note that the assumption $\a\geq 0$ implies
\begin{equation}
\label{eq:reaction_diffusion_divergence_1_limitation_p}
\frac{1}{p}+\frac{d}{2q}+\frac{s}{2}\leq \frac{h}{2(h-1)}.
\end{equation}
Since $d/2q+s/2<h/[2(h-1)]$ (thanks to the lower bound in \eqref{eq:reaction_diffusion_1_limitation_q}) the above inequality is always verified for $p$ sufficiently large.
\end{enumerate}

It remains to estimate $G$. To this end we can reasoning as in \eqref{eq:reaction_diffusion_estimate_F}. First, note that $X_{1/2}=H^{-s,q}$ (see \eqref{eq:H_complex_interpolation}) and let $r,\theta$ be as in \eqref{eq:reaction_diffusion_estimate_F}. By Assumption \ref{ass:RDC} one has
\begin{equation}
\label{eq:estimate_G_reaction_diffusion_div}
\begin{aligned}
\|G(\cdot, u)-G(\cdot, v)\|_{\g(\ell^2;H^{-s,q})}
&\lesssim \|G(\cdot, u)-G(\cdot, v)\|_{\g(\ell^2;L^r)}\\
&\stackrel{(i)}{\eqsim} \|G(\cdot, u)-G(\cdot, v)\|_{L^r(\ell^2)}\\
&\lesssim\|(|u|^{h-1}+|v|^{h-1})|u-v|\|_{L^r}\\
&\lesssim (\|u\|_{H^{\theta,q}}^{h-1}+\|v\|_{H^{\theta,q}}^{h-1})\|u-v\|_{H^{\theta,q}};
\end{aligned}
\end{equation}
where in $(i)$ we used the identification $\g(\ell^2,L^r)=L^r(\ell^2): = L^r(\R^d;\ell^2)$ (see \eqref{eq:gammaidentity}). The previous considerations show that $G$ verifies \ref{HGcritical} under the same assumptions on $F$.

Therefore, Theorem \ref{t:semilinear} gives the following result.

\begin{theorem}
\label{t:reaction_diffusion_divergence_local}
Let Assumptions \ref{ass:SND} and \ref{ass:RDC} be satisfied and $d\geq 2$. Let $s\in [0,1)$. Assume \eqref{eq:reaction_diffusion_1_limitation_q}. Let $\beta_1$ be as in \eqref{eq:beta_1_conservative_reaction_diffusion}. Assume that one of the following conditions is satisfied
\begin{itemize}
\item $1-(1+\a)/p\geq \beta_1$;
\item $1-(1+\a)/p<\beta_1$ and \eqref{eq:reaction_diffusion_div_critical_weights} holds.
\end{itemize}
Then for each $u_0\in L^0_{\F_0}(\O;B^{1-s-2(1+\a)/p}_{q,p}(\R^d))$ there exists a maximal local solution $(u,\sigma)$ to \eqref{eq:semilinear_reaction_diffusion}. Moreover, there exists a localizing sequence $(\sigma_n)_{n\geq 1}$ such that a.s.\ for all $n\geq 1$
$$
u\in L^p(\I_{\sigma_n},w_{\a};H^{1-s,q})\cap C(\overline{\I}_{\sigma_n};B^{1-s-2(1+\a)/p}_{q,p})\cap C((0,\sigma_n];B^{1-s-2/p}_{q,p}).
$$
\end{theorem}

\subsubsection{Critical spaces for \eqref{eq:semilinear_reaction_diffusion}}
\label{sss:critical_div_reaction_diffusion}
In this subsection we study the existence of critical spaces for \eqref{eq:semilinear_reaction_diffusion}.

To motivate the setting let $f,g_n$ be as in \eqref{eq:consevation_choice_nonlinearities} with $\tilde{f}$, $\tilde{g}\in \ell^2$ constant w.r.t. It will turn out that our abstract notion of critical spaces as introduced in Remark \ref{r:non_linearities} \eqref{it:critical_beta_grater_than} is consistent with the natural scaling of \eqref{eq:semilinear_reaction_diffusion}-\eqref{eq:consevation_choice_nonlinearities}.
First consider the deterministic setting, i.e.\ $b_{jn}\equiv \tilde{g}_n\equiv 0$. If $u$ is a (local smooth) solution to \eqref{eq:semilinear_reaction_diffusion}-\eqref{eq:consevation_choice_nonlinearities} on $(0,T)\times \R^d$, then $u_{\lambda}(x,t):=\lambda^{1/[2(h-1)]}u(\lambda t,\lambda^{1/2}x)$ is a (local smooth) solution to \eqref{eq:semilinear_reaction_diffusion} on $(0,T/\lambda)\times \R^d$ for each $\lambda>0$. Note that the map $ u\mapsto u_{\lambda}$ induces a mapping on the initial data $u_0$ given by $u_0\mapsto u_{0,\lambda}$ where $u_{0,\lambda}(x):=\lambda^{1/[2(h-1)]}u_0(\lambda^{1/2}x)$ for $x\in \R^d$.

In the theory of PDEs a function space is called critical for \eqref{eq:semilinear_reaction_diffusion}-\eqref{eq:consevation_choice_nonlinearities} (in absence of noise) if it is invariant under the above mapping $u_0\mapsto u_{0,\lambda}$. An example of a Besov spaces which is (locally) invariant under this scaling is $B^{d/q-1/(h-1)}_{q,p}$ for $q,p\in (1,\infty)$. This can be made precise by looking at the so-called homogeneous version of such spaces. Indeed, one has
\begin{equation}
\label{eq:scaling_B_conservative_reaction_diffusion}
\begin{aligned}
\|u_{0,\lambda}\|_{\dot{B}^{d/q-1/(h-1)}_{q,p}}
&\eqsim \lambda^{1/[2(h-1)]}(\lambda^{1/2})^{d/q-1/{(h-1)}-d/q}
\|u_0\|_{\dot{B}^{d/q-1/(h-1)}_{q,p}}\\
&=\|u_0\|_{{\dot{B}^{d/q-1/(h-1)}_{q,p}}};
\end{aligned}
\end{equation}
where the implicit constants do not depend on $\lambda>0$. It will turn out that this space appears naturally when equality in \eqref{eq:reaction_diffusion_div_critical_weights} is reached. This observation was made in \cite[Sections 2.3 and 3-6]{CriticalQuasilinear} for many PDEs.

Next consider the stochastic problem. At least formally, we can show that if $u$ is a (local smooth) solution to \eqref{eq:semilinear_reaction_diffusion}, then $u_{\lambda}$ is a (local smooth) solution to \eqref{eq:semilinear_reaction_diffusion} where the $(w_{t}^n\,:\,t\geq 0)_{n\geq 1}$ is replaced by the sequence of independent Brownian motions $(b_{t,\lambda}^n\,:\,t\geq 0)_{n\geq 1}:=(\lambda^{-1/2}w_{\lambda t}^n\,:\,t\geq 0)_{n\geq 1}$. To see this, let $t\in (0,T)$ and let us look at the strong formulation of \eqref{eq:semilinear_reaction_diffusion} as in Definition \ref{def:solution1}. As we have seen before, under the map $u\mapsto u_{\lambda}$ all the deterministic integrals have all the same scaling, therefore it is enough to study one of them. For instance,
\begin{align*}
\int_0^{t/\lambda}\Delta u_{\lambda}(s,x)ds= \lambda^{\frac{1}{2(h-1)}}\int_0^t \Delta u(s',\lambda^{1/2} x)ds'.
\end{align*}
Such scaling agrees with the scaling of the stochastic integrals,
\begin{align}
\nonumber \int_0^{t/\lambda} |u_{\lambda}(s,x)|^{h-1}u_{\lambda}(s,x) db_{s,\lambda}^n
&= \int_0^{t/\lambda} \lambda^{\frac{1}{2(h-1)}} |u(\lambda s,\lambda^{1/2}x)|^{h-1} u(\lambda s,\lambda^{1/2}x) dw_{\lambda s}^n
\\ \label{eq:stochscalinggrowth} &= \lambda^{\frac{1}{2(h-1)}} \int_0^{t} |u(s,\lambda^{1/2}x)|^{h-1} u(s,\lambda^{1/2}x) dw_{s}^n,
\end{align}
where $n\geq 1$ is fixed. The same holds for the stochastic integral for the $b$-term. Therefore, $u_{\lambda}$ is a solution to \eqref{eq:semilinear_reaction_diffusion} with a scaled noise.

After these formal calculations, let us turn to our setting. We will analyse when equality in \eqref{eq:reaction_diffusion_div_critical_weights} can be allowed. We begin by looking at the case $p\in (2,\infty)$. Note that $\a\in [0,\frac{p}{2}-1)$ if and only if $\frac{1+\a}{p}\in [\frac{1}{p},\frac{1}{2})$ and due to \eqref{eq:reaction_diffusion_divergence_1_limitation_p} to ensure the existence of a weight $\a$ which realizes equality in \eqref{eq:reaction_diffusion_div_critical_weights} we have to assume
\begin{equation}
\label{eq:critical_reaction_diffusion_divergence_bound}
\frac{1}{2}\frac{h}{h-1}-\frac{1}{2}\Big(\frac{d}{q}+s\Big)<\frac{1}{2}.
\end{equation}
Simple computations show that the previous is verified if and only if
\begin{equation}
\label{eq:critical_reaction_diffusion_divergence_limitation_q}
h\geq \frac{1+s}{s}\qquad \text{ or }\qquad \Big[ h<\frac{1+s}{s} \ \ \text{and}  \ \ q<\frac{d(h-1)}{1-s(h-1)}\Big].
\end{equation}
If \eqref{eq:reaction_diffusion_divergence_1_limitation_p} and \eqref{eq:critical_reaction_diffusion_divergence_limitation_q} hold, then we set
\begin{equation}\label{eq:criticalreacdiff}
\a_{\crit}=\frac{p}{2}\Big(\frac{h}{h-1}-\frac{d}{q}-s\Big)-1.
\end{equation}
Then  $\a_{\crit}\in [0,\frac{p}{2}-1)$ and the corresponding critical space is
\begin{equation}
\label{eq:conservation_reaction_diffusion_critical_spaces}
\Xapcrit=B^{1-s-2\frac{1+\a_{\crit}}{p}}_{q,p}(\R^d)=B^{\frac{d}{q}-\frac{1}{h-1}}_{q,p}(\R^d).
\end{equation}
Note that the above space coincides with the one appearing in the above discussion. Moreover, the space does not depend on the parameter $s>0$, and depends on $p$ only through the microscopic parameter. The independence on $s>0$ is in accordance with the independence of the scale founded in the deterministic case for \eqref{eq:QSEE} without noise and bilinear non-linearities, see \cite[Section 2.4]{CriticalQuasilinear}.

It remains to consider the case $p=q=2$ and $\a=0$. We expect that a similar space appears also in this case. Indeed, the condition \eqref{eq:reaction_diffusion_div_critical_weights} implies the identity
\begin{equation}
\label{eq:semilinear_h_identity}
h=\frac{2+d+2s}{d+2s}>1.
\end{equation}
Note that the lower bound in \eqref{eq:reaction_diffusion_1_limitation_q} is automatically verified and the upper bound in \eqref{eq:reaction_diffusion_1_limitation_q} is equivalent to $d>2s^2/(1-s)$. Therefore, in the case $p=q=2$, $\a=0$ and $h$ as in \eqref{eq:semilinear_h_identity}, the trace space for \eqref{eq:semilinear_reaction_diffusion} becomes
$$
\Xap=B^{-s}_{2,2}(\R^d) =B^{\frac{d}{2}-\frac{1}{h-1}}_{2,2}(\R^d)=H^{\frac{d}{2}-\frac{1}{h-1}}(\R^d).
$$
In the case $s=0$ one has $h=(2+d)/d=2/d+1$ and condition \eqref{eq:reaction_diffusion_1_limitation_q} is satisfied.

Let us summarize what we have proved in the following:
\begin{theorem}
\label{t:Reaction_diffusion_div_critical}
Let Assumptions \ref{ass:SND} and \ref{ass:RDC} be satisfied and $d\geq 2$. Let $s\in [0,1)$ and let one of the following conditions be satisfied:
\begin{itemize}
\item $p,q\in (2,\infty)$, \eqref{eq:reaction_diffusion_1_limitation_q}, \eqref{eq:reaction_diffusion_divergence_1_limitation_p} and \eqref{eq:critical_reaction_diffusion_divergence_limitation_q} hold;
\item $p=q=2$, $d>2s^2/(1-s)$, and $h$ is as in \eqref{eq:semilinear_h_identity}.
\end{itemize}
Let $\a_{\crit}$ be as in \eqref{eq:criticalreacdiff}. Then for each
$$
u_0\in L^0_{\F_0}(\O;B^{\frac{d}{q}-\frac{1}{h-1}}_{q,p}(\R^d))
$$
there exists a maximal local solution $(u,\sigma)$ to \eqref{eq:semilinear_reaction_diffusion}. Moreover, there exists a localizing sequence $(\sigma_n)_{n\geq 1}$ such that a.s.\ for all $n\geq 1$
$$
u\in L^{p}(\I_{\sigma_n},w_{\a_{\crit}};H^{1-s,q}(\R^d))
\cap C(\overline{\I}_{\sigma_n};B^{\frac{d}{q}-\frac{1}{h-1}}_{q,p}(\R^d))\cap
C((0,\sigma_n];B^{1-s-\frac{2}{p}}_{q,p}(\R^d)).
$$
\end{theorem}
Note that the space $B^{\frac{d}{q}-\frac{1}{h-1}}_{q,p}(\R^d)$ becomes larger as $p$ tends to $\infty$. Therefore, for $u_0$ as above and any $\delta<1-s$, there exists a maximal local solution $(u,\sigma)$ to \eqref{eq:semilinear_reaction_diffusion} such that $u \in C((0,\sigma_n];B^{\delta}_{q,\infty}(\R^d))$ a.s. In particular, if $s=0$, then for all $\delta<1$ we find a maximal local solution to \eqref{eq:semilinear_reaction_diffusion} such that $u\in C((0,\sigma_n];B^{\delta}_{q,\infty}(\R^d))$ a.s.\ Bootstrapping arguments related to such regularization phenomena will be investigated in the papers \cite{AV19_QSEE_2,AV19_QSEE_3}.

Let us conclude this section by giving an example which illustrates the usefulness of $s\in (0,1)$.
\begin{example}
Let $d=3$ and $h=2$. The restriction on $q\geq 2$ becomes
\begin{equation}
\label{eq:bound_conservative_reaction_diffusion_special_case}
\frac{3}{2-s}<q<\min\Big\{\frac{3}{s},\frac{3}{1-s}\Big\}, \qquad s\in [0,1).
\end{equation}
Therefore, in the weak setting $s=0$ one needs $q\in [2,3)$, and the critical space $B^{\frac{3}{q}-1}_{q,p}(\R^d)$ has strictly positive smoothness. To admit critical spaces with negative smoothness, we need $s>0$. Indeed, note that the choice $s=1/2$ optimizes the right hand-side of \eqref{eq:bound_conservative_reaction_diffusion_special_case}. Therefore, with $s=1/2$ we can allow $q\in [2,6)$ and thus we have a larger class of critical spaces which goes down to smoothness $-\frac{1}{2}$.
\end{example}

Also the space $L^{d(h-1)}(\R^d)$ is invariant under the scaling $u_0\mapsto u_{0,\lambda}$. From the previous result we obtain the following corollary.
\begin{corollary}
\label{cor:conservative_local_existence_L_r}
Let Assumptions \ref{ass:SND} and \ref{ass:RDC} be satisfied and $d\geq 2$. Let $h>1+\frac{2}{d}$, $q:=d(h-1)$ and $p\in (q,\infty)$. Then there exists $\bar{s}>0$ such that for all $s\in (0,\bar{s})$ and
$$
u_0\in L^0_{\F_0}(\O;L^{d(h-1)}(\R^d))
$$
there exists a maximal local solution to \eqref{eq:semilinear_Dirichlet_conditions}, and there exists a localizing sequence $(\sigma_n)_{n\geq 1}$ such that for any $n\geq 1$ and a.s.
$$
u\in L^p(\I_{\sigma_n},w_{\a_{\crit}};H^{1-s,q}(\R^d))\cap C(\overline{\I}_{\sigma_n};B^{0}_{q,p}(\R^d))\cap C((0,\sigma_n];B^{1-s-\frac{2}{p}}_{q,p}(\R^d)),
$$
where $\a_{\crit}$ is given by \eqref{eq:criticalreacdiff}.
\end{corollary}

Recall that $p$ in Theorem \ref{t:Reaction_diffusion_div_critical} can be chosen as large as one wants.

\begin{proof}
Since $h\geq 1+\frac{2}{d}$, $q\geq 2$. One can check that there exists $s_1>0$ such that \eqref{eq:reaction_diffusion_1_limitation_q} and \eqref{eq:critical_reaction_diffusion_divergence_limitation_q} hold for $q=d(h-1)$ and $s\in (0,s_1)$. Moreover, for $q=d(h-1)$, there exists $s_2>0$, such that \eqref{eq:reaction_diffusion_divergence_1_limitation_p} holds for all $p\in (2,\infty)$ and $s\in (0,s_2)$. Set $s:=\min\{s_1,s_2\}$. Thus, Theorem \ref{t:Reaction_diffusion_div_critical} ensures the existence of a maximal local solution to \eqref{eq:semilinear_reaction_diffusion} for any $s\in (0,\bar{s})$ and $u_0\in L^0_{\F_0}(\O;B^0_{q,p}(\R^d))$ with the required regularity. To conclude, it remains to recall that $L^{q}(\R^d)\hookrightarrow B^0_{q,p}(\R^d)$, since $p\geq q$.
\end{proof}

By choosing $s$ small enough such that $1-s-2/p>0$, the solution $u$ to \eqref{eq:semilinear_reaction_diffusion} provided by Corollary \ref{cor:conservative_local_existence_L_r}, instantaneously regularizes in space, i.e.\ $u\in C(\I_{\sigma_n};B^{1-s-\frac{2}{p}}_{q,p}(\R^d))\hookrightarrow C(\I_{\sigma_n};L^q(\R^d))$ a.s.\ for all $n\geq 1$.

\subsection{Stochastic reaction diffusion equations}
\label{ss:semilinear_reaction_diffusion}
 In this subsection we study local existence for the following non-conservative reaction-diffusion equation for the unknown $u: [0,T]\times \Omega\times\RR\to \R$,
\begin{equation}
\label{eq:semilinear_reaction_diffusion_l_m}
\begin{cases}
du +\A u dt= f(\cdot, u) dt + \sum_{n\geq 1}( \bb_n u +g_n(\cdot, u)) dw_t^n,  & \text{on } \RR,\\
u(0)=u_0, & \text{on }  \RR,
\end{cases}
\end{equation}
where $\A,\bb_n$ are as in \eqref{eq:SND_quasi_AB_n_Def}. In this subsection we assume that
\begin{assumption}\label{ass:RD}
The maps $f:\I_T\times\O\times \RR\times \R\to \R$, $g:=(g_n)_{n\geq 1}:\I_T\times\O\times\RR\times  \R\to \ell^2$ are $\Progress\otimes \Borel(\R^d)\otimes \Borel(\R)$-measurable with $f(\cdot, 0) = 0$ and $g(\cdot, 0) = 0$. Moreover, there exist $m,h>1$ and $C>0$ such that a.s.\ for all $t\in \I_T$, $x\in \RR$ and $z,z'\in \R$
\begin{align*}
|f(t,x,z)-f(t,x,z')|&\leq C (|z|^{m-1}+|z'|^{m-1})|z-z'|,\\
\|g(t,x,z)-g(t,x,z')\|_{\ell^2}&\leq C (|z|^{h-1}+|z'|^{h-1})|z-z'|.
\end{align*}
\end{assumption}

Typical choices for such non-linearities are:
\begin{equation}
\label{eq:semilinear_reaction_diffusion_l_m_power_law}
f(u)=|u|^{m-1}u ,\qquad g_n(\cdot,u)=\tilde{g}_n |u|^{h-1}u, \  \  \  n\geq 1,
\end{equation}
for some $\tilde{g}=(\tilde{g}_n)_{n\geq 1}\in L^{\infty}_{\Progress}(\I_T\times\Omega\times \RR;\ell^2)$.

To make the results more readable we choose to analyse \eqref{eq:semilinear_reaction_diffusion_l_m} only in the weak setting. The interested reader can adapt the argument below and the one given in Subsection \ref{ss:reaction_diffusion_divergence_semilinear_almost_weak}, to study \eqref{eq:semilinear_reaction_diffusion_l_m} in the almost weak setting. As we have seen before, the latter choice gives local existence in a wider set of critical spaces. This will be presented in Section \ref{ss:Allen_Cahn_stochastic_potentials} for the stochastic Allen--Cahn equation.

Again we will focus on the setting of critical spaces. Some noncritical cases could be included by simpler methods. Part of this is covered in the quasilinear setting in Subsection \ref{ss:quasilinear_divergence}.

\subsubsection{Weak setting}
\label{sss:reformulation}
As in Subsection \ref{ss:reaction_diffusion_divergence_semilinear_almost_weak} we rewrite \eqref{eq:semilinear_reaction_diffusion} in the form \eqref{eq:semilinearabstract} by setting $X_0:=H^{-1,q}(\R^d)$, $X_1:=W^{1,q}(\R^d)=H^{1,q}(\R^d)$ and, for $u\in X_1$,
\begin{equation}
\begin{aligned}
\label{eq:choice_AFGC_reaction_diffusion_for_intro}
A(t)u&=\A(t)u, &  B(t)u &=(\bb_n (t)u)_{n\geq 1},
\\ F(t,u)&=f(t,u), & G(t,u)&=(g_n(t,u))_{n\geq 1}.
\end{aligned}
\end{equation}
As before $(u,\sigma)$ is a maximal local solution to \eqref{eq:semilinear_reaction_diffusion_l_m} if $(u,\sigma)$ is a maximal local solution to \eqref{eq:semilinearabstract} in the sense of Definition \ref{def:solution2}.

To prove local existence we apply Theorem \ref{t:semilinear}. By Lemma \ref{l:SMR_semilinear_PDEs}, it is enough to estimate the nonlinearities $F,G$. We start by estimating $F$:
\begin{equation}
\label{eq:estimate_F_l_m}
\begin{aligned}
\|F(\cdot, u)-F(\cdot, v)\|_{H^{-1,q}}&\stackrel{(i)}{\lesssim}   \|F(\cdot, u)-F(\cdot, v)\|_{L^t}\\
&\lesssim \|(|u|^{m-1}+|v|^{m-1})|u-v|\|_{L^t}\\
&\stackrel{(ii)}{\lesssim}    (\|u\|_{L^{mt}}^{m-1}+\|v\|_{L^{mt}}^{m-1})\|u-v\|_{L^{mt}}\\
&\stackrel{(iii)}{\lesssim} (\|u\|_{H^{\theta,q}}^{m-1}+\|v\|_{H^{\theta,q}}^{m-1})\|u-v\|_{H^{\theta,q}}.
\end{aligned}
\end{equation}
where in $(i)$ we used the Sobolev embedding with $\frac{d}{t}:=1+\frac{d}{q}$, in $(ii)$ we applied the H\"{o}lder inequality, and in $(iii)$ we used Sobolev embedding with $\theta-\frac{d}{q}=-\frac{d}{mt}$. Note that to ensure that $t\in (1,\infty)$, it is enough to assume $q\neq 2$ if $d=2$. Further, we need $\theta\in (0,1)$ in order to obtain a space between $X_0$ and $X_1$. Combining the identities we obtain
$$
\frac{1}{q}-\frac{\theta}{d}= \frac{1}{mt}=\frac{1}{m}\Big(\frac{1}{q}+\frac{1}{d}\Big)
\Rightarrow \theta= \frac{d}{q}\Big(1-\frac{1}{m}\Big)-\frac{1}{m}.
$$
Therefore, $\theta\in (0,1)$ is equivalent to
\begin{equation}
\label{eq:reaction_diffusion_1_limitation_q_1}
d\Big(\frac{m-1}{m+1}\Big)<q<d(m-1).
\end{equation}
Since $q\geq 2$, we also need $m>1+\frac{2}{d}$.
Setting $\beta_1=\varphi_1=\frac{1+\theta}{2}<1$ we obtain $H^{\theta,q}=[H^{-1,q},H^{1,q}]_{\beta_1}$ by \eqref{eq:H_complex_interpolation}. More explicitly
$$\beta_1=\frac{\theta+1}{2}=\frac{1}{2}\Big(\frac{d}{q}+1\Big)\Big(1-\frac{1}{m}\Big).$$
As in Subsection \ref{ss:reaction_diffusion_divergence_semilinear_almost_weak} to check \ref{HFcritical} we split into three subcases:
\begin{enumerate}[{\rm(1)}]
\item If $1-(1+\a)/p>\beta_1$, then by Remark \ref{r:non_linearities}\eqref{it:non_linearities_continuous_trace} \eqref{it:non_linearities_varphi_equal_to_beta}, \ref{HFcritical} holds.
\item If $1-(1+\a)/p = \beta_1$, then by Remark \eqref{it:non_linearities_varphi_equal_to_beta}, \ref{HFcritical} holds.
\item If $1-(1+\a)/p<\beta_1$, then \ref{HFcritical} holds with $m_F=1$, $\beta_1=\varphi_1$, $\rho_1=m-1$ if the condition \eqref{eq:HypCritical} holds:
\begin{equation}
\label{eq:reaction_diffusion_1_lambda_1}
\frac{1+\a}{p}\leq \frac{\rho_1+1}{\rho_1}(1-\beta_1)
=\frac{m}{m-1}-\frac{1}{2}\Big(\frac{d}{q}+1\Big).
\end{equation}
To ensure that $\a\geq 0$ we have to assume that
\begin{equation}
\label{eq:reaction_diffusion_1_limitation_p}
\frac{1}{p}+\frac{d}{2q}\leq \frac{m}{m-1}-\frac{1}{2}=\frac{m+1}{2(m-1)}.
\end{equation}
From \eqref{eq:reaction_diffusion_1_limitation_q_1} one can check that \eqref{eq:reaction_diffusion_1_limitation_p} is solvable for $p$ sufficiently large.
\end{enumerate}

Next, we estimate $G$ using the same strategy of \eqref{eq:estimate_G_reaction_diffusion_div}. Indeed, since $X_{1/2}=L^q(\R^d)$ and $\g(\ell^2,L^q)=L^q(\R^d;\ell^2)=:L^q(\ell^2)$ (see \eqref{eq:gammaidentity}) one has
\begin{equation}
\begin{aligned}
\label{eq:estimate_G_reaction_diffusion_l_m}
\|G(\cdot, u)-G(\cdot, v)\|_{L^q(\ell^2)}
&\lesssim \|(|u|^{h-1}+|v|^{h-1})|u-v|\|_{L^q}\\
&\stackrel{(i)}{\lesssim} (\|u\|_{L^{hq}}^{h-1}+\|v\|_{L^{hq}}^{h-1})\|u-v\|_{L^{hq}}\\
&\stackrel{(ii)}{\lesssim} (\|u\|_{H^{\phi,q}}^{h-1}+\|v\|_{H^{\phi,q}}^{h-1})\|u-v\|_{H^{\phi,q}}.
\end{aligned}
\end{equation}
where in $(i)$ we applied the H\"{o}lder inequality and in $(ii)$ we used Sobolev embedding with $\phi-\frac{d}{q}=-\frac{d}{h q}$. Therefore, $\phi=\frac{d}{q}\frac{h-1}{h}$. Note that $\phi>0$ and to ensure that $\phi<1$ we have to assume
\begin{equation}
\label{eq:reaction_diffusion_1_limitation_q_2}
q>\frac{d(h-1)}{h}.
\end{equation}
In addition, let us set
$$\beta_2=\frac{\phi+1}{2}=\frac{1}{2}+\frac{d}{2q}\Big(1-\frac{1}{h}\Big),\qquad \varphi_2=\beta_2.$$
As in the previous cases, the discussion splits in two cases:
\begin{enumerate}[{\rm(1)}]
\item If $1-(1+\a)/p>\beta_2$, then \ref{HGcritical} holds by Remark \ref{r:non_linearities}\eqref{it:non_linearities_continuous_trace}.
\item If $1-(1+\a)/p= \beta_2$, then \ref{HGcritical} holds by Remark \ref{r:non_linearities}\eqref{it:non_linearities_varphi_equal_to_beta}.
\item If $1-(1+\a)/p<\beta_2$, then \ref{HGcritical} holds with $m_G=1$, $\rho_2=h-1$, $\beta_2=\varphi_2$ if the condition \eqref{eq:HypCriticalG} holds:
\begin{equation}
\label{eq:reaction_diffusion_1_lambda_2}
\frac{1+\a}{p}\leq \frac{h}{h-1}(1-\beta_2)=\frac{h}{2(h-1)}-\frac{d}{2q}.
\end{equation}
To ensure that $\a\geq 0$ we have to assume that
\begin{equation}
\label{eq:reaction_diffusion_1_limitation_p_2}
\frac{1}{p}+\frac{d}{2q}\leq \frac{h}{2(h-1)}.
\end{equation}
\end{enumerate}

These preparation give the following theorem.
\begin{theorem}
\label{t:reaction_diffusion_local_well_posedness}
Let Assumptions \ref{ass:SND} and \ref{ass:RD} be satisfied and $d\geq 2$. Let $m>1+\frac{2}{d}$ and $h>1$. Moreover, assume that \eqref{eq:reaction_diffusion_1_limitation_q_1} and \eqref{eq:reaction_diffusion_1_limitation_q_2} hold.
Assume one of the following conditions is satisfied
\begin{itemize}
\item $1-(1+\a)/p\geq \beta_1$ and $1-(1+\a)/p\geq \beta_2$;
\item $1-(1+\a)/p<\beta_1$, $1-(1+\a)/p\geq \beta_2$ and \eqref{eq:reaction_diffusion_1_lambda_1} holds;
\item $1-(1+\a)/p\geq \beta_1$, $1-(1+\a)/p<\beta_2$ and \eqref{eq:reaction_diffusion_1_lambda_2} holds;
\item $1-(1+\a)/p<\beta_1$ and $1-(1+\a)/p<\beta_2$ and \eqref{eq:reaction_diffusion_1_lambda_1}, \eqref{eq:reaction_diffusion_1_lambda_2} hold.
\end{itemize}
If $d=2$ we further assume further that $q\neq 2$. Then for each
$$
u_0\in L^0_{\F_0}(\O;B^{1-2\frac{1+\a}{p}}_{q,p}(\R^d)),
$$
the problem \eqref{eq:semilinear_reaction_diffusion_l_m} has a maximal local solution $(u,\sigma)$. Moreover, there exists a localizing sequence $(\sigma_n)_{n\geq 1}$ such that a.s.\ for all $n\geq 1$
$$
u\in L^p(\I_{\sigma_n},w_{\a};W^{1,q}(\R^d))\cap C(\overline{\I}_{\sigma_n};B^{1-2\frac{1+\a}{p}}_{q,p}(\R^d))\cap C((0,\sigma_n];B^{1-\frac{2}{p}}_{q,p}(\R^d)).
$$
\end{theorem}

\subsubsection{Critical spaces for \eqref{eq:semilinear_reaction_diffusion_l_m}}
\label{sss:critical_reaction_diffusion_l_m}
As in Subsection \ref{sss:critical_div_reaction_diffusion} we study critical spaces for \eqref{eq:semilinear_reaction_diffusion_l_m}. Therefore, we need to study when equality in \eqref{eq:reaction_diffusion_1_lambda_1} and \eqref{eq:reaction_diffusion_1_lambda_2} can be reached.

As in Subsection \ref{sss:critical_div_reaction_diffusion}, before embarking in this discussion let us analyse the scaling properties of the equation \eqref{eq:semilinear_reaction_diffusion_l_m} in the case that \eqref{eq:semilinear_reaction_diffusion_l_m_power_law} holds.

In the deterministic case, i.e.\ $b_{jn}\equiv \tilde{g}_n\equiv 0$,the map $u\mapsto u_{\lambda}$ where $u_{\lambda}(x,t):=\lambda^{1/(m-1)}u(\lambda t,\lambda^{1/2}x)$ for $\lambda>0$ preserves the set of (smooth local) solutions to \eqref{eq:semilinear_reaction_diffusion_l_m}. More precisely, if $u$ is a (smooth local) solution to \eqref{eq:semilinear_reaction_diffusion_l_m} on $(0,T)\times \R^d$ then $u_{\lambda}$ is a (smooth local) solution to \eqref{eq:semilinear_reaction_diffusion_l_m} on $(0,T/\lambda)\times \R^d$. Reasoning as \eqref{eq:scaling_B_conservative_reaction_diffusion}, one discovers that $B^{d/q-2/(m-1)}_{q,p}(\R^d)$ is `locally' invariant under the induced map $u_0\mapsto u_{0,\lambda}:=\lambda^{1/(m-1)}u_0(\lambda^{1/2}\cdot)$.

Since \eqref{eq:semilinear_reaction_diffusion_l_m}-\eqref{eq:semilinear_reaction_diffusion_l_m_power_law} presents two non-linearities, it is not immediate to see whether there is scaling-invariance as in Subsection \ref{sss:critical_div_reaction_diffusion}. To check this, we mimic the scaling argument performed in Subsection \ref{sss:critical_div_reaction_diffusion} to discover a relation between $h$ and $m$. Indeed, using the strong formulation of solutions given in Definition \ref{def:solution1}, substituting $s'=\lambda s$ for the deterministic integral one obtains
\begin{align*}
\int_0^{t/\lambda} |u_{\lambda}|^{m-1} u_{\lambda} ds
&=\int_0^t \lambda^{1+\frac{1}{m-1}}|u(s',\lambda x)|^{m-1}u(s',\lambda x)
 \frac{ds'}{\lambda}\\
 &=\lambda^{\frac{1}{m-1}} \int_0^t |u(s',\lambda x)|^{m-1}u(s',\lambda x)
 ds'.
\end{align*}
where, $u_{\lambda}$ is as above. For the stochastic term the same calculation as in \eqref{eq:stochscalinggrowth} gives that the scalings coincide if $\frac{h}{m-1}-\frac12 = \frac{1}{m-1}$, or in other words $\frac{h-1}{m-1}=\frac12$, thus $h=(m+1)/2$. This relation holds if and only if the right hand-sides of the inequalities \eqref{eq:reaction_diffusion_1_lambda_1} and \eqref{eq:reaction_diffusion_1_lambda_2} coincide. Moreover, if $h=(m+1)/2$ the lower bound in \eqref{eq:reaction_diffusion_1_limitation_q_1} coincides with \eqref{eq:reaction_diffusion_1_limitation_q_2}.

For the sake of simplicity, let us continue the discussion on critical spaces for \eqref{eq:semilinear_reaction_diffusion_l_m} under the assumption $h=(m+1)/2$. In this case, \eqref{eq:reaction_diffusion_1_lambda_1} and \eqref{eq:reaction_diffusion_1_lambda_2} coincide, and in order to have equality in the latter two we need to assume
$$\frac{m}{m-1}-\frac{1}{2}\Big(\frac{d}{q}+1\Big)<\frac{1}{2} \Leftrightarrow q<\frac{d(m-1)}{2}.$$
Since $q> 2$, to avoid trivial situations we assume $m>1+\frac{4}{d}$. Under the above assumption we can set
\begin{align}\label{eq:kriticalreactiondiff}
\a_{\crit}:= \frac{pm}{m-1} - \frac{p}{2} \Big(\frac{d}{q} + 1\Big)  -1
\end{align}
and the trace space for the solution to \eqref{eq:semilinear_reaction_diffusion_l_m}-\eqref{eq:semilinear_reaction_diffusion_l_m_power_law} becomes
$$
\Xap=B^{1-\frac{2(1+\a_{\crit})}{p}}_{q,p}(\R^d)=B^{1-2\frac{m}{m-1}+\frac{d}{q}+1}_{q,p}(\R^d)
=B^{\frac{d}{q}-\frac{2}{m-1}}_{q,p}(\R^d).
$$
Note that the above space depends on $p$ only through the microscopic parameter and it presents the same scaling as in the deterministic case, due to the choice $h=(m+1)/2$. Moreover, one can check that in the case $p=q=2$ and $\a=0$, no other critical space arises. Therefore, Theorem \ref{t:reaction_diffusion_local_well_posedness} implies the following result.

\begin{theorem}
\label{t:critical_space_reaction_diffusion}
Let Assumptions \ref{ass:SND} and \ref{ass:RD} be satisfied and $d\geq 2$. Let $m>1+\frac{4}{d}$ and $h = \frac{m+1}{2}$.
Assume that $q\in (\frac{d(m-1)}{m+1},\frac{d(m-1)}{2})$, and if $d=2$ we assume $q\neq 2$. Assume
$\frac{1}{p}+\frac{d}{2q}\leq \frac{m+1}{2(m-1)}$,
and let $\a_{\crit}$ be given by \eqref{eq:kriticalreactiondiff}. Then for each
$$
u_0\in L^0_{\F_0}(\O;B^{\frac{d}{q}-\frac{2}{m-1}}_{q,p}(\R^d))
$$
there exists a maximal local solution $(u,\sigma)$ to \eqref{eq:semilinear_reaction_diffusion_l_m}. Moreover, there exists a localizing sequence $(\sigma_n)_{n\geq 1}$ such that a.s.\ for all $n\geq 1$
$$
u\in L^{p}(0,\sigma_n,w_{\a_{\crit}};W^{1,q}(\R^d))
\cap C(\overline{I}_{\sigma_n};B^{\frac{d}{q}-\frac{2}{m-1}}_{q,p}(\R^d))\cap C((0,\sigma_n];B^{1-\frac{2}{p}}_{q,p}(\R^d)).
$$
\end{theorem}

\subsection{Stochastic reaction-diffusion with gradient nonlinearities}
\label{ss:reaction_diffusion_gradient_nonlinearities}
In this section we study reaction-diffusion equations with gradient non-linearities:
\begin{equation}
\label{eq:semilinear_reaction_diffusion_gradient_non_linearities}
\begin{cases}
du +\A u dt= f(\cdot,u,\nabla u) dt + \sum_{n\geq 1}( \bb_n u +g_n(\cdot,u)) dw_t^n,  & \text{on } \RR,\\
u(0)=u_0, & \text{on }  \RR;
\end{cases}
\end{equation}
where $\A,\bb_n$ are as in \eqref{eq:SND_quasi_AB_n_Def}. We study \eqref{eq:semilinear_reaction_diffusion_gradient_non_linearities} under the following assumption:
\begin{assumption}\label{ass:SND_reaction_diffusion_gradient}
The maps $f:\I_T\times \O\times \RR\times \R\times \R^d\to \R$, $g:=(g_n)_{n\geq 1}:\I_T\times \O\times \RR\times \R\to \ell^2$ are $\Progress\otimes \Borel(\R^d)\otimes \Borel(\R)$-measurable with $f(\cdot,0, 0) = 0$ and $g(\cdot, 0)=\nabla_y g(\cdot,0) = 0$. In addition there exist $m>2$ and $\eta\in (0,1)$ such that for each $R>0$ there exists $C_R>0$ for which one has
\begin{align*}
|f(t,x,y,z)-f(t,x,y',z')|\leq &C_R(1+|z|^{m-1}+|z'|^{m-1})  |z-z'| \\
+&C_R(1+|z|^{m-\eta}+|z'|^{m-\eta}) |y-y'|,\\
\|g(t,x,y)-g(t,x,y')\|_{\ell^2}&+\|\nabla_{y} g(t,x,y)-\nabla_{y} g(t,x,y')\|_{\ell^2}\leq C_R |y-y'|,
\end{align*}
a.s.\ for all $t\in \I_T$, $x\in \RR$, $y,y'\in \B_{\R}(R)$ and $z,z'\in\R^{d}$.
\end{assumption}

Typical choices for $f$ are
\begin{equation}
\label{eq:semilinear_reaction_diffusion_gradient_choice}
f(u,\nabla u)=u^c|\nabla u|^{r}, \ \ \ \text{or} \ \ f(u,\nabla u) = |\nabla u|^r; \quad\text{where } c\in [1, \infty),\,r>1;
\end{equation}
see the monograph \cite[Chapter 5, Section 34]{QS19} for related problems and motivations. For the first example it is straightforward to check that the assumption on $f$ holds for any $m>\max\{r,2\}$. For $c=1$ and $r=2$ we obtain a non-linearity similar to the one appearing in the study of harmonic maps into the sphere, see e.g.\ \cite[p.\ 225]{TayPDE3}. The second example in \eqref{eq:semilinear_reaction_diffusion_gradient_choice} satisfies the assumption for $m=r$ if $r>2$ or for any $m>2$ if $r\in (1,2]$. The latter example covers the stochastic version of \cite[eq.\ (34.5), p.\ 406]{QS19} and it appears in stochastic control theory see e.g.\ \cite{BDL04,PZ12} with $\bb_n=0$ and $g_n = 0$. A further motivation for \eqref{eq:semilinear_reaction_diffusion_gradient_non_linearities} comes from the analysis of high-order regularity of quasilinear equations in divergence form with gradient type nonlinearities (see e.g.\ \cite[Section 3, Example 2]{CriticalQuasilinear}). In such a case, one may take
$$
f(u,\nabla u)= a(u)|\nabla u|^2+ |\nabla u|^r,
$$
where $r>1$ and $a:\R\to \R$ is locally Lipschitz. As above, Assumption \ref{ass:SND_reaction_diffusion_gradient} holds for $m=r$ if $r>2$ or for any $m>2$ if $r\in (1,2]$.

As usual we consider \eqref{eq:semilinear_reaction_diffusion_gradient_non_linearities} as \eqref{eq:semilinearabstract} with $X_0:=L^{q}(\R^d)$, $X_1:=W^{2,q}(\R^d)$ and
\begin{align*}
A(t)u&=\A(t)u, &  B(t)u &=(\bb_n (t)u)_{n\geq 1},
\\ F(t,u)&=f(t,u,\nabla u), & G(t,u)&=(g_n(t,u))_{n\geq 1},
\end{align*}
for $u\in W^{2,q}(\R^d)$.
As before $(u,\sigma)$ is a maximal local solution to \eqref{eq:semilinear_reaction_diffusion_gradient_non_linearities} if $(u,\sigma)$ is a maximal local solution to \eqref{eq:semilinearabstract} in the sense of Definition \ref{def:solution2}.

The main result of this section reads as follows.

\begin{theorem}
\label{t:QND_application_gradient_control_growth_lipschitz_constant}
Let Assumptions \ref{ass:SND} and \ref{ass:SND_reaction_diffusion_gradient} be satisfied, $d\geq 1$ and $q>\frac{d(m-1)}{m}$. Let $\beta=\frac12+ \frac{d}{2q} \frac{m-1}{m}$. Assume that one of the following holds:
\begin{enumerate}[{\rm(1)}]
\item\label{it:gradient_RC_1} $1-\frac{1+\a}{p}\geq \beta$;
\item\label{it:gradient_RC_2} $1-\frac{1+\a}{p}<\beta$ and  $\frac{1+\a}{p}\leq \frac{m}{2(m-1)}-\frac{d}{2q}$.
\end{enumerate}
Then for each
\[u_0\in L^0_{\F_0}(\O;B^{2-2\frac{1+\a}{p}}_{q,p})\]
there exists a maximal local solution $(u,\sigma)$ to \eqref{eq:semilinear_reaction_diffusion_gradient_non_linearities}. Moreover, there exists a localizing sequence $(\sigma_n)_{n\geq 1}$ such that and a.s.\ for all $n\geq 1$
$$
u\in L^p(\I_{\sigma_n},w_{\a};W^{2,q})\cap C(\overline{\I}_{\sigma_n};B^{2-2\frac{1+\a}{p}}_{q,p})\cap C((0,\sigma_n];B^{2-\frac{2}{p}}_{q,p}).
$$
\end{theorem}

\begin{proof}
By Theorem \ref{t:semilinear} and Lemma \ref{l:SMR_semilinear_PDEs}, it remains to check that the nonlinearities satisfies the conditions \ref{HFcritical}-\ref{HGcritical}.

First observe that $2-2\frac{(1+\a)}{p}>\frac{d}{q}$ in each case. Indeed, if \eqref{it:gradient_RC_1} holds then the latter follows from $q>\frac{d(m-1)}{m}>\frac{d}{m}$. If \eqref{it:gradient_RC_2} holds, then $2-2\frac{1+\a}{p}\geq 2- \frac{m}{m-1}+\frac{d}{q}>\frac{d}{q}$ where in the last inequality we used that $m>2$.
The previous observation combined with Sobolev embedding gives that
\begin{align}\label{eq:XapembeddingCeps}
\Xap=B^{2-\frac{2(1+\a)}{p}}_{q,p}\hookrightarrow C^{\varepsilon},\qquad \text{ for some }\varepsilon>0.
\end{align}
Let $n\geq 1$ and let $u,v\in X_1$ be such that $u,v\in B_{\Xap}(n)$. By the previous embedding $\|u\|_{L^{\infty}(\R^d)}\leq C \|u\|_{\Xap}\leq Cn$ and the same for $v$. Let $\phi\in \big(2-2\frac{1+\a}{p},2\big)$ be arbitrary. Setting $R = Cn$, then by Assumption \ref{ass:SND_reaction_diffusion_gradient},
\begin{equation}
\label{eq:estimate_F_gradient_optimal}
\begin{aligned}
\|F(\cdot, u)-&F(\cdot, v)\|_{L^q}\\
&\leq C_R \|(1+|\nabla u|^{m-1}+|\nabla v|^{m-1})|\nabla u-\nabla v|\|_{L^q}\\
&\qquad + C_R \|(1+|\nabla u|^{m-\eta}+|\nabla v|^{m-\eta})| u- v|\|_{L^q}\\
&\lesssim_R \|\nabla u-\nabla v\|_{L^{q}}+ (\|\nabla u\|_{L^{qm}}^{m-1}+\|\nabla v\|_{L^{qm}}^{m-1})\|\nabla u-\nabla v\|_{L^{qm}}\\
&\qquad + \|u-v\|_{\Xap}+ (\|\nabla u\|_{L^{q(m-\eta)}}^{(m-\eta)}+\|\nabla u\|_{L^{q(m-\eta)}}^{(m-\eta)}) \|u-v\|_{C^{\epsilon}} \\
&\leq   (1+\| u\|_{H^{\theta,q}}^{m-1}+\| v\|_{H^{\theta,q}}^{m-1})\| u- v\|_{H^{\theta,q}}  \\
& \qquad +(1+\| u\|_{H^{\theta,q}}^{m-\eta}+\| v\|_{H^{\theta,q}}^{m-\eta})\|u-v\|_{H^{\phi,q}};
\end{aligned}
\end{equation}
where in the last line we used the Sobolev embedding with
$\theta-\frac{d}{q} = 1-\frac{d}{qm}$ and the fact that $H^{\phi,q}\hookrightarrow B^{2-2\frac{1+\a}{p}}_{q,p}$. Note that $\theta<2$ since $q>\frac{d(m-1)}{m}$ and $\beta=\theta/2$. Moreover, by \eqref{eq:H_complex_interpolation}, $H^{\theta,q}=[L^{q},W^{2,q}]_{{\beta}}$ and $H^{\phi,q}=[L^{q},W^{2,q}]_{\frac{\phi}{2}}$. To check \ref{HFcritical} we split the argument in two cases:
\begin{enumerate}
\item If $1-(1+\a)/p\geq \beta$, then $\phi>2(1-\frac{1+\a}{p})\geq \theta$. Since $\eta<1$, \eqref{eq:estimate_F_gradient_optimal} implies
$$
\|F(\cdot, u)-F(\cdot, v)\|_{L^q}\lesssim
(1+\| u\|_{H^{\phi,q}}^{m-\eta}+\| v\|_{H^{\phi,q}}^{m-\eta})\|u-v\|_{H^{\phi,q}}.
$$
Set $m_F=1$, $\rho_1=m-\eta$ and $\varphi_1=\beta_1=\phi/2$. Choosing $\phi=2(1-\frac{1+\a}{p})+\varepsilon$, for some $\varepsilon$ small,  \eqref{eq:HypCritical} is equivalent to
$$
(m-\eta)\Big(\varphi_1-1+\frac{1+\a}{p}\Big)+ \beta_1= (m-\eta+1)\frac{\varepsilon}{2} + 1-\frac{1+\a}{p}\leq 1.
$$
The latter inequality is satisfied if $\varepsilon>0$ is sufficiently small. In turn, \ref{HFcritical} is satisfied by setting $F_c=F$, $F_{\Tr}=F_L=0$.

\item If $1-(1+\a)/p<\beta$, then by \eqref{eq:estimate_F_gradient_optimal}, we may set $m_F=2$, $\rho_1=m-1$, $\rho_2=m-\eta$, $\varphi_1=\varphi_2=\theta/2$, $\beta_1=\varphi_1$ and $\beta_2=\phi/2$. It remains to verify \eqref{eq:HypCritical}, which is equivalent to the following
\begin{align}
\label{eq:critical_spaces_gradient_SND_epsilon_coincise_form}
&(m-1)\Big(\varphi_1-1+\frac{1+\a}{p}\Big)+\varphi_1\leq 1,\\
\label{eq:critical_spaces_gradient_SND_epsilon}
&(m-\eta)\Big(\varphi_1-1+\frac{1+\a}{p}\Big)+ \beta_2 \leq 1.
\end{align}
Note that \eqref{eq:critical_spaces_gradient_SND_epsilon_coincise_form} implies \eqref{eq:critical_spaces_gradient_SND_epsilon}. To see this, set $\phi=2-2\frac{1+\a}{p}+\varepsilon$ for $\varepsilon>0$ small. Then \eqref{eq:critical_spaces_gradient_SND_epsilon} holds provided $m\varphi_1 -(m-1)(1-\frac{1+\a}{p})\leq 1+\eta'$ where $\eta'>0$. Now, standard considerations show that \eqref{eq:critical_spaces_gradient_SND_epsilon_coincise_form} implies the latter. Thus, \ref{HFcritical} is satisfied by setting $F_c=F$, $F_{\Tr}=F_L=0$.

Finally, we note that \eqref{eq:critical_spaces_gradient_SND_epsilon_coincise_form} is equivalent to
\begin{equation}
\label{eq:critical_spaces_gradient_SND}
\frac{1+\a}{p}\leq \frac{\rho+1}{\rho}\Big(\frac{1}{2}-\frac{d}{2q}\frac{m-1}{m}\Big)
=\frac{m}{2(m-1)}-\frac{d}{2q}.
\end{equation}
\end{enumerate}
A more simple argument applies to $g$. Indeed,
\begin{equation}
\label{eq:estimate_G_gradient_nonlinearities_revision_stage}
\begin{aligned}
\|G(\cdot,u)-G(\cdot,v)\|_{W^{1,q}(\ell^2)}&\lesssim \|(g_n(\cdot,u)-g_n(\cdot,v))_{n\geq 1}\|_{L^q(\ell^2)}\\
& \quad +\|(\nabla g_n(\cdot,u)(\nabla u-\nabla v))_{n\geq 1}\|_{L^q(\ell^2)}\\
& \quad+\|(\nabla g_n(\cdot,u)-\nabla g_n(\cdot,v))\nabla v)_{n\geq 1}\|_{L^q(\ell^2)}\\
&\leq C_R \| u-v\|_{W^{1,q}}\lesssim C_R \|u-v\|_{\Xap};
\end{aligned}
\end{equation}
where we used that $\Xap \hookrightarrow L^{\infty}\cap W^{1,q}$ by \eqref{eq:XapembeddingCeps} and $2-2(1+\a)/p> 1$. Therefore, $G$ satisfies \ref{HGcritical} with $G_c=G_L=0$.
\end{proof}

\subsubsection{Critical spaces for \eqref{eq:semilinear_reaction_diffusion_gradient_non_linearities}}
Analogously to Subsections \ref{sss:critical_div_reaction_diffusion}, \ref{sss:critical_reaction_diffusion_l_m} let us first analyse the scaling property of the equation \eqref{eq:semilinear_reaction_diffusion_gradient_non_linearities} under the assumption
\begin{equation}
\label{eq:semilinear_reaction_diffusion_critical_space_choice_f}
f(u,\nabla u)=|\nabla u|^m,\qquad m>2,
\end{equation}
cf. \eqref{eq:semilinear_reaction_diffusion_gradient_choice}. In the deterministic case, i.e.\ $b_{jn}\equiv g_n\equiv 0$,
the equation \eqref{eq:semilinear_reaction_diffusion_gradient_non_linearities} with \eqref{eq:semilinear_reaction_diffusion_critical_space_choice_f} is `locally invariant' under the transformation $u\mapsto u_{\lambda}$ where
$$u_{\lambda}(t,x):=\lambda^{-\alpha/2} u(\lambda t,\lambda^{1/2}x),\quad \text{ for }\lambda>0,\,x\in \R^d,$$
and where we have set $\alpha:=\frac{m-2}{m-1}$.
As in \cite[Example 3, Section 3]{CriticalQuasilinear} one can see that the Besov space $B^{d/q+(m-2)/(m-1)}_{q,p}$ has the right `local' scaling for the problem \eqref{eq:semilinear_reaction_diffusion_gradient_non_linearities} with $f$ as in \eqref{eq:semilinear_reaction_diffusion_critical_space_choice_f}, i.e.\ the homogeneous version of this space is invariant under the induced map $u_0\mapsto u_{0,\lambda}:=\lambda^{-\alpha} u_0(\lambda^{1/2}\cdot)$. More precisely, one has
\begin{align*}
\|u_{0,\lambda}\|_{\dot{B}^{\frac{d}{q}+\alpha}_{q,p}}
&\eqsim
\lambda ^{-\alpha/2} (\lambda^{1/2})^{\frac{d}{q}+\alpha-\frac{d}{q}}\|u_0\|_{\dot{B}^{\frac{d}{q}+\alpha}_{q,p}}
=\|u_0\|_{\dot{B}^{\frac{d}{q}+\alpha}_{q,p}};
\end{align*}
here $\dot{B}^{d/q+(m-2)/(m-1)}_{q,p}$ denotes the homogeneous Besov space and the implicit constant does not depend on $\lambda>0$.

It turns out that the above spaces arise naturally as critical spaces for \eqref{eq:semilinear_reaction_diffusion_gradient_non_linearities} in our abstract framework. Moreover, using our abstract theory we do not assume that $f$ has the form in \eqref{eq:semilinear_reaction_diffusion_critical_space_choice_f} but Assumption \ref{ass:SND_reaction_diffusion_gradient} is enough. To this end, as in Subsections \ref{sss:critical_div_reaction_diffusion}, \ref{sss:critical_reaction_diffusion_l_m} we study when equality holds in \eqref{eq:critical_spaces_gradient_SND} for some $\a:=\a_{\crit}$.

Let us begin by analysing the case $p\in (2,\infty)$ and $\a\in [0,p/2-1)$. In this case, to ensure $\a\geq 0$, by \eqref{eq:critical_spaces_gradient_SND} we need
\begin{equation}
\label{eq:SND_gradient_limitation_p}
\frac{1}{p}+\frac{d}{2q}\leq \frac{m}{2(m-1)}.
\end{equation}
To ensure $\a<\frac{p}{2}-1$ we assume
$$
\frac{m}{2(m-1)}-\frac{d}{2q}<\frac{1}{2} \Leftrightarrow q<d(m-1).
$$
Since $q\geq 2$, we assume $m> 1+\frac{2}{d}$. Since $m>2$ by Assumption \ref{ass:SND_reaction_diffusion_gradient}, the latter is automatically satisfied in the case $d>1$. Under the previous conditions, we set
\begin{align}\label{eq:def:kappagraddiff}
\a_{\crit}=\frac{p m}{2(m-1)}- \frac{pd}{2q}-1.
\end{align}
Then the trace space becomes
$$
\Xapcrit=B^{2-\frac{2(1+\a_{\crit})}{p}}_{q,p}(\R^d)=B^{\frac{d}{q}+\frac{m-2}{m-1}}_{q,p}(\R^d).
$$
In the case $q=p=2$ and $\a=0$, if equality in \eqref{eq:critical_spaces_gradient_SND} holds, then $m= 1+2/d$, and therefore $d=1$ since $m>2$. Thus, we can also allow $d=1$, $m=3$, $p=q=2$, and $\a=0$, and the corresponding critical space becomes
$\Xap = B^{1}_{2,2}(\R)=H^1(\R)$.
Now Theorem \ref{t:QND_application_gradient_control_growth_lipschitz_constant} implies the following result.

\begin{theorem}
\label{t:QND_application_critical_spaces_gradient}
Let Assumptions \ref{ass:SND} and \ref{ass:SND_reaction_diffusion_gradient} be satisfied. Let either $d\geq 2$, or $d=1$ and $m>3$. Assume that $\frac{d(m-1)}{m}<q<d(m-1)$ and that $p\in (2,\infty)$ verifies \eqref{eq:SND_gradient_limitation_p}. Let $\a_{\crit}$ be given by \eqref{eq:def:kappagraddiff}. Then for each
$$u_0\in L^0_{\F_0}(\O;B^{\frac{d}{q}+\frac{m-2}{m-1}}_{q,p}(\R^d)),$$
there exists a maximal local solution to \eqref{eq:semilinear_reaction_diffusion_gradient_non_linearities}. Moreover, there exists a localizing sequence $(\sigma_n)_{n\geq 1}$ such that a.s.\ for all $n\geq 1$
$$
u\in L^p(\I_{\sigma_n},w_{\a_{\crit}};W^{2,q}(\R^d))\cap C(\overline{\I}_{\sigma_n};B^{\frac{d}{q}+\frac{m-2}{m-1}}_{q,p}(\R^d))\cap C((0,\sigma_n];B^{2-\frac{2}{p}}_{q,p}(\R^d)).
$$
Furthermore, the same is true if $d=1$, $m=3$, $p=q=2$ and $\a_{\crit}=0$.
\end{theorem}

\subsection{Stochastic Burgers' equation with white noise}
\label{ss:Burgers_semilinear}
In this section we consider a stochastic Burgers' equation with space-time white noise on $\Tor$. The space-time white noise will be denoted by $w_t$.
More precisely, we analyse the following problem for the unknown process $u:\I_T\times \O\times \Tor\to \R$
\begin{equation}
\label{eq:Burger_white_noise}
\begin{cases}
\displaystyle du +\A u dt= \partial_x(f(\cdot, u)) dt +
g(\cdot, u) d{w}_{t}, & \text{on } \Tor,\\
u(0)=u_0 & \text{on } \Tor.
\end{cases}
\end{equation}
Here $\A$ is as in \eqref{eq:SND_quasi_AB_n_Def} and for simplicity we took $\bb=0$.
For results with Dirichlet boundary conditions see Theorem \ref{t:Dirichlet_extensions_semilinear}\eqref{it:Dirichlet_extension}.

Compared with the previous sections, due to the space-time white noise we restrict ourselves to the one-dimensional torus, and require a suitable interpretation of the $g$-term. Indeed, the term $g(\cdot,u)d w_t$ in \eqref{eq:Burger_white_noise} will be interpreted as $M_{g(\cdot,u)} W_{L^{2}(\Tor)}$ where $M_{g(\cdot,u)}$ denotes multiplication by $g(\cdot,u)$, and $W_{L^2(\Tor)}$ is an $L^2(\Tor)$-cylindrical Brownian motion induced by the space-time white noise $w_t$.

\begin{assumption}\label{ass:SNDwhitenoise}\
\begin{enumerate}[{\rm(1)}]
\item Assumption \ref{ass:SND} is satisfied.
\item\label{it:SND_white_noise_f_g} The maps $f:\I_T\times\O\times \Tor\times \R\to \R$, $g:\I_T\times\O\times\Tor\times\R \to \R$ are $\Progress\otimes \Borel(\Tor)\otimes \Borel(\R)$-measurable with $f(\cdot, 0 ),g(\cdot, 0) \in L^{\infty}(\I_T\times \O\times \Tor)$. Moreover, there exist $h,m>1$ and $C>0$ such that
such that for all $z,z'\in \R$
\begin{align*}
|f(\cdot,z)-f(\cdot,z')|& \leq  C (1+|z|^{h-1}+|z'|^{h-1})|z-z'|,
\\ |g(\cdot,z)-g(\cdot,z')|& \leq C (1+|z|^{m-1}+|z'|^{m-1}) |z-z'|.
\end{align*}
\end{enumerate}
\end{assumption}
The Burgers' nonlinearity $f(u) = -u^2 $ satisfies the above condition for any $h\geq 2$.

As above, to prove local existence for \eqref{eq:Burger_white_noise} we employ Theorem \ref{t:semilinear}. Recall that the space-time white noise can be model as an $L^2(\Tor)$-cylindrical Brownian motion. Therefore, we set $H=L^2(\Tor)$. Fix $s\in (0,1)$ and $q\in [2,\infty)$. We rewrite \eqref{eq:Burger_white_noise} in the form \eqref{eq:semilinearabstract} by setting $X_0:=H^{-1-s,q}(\Tor)$, $X_1=H^{1-s,q}(\Tor)$. Note that by \eqref{eq:H_complex_interpolation},
\[X_{\frac12} = H^{-s,q}(\Tor)\ \ \ \text{and} \ \ \ \Xap = B^{1-s-2\frac{(1+\a)}{p}}_{q,p}(\Tor).\]
For $u\in X_1$ and $t\in \I_T$ we set
\begin{align*}
A(t) u&=\A(t)u, &  B(t)u &=0,
\\ F(t,u)&=\partial_x (f(t,u)), & G(t,u)&=i M_{g(t,u)}.
\end{align*}
Here $\A(t)$ is as \eqref{eq:SND_quasi_AB_n_Def} and for fixed $u\in L^{m \ell}(\Tor)$ measurable, $M_{g(t,u)}:L^2(\Tor)\to L^r(\Tor)$ is the multiplication operator
\[(M_{g(t,u)} h)(x) = g(t,u(x)) h(x),\]
for $r\in (1, 2)$ and $\ell\in (2, \infty)$ which satisfy $\frac{1}{r} = \frac{1}{2} + \frac{1}{\ell}$ and we will need $s-\frac{1}{r}>0$ later for the $G$ term. Moreover, $i:L^r(\Tor)\to H^{-s,q}(\Tor)=X_{\frac12}$ denotes the embedding which holds since $-s-\frac1q\leq  -\frac1r$. Since $s>\frac{1}{r}>\frac{1}{2}$ we only will consider $s\in (\frac{1}{2},1)$ below.

As usual, we say that $(u,\sigma)$ is a maximal local solution to \eqref{eq:Burger_white_noise} if $(u,\sigma)$ is a maximal local solution to \eqref{eq:semilinearabstract} in the sense of Definition \ref{def:solution2} with the above choice of $A,B,F,G,H$. To estimate the nonlinearity we start by looking at $F$. As in \eqref{eq:reaction_diffusion_estimate_F}, by Assumption \ref{ass:SNDwhitenoise}\eqref{it:SND_white_noise_f_g} we get
\begin{equation}
\label{eq:Burgers_semilinear_estimate_F}
\begin{aligned}
\|F(\cdot, u)-F(\cdot, v)\|_{H^{-1-s,q}}	
&\stackrel{(i)}{\lesssim} \|F(\cdot, u)-F(\cdot, v)\|_{H^{-1,\xi}}\\
&\lesssim (1+\|u\|_{L^{h \xi}}^{h-1}+\|v\|_{L^{h \xi}}^{h-1})\|u-v\|_{L^{h \xi}}\\
&\stackrel{(ii)}{\lesssim} (1+\|u\|_{H^{\theta,q}}^{h-1}+\|v\|_{H^{\theta,q}}^{h-1})\|u-v\|_{H^{\theta,q}};
\end{aligned}
\end{equation}
where in $(i)$ we used the Sobolev embedding with $\xi$ defined by $-1-\frac{1}{\xi} = -1-s -\frac{1}{q}$ and in $(ii)$ the Sobolev embedding with $\theta -\frac{1}{q} = -\frac{1}{h\xi}$. To ensure that $\xi\in (1,\infty)$ we have to assume $q>\frac{1}{1-s}$. Moreover,
$$
\frac{1}{q}-\theta= \frac{1}{h \xi}=\frac{1}{h}\Big(\frac{1}{q}+s\Big)\quad
\Rightarrow
\quad
\theta= \frac{1}{q}\Big(1-\frac{1}{h}\Big)-\frac{s}{h}.
$$
Since $\theta$ has to satisfy $\theta\in (0,1-s)$, we require $\frac{h-1}{h-s(h-1)}<q<\frac{h-1}{s}$. Since $\frac{h-1}{h-s(h-1)}<\frac{1}{1-s}$ for all $h\geq 1$ and $s\in (0,1)$, it is enough to assume
\begin{equation}
\label{eq:Burgers_limitation_q}
\frac{1}{1-s}<q<\frac{h-1}{s}.
\end{equation}
Note that since $s\in (\frac{1}{2},1)$ then $\frac{1}{1-s}>2$. Thus, if $q$ verifies \eqref{eq:Burgers_limitation_q}, then $q>2$. Moreover, the condition \eqref{eq:Burgers_limitation_q} is not empty provided
\begin{equation}
\label{eq:Burgers_limitation_h_s}
h>\frac{1}{1-s}.
\end{equation}
Since $s>\frac{1}{2}$, the previous implies $h>2$. If \eqref{eq:Burgers_limitation_q} holds, then $H^{\theta,q}=[H^{-1-s,q},H^{1-s,q}]_{\beta_1}$ where
\begin{equation}
\label{eq:beta1_Burgers}
\beta_1=\frac{1+\theta+s}{2}=\frac{1}{2}\Big[\Big(\frac{1}{q}+s\Big)\Big(1-\frac{1}{h}\Big)+1\Big]\in (0,1).
\end{equation}
To check the condition \ref{HFcritical} we may split the discussion into three cases:
\begin{enumerate}[{\rm(1)}]
\item If $1-\frac{1+\a}{p}>\beta_1$, by Remark \ref{r:non_linearities}\eqref{it:non_linearities_continuous_trace}, \ref{HFcritical} follows by setting $F_{\Tr}(t,u)=\partial_x(f(t,\cdot,u))$ and $F_{L}\equiv F_c\equiv 0$.
\item If $1-\frac{1+\a}{p}= \beta_1$, by \eqref{eq:Burgers_semilinear_estimate_F} and Remark \ref{r:non_linearities}\eqref{it:non_linearities_varphi_equal_to_beta}, \ref{HFcritical} follows by setting $F_L\equiv F_{\Tr}\equiv 0$, $F_{c}(t,u)=\partial_x(f(t,\cdot,u))$, $m_F=1$, $\rho_1=h-1$ and $\varphi_1=\beta_1$.
\item If $1-\frac{1+\a}{p}<\beta_1$ we set $F_{c}(t,u)=\partial_x(f(t,\cdot,u))$ and $F_{L}\equiv F_{\Tr}\equiv 0$. As in the previous item we set $m_F=1$, $\rho_1=h-1$ and $\varphi_1=\beta_1$. By \eqref{eq:Burgers_semilinear_estimate_F} it remains to check the condition \eqref{eq:HypCritical}. In this situation, \eqref{eq:HypCritical} becomes
\begin{equation}
\label{eq:Burgers_critical_equations_1}
\frac{1+\a}{p}\leq \frac{\rho_1+1}{\rho_1}(1-\beta_1)
=\frac{1}{2}\frac{h}{h-1}-\frac{1}{2}\Big(\frac{1}{q}+s\Big).
\end{equation}
\end{enumerate}

Next we estimate $G$. By Assumption \ref{ass:SNDwhitenoise}\eqref{it:SND_white_noise_f_g} it follows that
\begin{equation}
\label{eq:estimate_G_very_weak_Burgers}
\begin{aligned}
\|G(\cdot,u)-G(\cdot,v)\|_{\g(L^2;H^{-s,q})}
&\eqsim \|(I-\partial_x^2)^{-\frac{s}{2}} (M_{g(\cdot, u)}-M_{g(\cdot, v)})\|_{\g(L^2;L^{q})}
\\ &\stackrel{(i)}{\lesssim}
\|(I-\partial_x^2)^{-\frac{s}{2}} (M_{g(\cdot, u)}-M_{g(\cdot, v)})\|_{\calL(L^2;L^{\infty})}
\\ & \stackrel{(ii)}{\lesssim} \|M_{g(\cdot, u)}-M_{g(\cdot, v)}\|_{\calL(L^2;L^{r})}
\\ & \stackrel{(iii)}{\lesssim} \|g(\cdot, u)-g(\cdot, v)\|_{L^{\ell}}
\\ & \lesssim (1+\|u\|_{L^{m\ell}}^{m-1}+\|v\|_{L^{m\ell}}^{m-1})\|u-v\|_{L^{m\ell}}
\\ &\stackrel{(iv)}{\lesssim}(1+\|u\|_{H^{\phi,q}}^{m-1}+\|v\|_{H^{\phi,q}}^{m-1})\|u-v\|_{H^{\phi,q}};
\end{aligned}
\end{equation}
where in $(i)$ we used \cite[Lemma 2.1]{NVW3}, in $(ii)$ we used Sobolev embedding with $s-\frac{1}{r} >0$.  In $(iii)$ we used H\"older's inequality with $\frac{1}{\ell}=\frac{1}{r}-\frac{1}{2}$, and in $(iv)$ Sobolev embedding with $\phi-\frac{1}{q} =- \frac{1}{\ell m} =\frac{1}{m}( \frac12-\frac{1}{r})$. Thus, to ensure that $\phi\in (0,1-s)$ we require
\begin{equation}
\label{eq:Burger_limitation_q_2}
\frac{m}{m(1-s)+\frac1r-\frac12}<q< \frac{m}{\frac1r-\frac12}.
\end{equation}
The lower estimate in \eqref{eq:Burger_limitation_q_2} is immediate from $q>1/(1-s)$. The upper estimate gives a restriction, but we will take $r\in (1, 2)$ large enough to avoid any additional restrictions coming from \eqref{eq:Burger_limitation_q_2}.

Due to \eqref{eq:H_complex_interpolation} one has $H^{\phi,q}=[H^{-1-s,q},H^{1-s,q}]_{\beta_2}$ where
\begin{equation}
\label{eq:beta2_Burgers_real}
\beta_2=\frac{1+s+\phi}{2} = \frac{1+s}{2}+ \frac{1}{2q}-\frac{1}{2m}\Big( \frac{1}{r} - \frac12\Big) \in (0,1).
\end{equation}
As usual, to check assumption \ref{HGcritical} we split the discussion in several cases. Since $r\in (1, 2)$ will be chosen large, we will set
\begin{equation}
\label{eq:beta2_Burgers}
\tilde{\beta}_2=\frac{1+s}{2}+ \frac{1}{2q}\in (0,1).
\end{equation}
Then $\tilde{\beta}_2>\beta_2$.
\begin{enumerate}[{\rm(1)}]
\item If $1-\frac{1+\a}{p}\geq \tilde{\beta}_2$, then since $\tilde{\beta}_2>\beta_2$, by Remark \ref{r:non_linearities}\eqref{it:non_linearities_continuous_trace}, \ref{HGcritical} follows by setting $G_{\Tr}(t,u):=g(\cdot,u)$ and $G_{L}\equiv G_c\equiv 0$.
\item If $1-\frac{1+\a}{p}<\tilde{\beta}_2$, we can choose $r\in (1, 2)$ so large that the same holds with $\beta_2$ instead of $\tilde{\beta}_2$, and we set $G_{c}(t,u):=g(\cdot,u)$ and $G_{L}\equiv G_{\Tr}\equiv 0$. As in the previous item we set $m_G=1$, $\rho_2=m-1$ and $\varphi_2=\beta_2$. By \eqref{eq:estimate_G_very_weak_Burgers} it remains to check the condition \eqref{eq:HypCriticalG}. Now \eqref{eq:HypCriticalG} becomes
\begin{equation*}
\frac{1+\a}{p}\leq \frac{\rho_2+1}{\rho_2}(1-\beta_2)=
\frac{m}{m-1}(1-\beta_2)
\end{equation*}
Choosing $r\in(1, 2)$ large enough the latter holds if
\begin{equation}
\label{eq:Burgers_critical_equations_2}
\frac{1+\a}{p}<\frac{m}{(m-1)}(1-\tilde{\beta}_2)=\frac{m}{2(m-1)}\Big(1-s-\frac{1}{q}\Big).
\end{equation}
Since $\a\in [0,\frac{p}{2}-1)$ and $\tilde{\beta}_2<1$, then the above inequality is always verified for $p$ sufficiently large and $\a$ small.
\end{enumerate}
Combining the above considerations with Theorem \ref{t:semilinear} and Lemma \ref{l:SMR_semilinear_PDEs}, we obtain the following:

\begin{theorem}
\label{t:Burgers_local_existence}
Let $s\in (\frac{1}{2},1)$ and $h>1/(1-s)$. Assume that Assumption \ref{ass:SNDwhitenoise} holds. Let \eqref{eq:Burgers_limitation_q} be satisfied. Let $\beta_1$ be as in \eqref{eq:beta1_Burgers} and $\tilde{\beta}_2$ as in \eqref{eq:beta2_Burgers}. Assume that one of the following conditions is satisfied:
\begin{itemize}
\item $1-(1+\a)/p\geq \beta_1$ and $1-(1+\a)/p\geq \tilde{\beta}_2$;
\item $1-(1+\a)/p<\beta_1$, $1-(1+\a)/p\geq \tilde{\beta}_2$ and \eqref{eq:Burgers_critical_equations_1} holds;
\item $1-(1+\a)/p\geq \beta_1$, $1-(1+\a)/p<\tilde{\beta}_2$ and \eqref{eq:Burgers_critical_equations_2} holds;
\item $1-(1+\a)/p<\beta_1$ and $1-(1+\a)/p<\tilde{\beta}_2$ and \eqref{eq:Burgers_critical_equations_1}, \eqref{eq:Burgers_critical_equations_2} hold.
\end{itemize}
Then for each
$$
u_0\in L^0_{\F_0}(\O;\B^{1-s-2\frac{1+\a}{p}}_{q,p}(\Tor)),
$$
the problem \eqref{eq:Burger_white_noise} has a maximal local solution $(u,\sigma)$. Moreover, there exists a localizing sequence $(\sigma_n)_{n\geq 1}$ such that a.s.\ for all $n\geq 1$
$$
u\in L^p(\I_{\sigma_n},w_{\a};H^{1-s,q}(\Tor))\cap  C(\overline{\I}_{\sigma_n};\B^{1-s-2\frac{1+\a}{p}}_{q,p}(\Tor))\cap
C((0,\sigma_n];\B^{1-s-\frac{2}{p}}_{q,p}(\Tor)).
$$
\end{theorem}

\begin{example}
\label{ex:Burgers_semilinear}
In the case of Burgers' equation, i.e.\ $f(u)=-u^2$ and $h=2$, Theorem \ref{t:Burgers_local_existence} gives a sub-optimal result. To see this recall that $f(u) = -u^2$ verifies Assumption \ref{ass:SNDwhitenoise} for all $h\geq 2$.
Fix $\varepsilon>0$ and write $h=2+\varepsilon$. Then \eqref{eq:Burgers_limitation_h_s} implies
$
s\in (\frac{1}{2},\frac{1+\varepsilon}{2+\varepsilon})
$. Since $s\in (\frac{1}{2},\frac{1+\varepsilon}{2+\varepsilon})$ is arbitrary, choosing $s>\frac12$ small enough, the limitation \eqref{eq:Burgers_limitation_q} gives $2<q<2(1+\varepsilon)$.
Since $\beta_1,\tilde{\beta}_2<1$, by choosing $p$ large enough, Theorem \ref{t:Burgers_local_existence} gives the existence of a maximal solution to \eqref{eq:Burger_white_noise} with $f(u)=-u^2$.
\end{example}

\subsubsection{Critical spaces for \eqref{eq:Burger_white_noise}}
Here we analyse the existence of critical spaces for \eqref{eq:Burger_white_noise}. By definition, it means that \eqref{eq:Burgers_critical_equations_1} or \eqref{eq:Burgers_critical_equations_2} has to be satisfied with equality. Since in \eqref{eq:Burgers_critical_equations_2} equality is not allowed, we have to require that the right-hand side of \eqref{eq:Burgers_critical_equations_1} is smaller than the one in  \eqref{eq:Burgers_critical_equations_2}. A straightforward computation shows that this holds if and only if
\begin{equation}
\label{eq:Burgers_m_h_relation}
m<h+(1-h)\Big(s+\frac{1}{q}\Big).
\end{equation}
In particular, the latter implies $m<h$. Note that \eqref{eq:Burgers_m_h_relation} is not empty since $h+(1-h)(s+\frac{1}{q})>1$, by \eqref{eq:Burgers_limitation_q}. If \eqref{eq:Burgers_m_h_relation} holds, then the critical spaces arise when equality in \eqref{eq:Burgers_critical_equations_1} is reached. Reasoning as in Subsection \ref{sss:critical_div_reaction_diffusion}, for $p\in (2,\infty)$ equality in \eqref{eq:Burgers_critical_equations_1} holds for some $\a\in [0,\frac{p}{2}-1)$ if the following are satisfied
\begin{align}
\label{eq:Burgers_critical_p}
\frac{1}{p}+\frac{1}{2}\Big(\frac{1}{q}+s\Big)&\leq \frac{1}{2}\frac{h}{h-1},
\\
\label{eq:Burgers_critical_q}
h\geq \frac{1+s}{s}\qquad \text{ or }\qquad& \Big[ h<\frac{1+s}{s} \ \ \text{and}  \ \ q<\frac{h-1}{1-s(h-1)}\Big].
\end{align}
Note that if $h<\frac{1+s}{s}$ one always has $\frac{h-1}{1-s(h-1)}>\frac{h-1}{s}$
as follows from $s>\frac{1}{2}$. Therefore, by \eqref{eq:Burgers_limitation_q}, condition \eqref{eq:Burgers_critical_q} is always verified. Defining $\a_{\crit}$ as in \eqref{eq:criticalreacdiff}, one obtains
$\Xapcrit=B^{\frac{1}{q}-\frac{1}{h-1}}_{q,p}(\Tor)$.
These considerations and Theorem \ref{t:Burgers_local_existence} give the following.
\begin{theorem}
\label{t:Burgers_critical}
Let $s\in (\frac{1}{2},1)$ and $h>1/(1-s)$. Assume that Assumption \ref{ass:SNDwhitenoise} holds. Assume that \eqref{eq:Burgers_limitation_q}, \eqref{eq:Burgers_m_h_relation} and \eqref{eq:Burgers_critical_p} hold. Let $\a_{\crit}:=\frac{p}{2}(\frac{h}{h-1}-\frac{1}{q}-s)-1$. Then for each
$$
u_0\in L^0_{\F_0}(\O;B^{\frac{1}{q}-\frac{1}{h-1}}_{q,p}(\Tor))
$$
there exists a maximal local solution $(u,\sigma)$ to \eqref{eq:semilinear_reaction_diffusion}. Moreover, there exists a localizing sequence $(\sigma_n)_{n\geq 1}$ such that a.s.\ for all $n\geq 1$
$$
u\in L^{p}(\I_{\sigma_n},w_{\a_{\crit}};H^{1-s,q}(\Tor))
\cap C(\overline{\I}_{\sigma_n};B^{\frac{1}{q}-\frac{1}{h-1}}_{q,p}(\Tor))\cap
C((0,\sigma_n];B^{1-s-\frac{2}{p}}_{q,p}(\Tor)).
$$
\end{theorem}

\begin{example}
Here we continue the study of \eqref{eq:Burger_white_noise} in the case of Burgers' equation, i.e.\ \eqref{eq:Burger_white_noise} with $f(u)=-u^2$. As in Example \ref{ex:Burgers_semilinear}, let $\varepsilon>0$ and $h=2+\varepsilon$. Thus, \eqref{eq:Burgers_limitation_h_s} and  \eqref{eq:Burgers_limitation_q} gives $s\in (\frac{1}{2},\frac{1+\varepsilon}{2+\varepsilon})$ and $q\in (\frac{1}{1-s},\frac{1+\varepsilon}{s})$. In addition, \eqref{eq:Burgers_m_h_relation} is equivalent to $m\in (1,2+\varepsilon-(1+\varepsilon)(s+\frac{1}{q}))$. Therefore, if $p\in (2,\infty)$ verifies \eqref{eq:Burgers_critical_p} and $q,s,m,h$ are as above, then Theorem \ref{t:Burgers_critical} ensure the existence of a maximal local solution to \eqref{eq:Burger_white_noise} for
$
u_0\in L^0_{\F_0}(\O;B^{\frac{1}{q}-\frac{1}{1+\varepsilon}}_{q,p}(\Tor)).
$
\end{example}

\subsection{Discussion and further extensions}
\label{ss:Discussion_further_extensions}

\subsubsection{$x$-dependent coefficients}
In the results of Sections \ref{ss:conservative_RD}-\ref{ss:Burgers_semilinear} we only used the assertion $(A,B)\in \MRta$ of Lemma \ref{l:SMR_semilinear_PDEs}. If $(A,B)$ in Assumption \ref{ass:SND} have $x$-dependent coefficients but still satisfies $(A,B)\in \MRta$, then all local existence and regularity results extend to this setting. In the time-independent case (or time-continuous case) many of such results are available as follows from Theorem \ref{t:H_infinite_SMR} (and \cite[Section 5]{NVW11eq}). However, only under a smallness condition on $b_{jn}$.

In the case $p=q$ much more is known on $(A,B)\in \MRta$ with $x$-dependent coefficients.
In particular, from \cite{Kry} and the discussion in Section \ref{ss:operators_with_SMR} we see that stochastic maximal $L^p$-regularity holds in the case the coefficients $a_{ij}$ and $b_{jn}$ are smooth in space. Moreover, some results can be extended to systems as in \cite[Section 5]{VP18}.
In our opinion the restriction $p=q$ seem quite unnatural for the $x$-dependent variant of the SPDEs considered in the previous sections. This motivates to extend the theory to $p\neq q$ as well. At the moment this seems out of reach if the coefficients $a_{ij}$ and $b_{jn}$ only have measurable dependence in $(t,\omega)\in [0,T]\times \Omega$ or if the $b_{jn}$ are not small.

As an illustrations let us mention that for $s=0$ and $p=q$, the conditions of Theorem \ref{t:Reaction_diffusion_div_critical} become
\begin{align*}
\frac{d(h-1)}{h}<p<d(h-1) \ \ \text{and} \ \ p\geq \frac{d+2}{h-1}.
\end{align*}
One can check that this will create cases in which not all $h>1$ can be treated. For instance for $d=2$, $h\in (1, 2]$ has to be excluded. Similar restrictions occur in Theorems \ref{t:critical_space_reaction_diffusion} and \ref{t:QND_application_critical_spaces_gradient}. On the other hand, as explained before we can allow $x$-dependent coefficients $a_{ij}$ and $b_{jn}$ using the pointwise extension of Assumption \ref{ass:SND} to the $x$-dependent setting under some smoothness conditions in $x$.

\subsubsection{Lower order terms}
\label{sss:lower_order_terms}
The results of the previous subsections hold if we add lower order terms in the differential operators \eqref{eq:SND_quasi_AB_n_Def}. For instance, one may substitute $\A$ by $\A+\A_{\ell}$ where $\A_{\ell}(t)u:=\sum_{j=1}^d a_{j}(t,\cdot) \partial_j u + a_0(t,\cdot) u$. To see this, one can take $F_L(t,u):=\A_{\ell}(t)u$ and, under suitable assumptions on $a_{0},\dots,a_d$, the assumption \ref{HFcritical_weak} is satisfied. Another possibility, to allow lower order terms in \eqref{eq:SND_quasi_AB_n_Def} is to use a perturbation theorem to check stochastic maximal $L^p$-regularity. Yet another possibility is to include the lower order terms in the nonlinearity $f$. It depends on each specific case what is the best solution.

\subsubsection{Results on $\Tor^d$}\label{ss:stochastic_reaction_diffusion_torus}
The results of Subsections \ref{ss:conservative_RD}-\ref{ss:reaction_diffusion_gradient_nonlinearities} hold if $\R^d$ is replaced by the torus $\Tor^d$. Moreover, in this case, the assumptions on the nonlinearities can be slightly weakened. Indeed, for instance in Section \ref{ss:conservative_RD} the Lipschitz condition can be replaced by the following: there exist $h>1$ and $C>0$ such that a.s.\ for all $t\in \I_T$, $z,z'\in \R$ and $x\in \RR$,
\[|f(t,x,z)-f(t,x,z')|+\|g(t,x,z)-g(t,x,z')\|_{\ell^2}\leq C (1+|z|^{h-1}+|z'|^{h-1})|z-z'|.\]
The only difference is that an additional constant $C$ is added on the right-hand side. Since $\Tor^d$ has finite volume this does not lead to any problems. The same applies to Sections \ref{ss:semilinear_reaction_diffusion} and \ref{ss:reaction_diffusion_gradient_nonlinearities}.
Moreover, the conditions on $f$ and $g$ in Section \ref{ss:Burgers_semilinear} can be weakened in the same way.

\subsubsection{Results on domains with Dirichlet boundary conditions}\label{ss:stochastic_reaction_diffusion_domains}

In the subsection we assume that $\Dom$ is a $C^2$-boundary with compact boundary. Here, we study \eqref{eq:semilinear_reaction_diffusion_prototype} on $\Dom$ with homogeneous Dirichlet boundary condition, i.e.\
\begin{equation}
\label{eq:semilinear_Dirichlet_conditions}
u=0\qquad \text{on} \qquad \partial\Dom.
\end{equation}
We show that the results in Subsection \ref{ss:conservative_RD}-\ref{ss:Burgers_semilinear} still hold in this case with the same nonlinearities, $\A=\Delta$ and $b_{jn}$ sufficiently small in a suitable norm.
In this case, the scales $\HD^{s,q}(\Dom)$ and $\BD^{s,q}(\Dom)$ introduced in Example \ref{ex:extrapolated_Laplace_dirichlet} play the role of $H^{s,q}(\R^d)$ and $B^{s,q}(\R^d)$. Note that
\[\HD^{2,q}(\Dom)=W^{2,q}(\Dom)\cap W^{1,q}_0(\Dom), \HD^{1,q}(\Dom)=W^{1,q}_0(\Dom) \ \text{and} \ \HD^{0,q}(\Dom)=L^q(\Dom).\]
We refer to Example \ref{ex:extrapolated_Laplace_dirichlet} for more details and other properties.

Under the previous assumptions, in each case considered in Subsections \ref{ss:conservative_RD}-\ref{ss:reaction_diffusion_gradient_nonlinearities} we may rewrite \eqref{eq:semilinear_reaction_diffusion_prototype} with $\A=\Delta$ and boundary value \eqref{eq:semilinear_Dirichlet_conditions}, as a stochastic evolution equation on $X_0:=\HD^{-1-s,q}(\Dom)$, $X_1:=\HD^{1-s,q}(\Dom)$ for some $s\in[-1,1]$, $q\in [2,\infty)$ and
\begin{align*}
A(t)u&=\Dd_{-1-s,q}u, &  B(t)u &=0,
\\ F(t,u)&=f(t,u,\nabla u), & G(t,u)&=G_L(t,u)+\wt{G}(t,u),
\\ G_L(t,u)&=(\bb_n (t)u)_{n\geq 1}, &\wt{G}(t,u)&=(g_n(\cdot,u))_{n\geq 1};
\end{align*}
where $\Dd_{-1-s,q}$ is defined in \eqref{eq:extrapolated_Laplacian_D}. To extend the results we let $s\in [0,1]$ in Subsections \ref{ss:conservative_RD}-\ref{ss:semilinear_reaction_diffusion} and $s=-1$ in Subsection \ref{ss:reaction_diffusion_gradient_nonlinearities}.

Let us note that the estimates for $F,\wt{G}$, performed in Subsections \ref{ss:conservative_RD}-\ref{ss:reaction_diffusion_gradient_nonlinearities}, are obtained by factorization through an $L^r$-space, for some $r\in (1,\infty)$. Therefore, by \eqref{eq:HD_embedding} and the identity $\HD^{0,r}(\Dom)=L^r(\Dom)$, the same estimates hold for $F,\wt{G}$ provided $H^{s,q}$ is replaced by $\HD^{s,q}$. The only new feature is the presence of a non-trivial $G_L$ which takes care of the $b_{jn}$-term. The next result shows that $G_L$ verifies \ref{it:G_L}.

\begin{lemma}
\label{l:small_gradient_noise_domain}
Let $q\in (1,\infty)$ and $T>0$. Let $G_L:\I_T\times \O\times \HD^{1,q}(\Dom)\to \g(\ell^2;L^{q}(\Dom))$ be given by $G_L(t,u):=(\sum_{j=1}^d b_{jn}(t)\partial_j u)_{n\geq 1}$. Then a.s.\ for all $t\in \I_T$,
\begin{equation}
\label{eq:estimate_G_L_allen_cahn}
\|G_L(t,u)\|_{\g(\ell^2;\HD^{-s,q}(\Dom))}\lesssim
\Big( \sup_{j}\|(b_{jn})_{n\geq 1}\|_{L^{\infty}(\I_T\times \O;Y)}\Big)\|u\|_{\HD^{1-s,q}(\Dom)},
\end{equation}
in each of the following cases:
\begin{enumerate}[{\rm(1)}]
\item\label{it:small_gradient_weak_setting} $s=0$ and $Y=L^{\infty}(\Dom;\ell^2)$;
\item\label{it:small_gradient_strong_setting} $s=-1$, $Y=W^{1,\infty}(\Dom;\ell^2)$ and $b_{jn}=0$ on $\partial \Dom$ for all $j,n$;
\item\label{it:small_gradient_ultra_weak_setting} $s\in [0,1]$ and $Y=W^{1,\infty}(\Dom;\ell^2)$.
\end{enumerate}
\end{lemma}

\begin{proof}
Since the estimate are pointwise with respect to $(t,\om)$, we fix $(t,\om)$ but we omit it from our notation.

\eqref{it:small_gradient_weak_setting} : In this case, \eqref{eq:estimate_G_L_allen_cahn} follows immediately from $\HD^{1,q}(\Dom)=W^{1,q}_0(\Dom)$ (see \eqref{eq:HD_identification}) and \eqref{eq:gammaidentity}.

\eqref{it:small_gradient_strong_setting}: Note that \eqref{eq:HD_identification} and \eqref{eq:gammaidentity} implies
\begin{equation}
\label{eq:identification_g_values_W}
\g(\ell^2,\HD^{1,q}(\Dom))=W^{1,q}_0(\Dom;\ell^2).
\end{equation}
Therefore, the chain rule yields $G_L:\HD^{2,q}(\Dom)\to W^{1,q}(\Dom;\ell^2)$. Since $b_{jn}=0$ on $\partial\Dom$, it follows that $G_L(u)$ takes values in $\g(\ell^2,\HD^{1,q}(\Dom))$ by \eqref{eq:identification_g_values_W}.

The prove of \eqref{eq:estimate_G_L_allen_cahn} in the case \eqref{it:small_gradient_ultra_weak_setting} is more involved. Let us note that since the map $u\mapsto G_L(u)$ is linear and $\HD^{1,q}(\Dom)\hookrightarrow L^q(\Dom)$ is dense, the extension of $G_L$ is solely determined by $G_L$ on $\HD^{1,q}(\Dom)$. We prove \eqref{eq:estimate_G_L_allen_cahn} by complex interpolation. For this we use \eqref{eq:HD_complex_interpolation} and \cite[Theorem 9.1.25]{Analysis2}. Therefore, it is enough to prove \eqref{eq:estimate_G_L_allen_cahn} in the cases $s\in\{0,1\}$. Since $s=0$ was already considered, it remains to consider $s=1$. By \eqref{eq:HD_duality} and \eqref{eq:HD_identification} we have $\HD^{-1,q}(\Dom)=(\HD^{1,q'}(\Dom))^*=(W^{1,q'}_0(\Dom))^*$. By trace duality, i.e.\ \cite[Theorem 9.4.1]{Analysis2}, it follows that $\g(\ell^2,\HD^{-1,q}(\Dom))=(\g(\ell^2,\HD^{1,q'}(\Dom)))^*$. Moreover, $\g(\ell^2,\HD^{1,q'}(\Dom))=W^{1,q'}_0(\Dom;\ell^2)$ by \eqref{eq:identification_g_values_W}. Thus, we define $G_L: L^q(\Dom)\to (\g(\ell^2,\HD^{1,q}(\Dom)))^*$ by setting
\begin{equation}
\label{eq:def_G_L_domain_weak}
\l v,G_L(u)\r:=-\sum_{n\geq 1}\sum_{j=1}^d\int_{\Dom} u\,\partial_j(b_{jn} v_n) dx,\qquad \forall v\in W^{1,q'}_0(\Dom;\ell^2).
\end{equation}
Using integration by parts argument one sees that this coincide with $G_L$ in the lemma for $u\in W^{1,q}_0(\Dom)$. By \eqref{eq:def_G_L_domain_weak} and H\"{o}lder's inequality
\begin{align*}
\l v,G_L(t,u)\r&\leq \|u\|_{L^q(\Dom)}\sum_{n\geq 1}\sum_{j=1}^d\|\partial_j(b_{jn}(t)v_n)\|_{L^{q'}(\Dom)}\\
&\lesssim \|u\|_{L^q(\Dom)}\Big( \sup_{j}\|b_{jn}(t)\|_{W^{1,\infty}(\Dom;\ell^2)}\Big)
\|v\|_{W^{1,q'}_0(\Dom;\ell^2)}.
\end{align*}
Taking the supremum over all $\|v\|_{W^{1,q'}_0(\Dom;\ell^2)}\leq 1$, \eqref{eq:estimate_G_L_allen_cahn} one obtains for $s=1$.
\end{proof}

\begin{remark}\
\label{r:AC_noise}
\begin{itemize}
\item The argument given in the proof of Lemma \ref{l:small_gradient_noise_domain} can be refined using bilinear interpolation (see e.g.\ \cite[Section 4.4]{BeLo}) and in the case \eqref{it:small_gradient_ultra_weak_setting} we may choose $Y=C^{\alpha}(\Dom;\ell^2)$ for some $\alpha>|s|$.
\item The proof of \eqref{eq:estimate_G_L_allen_cahn}, in the case \eqref{it:small_gradient_strong_setting}, shows that $b_{jn}|_{\partial\Dom}=0$ can be replaced by an `orthogonality condition' on $\partial \Dom$, see Assumption \ref{ass:QND_domain_ortogonality} below. Indeed, since \eqref{eq:semilinear_Dirichlet_conditions} holds, $\nabla u$ is parallel to the exterior normal field $\n$. Thus, if $(b_{jn})_{j\in \{1,\dots,d\}}$ is orthogonal to $\n$, then $\sum_{j=1}^d b_{jn}\partial_j u=0$ on $\partial\Dom$.
\end{itemize}
\end{remark}

The considerations at the beginning of this subsection, Lemma \ref{l:small_gradient_noise_domain} and the smallness condition \eqref{eq:smallness_semilinear} in Theorem \ref{t:semilinear} show the following.

\begin{theorem}
\label{t:Dirichlet_extensions_semilinear}
Let $\varepsilon>0$. Assume that $\A=\Delta$ and $b_{jn}\in L^{\infty}_{\Progress}(\I_T\times \O;Y)$ be such that $\sup_{j\in \{1,\dots,d\}}\|(b_{nj})_{n\geq 1}\|_{ L^{\infty}(\I_T\times \O;Y)}\leq \varepsilon$. Then in each of the following cases there exists $\bar{\varepsilon}>0$ such that the statement holds for all $\varepsilon\leq \bar{\varepsilon}$:
\begin{enumerate}[{\rm(1)}]
\item\label{it:Dirichlet_extension} If $Y=W^{1,\infty}(\Dom;\ell^2)$ and Assumption \ref{ass:RDC} holds, then the results in Subsections \ref{ss:conservative_RD}, \ref{ss:Burgers_semilinear} hold for \eqref{eq:semilinear_reaction_diffusion} on $\Dom$ with the boundary condition \eqref{eq:semilinear_Dirichlet_conditions}.
\item\label{it:Dirichlet_L_infty_domains} If $Y=L^{\infty}(\Dom;\ell^2)$ and Assumption \ref{ass:RD} holds, then the results in Subsection \ref{ss:semilinear_reaction_diffusion} hold for \eqref{eq:semilinear_reaction_diffusion_l_m} on $\Dom$ with the boundary condition \eqref{eq:semilinear_Dirichlet_conditions}.
\item If $Y=W^{1,\infty}(\Dom;\ell^2)$, $b_{jn}|_{\partial\Dom}=0$ and Assumption \ref{ass:SND_reaction_diffusion_gradient} holds, then the results in Subsection \ref{ss:reaction_diffusion_gradient_nonlinearities} holds for \eqref{eq:semilinear_reaction_diffusion_gradient_non_linearities} on $\Dom$ with boundary condition \eqref{eq:semilinear_Dirichlet_conditions}.
\end{enumerate}
\end{theorem}
\begin{proof}
By the previous discussion it remains to prove that $\Dd_{2\alpha,q}\in \MRta$ for all $q\in (1,\infty)$ and $\alpha\geq -1$. By Example \ref{ex:extrapolated_Laplace_dirichlet}, $\Dd_{2\alpha,q}$ has a bounded $H^{\infty}$-calculus with angle $<\pi/2$. Therefore, the claim follows from Theorem \ref{t:H_infinite_SMR}.
\end{proof}

\begin{remark}
By employing the $\HN$-scale constructed in Example \ref{ex:Laplacian_Neumann_scales} for the Neumann Laplacian, one can extend Theorem \ref{t:Dirichlet_extensions_semilinear} to homogeneous Neumann boundary conditions
$\partial_{\n}u=0$ on $\partial\Dom$, see Subsection \ref{ss:Allen_Cahn_mass_conservative} below.
\end{remark}

\section{Applications to quasilinear SPDEs with gradient noise}
\label{s:quasi_second_gradient_noise}
In this section we study quasilinear SPDEs which can be rewritten in the form \eqref{eq:QSEE} with $H=\ell^2$ (Subsections \ref{ss:QND_R_torus}-\ref{ss:porous_media_equations}) or $H=H^{\delta,2}$ (Subsection \ref{ss:Burgers_quasilinear}). In the next subsection we motivate and explain the class of equations which will be considered.

\subsection{Introduction and motivations}
Quasilinear parabolic SPDEs have been intensively studied in literature. In the deterministic case the monograph \cite{LU68} contains the classical theory. Quasilinear SPDEs arise in many areas of applied science since they model reaction-diffusion equations in which the diffusivity depends strongly on the property itself. For this and more physical motivations we refer to \cite{D96,DV12,DFVE14,K98,MT99}. For a mathematical perspective one may consult \cite{DdMH15,HZ17,KimKimquasi2,KuehnNeamtu2020}. To the best of our knowledge, except for the paper \cite{FG19}, there is no other treatment in the literature for quasilinear stochastic PDEs where the coefficients $b_{jnk}$ appearing in the gradient noise term (see \eqref{eq:QND_quasi_AB_n_Def} below) may depend on $u$. However, our approach and setting is quite different from the one used in \cite{FG19} due to a different choice of the leading operators (in \cite{FG19} they may be degenerate) and a different choice of the noise.

In this section we analyse quasilinear systems of second order stochastic PDEs in non-divergence form with nonlinear gradient noise on a domain $\Dom\subseteq \R^d$:
\begin{equation}
\label{eq:quasilinear_non_divergence}
\begin{cases}
\displaystyle du +\A(\cdot, u,\nabla u) u dt= f(\cdot, u,\nabla u) dt +
\sum_{n\geq 1}(\bb_n(\cdot, u)\cdot \nabla u +g_n(\cdot, u)) d{w}_{t}^n,\\
u(0)=u_0.
\end{cases}
\end{equation}
Here $(w_t^n:t\geq 0)_{n\geq 1}$ denotes a sequence of independent Brownian motions and $u:[0,T]\times \Omega\times \Dom\to \R^N$ is the unknown process where $N\geq 1$. The differential operators $\A,\bb_n$ for each $x\in \Dom$, $\om\in\O$, $t\in(0,T)$, are given by
\begin{equation}
\label{eq:QND_quasi_AB_n_Def}
\begin{aligned}
(\A(t,\om,v,\nabla v)u)(t,\om,x)&:=-\sum_{i,j=1}^d a_{ij}(t,\om,x,v(x),\nabla v(x))\partial_{ij}^2 u(x),\\
(\bb_n(t,\om,v)u)(t,\om,x)&:=
\Big(\sum_{j=1}^d b_{jkn}(t,\om,x,v(x))\partial_j u_k(x)\Big)_{k=1}^N.
\end{aligned}
\end{equation}
Note that $\A,\bb_n$ generalize the differential operators in \eqref{eq:SND_quasi_AB_n_Def} studied in Section \ref{s:semilinear_gradient}. As we saw in Subsection \ref{sss:lower_order_terms}, lower order terms in \eqref{eq:QND_quasi_AB_n_Def} can be allowed here as well. Furthermore, as in Subsection \ref{ss:introduction_motivation_semilinear}, the following splitting arises naturally:
\begin{itemize}
\item $\Dom=\R^d$ or $\Dom=\Tor^d$;
\item $\Dom$ is a sufficiently smooth domain in $\R^d$.
\end{itemize}
In Subsection \ref{ss:QND_R_torus} we will only consider $\R^d$ in detail since the case $\Tor^d$ is similar. Under additional assumptions, in Subsection \ref{ss:QND_Dirichlet} we study \eqref{eq:quasilinear_non_divergence} with Dirichlet boundary condition. Subsection \ref{ss:quasilinear_divergence} is devoted to equations in divergence form. We remark that in the latter section, we can deal only with a small gradient noise term.

The following assumption will be in force in Subsections \ref{ss:QND_R_torus}-\ref{ss:QND_Dirichlet}:
\begin{assumption}
\label{ass:QND}
Suppose that one of the following two conditions hold:
\begin{itemize}
\item $p\in (2,\infty)$ and $\a\in [0,\frac{p}{2}-1)$.
\item $p=2$ and $\a = 0$.
\end{itemize}
Assume the following conditions on $a_{ij}, b_{jkn}$:
\begin{enumerate}[{\rm(1)}]
\item \label{it:QND_a_b}
For each $i,j\in \{1,\dots,d\}$ and $n\geq 1$, the maps $a_{ij}:(0,T)\times \O\times \Dom\times \R^N\times \R^{N\times d}\to \R^{N\times N}$ and
$b_{jkn}:(0,T)\times \O\times \Dom\times \R^N\to \R$ are $\Progress\otimes \Borel(\Dom)\otimes \Borel(\R^N)\otimes \Borel(\R^{N\times d})$ and $\Progress\otimes \Borel(\Dom)\otimes \Borel(\R^N)$-measurable, respectively.

Moreover, for every $r>0$ there exist constants $L_r,M_r>0$ and an increasing continuous function $K_r:\R_+\to\R_+$ such that $K_r(0)=0$ and for a.a.\ $\om\in\O$ for all $t\in [0,T]$, $i,j\in\{1,\dots,d\}$, $x,x'\in \Dom$, $y\in\B_{\R^N}(r)$, $z\in \B_{\R^{N\times d}}(r)$,
\begin{align*}
|a_{ij}(t,\om,x,y,z)|+\|(b_{jkn}(t,\om,\cdot,y))_{n\geq 1}\|_{W^{1,\infty}(\Dom;\ell^2)}&\leq M_r,\\
|a_{ij}(t,\om,x,y,z)-a_{ij}(t,\om,x',y,z)|&\leq K_r(|x-x'|).
\end{align*}
\item \label{it:QND_continuity}
For each $r>0$ there exists $C_r>0$ such that for all $i,j\in \{1,\dots,d\}$, $x\in \Dom$, $y,y'\in \B_{\R^N}(r)$, $z,z'\in \B_{\R^{N\times d}}(r)$, $t\in [0,T]$, $k\in \{1,\dots,N\}$ and a.a.\ $\om\in \O$,
\begin{align*}
&|a_{ij}(t,\om,x, y,z)-a_{ij}(t,\om,x,y',z')|
+\| (b_{jkn}(t,\om,x,y)-  b_{jkn}(t,\om,x',y'))_{n\geq 1}\|_{\ell^2}\\
& \ \  \ \ \ \|(\nabla_{y} b_{jkn}(t,\om,x,y)- \nabla_{y} b_{jkn}(t,\om,x',y'))_{n\geq 1}\|_{\ell^2\times \R^{N}}
\leq C_r(|y-y'|+|z-z'|).
\end{align*}
\item \label{it:QND_ellipticity}
For each $r>0$ there exists $\epsilon_r>0$ such that a.s.\ for all $\xi\in \R^d$, $\theta\in \R^N$, $t\in [0,T]$, $x\in \Dom$, $y\in \B_{\R^N}(r)$ and $z\in \B_{\R^{N\times d}}(r)$ one has
$$
\sum_{i,j=1}^d \xi_i\xi_j ((a_{ij}(t,\om,x,y,z)-\Sigma_{ij}(t,\om,x,y))\theta,\theta)_{\R^N}\geq \epsilon_r |\xi|^2|\theta|^2.
$$
Here for each fixed $i,j\in \{1,\dots,d\}$, $\Sigma_{ij}(t,\om,x,y)$ is the $N\times N$ matrix with the diagonal elements
$$\Big(\frac{1}{2}\sum_{n\geq 1} b_{ikn}(t,\om,x,y)b_{jkn}(t,\om,x,y)\Big)_{k=1}^N.$$
\end{enumerate}
\end{assumption}

In Subsection \ref{ss:QND_R_torus} we study \eqref{eq:quasilinear_non_divergence} under the following assumption.

\begin{assumption}
\label{ass:QND_f_g_regular_trace_space}
The maps $f:\I_T\times \O\times \Dom\times \R^N\times \R^{N\times d}\to \R^N$, $g:=(g_n)_{n\geq 1}:\I_T\times \O\times \Dom\times \R^N\times \R^{N\times d}\to \ell^2\times\R^N$ are $\Progress\otimes \Borel(\Dom)\otimes \Borel(\R^N)\otimes \Borel(\R^{N\times d})$ and $\Progress\otimes \Borel(\Dom)\otimes \Borel(\R^N)$-measurable respectively. Assume $f(\cdot, 0) = 0$ and $g(\cdot, 0)=\nabla_y g(\cdot,0) = 0$. Moreover, for each $r>0$ there exists $C_r>0$ such that a.a.\ $\om\in\O$, for all $t\in [0,T]$, $x\in \Dom$, $y,y'\in\B_{\R^N}(r)$ and $z,z'\in\B_{\R^{N\times d}}(r)$,
\begin{align*}
|f(t,x,y,z)-f(t,x,y',z')|&\leq C_r (|y-y'|+|z-z'|),\\
\|g(t,x,y)-g(t,x,y')\|_{\ell^2} + \|\nabla_y g(t,x,y)-\nabla_y g(t,x,y')\|_{\ell^2}&\leq C_r |y-y'|.
\end{align*}
\end{assumption}

In the next subsection, under additional assumption on $f,g$, we extend the results in Subsection \ref{ss:reaction_diffusion_gradient_nonlinearities} to suitable quasilinear equations; see Theorems \ref{t:QND_gradient_not_critical}-\ref{t:QND_gradient_critical} there.

\begin{remark}
The parabolicity condition in Assumption \ref{ass:QND}\eqref{it:QND_ellipticity} extends the one we have seen in Assumption \ref{ass:SND}\eqref{it:SND_ellipticity} to the case of $x$-dependent coefficients and systems. It was considered in the above form in \cite{VP18}, where complex matrix-valued $a_{ij}$ were allowed as well. Some diagonal condition is assumed for the $b$-term, because otherwise the result does not hold in general (see \cite{BrzVer11, DuLiuZhang, KimLeesystems} for further discussion on this topic).

Unlike in Sections \ref{ss:conservative_RD}, \ref{ss:semilinear_reaction_diffusion}, and \ref{ss:reaction_diffusion_gradient_nonlinearities} we will be assuming $p=q$ in many of the results below. This is mainly because the quasilinear structure of the equation will imply that our operators will have coefficients depending on $(t,\omega,x)$. Unfortunately, no $L^p(L^q)$-theory is available for $p\neq q$ if only measurability in time is assumed. Of course in the case the coefficients are $(\omega,x)$-dependent, there is a theory with $p\neq q$ as follows from Theorem \ref{t:H_infinite_SMR}. However, at the same time we would like the $b$-term to satisfy the right parabolicity condition, and almost no general $L^p(L^q)$-theory with $p\neq q$ is available in this case.
\end{remark}

\subsection{Quasilinear SPDEs in non-divergence form on $\R^d$}
\label{ss:QND_R_torus}
In this section we study \eqref{eq:quasilinear_non_divergence} on $\R^d$. For the function spaces needed below, we employ the notation introduced in Subsection \ref{ss:introduction_motivation_semilinear}.

To begin, we recast \eqref{eq:quasilinear_non_divergence} as a quasilinear evolution equations in the form \eqref{eq:QSEE} on $X_0:=L^p(\R^d;\R^N)$ and $X_1:=W^{2,p}(\R^d;\R^N)$ by setting, for $u\in C^1(\R^d;\R^N)$ and $v\in W^{2,p}(\R^d;\R^N)$
\begin{align*}
A(t,u)v&=\A(t,u,\nabla u)v, &  B(t,u)v &=(\bb_n (t,u)v)_{n\geq 1},
\\ F(t,u)&=f(t,u,\nabla u), & G(t,u)&=(g_n(t,\cdot,u))_{n\geq 1}.
\end{align*}
By \eqref{eq:QND_quasi_AB_n_Def} and $u\in C^1(\R^d;\R^N)$ all the above maps are well-defined. As usual, we say that $(u,\sigma)$ is a maximal local solution to \eqref{eq:quasilinear_non_divergence} on $\RR$ if $(u,\sigma)$ is a maximal local solution to \eqref{eq:QSEE} in the sense of Definition \ref{def:solution2}.

The first result of this section is as follows:
\begin{theorem}
\label{t:QND_general_R_Tor}
Let the Assumptions \ref{ass:QND}-\ref{ass:QND_f_g_regular_trace_space} be satisfied for $\Dom=\R^d$. Assume that $p>2(1+\a)+d$. Then for any
\[u_0\in L^0_{\F_0}(\O;W^{2-\frac{2(1+\a)}{p},p})\]
there exists a maximal local solution $(u,\sigma)$ to \eqref{eq:quasilinear_non_divergence}. Moreover, there exists a localizing sequence $(\sigma_n)_{n\geq 1}$ such that for all $n\geq 1$ and a.s.
$$
u\in L^p(\I_{\sigma_n},w_{\a};W^{2,p})\cap
C(\overline{\I}_{\sigma_n};W^{2-2\frac{1+\a}{p},p})\cap C((0,\sigma_n];W^{2-\frac{2}{p},p}).
$$
\end{theorem}

\begin{proof}
We apply Theorem \ref{t:local_Extended} with $F_L\equiv F_c\equiv G_L\equiv G_c\equiv 0$, $F_{\Tr}:=f$ and $G_{\Tr}:=(g_n)_{n\geq 1}$.
For this it remains to check \ref{HAmeasur}, \ref{HFcritical_weak}, \ref{HGcritical_weak} and \eqref{eq:stochastic_maximal_regularity_assumption_local_extended}. For the sake of clarity we split the proof into several steps.

\textit{Step 1: \emph{\ref{HAmeasur}} holds}. Since $p>2(1+\a)+d$, by Sobolev embedding one has
\begin{equation}
\label{eq:QND_embedding}
\Xap=W^{2-\frac{2(1+\a)}{p},p}\hookrightarrow C^{1+\epsilon},\quad\text{ for some }\epsilon>0.
\end{equation}
Fix $r>0$, and let $u_1,u_2\in \B_{\Xap}(r)$. By \eqref{eq:QND_embedding} it follows that $\|u_1\|_{W^{1,\infty}},\|u_2\|_{W^{1,\infty}}\leq C r=:R$ where $C$ depends only on $p,d$. Thus,
\begin{align*}
\|A(t,u_1)v-A(t,u_2)v\|_{L^q}\leq C_{R} \|u_1-u_2\|_{W^{1,\infty}}\|v\|_{W^{2,q}}
\leq C_R \|u_1-u_2\|_{\Xap}\|v\|_{W^{2,q}},
\end{align*}
where $C_R$ is as in Assumption \ref{ass:QND} \eqref{it:QND_continuity}. The same argument holds for $B$.

\textit{Step 2: \eqref{eq:stochastic_maximal_regularity_assumption_local_extended} holds}. It is enough to prove that $(A(\cdot,w_0),B(\cdot,w_0))\in\MRta$ for all $w_0\in L^{\infty}_{\F_0}(\O;\Xap)$. By \eqref{eq:QND_embedding}, it follows that $ w_0\in  L^{\infty}_{\F_0}(\O;C^{1+\epsilon})$. Now the claim follows from \cite[Theorem 5.4]{VP18} and Assumption \ref{ass:QND}.

\textit{Step 3: \emph{\ref{HFcritical_weak}} and \emph{\ref{HGcritical_weak}} holds}. By \eqref{eq:QND_embedding} and the assumption on $f,g_n$ it follows easily that for any $n\geq 1$ and any $u,v\in \B_{\Xap}(n)$ one has
\begin{equation*}
\|f(\cdot,u,\nabla u)-f(\cdot,v,\nabla v)\|_{L^p}+\|g(\cdot,u)-g(\cdot,v)\|_{W^{1,p}(\ell^2)}\leq C_n \|u-v\|_{\Xap};
\end{equation*}
where $C_n>0$ may depends on $n\geq 1$.
\end{proof}

Theorem \ref{t:QND_general_R_Tor} gives local existence for \eqref{eq:quasilinear_non_divergence} under quite general assumptions on $f,g_n$. The drawback in applying Theorem \ref{t:QND_general_R_Tor} is that the trace space in \eqref{eq:QND_embedding} is very regular and therefore the initial values have to be rather smooth. Under additional assumptions on $a_{ij},b_{jnk}$ we can admit rougher trace spaces $\Xap$ for \eqref{eq:quasilinear_non_divergence}. To do so we will partially extend the results in Subsection \ref{ss:reaction_diffusion_gradient_nonlinearities}. In particular, the following extends Theorem \ref{t:QND_application_gradient_control_growth_lipschitz_constant} in the case $q=p$.

\begin{theorem}
\label{t:QND_gradient_not_critical}
Suppose that Assumptions \ref{ass:SND_reaction_diffusion_gradient} and \ref{ass:QND} hold. Assume $d\geq 1$. Assume that $a_{ij}(t,\om,x,y,z)$ does not depend on the $z$-variable and $b_{jkn}(t,\om,x,y)$ does not depend on the $y$ variable. Moreover, suppose that
\begin{equation}
\label{eq:limitations_p_QND_gradient_not_critical}
p\geq \frac{m-1}{m}(2(1+\a)+d).
\end{equation}
Then for each
\[u_0\in L^0_{\F_0}(\O;W^{2-\frac{2(1+\a)}{p},p})\]
there exists a maximal local solution $(u,\sigma)$ to \eqref{eq:quasilinear_non_divergence}. Moreover, there exists a localizing sequence $(\sigma_n)_{n\geq 1}$ such that for all $n\geq 1$ and a.s.
$$
u\in L^p(\I_{\sigma_n},w_{\a};W^{2,p})\cap C(\overline{I}_{\sigma_n};W^{2-\frac{2(1+\a)}{p},p})
\cap C((0,\sigma_n];W^{2-\frac{2}{p},p}).
$$
\end{theorem}

Recall that typical examples of non-linearities which satisfies Assumptions \ref{ass:SND_reaction_diffusion_gradient}
are $f(u,\nabla u)=|u|^c|\nabla u|^r$ with $c,r>1$ and $f(\nabla u) = |\nabla u|^r$ with $r>2$.

\begin{proof}
The proof is similar to the one proposed in Theorem \ref{t:QND_application_gradient_control_growth_lipschitz_constant} with $q=p$. Note that if $q=p$, the restrictions in Theorem \ref{t:QND_application_critical_spaces_gradient} reduce to \eqref{eq:limitations_p_QND_gradient_not_critical}.

Additionally, we need to check that for $w_0\in L^{\infty}_{\F_0}(\O;\Xap)$ and $q=p$, one has $(A(w_0),B(w_0))\in \MRta$. Since these operators have $x$-dependent coefficients, Lemma \ref{l:SMR_semilinear_PDEs} is not applicable.
By \eqref{eq:limitations_p_QND_gradient_not_critical} it follows that $2-2(1+\a)/p>d/p$. Therefore, by Sobolev embedding
\begin{equation}
\label{eq:Sob_embeddings_QND_gradient}
\Xap=W^{2-2\frac{1+\a}{p},p}\hookrightarrow C^{\eta},\qquad \text{ for some }\eta>0.
\end{equation}
Thus, $(A(w_0),B(w_0))\in \MRta$ follows from \eqref{eq:Sob_embeddings_QND_gradient}, Assumption \ref{ass:QND} and \cite[Theorem 5.4]{VP18}.
\end{proof}
As a consequence we obtain the following result in the critical case in the same way as in the proof of Theorem \ref{t:QND_application_critical_spaces_gradient}. However, since we need $p=q$, we need further restrictions on $q$. To avoid this, one needs further results on stochastic maximal $L^p(L^q)$-regularity with $x$-dependent coefficients.
\begin{theorem}
\label{t:QND_gradient_critical}
Suppose that Assumptions \ref{ass:SND_reaction_diffusion_gradient} and \ref{ass:QND} hold. Assume $d\geq 1$ and $m>1+\frac{2}{d}$. Assume that $a_{ij}(t,\om,x,y,z)$ does not depend on the $z$-variable and $b_{jkn}(t,\om,x,y)$ does not depend on the $y$ variable.
Suppose that
\begin{equation}
\label{eq:range_p_critical_quasilinear_gradient_nonlinearities}
\frac{(m-1)}{m}(2+d)< p< d(m-1).
\end{equation}
Then for any
$$u_0\in L^0_{\F_0}\Big(\O;W^{\frac{d}{p}+\frac{m-2}{m-1},p}\Big)$$
there exists a maximal local solution $(u,\sigma)$ to \eqref{eq:quasilinear_non_divergence}. Moreover, there exists a localizing sequence $(\sigma_n)_{n\geq 1}$ such that a.s.\ for all $n\geq 1$
$$
u\in L^p(\I_{\sigma_n},w_{\a_{\crit}};W^{2,p})\cap C(\overline{I}_{\sigma_n};W^{\frac{d}{p}+\frac{m-2}{m-1},p})
\cap C((0,\sigma_n];W^{2-\frac{2}{p},p}),
$$
where $\a_{\crit}:=\frac{pm}{2(m-1)}-\frac{d}{2}-1$.
\end{theorem}

Note that since $m>1+\frac{2}{d}$, one has $d(m-1)>2$. Therefore, the set of $p$ which satisfies \eqref{eq:range_p_critical_quasilinear_gradient_nonlinearities} is not empty.

\subsection{Quasilinear SPDEs in non-divergence form on domains}
\label{ss:QND_Dirichlet}
In this subsection we investigate the quasilinear problem \eqref{eq:quasilinear_non_divergence} with Dirichlet boundary conditions
\begin{equation}
\label{eq:Dirichlet_boundary_condition}
u=0\qquad \text{on }\partial\Dom.
\end{equation}
Here we assume $\Dom$ is a bounded domain with $C^2$-boundary. Moreover, we let $N = 1$ and write $b_{jn} := b_{j1n}$.

As usual, we recast  \eqref{eq:quasilinear_non_divergence} in the form \eqref{eq:QSEE}.
To this end, for $p\in (1,\infty)$ and $s\in (0,1)$ we set
\begin{align*}
W^{1,p}_0(\Dom)& =\{u\in W^{1,p}(\Dom)\,:\,u|_{\Dom}=0\},
\\ \HD^{2,p}(\Dom)& =W^{2,p}(\Dom)\cap W^{1,p}_0(\Dom),
\\ \WD^{2s,p}(\Dom)& =(L^p(\Dom),\HD^{2,p}(\Dom))_{s,p},
\end{align*}
(see \eqref{eq:def_BD} and \eqref{eq:BD_identifications}). For more on spaces with Dirichlet type boundary conditions see Example \ref{ex:extrapolated_Laplace_dirichlet}. For the reader's convenience, we recall that (see \eqref{eq:HD_identification})
\begin{equation*}
[L^q(\Dom),\HD^{2,p}(\Dom)]_{1/2}= W^{1,p}_0(\Dom) \ \ \ \text{ for all }p\in (1,\infty).
\end{equation*}
To proceed further, let $X_0=L^p(\Dom)$, $X_1=\HD^{2,q}(\Dom)$ and for $u\in \Xap=\BD^{1-2\frac{1+\a}{p}}_{q,p}(\Dom)$ (see \eqref{eq:def_BD}), $v\in X_1$ we set
\begin{equation}
\begin{aligned}
\label{eq:non_divergence_domain_def_ABFG}
A(t,u)v&=\A(t,u,\nabla u)v, \qquad&  B(t,u)v &=(\bb_n (t,u)v)_{n\geq 1},
\\ F(t,u)&=f(t,u,\nabla u), \qquad & G(t,u)&=(g_n(t,\cdot,u))_{n\geq 1}.
\end{aligned}
\end{equation}
where $\A,\bb_n$ are as in \eqref{eq:QND_quasi_AB_n_Def}.
We say that $(u,\sigma)$ is a maximal local solution to \eqref{eq:quasilinear_non_divergence} with boundary condition \eqref{eq:Dirichlet_boundary_condition} if $(u,\sigma)$ is a maximal local solution to \eqref{eq:QSEE} with $A,B,F,G$ in \eqref{eq:non_divergence_domain_def_ABFG}.

Below we will show that for $p>d+2$
\begin{equation}
\label{eq:B_boundary_conditions_Du_revised_version}
B(\cdot,u)v\in \g(\ell^2,W^{1,p}_0(\Dom)), \text{ a.e.\ on }\I_T\times \O \text{ for all }u\in \Xap,\, v\in X_1.
\end{equation}
As remarked in \cite{Du2018} (see the text below Assumption 1.4), to check \eqref{eq:B_boundary_conditions_Du_revised_version} it is sufficient to require an ``orthogonality condition" for $b$ at the boundary of $\Dom$. In the quasilinear setting, this condition reads as follows
\begin{equation}
\label{eq:QND_domain_ortogonality_condition_for_SMR_boundary_Du_revision_stage}
\sum_{j=1}^d b_{jn}(t,\om,x,0)\n_j(x)=0,  \text{ for a.a.\ }(t,\om,x)\in\I_T\times\O\times \partial\Dom
\text{ and all }n\geq 1,
\end{equation}
where $\n=(\n_j)_{j=1}^d$ is the exterior normal field on $\partial\Dom$. To see that \eqref{eq:QND_domain_ortogonality_condition_for_SMR_boundary_Du_revision_stage} implies \eqref{eq:B_boundary_conditions_Du_revised_version}, we argue as follows. Since $p>d+2$, one has
\begin{equation}
\begin{aligned}
\label{eq:Xap_divergence_domain_Du_revision_stage}
\Xap = \WD^{2-\frac{2}{p},p}(\Dom)
&=\big\{ u\in W^{2-\frac{2}{p},p}(\Dom)\,:\,u=0\text{ a.e.\ on }\partial\Dom\big\} \\
&\hookrightarrow
\big\{u\in C^{1+\varepsilon}(\Dom)\,:\,u=0\text{ a.e.\ on }\partial\Dom\big\}
\end{aligned}
\end{equation}
for some $\varepsilon>0$. Note that $\g(\ell^2,W^{1,p}_0(\Dom))=W^{1,p}_0(\Dom;\ell^2)$ by  \eqref{eq:gammaidentity}.
Thus, \eqref{eq:B_boundary_conditions_Du_revised_version} holds thanks to $u=v=0$ a.e.\ on $\partial\Dom$ (thus $\nabla v$ is parallel to $\n$). Assumption \eqref{eq:QND_domain_ortogonality_condition_for_SMR_boundary_Du_revision_stage} was first introduced in \cite{Du2018} (for $\a=0$) where the author showed that under suitable conditions on the coefficients, \eqref{eq:QND_domain_ortogonality_condition_for_SMR_boundary_Du_revision_stage} yields stochastic maximal $L^p$-regularity estimates for the linear case of \eqref{eq:quasilinear_non_divergence} on domains. Ellipticity and smoothness of the coefficients alone are not enough to show maximal regularity estimates for parabolic SPDEs on domains with the choice $X_0=L^q(\Dom)$ and $X_1=\HD^{2,q}(\Dom)$, see \cite[Theorem 5.3]{Krylov03}. To reduce the conditions on the $b$-term one needs to use suitable weighted Sobolev spaces (see Subsection \ref{ss:QND_Dirichletweights} below).

A similar argument shows that item \eqref{it:QND_domain_ortogonality_g_null_at_the_boundary_revision_stage} in the following assumption is sufficient to obtain \eqref{eq:B_boundary_conditions_Du_revised_version} with $B(\cdot,u)v$ replaced by $G(\cdot,v)$ where $G$ is as in \eqref{eq:non_divergence_domain_def_ABFG}.
\begin{assumption}
\label{ass:QND_domain_ortogonality}
Let Assumption \ref{ass:QND} be satisfied. Assume that $N=1$. Suppose that  \eqref{eq:QND_domain_ortogonality_condition_for_SMR_boundary_Du_revision_stage} holds and that the following are satisfied.
\begin{enumerate}[{\rm(1)}]
\item\label{it:QND_domain_ortogonality_domain_regularity_revision_stage} $\Dom$ is a bounded $C^2$-domain in $\R^d$;
\item\label{it:QND_domain_ortogonality_g_null_at_the_boundary_revision_stage} $g_n(t,\om,0)=0$ for a.a.\ $\om\in\O$ and for all $x\in \partial\Dom$.
\end{enumerate}
\end{assumption}

The main result of this subsection is an extension of Theorem \ref{t:QND_general_R_Tor} to domains in case $\a=0$. Using the results of Example \ref{ex:extrapolated_Laplace_dirichlet}, the reader can check that also Theorem \ref{t:QND_gradient_not_critical} (resp.\ \ref{t:QND_gradient_critical}) extends to the problem \eqref{eq:quasilinear_non_divergence} with boundary condition \eqref{eq:Dirichlet_boundary_condition} provided $\a=0$ (resp.\ $p=\frac{m-1}{m}(d+2)$ i.e.\ $\a_{\crit}=0$). For the sake of brevity, we do not include any statement here.

\begin{theorem}
\label{t:QND_domain}
Suppose Assumptions \ref{ass:QND_f_g_regular_trace_space} and \ref{ass:QND_domain_ortogonality} hold. Let $p\in (d+2,\infty)$.
Then for each $u_0\in L^0_{\F_0}\big(\O;\WD^{2-\frac{2}{p},p}(\Dom)\big)$ there exists a maximal local solution to \eqref{eq:quasilinear_non_divergence} with boundary condition \eqref{eq:Dirichlet_boundary_condition}, and a localizing sequence $(\sigma_n)_{n\geq 1}$ such that a.s.\ for all $n\geq 1$
$$
u\in L^p(\I_{\sigma_n};W^{2,p}(\Dom)\cap W^{1,p}_0(\Dom))\cap C(\overline{I}_{\sigma_n};\WD^{2-\frac{2}{p},p}(\Dom)).
$$
\end{theorem}

\begin{proof}
Similar to the proof of Theorem \ref{t:QND_general_R_Tor}, we set $F_L\equiv F_c\equiv G_L\equiv G_c\equiv 0$, $F_{\Tr}(t,u):=F(t,u)$ and $G_{\Tr}(u):=G(t,u)$. Here $F,G$ are as in \eqref{eq:non_divergence_domain_def_ABFG}. As before one sees that \ref{HFcritical_weak} and \ref{HAmeasur} hold. To check \ref{HGcritical_weak} recall that $\g(\ell^2,W^{1,p}_0(\Dom))=W^{1,p}_0(\Dom;\ell^2)$. By Assumption
\ref{ass:QND_domain_ortogonality}\eqref{it:QND_domain_ortogonality_g_null_at_the_boundary_revision_stage} one has $g_n(\cdot,u)=g_n(\cdot,v)=0$ a.e.\ on $\I_T\times\O\times \partial\Dom$  for all $u,v\in B_{\Xap}(n)$. The latter considerations imply, for all $u,v\in B_{\Xap}(n)$,
$$
\|(g_n(\cdot,u)-g_n(\cdot,v))_{n\geq 1}\|_{\g(\ell^2,W^{1,p}_0(\Dom))}\eqsim
 \|(g_n(\cdot,u)-g_n(\cdot,v))_{n\geq 1}\|_{W^{1,p}(\Dom;\ell^2)},
$$
where the implicit constants are independent of $u,v$.
By \eqref{eq:Xap_divergence_domain_Du_revision_stage} and the former one can show that \ref{HGcritical_weak} holds.

To apply Theorem \ref{t:local_Extended} it remains to check that the stochastic maximal $L^p$-regularity assumption. Fix $w_0\in L^\infty_{\F_0}(\Omega;\Xap)$. By \eqref{eq:Xap_divergence_domain_Du_revision_stage}, $w_0=0$ a.e.\ on $\O\times\partial\Dom$. The latter and \eqref{eq:QND_domain_ortogonality_condition_for_SMR_boundary_Du_revision_stage} yield $
\sum_{j=1}^d b_{j n}(\cdot,w_0)\n_j =0
$
a.e.\ on $\I_T\times \O\times \partial\Dom$. Therefore, by \cite[Theorem 2.5]{Du2018} one has $(A(\cdot,w_0),B(\cdot,w_0))\in \MRtz$. Moreover, by Example \ref{ex:extrapolated_Laplace_dirichlet} and Assumption \ref{ass:QND_domain_ortogonality}\eqref{it:QND_domain_ortogonality_domain_regularity_revision_stage}, the operator $-\Dd_p$ (see \eqref{eq:strong_Dirichlet_Laplacian}) has a bounded $H^{\infty}$-calculus of angle $<\pi/2$. Thus, by Theorem \ref{t:H_infinite_SMR} and the transference result Proposition \ref{prop:transferenceSMR} we also obtain $(A(\cdot,w_0),B(\cdot,w_0))\in \MRt$.
\end{proof}

\begin{remark}
\label{r:QND_Dom}
We believe that Theorem \ref{t:QND_domain} can be extended to an $L^p-L^q$ and weighted in time setting. We plan to address this issue in a forthcoming paper.
\end{remark}

\subsection{Quasilinear SPDEs in non-divergence form on domains with weights}
\label{ss:QND_Dirichletweights}
In a series of papers by Krylov and his collaborators stochastic maximal $L^p$-regularity is derived on weighted $L^p$ spaces on bounded domains. For special choices of the weight no additional conditions on $b$ and $g$ arise. We consider exactly the same problem as in Section \ref{ss:QND_Dirichlet}, but this time with weighted function spaces which are more complicated. For function spaces on $\R^d$ with weights we refer to \cite{MV12,MV14_embedding} and the references therein. In particular, to define Besov spaces on $\R^d$ with weights we employ the Definition 3.2 in \cite{MV12}.

Let $v_{\alpha}:\R^d\to (0,\infty)$ be given by $v_{\alpha}(x) = \dist( x,\partial \Dom)^{\alpha}$ where $\alpha\in \R$. For an integer $n\geq 1$, let $W^{n,p}(\Dom,v_{\alpha})$ be the space of all $u\in L^p(\Dom,v_{\alpha})$ for which $\partial^{\beta} u\in L^p(\Dom,v_{\alpha})$ for all $|\beta|\leq n$ endowed with its natural norm.
Let
\[\WD^{n,p}(\Dom,v_{\alpha}) = \{u\in W^{n,p}(\Dom,v_{\alpha}): \tr_{\partial \Dom} u = 0 \ \text{if}\ n>(1+\alpha)/p\}.\]
The trace operator is a bounded operator into $L^p(\partial \Dom)$ (see \cite[Section 3.2]{LV18}). We will only use the above space for $n\in \{1,2\}$ below.
For $s\in (0,1)$ let
\begin{align}
\label{eq:Wspspace}
\Vspace^{2s,p}(\Dom,v_{\alpha}) := (L^p(\Dom,v_{\alpha}),\Vspace^{2,p}(\Dom,v_{\alpha}))_{s,p},\qquad \Vspace\in \{\WD,W\}.
\end{align}
The latter definition requires some care. In the case $\alpha\in (-1,p-1)$ the space $W^{2s.p}(\Dom,v_{\alpha})$ is equivalent to the Besov space $B^{2s}_{p,p}(\Dom,v_{\alpha})$. Here $B^{2s}_{p,p}(\Dom,v_{\alpha})$ is the restricted space to $\Dom$ of $B^{2s}_{p,p}(\R^d,v_{\alpha})$, see e.g.\ \cite[Definition 5.2]{LMV18}. This follows by combining \cite{C92_Extension} and \cite[Proposition 6.1]{MV12}. To see this, it is enough to note that by \cite[Theorem 1.1]{C92_Extension} and \eqref{eq:Wspspace}, for each $s\in (0,1)$, $p\in (1,\infty)$ and $\alpha\in (-1,p-1)$ there exists an extension operator $E$ (cf.\ Definition \ref{def:extension_operator}), i.e.\ a bounded linear operator
\begin{equation}
\label{eq:extension_weights}
E:W^{2s,p}(\Dom,v_{\alpha})\to W^{2s,p}(\R^d,v_{\alpha}), \quad\text{such that}\quad Ef|_{\Dom}=f,
\end{equation}
where $W^{2s,p}(\R^d,v_{\alpha})=B^{2s}_{p,p}(\R^d,v_{\alpha})$. In the case $\alpha\geq p-1$, the space $W^{2s,p}(\Dom,v_{\alpha})$ does not coincide with a weighted Besov space. However, it densely contains the Besov space $B^{2s}_{p,p}(\Dom,v_{\alpha})$ (see \cite[Remark 7.14]{LV18}).

The following is the main assumption of this subsection.

\begin{assumption}\label{ass:Qweight}
Suppose that Assumption \ref{ass:QND} holds with $N =1$, $\kappa=0$ and write $b_{jn} = b_{j1n}$. Suppose that $a_{ij}(t,\om,x,y,z)$ does not depend on the $z$-variable and $b_{jn}(t,\om,x,y)$ does not depend on the $y$ variable. Let $\Dom$ be a bounded $C^2$-domain. Moreover, let $\delta\in (0,1]$ and suppose that
for each $r>0$ there exists $\epsilon_r>0$ such that a.s.\ for all $\xi\in \R^d$, $\theta\in \R^N$, $t\in [0,T]$, $x\in \Dom$, $y\in \B_{\R^N}(r)$ one has
$$
\sum_{i,j=1}^d \xi_i\xi_j \big(a_{ij}(t,\om,x,y)-\Sigma_{ij}(t,\om,x)\big)\geq \delta \sum_{i,j=1}^d \xi_i\xi_j a_{ij}(t,\om,x,y) \geq \epsilon_r |\xi|^2.
$$
Here for each fixed $i,j\in \{1,\dots,d\}$,
$$\Sigma_{ij}(t,\om,x) = \frac{1}{2}\sum_{n\geq 1} b_{in}(t,\om,x)b_{jn}(t,\om,x).$$
Finally, suppose that $p\in (d+2,\infty)$ and $\delta$ satisfy
\[2p-1 - \frac{p}{p(1-\delta)+\delta}<\alpha<2p-d-2.\]
\end{assumption}
The above assumptions imply that $\alpha > p-1$. In the special case that $b_{jn}\equiv 0$, we can take $\delta = 1$, and thus $p-1<\alpha<2p-d-2$. The above parabolicity condition is introduced in \cite{KryLot} and also considered in \cite{Kim04b}.

In this subsection we let
\begin{equation}
\label{eq:choice_X_0_X_1_quasilinear_domain_weights}
X_0 := L^p(\Dom,v_{\alpha}),\quad X_1:= \WD^{2,p}(\Dom,v_{\alpha}),\quad
 \Xp=\WD^{2-2/p,p}(\Dom,v_{\alpha}),
\end{equation}
where the last equality follows from \eqref{eq:Wspspace}. Moreover, we define $A,B,F,G$ be as in \eqref{eq:non_divergence_domain_def_ABFG}. Let us first analyse the linear problem.
\begin{lemma}\label{lem:maxregweightsspace}
Suppose that Assumption \ref{ass:Qweight} holds. Then the following hold:
\begin{enumerate}[{\rm(1)}]
\item\label{it:embedding_weighted_spaces} There exists $\eta>0$ such that
$
\WD^{2-\frac{2}{p},p}(\Dom,v_{\alpha})\hookrightarrow C^{\eta}(\Dom);
$
\item\label{it:stochastic_maximal_L_p_regularity_weighted} For every
\[w_0\in L^\infty_{\F_0}(\Omega;\WD^{2-\frac{2}{p},p}(\Dom,v_{\alpha}))\]
one has $(A(\cdot, w_0), B(\cdot, w_0))\in \MRt$.
\end{enumerate}
\end{lemma}

\begin{proof}
\eqref{it:embedding_weighted_spaces}:
By \eqref{eq:Wspspace} and \cite[Theorem 4.7.2]{BeLo},
\[\WD^{2-\frac{2}{p},p}(\Dom,v_{\alpha}) = \Xp = (X_0, X_1)_{1-\frac1p,p} = (X_{1/2}, X_1)_{1-\frac{2}{p},p}.\]
By \cite[Proposition 3.16]{LV18} one has
\begin{equation}
\label{eq:X12_equivalence_weights_revision_stage}
X_{\frac{1}{2}} = W^{1,p}(\Dom,v_{\alpha}).
\end{equation}
Therefore, by Hardy's inequality (see \cite[Corollary 3.4]{LV18}),
\begin{equation*}
\begin{aligned}
\Xp  & = (X_{\frac{1}{2}}, X_1)_{1-\frac{2}{p}}
\hookrightarrow (L^p(\Dom,v_{\alpha-p}),W^{1,p}(\Dom,v_{\alpha-p}))_{1-\frac{2}{p},p}=W^{1-\frac{2}{p},p}(\Dom,v_{\alpha-p}),
\end{aligned}
\end{equation*}
where the last equality follows from \eqref{eq:extension_weights}. By Assumption \ref{ass:Qweight} one has $\alpha-p\in (-1,p-1)$, therefore the considerations at the beginning of this section imply that $W^{1-\frac{2}{p},p}(\Dom,v_{\alpha-p})=B^{1-\frac{2}{p}}_{p,p}(\Dom,v_{\alpha-p})$. To complete the proof of \eqref{it:embedding_weighted_spaces} it is enough to show that
$
B^{1-\frac{2}{p}}_{p,p}(\Dom,v_{\alpha-p})\hookrightarrow C^{\eta}(\Dom)
$
for some $\eta>0$.  Since $\Dom$ is bounded, using a standard localization argument (see e.g.\ \cite[Section 2.2]{LV18} and the references therein) it is enough to prove
\begin{equation}
\label{eq:intermediate_embedding_weights_in_space_step_2}
B^{1-\frac{2}{p}}_{p,p}(\R^d,g_{\alpha-p})\hookrightarrow B_{\infty,\infty}^{\eta}(\R^d);
\end{equation}
where, for $x=(x_1,x_2,\dots,x_d)$, $g_{\beta}(x):=x_1^{\beta}$ on $|x|\leq 1$ and $g_{\beta}(x):=1$ otherwise.

The embedding in \eqref{eq:intermediate_embedding_weights_in_space_step_2} follows from \cite[Proposition 4.2]{MV14_embedding} and the fact that $1-\frac{2}{p} - \frac{\alpha-p+d}{p}>0$ and $1-\frac{2+d}{p}>0$. The latter are equivalent to $\alpha<2p-d-2$ and $p>d+2$, respectively, which hold by Assumption \ref{ass:Qweight}.

\eqref{it:stochastic_maximal_L_p_regularity_weighted}: Combining \eqref{it:embedding_weighted_spaces}, Assumption \ref{ass:Qweight} and \cite[Theorem 2.9]{Kim04b} it follows that $(A(\cdot, w_0), B(\cdot, w_0))\in \MRtz$. To see the latter note that by Hardy's inequality (see \cite[Corollary 3.4]{LV18}), and \cite[Proposition 2.2]{Lot2000} (also see \cite[Remark 2.9]{KimKry04}) the spaces in \cite{Kim04b} coincide with the ones considered here.

Since by \cite[Theorem 1.1]{LV18}, $-\Dd_p$ has a bounded $H^\infty$-calculus of angle zero on $L^p(\Dom,v_{\alpha})$ (with domain $\WD^{2,p}(\Dom,v_{\alpha})$), by Theorem \ref{t:H_infinite_SMR} and the transference result Proposition \ref{prop:transferenceSMR} we also obtain that $(A(\cdot,w_0),B(\cdot,w_0))\in \MRt$.
\end{proof}

In the next result, we say that $(u,\sigma)$ is maximal local solution to \eqref{eq:quasilinear_non_divergence} if $(u,\sigma)$ is a maximal local solution to \eqref{eq:QSEE} with $A,B,F,G$ and $X_0,X_1$ are as in \eqref{eq:non_divergence_domain_def_ABFG} and \eqref{eq:choice_X_0_X_1_quasilinear_domain_weights} respectively.

\begin{theorem}
\label{t:QND_domainweights}
Suppose Assumptions \ref{ass:QND_f_g_regular_trace_space} and \ref{ass:Qweight} hold, and that $f$ does not dependent on the $z$-variable.
Then for each
$$u_0\in L^0_{\F_0}(\O;\WD^{2-\frac{2}{p},p}(\Dom,v_{\alpha})),$$
there exists a maximal local solution $(u,\sigma)$ to \eqref{eq:quasilinear_non_divergence}. Moreover, there exists a localizing sequence $(\sigma_n)_{n\geq 1}$ such that a.s.\ for all $n\geq 1$
$$
u\in L^p(\I_{\sigma_n};\WD^{2,p}(\Dom,v_{\alpha}))\cap C(\overline{I}_{\sigma_n};\WD^{2-\frac{2}{p},p}(\Dom,v_{\alpha})).
$$
\end{theorem}
Recall that the space $\WD^{2-\frac{2}{p},p}(\Dom,v_{\alpha})$ is defined as in \eqref{eq:Wspspace} and does not coincide with a weighted Besov space if $\alpha\geq p-1$.
\begin{proof}
By Lemma \ref{lem:maxregweightsspace} and the fact that $\Dom$ is bounded, one can argue in the same way as in Theorem \ref{t:QND_domain}. We remark that \ref{HGcritical_weak} is satisfied by setting $G_c=G_L=0$ and $G_{\Tr}=G$. To see this one can argue as in \eqref{eq:estimate_G_gradient_nonlinearities_revision_stage} since $\Xap\hookrightarrow W^{1,p}(\Dom,v_{\alpha})\cap C^{\eta}(\Dom)$ for some $\eta>0$. The latter embedding follows from Lemma \ref{lem:maxregweightsspace}\eqref{it:embedding_weighted_spaces}, \eqref{eq:X12_equivalence_weights_revision_stage} and $\Xap \hookrightarrow X_{1/2}$ due to $1-2\frac{1+\a}{p}>\frac{1}{2}$.
\end{proof}

\begin{remark}\
\begin{enumerate}[\rm(1)]
\item It would be interesting to extend the above to $\kappa\neq 0$ and $p\neq q$. However, at the moment almost no weighted theory is available in the case $a_{ij}$ depend on $(t,\omega)$. Except in the case $\mathcal{A} = -\Delta$, one has a bounded $H^\infty$-calculus on $L^q(\Dom,v_{\alpha})$ by \cite{LV18}, and thus Theorem \ref{t:H_infinite_SMR} implies stochastic maximal $L^p$-regularity in the full range. The latter can very likely be extended to elliptic second order operators in non-divergence form with smooth $x$-dependent coefficients by standard arguments. This would make it possible to do a variant of Theorem \ref{t:QND_domainweights} with general $(p,q,\kappa)$ as long as the coefficients $a_{ij}$ are independent of time.

\item In \cite{KimKimquasi2} a quasilinear SPDE is considered in weighted spaces as well. However, the results seem not comparable. For instance, they consider operators in divergence form and they do not allow a gradient type noise term.
\end{enumerate}
\end{remark}

\subsection{Quasilinear SPDEs in divergence form on domains}
\label{ss:quasilinear_divergence}
Unlike in the previous sections we will consider an example where there is no time-dependence in the operator $A$ and $B=0$. In this way we can obtain a full $L^p(L^q)$-theory. We study the following differential problem for the unknown $u:\I_T\times \O\times \Dom\to \R$:
\begin{equation}
\label{eq:second_quasilinear_divergence}
\begin{cases}
du -\div(a(u)\nabla u) dt= (\div(f_1(\cdot, u,\nabla u)) +f_2(\cdot, u,\nabla u))dt \\  \qquad \qquad \qquad \qquad \qquad \qquad + \sum_{n\geq 1}g_n(\cdot, u,\nabla u) dw_t^{n}, & \text{on }\Dom,\\
u=0 ,\qquad &\text{on }\partial\Dom;\\
u(0)=u_0, & \text{on }\Dom.
\end{cases}
\end{equation}

The problem \eqref{eq:second_quasilinear_divergence} was already considered in \cite[Secion 5.5]{Hornung}. The aim of this section is to partially extend \cite[Theorem 5.6]{Hornung} and at the same time correct it (see the discussion in \cite[p.\ 66]{HornungDissertation} on this matter). Note that in \cite[Section 5.5]{Hornung} equations in divergence form have been considered with Neumann and/or mixed type boundary conditions. Our framework also allows this setting, but we will only consider Dirichlet conditions here. The interested reader can adapt the proofs below with the functional analytic set-up proposed \cite[Section 5.5]{Hornung} to correct \cite[Theorem 4.11]{Hornung} with different boundary conditions.

We study \eqref{eq:second_quasilinear_divergence} under the following assumption.

\begin{assumption}\
\label{ass:QE_divergence}
\begin{enumerate}[{\rm(1)}]
\item\label{it:QE_divergence_p_q_a} Let $q\in [2,\infty)$, $p\in (2,\infty)$ and $\a\in [0,\frac{p}{2}-1)$ be such that $1-\frac{2(1+\a)}{p}>\frac{d}{q}$.
\item\label{it:QE_divergence_domain_smoothness} $\Dom\subseteq \R^d$ is a bounded $C^1$-domain.
\item\label{it:QE_divergence_form_matrix_a}
The map $a:\O\times \Dom\times \R\to \R^{d\times d}$ is $\F_0\otimes \Borel(\Dom)\otimes \Borel(\R)$-measurable. Assume that $a(\cdot,0)\in L^{\infty}(\O\times\Dom)$ and for each $r>0$ there exists an increasing continuous function $K_r:\R_+\to \R_+$ such that $K_r(0)=0$ and for each $i,j\in \{1,\dots,d\}$, $x,x'\in \Dom$ and $y\in\B_{\R}(r)$,
\[|a(x,y)-a(x',y)|\leq K_r(|x-x'|).\]
Moreover, $a$ is locally Lipschitz w.r.t. $y\in\R$ uniformly in $(\om,x)$, i.e.\ for each $r>0$ there exists $C_r>0$ such that a.s.\ for all $x\in \Dom$ and $y,y'\in \B_{\R}(r)$ one has
$$
|a(x,y)-a(x,y' )|_{\R^{d\times d}}\leq C_r |y-y'|.
$$
Furthermore, $a$ is locally uniformly ellipticity, i.e.\ for each $r>0$ there exists $\epsilon_r>0$ such that a.s.\ for all $x\in \Dom$ and $y\in \B_{\R}(r)$ one has
$$
\sum_{i,j=1}^d \xi_i \xi_j a_{ij}(x,y)\geq \epsilon_r |\xi|^2.
$$
\item\label{it:QE_divergence_f_1_f_2_g}
Let $\varepsilon\geq 0$. The mappings $f_1:\I_T\times \O\times \Dom\times \R\times \R^d\to \R^d$, $f_2:\I_T\times \O\times \Dom\times \R\times \R^d\to \R$ and $g:=(g_n)_{n\geq 1}:\I_T\times \O\times \Dom\times \R\times \R^d\to \ell^2$ are $\Progress\otimes \Borel(\Dom)\otimes \Borel(\R^d)\otimes \Borel(\R)$-measurable. Assume that $f_1(\cdot,0,0)=0$, $f_2(\cdot,0)=g_n(\cdot,0)=0$ for all $n\geq 1$. Finally, we assume that for each $r>0$ there exists $C_r>0$ such that a.s.\ for all $x\in \Dom$, $y,y'\in \B(r)$ and $z,z'\in \R$
$$
\sum_{i=1}^2 |f_i(t,x,y,z)-f_i(t,x,y',z' )|+\|g(t,x,y,z)-g(t,x,y',z')\|_{\ell^2}\leq C_{r,\varepsilon} |y-y'|+\varepsilon |z-z'|.
$$
\end{enumerate}
\end{assumption}

Typical examples of $f_1,f_2,g$ are
\begin{align*}
  f_1(x,u,\nabla u) &=\varepsilon \nabla u, & f_1(x,u)&=(\wt{f}_i(x)|u|^{h-1}u)_{i=1}^d,
\\ f_2(t,u)&=|u|^{m-1}u, & (g_n(x,u))_{n\geq 1}&=(\wt{g}_n(x)|u|^{r-1}u)_{n\geq 1}.
\end{align*}
where $h,m,r>1$, $\varepsilon>0$, $(\wt{f}_i)_{i=1}^d\in L^{\infty}(\O\times\Dom;\R^d)$ and $(\wt{g}_n)_{n\geq 1}\in L^{\infty}(\O\times \Dom;\ell^2)$.

Let us briefly recall the function spaces which will be needed below. Let $s\in (-1,1)$ and $q,p\in (1,\infty)$,  we set
\begin{equation}
\label{eq:def_X_0_X_1_quasilinear_divergence_form}
\begin{aligned}
W^{1,q}_0(\Dom)&=\{u\in W^{1,q}(\Dom)\,:\,u|_{\partial\Dom}=0\},
\\ W^{-1,q}(\Dom)& =(W^{1,q'}_0(\Dom))^*
\\ \BD^{s}_{q,p}(\Dom)&=(W^{-1,q}(\Dom),W^{1,q}_0(\Dom))_{\frac{s+1}{2},p}.
\end{aligned}
\end{equation}
For further properties we refer to Example \ref{ex:extrapolated_Laplace_dirichlet} and the references therein.

To recast the problem \eqref{eq:second_quasilinear_divergence} in the form \eqref{eq:QSEE} let us set
$ X_0:=W^{-1,q}(\Dom)$, $X_1:=W^{1,q}_0(\Dom)$, and for $u\in C(\overline{\Dom})$ and $v\in X_1$
\begin{align*}
A(t,u)v&=-\div(a(u)\nabla v), &  B(t,u)v &=0,
\\ F(t,u)&=\div(f_1(t,u,\nabla u))+f_2(t,u,\nabla u), & G(t,u)&=(g_n(t,u,\nabla u))_{n\geq 1}.
\end{align*}
Here the divergence operator is defined as in \eqref{eq:weak_divergence_operator}, i.e.\ for $u\in C(\overline{\Dom})$ and $v\in X_1$,
\begin{equation}
\label{eq:variational_divergence_operator}
\l \phi, A(u)v\r=\int_{\Dom}(a(u)\cdot \nabla v)\cdot \nabla \phi\,dx, \qquad \phi\in W^{1,q'}(\Dom).
\end{equation}
The same applies to $F(t,u)$. As usual we say that $(u,\sigma)$ is a maximal local solution of \eqref{eq:second_quasilinear_divergence} if $(u,\sigma)$ is a maximal local solution of \eqref{eq:QSEE} with the above choice of $A,B,F,G$ and $H=\ell^2$.

Before stating the main result of this subsection, let us note that a maximal local solution to \eqref{eq:second_quasilinear_divergence} verifies the natural weak formulation of \eqref{eq:second_quasilinear_divergence}: a.s.\ for all $t\in [0,\sigma)$  and all $\phi\in C^1_{c}(\Dom)$,
\begin{align*}
\int_{\Dom} (u(t)-u_0)\phi \,dx&+\int_0^t\int_{\Dom}(a(u)\cdot \nabla u)\cdot\nabla \phi\, dx\,ds
= -\int_0^t \int_{\Dom} f_1(u,\nabla u)\cdot \nabla \phi \,dx\,ds\\
&+\int_0^t \int_{\Dom} f_2(u,\nabla u) \phi \,dx\,ds 
+\sum_{n\geq 1}\int_0^t\int_{\Dom} g_n(u,\nabla u) \phi \,dx\,dw_s^n .
\end{align*}
To see this, use \eqref{eq:identity_sol} and note that $\phi\in C^1_c(\Dom) \subseteq (W^{-1,q}(\Dom))^*=W^{1,q'}_0(\Dom)$.

\begin{theorem}
\label{t:divergence_quasi_local}
Suppose Assumption \ref{ass:QE_divergence} holds. Then for each $N\geq 1$ there exists $\bar{\varepsilon}_N>0$ such that if $\varepsilon\in (0,\bar{\varepsilon}_N)$ and
\[u_0\in L^{\infty}_{\F_0}(\O;\BD^{1-\frac{2(1+\a)}{p}}_{q,p}(\Dom))\]
has norm $\leq N$, then
there exists a maximal local solution $(u,\sigma)$ to \eqref{eq:second_quasilinear_divergence}. Moreover, there exists a localizing sequence $(\sigma_n)_{n\geq 1}$ such that for all $n\geq 1$ and a.s.
$$
u\in L^p(\I_{\sigma_n},w_{\a};W^{1,q}_0(\Dom))\cap C(\overline{I}_{\sigma_n};\BD^{1-\frac{2(1+\a)}{p}}_{q,p}(\Dom))\cap C((0,\sigma_n];\BD^{1-\frac{2}{p}}_{q,p}(\Dom)).
$$
\end{theorem}

\begin{proof}
By Assumption \ref{ass:QE_divergence}\eqref{it:QE_divergence_p_q_a}, \eqref{eq:def_X_0_X_1_quasilinear_divergence_form}, and Sobolev embeddings one has
\begin{equation}
\label{eq:embedding_divergence_form_equations}
\Xap=\BD^{1-\frac{2(1+\a)}{p}}_{q,p}(\Dom)\hookrightarrow B^{1-\frac{2(1+\a)}{p}}_{q,p}(\Dom) \hookrightarrow C^{\eta}(\Dom)\hookrightarrow  L^{\infty}(\Dom);
\end{equation}
for some $\eta>0$. Therefore, $A(u)v:=-\div(a(u)\cdot\nabla v)$ for $u\in \Xap$ and $v\in X_1$ is well-defined. By \cite[Remark 4.3(ii) and Theorem 11.5]{DivH}, \cite[Remark 2.4(3)]{EHDT} and Assumption \ref{ass:QE_divergence} $A(u_0)$ has a bounded $H^\infty$-calculus of angle $<\pi/2$.
Therefore, by Theorem \ref{t:H_infinite_SMR} (see also Remark \ref{r:time_transference}\eqref{it:uniformOmegaHinfty}) we find that $A(u_0)\in \MRta$ and for each $\theta\in [0,1/2)$
\begin{equation}
\label{eq:QE_divergence_maximal_regularity_u_0}
\max\Big\{ K^{\deter,\theta}_{A(u_{0})},K^{\stoc,\theta}_{A(u_{0})}\Big\} \leq C_N,
\end{equation}
where $C_N$ depends only on $N>0$. To check \ref{HAmeasur} let us fix $n\geq 1$ and $u_1,u_2\in \Xap$ of norm $\leq n$. Then by \eqref{eq:embedding_divergence_form_equations} it follows that $\|u_1\|_{L^{\infty}(\Dom)},\|u_2\|_{L^{\infty}(\Dom)}\leq C_n=:R$, and for each $v\in X_1$
\begin{equation*}
\begin{aligned}
\|\div(a(u_1)\cdot\nabla v)-\div(a(u_2)\cdot\nabla v)\|_{W^{-1,q}(\Dom)}
&\lesssim \|(a(u_1)-a(u_2))\nabla v\|_{L^q(\Dom)}\\
&\leq C_R \|u_1-u_2\|_{L^{\infty}(\Dom)}\|v\|_{W^{1,q}(\Dom)}\\
&\lesssim_{R}\|u_1-u_2\|_{\Xap}\|v\|_{X_1};
\end{aligned}
\end{equation*}
where we used \eqref{eq:weak_divergence_operator} and Assumption \ref{ass:QE_divergence}\eqref{it:QE_divergence_form_matrix_a}.

Since $X_{1/2}=L^q(\Dom)$ by \eqref{eq:HD_complex_interpolation}, using the same argument as above combined with Assumption \ref{ass:QE_divergence}\eqref{it:QE_divergence_f_1_f_2_g} one obtains
\begin{equation}
\begin{aligned}
\label{eq:QE_divergence_estimate_F_G}
\|F(\cdot,u)-F(\cdot,v)\|_{X_0}&+\|G(\cdot,u)-G(\cdot,v)\|_{\g(\ell^2,X_{1/2})}\\
&\leq C_R \|u-v\|_{\Xap} +C\varepsilon \|u-v\|_{X_1};
\end{aligned}
\end{equation}
where $C>0$ does not depend on $n\geq 1$. By setting $F_L=F$ and $G_L=G$ the assumptions \ref{HFcritical_weak}-\ref{HGcritical_weak} are verified. Moreover, the inequalities \eqref{eq:QE_divergence_maximal_regularity_u_0}, \eqref{eq:QE_divergence_estimate_F_G} and Remark \ref{r:smallness_n} show that the condition \eqref{eq:smallness_condition_nonlinearities_QSEE_extended} holds. The result now follows from Theorem \ref{t:local_Extended}.
\end{proof}

\begin{remark}\
\begin{enumerate}[{\rm(1)}]
\item The assumption $u_0\in L^{\infty}_{\F_0}(\O;\Xap)$ is automatically satisfied if $\Filtr$ is generated by $W_{\ell^2}$ (see Remark \ref{r:W_H_filtration}).
\item  In the companion paper \cite{AV19_QSEE_2} we will see that the instantaneous regularization effect in Theorem \ref{t:divergence_quasi_local} can be bootstrapped to prove further regularization of solutions to \eqref{eq:second_quasilinear_divergence}. In such situation weights in time play a basic role.
\end{enumerate}
\end{remark}

In the case $u_0\not\in L^{\infty}_{\F_0}(\O;\Xap)$, we do not have any control on the constants of maximal regularity of $A(u_{0,n})$ as $n$ grows see \cite[p.\ 66]{HornungDissertation} (here $(u_{0,n})_{n\geq 1}$ is as in \eqref{eq:truncation_operator}). However, by choosing $\varepsilon_n\downarrow 0$ appropriately, the arguments used in the proof of Theorem \ref{t:divergence_quasi_local} still lead to the following.

\begin{theorem}
\label{t:divergence_quasi_local_varepsilon_0}
Let the Assumption \ref{ass:QE_divergence} be satisfied for any $\varepsilon>0$.
Then for each
\[u_0\in L^{0}_{\F_0}(\O;\BD^{1-\frac{2(1+\a)}{p}}_{q,p}(\Dom))\]
there exists a maximal local solution $(u,\sigma)$ to \eqref{eq:second_quasilinear_divergence}. Moreover, there exists a localizing sequence $(\sigma_n)_{n\geq 1}$ such that for all $n\geq 1$ and a.s.
$$
u\in L^p(\I_{\sigma_n},w_{\a};W^{1,q}_0(\Dom))\cap C(\overline{I}_{\sigma_n};\BD^{1-\frac{2(1+\a)}{p}}_{q,p}(\Dom))\cap C((0,\sigma_n];\BD^{1-\frac{2}{p}}_{q,p}(\Dom)).
$$
\end{theorem}

\subsection{Stochastic porous media equations with positive initial data}
\label{ss:porous_media_equations}
In this subsection we investigate porous media type equations on the $d$-dimensional torus $\Tor^d$ with uniformly positive initial data. More precisely, we investigate the following problem for the unknown $u:\I_T\times \O\times \Tor^d\to \R$
\begin{equation}
\label{eq:porous_media}
\begin{cases}
du - \big(\Delta(|u|^{r-1} u)-\sum_{i,j=1}^d\Xi_{ij}(\cdot,u)\partial_{ij}^2 u \big)dt= f(u,\nabla u)dt\\ \qquad \qquad \qquad
  \qquad\; +\sum_{n\geq 1}\big(\sum_{j=1}^d b_{nj}(\cdot,u)\partial_j u+ g_n(u) \big) dw^n_t,&\text{ on }\Tor^d,\\
u(0)=u_0,& \text{ on } \Tor^d;
\end{cases}
\end{equation}
where $r\in [1,\infty)$, $u_0\geq c>0$ a.e.\ on $\Tor^d$ and
$
\Xi_{i,j}(\cdot,u)=\frac{1}{2}\sum_{n\geq 1}b_{jn}(\cdot,u)b_{jn}(\cdot,u)
$. The problem \eqref{eq:porous_media} in the case $r=1$ fits in the framework of Section \ref{s:semilinear_gradient}, in such a case the condition $u_0\geq c$ can be avoided. We will only consider the range $r\geq 3$ for technical reasons. The range $r\in (1, 3)$ is more sophisticated and requires other solution concepts  than to the one below. For physical motivations we refer to \cite{BDPR16}, \cite[Subsection 1.1]{FG19} and the references therein. To see the link with the works \cite{DG18,FG19}, let us note that at least formally (see \cite[Remark 2.1]{DG18})
$$
\sum_{n\geq 1}\sum_{j=1}^d b_{nj} \partial_j u\circ dw^n_t =
\sum_{n\geq 1}\sum_{j=1}^d b_{nj} \partial_j u\, dw^n_t +\sum_{i,j=1}^d\Big(\Xi_{ij}(\cdot,u)\partial_{ij}^2 u+\text{lower order terms}\Big) dt,
$$
where $\circ$ denotes the Stratonovich integration. We refer to Subsection \ref{sss:porous_media_discussion} for a comparison to the literature.

To study \eqref{eq:porous_media}, we exploit the fact that in Theorem \ref{t:local_Extended}, stochastic maximal $L^p$-regularity is required on $(A(u_{0,n}), B(u_{0,n}))$ for appropriate $A$ and $B$ (see \eqref{eq:stochastic_maximal_regularity_assumption_local_extended}). We mainly deal with the strong setting and we refer to Remark \ref{r:porous_weak} for the weak one. To begin, let us note that at least formally,
$$
\Delta(|u|^{r-1}u)=r|u|^{r-1} \Delta u +r (r-1) u |u|^{r-3}  |\nabla u|^2.
$$
Therefore, in the case $u\geq c>0$, the porous media operators acts like $\Delta$ plus a lower order term. For notational convenience, we set $\A_r(t,u)v:=-r|u|^{r-1} \Delta v $ and $ f_{r}(u,\nabla v):= -r (r-1) u |u|^{r-3}  | \nabla v|^2$. To recast \eqref{eq:porous_media} in the form \eqref{eq:QSEE}, we set $X_0=L^q(\Tor^d)$, $X_1=W^{2,q}(\Tor^d)$ and for $v\in X_1$, $u\in C(\Tor^d)\cap W^{1,q}(\Tor^d)$
\begin{align*}
A(t,u)v&=\A_r(t,u)v+\sum_{i,j=1}^d \Xi_{ij}(t,u)\partial_{ij}^2 v, &  B(t,u)v &=\Big(\sum_{j=1}^d b_{jn}(t,u)\partial_j v\Big)_{n\geq 1},
\\ F(t,v)&=f(t,v,\nabla v)- f_{r}(v,\nabla v), & G(t,v)&=(g_n(t,v))_{n\geq 1}.
\end{align*}
Here $f$ and $g_n$ are as in Assumption \ref{ass:QND_f_g_regular_trace_space}. The following is the main result of this subsection.

\begin{theorem}
\label{t:porous_media_local_general}
Let $r\geq 3$. Let $p\in (2,\infty)$ and $\a\in [0,\frac{p}{2}-1)$ be such that $p>2(1+\a)+d$. Assume that $b_{jn}$ and $f,g$ verifies Assumption \ref{ass:QND}\eqref{it:QND_a_b}-\eqref{it:QND_continuity} and Assumption \ref{ass:QND_f_g_regular_trace_space}, respectively. Then for each
$$u_0\in L^{0}_{\F_0}(\O;W^{2-2\frac{1+\a}{p},p}(\Tor^d)),\qquad
u_0\geq c>0\text{ a.e.\ on }\Tor^d\times \O,$$
there exists a maximal local solution $(u,\sigma)$ to \eqref{eq:porous_media}. Moreover, there exists a localizing sequence $(\sigma_n)_{n\geq 1}$ such that for all $n\geq 1$ and a.s.
$$
u\in L^p(\I_{\sigma_n},w_{\a};W^{2,q}(\Tor^d))\cap C(\overline{I}_{\sigma_n};W^{2-2\frac{1+\a}{p},p}(\Tor^d))\cap C((0,\sigma_n];W^{2-\frac{2}{p},p}(\Tor^d)).
$$
\end{theorem}

\begin{proof}
The proof is similar to the one given for Theorem \ref{t:QND_general_R_Tor}. As in the proof of the latter theorem, by Sobolev embedding $\Xap=W^{2-2\frac{1+\a}{p},p}(\Tor^d)\hookrightarrow C^{1+\eta}(\Tor^d)$ for some $\eta>0$. Thus, using $r\geq 3$ the estimates on the nonlinearities can be performed as in Theorem \ref{t:QND_general_R_Tor}. The fact that \ref{HAmeasur} holds follows from standard computations.

To check the stochastic maximal regularity condition \eqref{eq:stochastic_maximal_regularity_assumption_local_extended},
for all $n\geq 1$ we set
\begin{equation}
\label{eq:choice_approximating_sequence_initial_data_porous_media_equations}
u_{0,n}:=\one_{\Gamma_n} u_0+ \one_{\O\setminus \Gamma_n} (c\one_{\Tor^d}),
\quad
\text{ where }
\quad
\Gamma_n:=\{\|u_0\|_{\Xap}\leq n\}.
\end{equation}
Thus, $u_{0,n}\in L^{\infty}(\O;C^{1,\eta}(\Tor^d))$ verifies $u_{0,n}\geq c$.
Reasoning as in the proof of Theorem \ref{t:QND_general_R_Tor}, $(A(\cdot,u_{0,n}),B(\cdot,u_{0,n}))\in \MRta$ by \cite[Theorem 5.4]{VP18} and $u_{0,n}\geq c_1$. We remark that the ellipticity condition in \cite[Assumption 5.2(2)]{VP18} is satisfied since $|u_{0,n}|^{r-1}\geq c^{r-1}>0$ a.s.\ and a.e.\ on $\Tor^d$.
\end{proof}

In the above proof we used the choice \eqref{eq:choice_approximating_sequence_initial_data_porous_media_equations} instead of \eqref{eq:truncation_operator}. Indeed, if $u_{0,n}$ is as in \eqref{eq:truncation_operator}, then $u_{0,n}$ are not uniformly bounded from below in general.

The proof of Theorem \ref{t:porous_media_local_general} shows that Theorems \ref{t:QND_gradient_not_critical}-\ref{t:QND_gradient_critical} extends to \eqref{eq:porous_media}. To avoid repetitions, we only state the extension of Theorem \ref{t:QND_gradient_critical} to \eqref{eq:porous_media}.
\begin{theorem}
\label{t:porous_media_critical}
Let $r \geq 3$. Assume that $b_{jn}$ and $f,g$ verifies Assumption \ref{ass:QND}\eqref{it:QND_a_b}-\eqref{it:QND_continuity} and Assumption \ref{ass:SND_reaction_diffusion_gradient}, respectively. Moreover, assume that $m>1+\frac{2}{d}$ and $b_{jn}(t,\om,x,y)$ does not depend on the $y$ variable.
Suppose that $p\in (2,\infty)$ verifies \eqref{eq:range_p_critical_quasilinear_gradient_nonlinearities}. Then for any
$$u_0\in L^0_{\F_0}(\O;W^{\frac{d}{p}+\frac{m-2}{m-1},p}(\Tor^d)), \ \ \text{with} \ \ u_0\geq c>0
\,\text{ a.e.\ on }\,\Tor^d\times\Omega,$$
there exists a maximal local solution $(u,\sigma)$ to \eqref{eq:quasilinear_non_divergence}. Moreover, there exists a localizing sequence $(\sigma_n)_{n\geq 1}$ such that a.s.\ for all $n\geq 1$
$$
u\in L^p(\I_{\sigma_n},w_{\a_{\crit}};W^{2,p}(\Tor^d))\cap C(\overline{I}_{\sigma_n};W^{\frac{d}{p}+\frac{m-2}{m-1},p}(\Tor^d))
\cap C((0,\sigma_n];W^{2-\frac{2}{p},p}(\Tor^d)),
$$
where $\a_{\crit}:=\frac{pm}{2(m-1)}-\frac{d}{2}-1$.
\end{theorem}

\begin{proof}
Comparing the proof of Theorem \ref{t:porous_media_local_general} and Theorem \ref{t:QND_gradient_critical}, it remains to estimate $f_r$. To this end let us note that for each $R>0$, $y,y'\in \B_{\R}(R)$ and $z,z'\in \R^d$,
\begin{equation}
\label{eq:estimate_f_r_porous_media_equation}
|f_r(y,z)-f_r(y',z')|\leq C_R\big[(1+|z|^2+|z'|^2)|y-y'|+(1+|z|+|z'|)|z-z'|\big],
\end{equation}
for some $C_R>0$ independent of $y,y',z,z'$. Therefore, due to \eqref{eq:estimate_f_r_porous_media_equation}, if $f$ verifies Assumption \ref{ass:SND_reaction_diffusion_gradient} for $m>2$, then $f-f_r$ verifies Assumption \ref{ass:SND_reaction_diffusion_gradient} with the same $m$. Thus, reasoning as in the proof of Theorem \ref{t:QND_gradient_critical}, the conclusion follows.
\end{proof}

\begin{remark}
\label{r:porous_weak}
Equation \eqref{eq:porous_media} has a natural weak formulation. One can check that the arguments used in Theorems \ref{t:porous_media_local_general}-\ref{t:porous_media_critical} can be adapted to
prove local existence in the weak setting (see Subsection \ref{ss:quasilinear_divergence}). In such a case, $r\in (2,3)$ is also allowed.
\end{remark}

\subsubsection{Discussions}
\label{sss:porous_media_discussion}
Under some structural assumptions on the nonlinearities $b_{jn},f,g$, \eqref{eq:porous_media} (and its generalizations) has been extensively studied (see for instance \cite{DG18,FG19,GS17,GH18,DGG19} and the references therein). One of the first paper on the topic is \cite{GS17} where only $x$-independent $b_{nj}$ are considered. In the $x$-dependent case the situation is more complicated and one often needs the assumption $m\geq 2$, see \cite{GH18,FG19}.
In \cite{DG18,DGG19}, the authors allow the more complicated range $r\in (1,2)$ as well, in some cases they need to work with other type of solutions such as kinetic or entropy solutions.
Our results appear weaker than the ones in \cite{DG18}. For instance, the assumption $u_0\geq c$ is unnatural. However, this case was also considered in the deterministic setting, see e.g.\ \cite{RoiSch}. Moreover, our setting differs from the one in \cite{DG18}, and the main differences are:
\begin{itemize}
\item the functions spaces considered for the initial data are different;
\item the nonlinearity $f$ can be of arbitrary polynomial growth in $u$ and $|\nabla u|$;
\item less regularity is required for $b_{jn}$.
\end{itemize}

It seems to us that the theory developed here can be used to study \eqref{eq:porous_media} with general $u_0$, employing a standard approximation argument (see e.g.\ \cite[eq.\ (3.2)]{FG19}). Firstly, one replaces $\Delta(|u|^{r-1}u)$ by $\Delta(\varepsilon+|u|^{r-1}u)$ in \eqref{eq:porous_media}. With such modification, we can apply Theorem \ref{t:local_Extended} to \eqref{eq:porous_media}, obtaining a family of maximal local solutions $(u_{\varepsilon},\sigma_{\varepsilon})_{\varepsilon>0}$ to the modified equations. Secondly, one provides a-priori bound (uniform in $\varepsilon>0$) in $C^{\alpha}$-norm for $(u_{\varepsilon})_{\varepsilon>0}$ for some uniform $\alpha>0$. Thus, by the blow-up criteria in \cite{AV19_QSEE_2}, $\sigma_{\varepsilon}=T$ and one can study the behaviour of $u_{\varepsilon}$ as $\varepsilon\downarrow 0$. We remark that a-priori estimates for the $C^{\alpha}$-norm for the deterministic version of \eqref{eq:porous_media} are known, see the discussion in \cite[p.\ vii-viii]{DiB93}. However, we are not aware of any contribution on this topic for \eqref{eq:porous_media}.
Note that the arguments used for \eqref{eq:porous_media} seem to be applicable to other degenerate parabolic equations.

\subsection{Stochastic Burgers' equation with coloured noise}
\label{ss:Burgers_quasilinear}
Here, we consider a quasilinear version of the stochastic Burgers' equation on $\Tor$ with space-time coloured noise, which can be seen as the quasilinear analogue of \eqref{eq:Burger_white_noise}. However, for technical reasons, we cannot deal with white noise as in Subsection \ref{ss:Burgers_semilinear}.

More precisely, we consider the following problem for $u:\I_T\times \O\times \Tor\to \R$,
\begin{equation}
\label{eq:quasilinearwhitenoise}
\begin{cases}
\displaystyle du -\partial_x(a(\cdot,u)\partial_x u) dt=\big(\partial_x(f_1(\cdot, u)) +f_2(\cdot,u)\big)dt +
g(\cdot, u) d{w}_{t}^{c}, & \text{on } \Tor,\\
u(0)=u_0, & \text{on } \Tor;
\end{cases}
\end{equation}
here $w_t^c$ denotes a coloured space-time noise on $\Tor$. More precisely, for some $\delta>0$, we assume that $w_t^c$ induces an $H^{\delta,2}(\Tor)$-cylindrical Brownian motion in the sense of Definition \ref{def:Cylindrical_BM}.

The noise in \eqref{eq:quasilinearwhitenoise} is different than in Subsections \ref{ss:QND_R_torus}-\ref{ss:porous_media_equations}. The setting in \eqref{eq:quasilinearwhitenoise} is as in Subsection \ref{ss:Burgers_semilinear}, but with a coloured noise.

\begin{assumption}\label{ass:QNDwhitenoise}\
\begin{enumerate}[{\rm(1)}]
\item\label{it:QND_white_noise_q_p} $q\in [2,\infty)$, $p\in(2,\infty)$ and $\a\in [0,\frac{p}{2}-1)$ verifies $2\delta-2\frac{1+\a}{p}>\frac{1}{q}$.
\item\label{it:QND_white_noise_a} The map $a:\O\times \Tor\times \R\to \R$ is $\F_0\otimes \Borel(\Tor)\otimes \Borel(\R)$-measurable and it verifies the Assumption \ref{ass:QE_divergence}\eqref{it:QE_divergence_form_matrix_a} with $d=1$ and $\Dom$ replaced by $\Tor$.
\item\label{it:QND_white_noise_f_g} The maps $f_1,f_2,g:\I_T\times \O\times \Tor\times \R\to \R$ are $\Progress\otimes \Borel(\Tor)\otimes \Borel(\R)$-measurable. Assume that $f_1(\cdot,0),f_2(\cdot,0)\in L^{\infty}(\I_T\times \O;L^q(\Tor))$ and $g(\cdot,0) \in L^{\infty}(\I_T\times \O\times\Tor)$. Moreover, for each $r>0$ there exists $C_r>0$ such that for all $t\in \I_T$, $x\in \Tor$ and $y,y'\in \B_{R}(r)$,
\begin{align*}
\sum_{i\in \{1,2\}}|f_i(t,x,y)-f_i(t,x,y')|+|g(t,x,y)-g(t,x,y')|& \leq  C_r|y-y'|.
\end{align*}
\end{enumerate}
\end{assumption}

\begin{remark}\
\begin{itemize}
\item For any $\delta>0$, Assumption \ref{ass:QNDwhitenoise}\eqref{it:QND_white_noise_q_p} is satisfied for $p,q$ large and $\a$ small.
\item Assumption \ref{ass:QNDwhitenoise}\eqref{it:QND_white_noise_f_g} includes the Burgers' type nonlinearity $f(u) = -u^2$.
\end{itemize}
\end{remark}
In what follows, we only consider the case $\delta\in (0,\frac{1}{2})$, the other cases being simpler. To begin, note that by Assumption \ref{ass:QNDwhitenoise}\eqref{it:QND_white_noise_q_p}, there exists $s>\frac{1}{2}$ such that $1-2s+2\delta-2\frac{(1+\a)}{p}>\frac{1}{q}$. With such a choice, we rewrite \eqref{eq:quasilinearwhitenoise} in the form \eqref{eq:QSEE}. To this end, set $H=H^{2\delta,2}(\Tor)$, $X_0:=H^{-1-s+\delta,q}(\Tor)$ and $X_1=H^{1-s+\delta,q}(\Tor)$. Then by \eqref{eq:H_complex_interpolation},
\begin{equation}
\label{eq:QND_white_noise_Xap}
X_{\frac12} = H^{-s+\delta,q}(\Tor) \ \ \ \text{and} \ \ \ \Xap = B^{1-s+\delta-\frac{2(1+\a)}{p}}_{q,p}(\Tor).
\end{equation}
As in Subsection \ref{ss:QND_Dirichlet}, by Sobolev embedding and Assumption \ref{ass:QNDwhitenoise}\eqref{it:QND_white_noise_q_p}, one has
\begin{equation}
\label{eq:Burges_QND_embeddings_Holder_continuous_function}
B^{1-s+\delta-\frac{2(1+\a)}{p}}_{q,p}(\Tor)\hookrightarrow C^{\eta}(\Tor),\qquad
\eta:=1-s+\delta-2\frac{1+\a}{p}-\frac{1}{q}>s-\delta.
\end{equation}
For $v\in \Xap$, $u\in X_1$ let
\begin{align*}
A(v) u&=- \partial_x (a(v)\partial_x u), &  B(t)u &=0,
\\ F(t,u)&=\partial_x (f(t,u)), & G(t,u)&=i M_{g(t,u)}.
\end{align*}
Similar to Subsection \ref{ss:Burgers_semilinear}, for fixed $u\in C(\Tor)$, $M_{g(t,u)}:L^{\xi}(\Tor)\to L^{\xi}(\Tor)$ is the multiplication operator $(M_{g(t,u)} h)(x) = g(t,u(x)) h(x)$ where $\xi\in (2,\infty)$ verifies $\delta-\frac{1}{2}=-\frac{1}{\xi}$, which is needed below for the Sobolev embedding $H^{\delta,2}\hookrightarrow L^{\xi}$ (and here we need $\delta\in (0,\frac12)$). Moreover, $i:L^{\xi}(\Tor)\to X_{\frac12}$ denotes the embedding. As usual, we say that $(u,\sigma)$ is a maximal local solution to \eqref{eq:quasilinearwhitenoise} if $(u,\sigma)$ is a maximal local solution to \eqref{eq:semilinearabstract} in the sense of Definition \ref{def:solution2} with the above choice of $A,B,F,G,H$.

To estimate $F$, similar to \eqref{eq:QE_divergence_estimate_F_G}, one has
\begin{equation*}
\|F(\cdot,u)-F(\cdot,v)\|_{H^{-s,q}}\lesssim \sum_{i\in \{1,2\}}\|f_i(\cdot,u)-f_i(\cdot,v)\|_{L^q}\lesssim_r \|u-v\|_{\Xap},
\end{equation*}
where in the last inequality we used Assumption \ref{ass:QNDwhitenoise}\eqref{it:QND_white_noise_f_g} and \eqref{eq:Burges_QND_embeddings_Holder_continuous_function}. Therefore, $F$ verifies \ref{HFcritical_weak} by setting $F_{\Tr}=F$, $F_L=F_c=0$. To estimate $G$, we argue as in \eqref{eq:estimate_G_very_weak_Burgers}, \eqref{eq:QE_divergence_estimate_F_G}. Then for $u,v\in \Xap$ such that $\|u\|_{\Xap},\|v\|_{\Xap}\leq r$, one has
\begin{equation*}
\begin{aligned}
\|G(\cdot,u)-G(\cdot,v)&\|_{\g(H^{\delta,2};H^{-s+\delta,q})}\\
&\eqsim
\|(I-\partial_x^2)^{-\frac{s}{2}+\frac{\delta}{2}} (M_{g(\cdot, u)}-M_{g(\cdot, v)})(1-\partial_x^2)^{-\frac{\delta}{2}}\|_{\g(L^2,L^q)}
\\ &\stackrel{(i)}{\lesssim}
\|(I-\partial_x^2)^{-\frac{s}{2}+\frac{\delta}{2}} (M_{g(\cdot, u)}-M_{g(\cdot, v)})(1-\partial_x^2)^{-\frac{\delta}{2}}\|_{\calL(L^2,L^{\infty})}
\\ & \stackrel{(ii)}{\lesssim} \|(I-\partial_x^2)^{-\frac{s}{2}+\frac{\delta}{2}} (M_{g(\cdot, u)}-M_{g(\cdot, v)})\|_{\calL(L^{\xi},L^{\infty})}
\\ & \stackrel{(iii)}{\lesssim}
\|M_{g(\cdot, u)}-M_{g(\cdot, v)}\|_{\calL(L^{\xi},L^{\xi})}
\\ & \leq \|g(\cdot, u)-g(\cdot, v)\|_{L^{\infty}}  \stackrel{(iv)}{\lesssim_r} \|u-v\|_{\Xap};
\end{aligned}
\end{equation*}
where in $(i)$ we used  \cite[Corollary 9.3.3]{Analysis2}, in $(ii)$ we used that $(1-\partial_x^2)^{-\frac{\delta}{2}}:L^2(\Tor)\to H^{\delta,2}(\Tor)\hookrightarrow L^{\xi}(\Tor)$ as mentioned before. In $(iii)$ we used $(1-\partial_x^2)^{-\frac{s}{2}+\frac{\delta}{2}}:L^{\xi}(\Tor)\to H^{s-\delta,\xi}(\Tor)\hookrightarrow L^{\infty}(\Tor)$ and Sobolev embedding with $s-\delta-\frac{1}{\xi}=s-\frac{1}{2}>0$. Finally, $(iv)$ follows from Assumption \ref{ass:QNDwhitenoise}\eqref{it:QND_white_noise_f_g} and \eqref{eq:Burges_QND_embeddings_Holder_continuous_function}. Thus, \ref{HGcritical_weak} is verified by setting $G_{\Tr}=G$, $G_L=G_c=0$.

The following is the main result of this subsection.

\begin{theorem}
Assume that the Assumption \ref{ass:QNDwhitenoise} holds. Let $s>\frac{1}{2}$ be such that $1-2s+2\delta-2\frac{(1+\a)}{p}>\frac{1}{q}$. Set $s_{\delta}:=1-s+\delta$. Then for each
$$
u_0\in L^0_{\F_0}(\O;B^{s_{\delta}-2\frac{1+\a}{p}}_{q,p}(\Tor)),
$$
there exists a maximal local solution to \eqref{eq:quasilinearwhitenoise}. Moreover, there exists a localizing sequence $(\sigma_n)_{n\geq 1}$ such that a.s.\ for all $n\geq 1$
\[
u\in L^p(\I_{\sigma_n},w_{\a};H^{s_{\delta},q}(\Tor))\cap C(\overline{I}_{\sigma_n};B^{s_{\delta}-2\frac{1+\a}{p}}_{q,p}(\Tor))\cap C((0,\sigma_n];B^{s_{\delta}-\frac{2}{p}}_{q,p}(\Tor)).
\]
\end{theorem}

\begin{proof}
To apply Theorem \ref{t:local_Extended} it remains to check the condition \ref{HAmeasur} and \eqref{eq:stochastic_maximal_regularity_assumption_local_extended}.

To prove that $A$ verify \ref{HAmeasur}, it is enough to note that for any $u\in X_1$, $r>0$ and $v_1,v_2\in \Xap$ such that $\|v_1\|_{\Xap},\|v_2\|_{\Xap}<r$,
\begin{align*}
\|A(v_1) u - A(v_2) u\|_{H^{-1-s+\delta,q}(\Tor)} &
\lesssim\|(a(v_1) -a(v_2))\partial_{x}u\|_{H^{-s+\delta,q}(\Tor)}
\\ &  \stackrel{(i)}{\lesssim} \|a(v_1)-a(v_2)\|_{C^{\eta}(\Tor)} \|\partial_x u\|_{H^{-s+\delta,q}(\Tor)}
\\ &\stackrel{(ii)}{\lesssim_r}  \|v_1-v_2\|_{\Xap} \| u\|_{H^{1+s-\delta,q}(\Tor)},
\end{align*}
where in $(i)$ follows by combining $\eta>s-\delta$, by \eqref{eq:Burges_QND_embeddings_Holder_continuous_function}, and \cite[Chapter 14, eq.\ (4.14)]{TayPDE3} (or \cite[Proposition 3.8]{MV15}) and $(ii)$ by Assumption \ref{ass:QNDwhitenoise}\eqref{it:QND_white_noise_a}, \eqref{eq:QND_white_noise_Xap} and \eqref{eq:Burges_QND_embeddings_Holder_continuous_function}.

It remains to check the stochastic maximal regularity assumption \eqref{eq:stochastic_maximal_regularity_assumption_local_extended}, where in this case $B = 0$. By Theorem \ref{t:H_infinite_SMR}  and Remark \ref{r:time_transference}\eqref{it:uniformOmegaHinfty}, it is enough to show that for any $N\geq 1$ there exists $\lambda_N>0$ such that for any $w_0\in L^{\infty}_{\F_0}(\O;\Xap)$ the operator
$
\lambda_N+ A(w_0)
$
has a bounded $H^{\infty}$-calculus on $H^{-1-\varepsilon,q}(\Tor)$ with angle $<\pi/2$ and the estimates of the $H^{\infty}$-calculus are uniform in $\om\in\O$. To see this, recall that by \eqref{eq:Burges_QND_embeddings_Holder_continuous_function}, $w_{0}\in L^{\infty}(\O;C^{\eta}(\Tor))$. Let $s'>s$ such that $1-2s'+2\delta-2\frac{1+\a}{p}>\frac{1}{q}$ and $\eta>s'-\delta$. Combining the proof of \cite[Theorem 6.4.3]{pruss2016moving} and the multiplication property in \cite[Chapter 14, eq.\ (4.14)]{TayPDE3} one can check that there exists $\lambda_N>0$ such that $\lambda_N+ A(w_0)$ is $R$-sectorial on $H^{-1-\rho+\delta,q}(\Tor)$ with $\rho\in \{0,s'\}$ (see e.g.\ \cite[Definition 4.4.1]{pruss2016moving} or \cite[Defintion 10.3.1]{Analysis2}) with angle $<\pi/2$.
As we have seen in the proof of Theorem \ref{t:divergence_quasi_local}, up to enlarging $\lambda_N>0$, $\lambda_N+ A(w_0)$ has a bounded $H^{\infty}$-calculus on $H^{-1,q}(\Tor)$. The claim follows by using the argument in \cite[Theorem 5]{Stokes}, choosing $A=\lambda_N+ A(w_0)$, $B=1-\partial_x^2$ and replacing $L^2,L^{p_0}$ by $H^{-1,q}(\Tor),H^{-1-s'+\delta,q}(\Tor)$ respectively.
\end{proof}

\section{Further applications: Allen-Cahn and Cahn-Hilliard equations}
\label{s:AC_CH}
In this section we present additionally applications of Theorem \ref{t:semilinear}. More precisely, in Subsection \ref{ss:Allen_Cahn_stochastic_potentials}-\ref{ss:Allen_Cahn_mass_conservative} we investigate the Allen-Cahn type equations and in Subsection \ref{ss:CH} the Cahn-Hilliard equations. In both sections we study the equations on domains since boundary conditions are important from a modelling perspective. However, the case $\Dom=\R^d$ or $\Dom=\Tor^d$ can be analysed with the same technique.

\subsection{Stochastic Allen-Cahn equations}
\label{ss:Allen_Cahn_stochastic_potentials}
Allen-Cahn type equations have been extensively studied in literature. From a physical point of view, Allen-Cahn equation is a prototype for phase separation processes in melts or alloys that is of fundamental interest for both theory and applications. For additional motivations and further results one may consult \cite{ABBK16,BBP17_2,BBP17,F16,FS19,RW13} and the references therein. To the best of our knowledge no results in an $L^q$-setting are available.

Here, we study the following stochastic perturbation of Allen-Cahn equation for the unknown process $u:\I_T\times \O\times \Dom\to \R$
\begin{equation}
\label{eq:Allen_Cahn_stochastic_potentials}
\begin{cases}
du-\Delta u dt= V(\cdot,u)dt+ \sum_{n\geq 1} \big(\sum_{j=1}^d b_{nj}(\cdot)\partial_j u+ g_n(\cdot,u)\big) dw_t^n,&\text{in }\Dom,\\
u=0,&\text{in }\partial\Dom,\\
u(0)=u_0, &\text{in }\Dom.
\end{cases}
\end{equation}

We study \eqref{eq:Allen_Cahn_stochastic_potentials} under the following assumption.

\begin{assumption} Let $d\geq 2$.
\label{ass:AC_W_first}
\begin{enumerate}[{\rm(1)}]
\item Suppose one of the following conditions holds:
\begin{itemize}
\item $q\in[2,\infty)$, $p\in (2,\infty)$ and $\a\in [0,p/2-1)$;
\item $p=q=2$ and $\a=0$.
\end{itemize}
\item $\Dom\subseteq\R^d$ is a bounded $C^2$-domain.
\item\label{it:AC_V_g_n_potentials} The mappings $V:\I_T\times \O\times \Dom\times \R\to \R$, $g:=(g_n)_{n\geq 1}:\I_T\times \O\times \Dom\times \R\to \ell^2$ are $\Progress\otimes \Borel(\Dom)\otimes \Borel(\R)$-measurable, $V(\cdot, 0) \in L^{\infty}(\I_T\times \O;L^q(\Dom))$ and $g(\cdot,0)\in L^{\infty}(\I_T\times \O;L^q(\Dom;\ell^2))$.  Moreover, there exists $C>0$ such that  a.s.\ for all $t\in \I_T$, $x\in \Dom$ and $y,y'\in \R$
\begin{align*}
|V(t,x,y)-V(t,x,y')|&\leq C (1+|y|^2+|y'|^2)|y-y'|,\\
\|g(t,x,y)-g(t,x,y')\|_{\ell^2}&\leq C  (1+|y|+|y'|)|y-y'|.
\end{align*}
\item\label{it:AC_smallness_b} Let $\varepsilon\geq 0$. For each $j\in\{1,\dots,d\}$, the maps $(b_{nj})_{n\geq 1}:\I_T\times \O\times \Dom\to \ell^2$ are $\Progress\otimes \Borel(\Dom)$-measurable and
\begin{align*}
\|(b_{nj}(t))_{n\geq 1}\|_{W^{1,\infty}(\R^d;\ell^2)}\leq \varepsilon,\qquad \forall t\in \I_T \text{ and a.s.}
\end{align*}
\end{enumerate}
\end{assumption}

Note that the usual potential $V(t,u)=u(1-u^2)$ verifies Assumption \ref{ass:AC_W_first}\eqref{it:AC_V_g_n_potentials}.

\begin{remark}
\label{r:AC_condition}
Some remarks may be in order.
\begin{enumerate}[{\rm(1)}]
\item The problem \eqref{eq:Allen_Cahn_stochastic_potentials} under the Assumption \ref{ass:AC_W_first} is similar to \eqref{eq:semilinear_reaction_diffusion_l_m} with $m=3$ and $h=2$. However, we will study \eqref{eq:Allen_Cahn_stochastic_potentials} in the almost very weak setting instead of the weak one. Moreover, we consider the problem on a bounded domain with Dirichlet boundary conditions.
\item\label{it:AC_scaling_remark} The growth of $(g_n)_{n\geq 1}$ is chosen in such a way that the all the nonlinearities in \eqref{eq:Allen_Cahn_stochastic_potentials} have the same scaling. Indeed, $V$ and $(g_n)_{n\geq 1}$ verify Assumption \ref{ass:RD} with $h=2$ and $m=3$. As we have seen in Subsection \ref{sss:critical_reaction_diffusion_l_m}, the nonlinearities in \eqref{eq:semilinear_reaction_diffusion_l_m} have the same scaling if $h=(1+m)/2$.
\item As in Section \ref{s:semilinear_gradient} due to Lemma \ref{l:SMR_semilinear_PDEs}, if $\Dom=\R^d$ or $\Dom=\Tor^d$, then the smallness assumption on the gradient noise term can be sometimes be dropped. However, in the case of $x$-dependent coefficients, one needs to take $p=q$ as explained in Section \ref{ss:Discussion_further_extensions}.
\item In the weak setting, i.e.\ $s=0$, the regularity condition in Assumption \ref{ass:AC_W_first}\eqref{it:AC_smallness_b} can be weakened (see Lemma \ref{l:small_gradient_noise_domain} and Theorem \ref{t:Dirichlet_extensions_semilinear}\eqref{it:Dirichlet_L_infty_domains}).
\end{enumerate}
\end{remark}

By Assumption \ref{ass:AC_W_first} we can study \eqref{eq:Allen_Cahn_stochastic_potentials} in the scale $(\HD^{s,q}(\Dom))_{s\geq -2 }$ of the Dirichlet Laplacian. This scale of Banach spaces fits the boundary condition appearing in \eqref{eq:Allen_Cahn_stochastic_potentials}. Indeed, $\HD^{2,q}(\Dom)=W^{2,q}(\Dom)\cap W^{1,q}_0(\Dom)$, $\HD^{1,q}(\Dom)=W^{1,q}_0(\Dom)$ and $\HD^{0,q}(\Dom)=L^q(\Dom)$. We refer to Example \ref{ex:extrapolated_Laplace_dirichlet} for additional properties of these spaces.

\subsubsection{Almost very weak setting}
\label{ss:Allen_Cahn_semilinear_almost_weak} Let $s\in [0,1)$ and $q\in [2,\infty)$. We rewrite \eqref{eq:Allen_Cahn_stochastic_potentials} in the form \eqref{eq:semilinearabstract} by setting $X_0:=\HD^{-1-s,q}(\Dom)$, $X_1=\HD^{1-s,q}(\Dom)$ and for $u\in X_1$
\begin{align*}
A(t)u&=-\Dd_{-1-s,q} u, &  B(t)u &=0,
\\ F(t,u)&=V(t,u), & G(t,u)&=G_1(t,u)+G_2(t,u),
\\ G_1(t,u)&=(g_n(t,u))_{n\geq 1}, & G_2(t,u)&=\Big(\sum_{j=1}^d b_{nj}(t)\partial_j u\Big)_{n\geq 1}.
\end{align*}
Here $\Dd_{-1-s,q}$ is the extrapolated Dirichlet Laplacian (see \eqref{eq:extrapolated_Laplacian_D}). As usual, we say that $(u,\sigma)$ is a maximal local solution to \eqref{eq:Allen_Cahn_stochastic_potentials} if $(u,\sigma)$ is a maximal local solution to \eqref{eq:semilinearabstract} in the sense of Definition \ref{def:solution2} with the above choice of $A,B,F,G$ and $H=\ell^2$.

To apply Theorem \ref{t:semilinear}, we estimate the nonlinearities. As usual we begin by estimating $F$.  By Assumption \ref{ass:AC_W_first}\eqref{it:AC_V_g_n_potentials} one has
\begin{equation}
\label{eq:estimate_F_allen_cahn_almost_ultra_weak}
\begin{aligned}
\|F(\cdot,u)-F(\cdot,v)&\|_{\HD^{-1-s,q}(\Dom)}\\
&\stackrel{(i)}{\lesssim} \|V(\cdot,u)-V(\cdot,v)\|_{L^{m}(\Dom)}\\
&{\lesssim} \|(1+|u|^2+|v|^2)|u-v|\|_{L^m(\Dom)},\\
&\lesssim (1+\|v\|_{L^{3m}(\Dom)}^{2}+\|v\|_{L^{3m}(\Dom)}^{2})\|u-v\|_{L^{3m}(\Dom)}\\
&\stackrel{(ii)}{\lesssim} (1+\|u\|_{\HD^{\phi,q}(\Dom)}^{2}+\|v\|_{\HD^{\phi,q}(\Dom)}^{2})\|u-v\|_{\HD^{\phi,q}(\Dom)}.
\end{aligned}
\end{equation}
where in $(i)$ and $(ii)$ we used Sobolev embedding with $-\frac{d}{m}=-1-s-\frac{d}{q}$
and $\phi-\frac{d}{q} = -\frac{d}{3m}$ where $\phi\in (0,1-s)$ (see \eqref{eq:HD_embedding}).
To ensure $m\in (1,\infty)$ we have to assume $q>\frac{d}{d-1-s}$ (recall that $d\geq 2$). Note that $\phi$ is given by
$$
\phi=\frac{d}{q}-\frac{d}{3m}=\frac{2d}{3q}-\frac{1+s}{3}.
$$
To ensure $\phi\in (0,1-s)$ we assume $\frac{d}{2-s}<q<\frac{2d}{1+s}$.
Combining the previous requirements gives
\begin{equation}
\label{eq:restriction_q_Allen_Cahn_almost_ultra_weak}
\max\Big\{\frac{d}{d-1-s},\frac{d}{2-s}\Big\}<q<\frac{2d}{1+s}.
\end{equation}
Set
\begin{equation}
\label{eq:AC_beta_1}
\beta_1:=\frac{1+s+\phi}{2}=\frac{d}{3q}+\frac{1+s}{3}.
\end{equation}
By \eqref{eq:HD_complex_interpolation} one has $\HD^{\phi,q}(\Dom)=[\HD^{-1-s,q}(\Dom),\HD^{1-s,q}(\Dom)]_{\beta_1}$ and, under the previous assumptions, \eqref{eq:estimate_F_allen_cahn_almost_ultra_weak} shows that $F:X_{\beta_1}\to X_0$ is locally Lipschitz. As in Subsection \ref{ss:conservative_RD}, the argument splits in three cases:

\begin{enumerate}[{\rm(1)}]
\item If $1-(1+\a)/{p}>\beta_1$, \ref{HFcritical} follows from Remark \ref{r:non_linearities}\eqref{it:non_linearities_continuous_trace}, by setting $F_{\Tr}=F$ and $F_{L}\equiv F_c\equiv 0$.
\item If $1-(1+\a)/{p}= \beta_1$, \ref{HFcritical} follows from \eqref{eq:estimate_F_allen_cahn_almost_ultra_weak} and Remark \ref{r:non_linearities}\eqref{it:non_linearities_varphi_equal_to_beta}, by setting $F_L\equiv F_{\Tr}\equiv 0$, $F_{c}=F$, $m_F=1$, $\rho_1=2$ and $\varphi_1=\beta_1$.
\item If $1-(1+\a)/{p}<\beta_1$ \ref{HFcritical} holds with $F_{c}=V$ and $F_{L}\equiv F_{\Tr}\equiv 0$, under the condition that \eqref{eq:HypCritical} holds with $m_F=1$, $\rho_1=2$, $\varphi_1=\beta_1$. The latter condition becomes
\begin{equation}
\label{eq:allen_cahn_first_critical_weights}
\frac{1+\a}{p}\leq \frac{3}{2}(1-\beta_1)=1-\frac{d}{2q}-\frac{s}{2}.
\end{equation}
\end{enumerate}
Next we estimate $G:=G_1+G_2$. By Assumption \ref{ass:AC_W_first}\eqref{it:AC_smallness_b}, Lemma \ref{l:small_gradient_noise_domain} holds, therefore $G_2$ satisfies \eqref{eq:estimate_G_L_allen_cahn}. It remains to estimate $ G_1$. By \eqref{eq:HD_complex_interpolation}, we have $X_{1/2}=\HD^{-s,q}(\Dom)$. Therefore, by Assumption \ref{ass:AC_W_first}\eqref{it:AC_V_g_n_potentials}
\begin{equation}
\label{eq:estimate_G_allen_cahn_almost_ultra_weak}
\begin{aligned}
\|G_1(\cdot,u)-G_1(\cdot,v)&\|_{\g(\ell^2;\HD^{-s,q}(\Dom))}\\
&\stackrel{(i)}{\lesssim}\|G_1(\cdot,u)-G_1(\cdot,v)\|_{\g(\ell^2;L^r(\Dom))}\\
&\stackrel{(ii)}{\eqsim}\|G_1(\cdot,u)-G_1(\cdot,v)\|_{L^r(\Dom;\ell^2)}\\
&\lesssim (1+\|u\|_{L^{2r}(\Dom)}+\|v\|_{L^{2r}(\Dom)})\|u-v\|_{L^{2r}(\Dom)}\\
&\stackrel{(iii)}{\lesssim} (1+\|u\|_{\HD^{\rho,q}(\Dom)}+\|v\|_{\HD^{\rho,q}(\Dom)})\|u-v\|_{\HD^{\rho,q}(\Dom)};
\end{aligned}
\end{equation}
where in $(i)$ and $(iii)$ we used Sobolev embedding with $-s-\frac{d}{q} = -\frac{d}{r}$ and $\rho - \frac{d}{q}=-\frac{d}{2r}$ (see \eqref{eq:HD_embedding}). In $(ii)$ we used \eqref{eq:gammaidentity}. It follows that $\rho=\frac{d}{2q}-\frac{s}{2}$. Moreover, \eqref{eq:restriction_q_Allen_Cahn_almost_ultra_weak} gives that $r\in (1,\infty)$ and $0<\rho<1-s$. Setting
\begin{equation}
\label{eq:AC_beta_2}
\beta_2=\frac{1+s+\rho}{2}=\frac{1}{4}\Big(\frac{d}{q}+s\Big)+\frac{1}{2}\in (0,1)
\end{equation}
it follows that $\HD^{\rho,q}(\Dom)=[\HD^{-1-s,q}(\Dom),\HD^{1-s,q}(\Dom)]_{\beta_2}$ by \eqref{eq:HD_complex_interpolation}.

By \eqref{eq:estimate_G_allen_cahn_almost_ultra_weak} it follows that $G_1(t,\om,\cdot):X_{\beta_2}\to X_0$ is locally Lipschitz for a.a.\ $(t,\om)\in \I_T\times\O$, and as before:

\begin{enumerate}[{\rm(1)}]
\item If $1-(1+\a)/{p}>\beta_2$, \ref{HGcritical} holds with $G_{\Tr}=G_1$, $G_c\equiv 0$ and $G_L=G_2$ .
\item If $1-(1+\a)/{p}= \beta_2$, \ref{HGcritical} holds with $G_{\Tr}\equiv 0$, $G_{c}:=G_1$, $G_L:=G_2$, $m_G=1$, $\rho_2=1$ and $\varphi_2=\beta_2$.
\item If $1-(1+\a)/{p}<\beta_2$, \ref{HGcritical} holds with $G_{\Tr}\equiv 0$, $G_{c}:=G_1$, $G_L:=G_2$ under the condition \eqref{eq:HypCritical} with $m_G=1$, $\rho_2=1$ and $\varphi_2=\beta_2$. In this situation. The latter condition becomes
\begin{equation}
\label{eq:allen_cahn_first_critical_weights_2}
\frac{1+\a}{p}\leq 2(1-\beta_2)=1-\frac{d}{2q}-\frac{s}{2},
\end{equation}
which coincides with \eqref{eq:allen_cahn_first_critical_weights}. This is in accordance with Remark \ref{r:AC_condition}\eqref{it:AC_scaling_remark}.
\end{enumerate}

Let us summarize what we have proven so far in the following result.

\begin{theorem}
\label{t:allen_cahn_first_local}
Suppose Assumption \ref{ass:AC_W_first} holds. Let $s\in [0,1)$ and let $q\in [2,\infty)$ be such that \eqref{eq:restriction_q_Allen_Cahn_almost_ultra_weak} holds. Let $\beta_2$ be as in \eqref{eq:AC_beta_2}. Assume one of the following conditions is satisfied:
\begin{itemize}
\item $1-(1+\a)/p\geq \beta_2$.
\item $1-(1+\a)/p<\beta_2$ and \eqref{eq:allen_cahn_first_critical_weights} holds.
\end{itemize}
Then there exists an $\bar{\varepsilon}>0$ such that for any $\varepsilon\in (0,\bar{\varepsilon})$, and any
\[u_0\in L^0_{\F_0}(\O;\BD^{1-s-2\frac{(1+\a)}{p}}_{q,p}(\Dom))\]
there exists a maximal local solution $(u,\sigma)$ to \eqref{eq:Allen_Cahn_stochastic_potentials}. Moreover, there exists a localizing sequence $(\sigma_n)_{n\geq 1}$ such that for any $n\geq 1$ and a.s.
$$
u\in L^p(\I_{\sigma_n},w_{\a};\HD^{1-s,q}(\Dom))\cap C(\overline{I}_{\sigma_n};\BD^{1-s-2\frac{(1+\a)}{p}}_{q,p}(\Dom))\cap C(\I_{\sigma_n};\BD^{1-s-\frac{2}{p}}_{q,p}(\Dom)).
$$
\end{theorem}

\begin{proof}
To conclude it remains to check the conditions of Theorem \ref{t:semilinear}. Firstly, recall that $-\Dd_{-1-s,q}$ has a bounded $H^{\infty}$-calculus (see Example \ref{ex:extrapolated_Laplace_dirichlet}). Therefore, by Theorem \ref{t:H_infinite_SMR} $-\Dd_{-1-s,q}\in \MRta$.

Moreover, let $\beta_1,\beta_2$ be as in \eqref{eq:AC_beta_1}, \eqref{eq:AC_beta_2} respectively. Since $q>d/(2-s)$ by assumption (see \eqref{eq:restriction_q_Allen_Cahn_almost_ultra_weak}) it follows that $\beta_2>\beta_1$.

By Assumption \ref{ass:AC_W_first}\eqref{it:AC_smallness_b}, the estimate \eqref{eq:estimate_G_L_allen_cahn} holds. Therefore, \eqref{eq:smallness_semilinear} is satisfied if $\varepsilon$ is sufficiently small.
\end{proof}

\subsubsection{Critical spaces for \eqref{eq:Allen_Cahn_stochastic_potentials}}
To find critical spaces for \eqref{eq:Allen_Cahn_stochastic_potentials} we look for some $\a=\a_{\crit}$ such that \eqref{eq:allen_cahn_first_critical_weights} becomes an equality. We first analyse the case $p\in (2,\infty) $ and $\a\in [0,p/2-1)$. By \eqref{eq:allen_cahn_first_critical_weights_2} to ensure $\a\geq 0$ one has to impose
\begin{equation}
\label{eq:allen_cahn_s_restriction_p}
\frac{1}{p}+\frac{d}{2q}+\frac{s}{2}\leq 1.
\end{equation}
Since $q>d/(2-s)$ the above restriction is verified if $p$ is sufficiently large. The condition $\a<p/2-1$ becomes
\begin{equation}
\label{eq:AC_critical_condition_q}
1-\frac{d}{2q}-\frac{s}{2}<\frac{1}{2} \quad \Leftrightarrow\quad q<\frac{d}{1-s}.
\end{equation}
Since $q>d/(d-1-s)$, by \eqref{eq:restriction_q_Allen_Cahn_almost_ultra_weak}, we also need $d>2$. If \eqref{eq:allen_cahn_s_restriction_p}-\eqref{eq:AC_critical_condition_q} hold, then we set
\begin{equation}
\label{eq:AC_critical_weight_formula}
\a_{\crit}:=p\Big(1-\frac{d}{2q}-\frac{s}{2}\Big)-1.
\end{equation}
Therefore,
\begin{equation}
\label{eq:AC_critical_s_first}
\Xapcrit=\BD^{1-s-\frac{2(1+\a_{\crit})}{p}}_{q,p}(\Dom)
=\BD^{\frac{d}{q}-1}_{q,p}(\Dom);
\end{equation}
where as above we used \eqref{eq:AC_critical_weight_formula} and \eqref{eq:def_BD}. Note that the above space does not depend on $s$ and depends on $p$ only through the microscopic parameter.

In the case $q=p=2$ and $\a=0$, the condition \eqref{eq:allen_cahn_first_critical_weights} gives $1=d/2+s\geq d/2$. The latter forces $s=0$ and $d=2$. However, $s=0$ implies $q>2$, thus this case has to be avoided here.

The following is the main result of this subsection.
\begin{theorem}
\label{t:allen_cahn_first_local_critical}
Let the Assumption \ref{ass:AC_W_first} be satisfied. Let $d>2$, $s\in [0,1/3]$ and $q\in (2,\infty)$ be such that
$$\frac{d}{2-s}<q<\frac{d}{1-s}.$$
Let $p\in (2,\infty)$ be such \eqref{eq:allen_cahn_s_restriction_p} holds.
Then there exists $\bar{\varepsilon}>0$ such that if $\varepsilon\in (0,\bar{\varepsilon})$, then the following hold: For any
$$u_0\in L^0_{\F_0}(\O;\BD^{\frac{d}{q}-1}_{q,p}(\Dom))$$
there exists a maximal local solution $(u,\sigma)$ to \eqref{eq:Allen_Cahn_stochastic_potentials}. Moreover, there exists a localizing sequence $(\sigma_n)_{n\geq 1}$ such that for any $n\geq 1$ and a.s.
$$
u\in L^p(\I_{\sigma_n},w_{\a_{\crit}};\HD^{1-s,q}(\Dom))\cap C(\overline{I}_{\sigma_n};\BD^{\frac{d}{q}-1}_{q,p}(\Dom))\cap C(\I_{\sigma_n};\BD^{1-s-\frac{2}{p}}_{q,p}(\Dom));
$$
where $\a_{\crit}$ is given in \eqref{eq:AC_critical_weight_formula}.
\end{theorem}

\begin{proof}
By \eqref{eq:restriction_q_Allen_Cahn_almost_ultra_weak} and \eqref{eq:AC_critical_condition_q}, the restriction on $q$ is equivalent to
$$
\max\Big\{\frac{d}{d-1-s},\frac{d}{2-s}\Big\}<q<\min\Big\{\frac{2d}{1+s}, \frac{d}{1-s}\Big\}.
$$
Since $d>2$, one has $\frac{d}{d-1-s}\leq \frac{d}{2-s}$. Optimizing the right-hand side of the upper bound on $q$ we see that $s\in [0,1/3]$ leads to the  least restrictions on $q$, because in that case $\frac{d}{1-s}\leq \frac{2d}{1+s}$. Now the result follows from Theorem \ref{t:allen_cahn_first_local}.
\end{proof}

\begin{remark}\
\label{r:smoothness_AC}
For $s=1/3$ we obtain the restriction $q<\frac{3d}{2}$. Thus, Theorem \ref{t:allen_cahn_first_local_critical} ensures local existence for initial data which takes values in $\BD^{\frac{d}{q}-1}_{q,p}(\Dom)$ with smoothness $\frac{d}{q}-1>-\frac{1}{3}$. The optimality of this threshold is not known.
\end{remark}

Let us conclude this section by deriving local existence in the space $L^{d}(\Dom)$. Note that the latter space has the same `local scaling' of $\BD^0_{d,p}(\Dom)$, which is a critical space for \eqref{eq:Allen_Cahn_stochastic_potentials} by Theorem \ref{t:allen_cahn_first_local_critical}. Recall that $\varepsilon>0$ is as in Assumption \ref{ass:AC_W_first}\eqref{it:AC_smallness_b}.

\begin{corollary}
\label{cor:AC_local_L_d}
Let the Assumption \ref{ass:AC_W_first} be satisfied. Let $d>2$ and $p\in [d,\infty)$. Then there exist $\bar{s},\bar{\varepsilon} >0$ such that if $\varepsilon\in (0,\bar{\varepsilon})$, $s\in (0,\bar{s})$ the following holds: for all
\begin{equation}
\label{eq:AC_initial_data_L_d}
u_0\in L^0_{\F_0}(\O;L^d(\Dom))
\end{equation}
there exists a maximal local solution $(u,\sigma)$ to \eqref{eq:Allen_Cahn_stochastic_potentials} and there exists a localizing sequence $(\sigma_n)_{n\geq 1}$ such that for any $n\geq 1$ and a.s.
$$
u\in L^p(\I_{\sigma_n},w_{\a_{\crit}};\HD^{1-s,d}(\Dom))\cap
C(\overline{I}_{\sigma_n};\BD^{0}_{d,p}(\Dom))\cap C(\I_{\sigma_n};\BD^{1-s-\frac{2}{p}}_{d,p}(\Dom)),
$$
where $\a_{\crit}:=\frac{p(1-s)}{2}-1$.
\end{corollary}

\begin{proof}
The proof is analogous to Corollary \ref{cor:conservative_local_existence_L_r}. The argument used in Corollary \ref{cor:conservative_local_existence_L_r} shows that there exists $\bar{s}>0$ such that $\frac{d}{2-s}<d<\frac{d}{1-s}$ and \eqref{eq:allen_cahn_s_restriction_p} hold for any $s\in (0,\bar{s})$ and $p\in [d,\infty)$. Fix $s\in (0,\bar{s})$. Theorem \ref{t:allen_cahn_first_local_critical}, applied with $s\in (0,\bar{s})$, $q=d$ and $p\in [d,\infty)$, gives the existence of $\bar{\varepsilon}>0$ such that if $\varepsilon\in (0,\bar{\varepsilon})$, then there exists a maximal local solution to \eqref{eq:Allen_Cahn_stochastic_potentials} with initial data $u_0\in L^0_{\F_0}(\O;\BD^{0}_{d,p}(\Dom))$ with the required regularity. To conclude it is enough to recall that $L^d(\Dom)\hookrightarrow \BD^{0}_{d,p}(\Dom)$ since $p\geq d$ and \eqref{eq:BD_identifications} holds.
\end{proof}

As in the previous sections, in Theorem \ref{t:allen_cahn_first_local_critical} and Corollary \ref{cor:AC_local_L_d}, the solution instantaneously regularizes in space.

\subsection{Mass conservative stochastic Allen-Cahn equations}
\label{ss:Allen_Cahn_mass_conservative}
In this subsection we study local existence for the following mass conservative Allen-Cahn equation:
\begin{equation}
\label{eq:mass_conservative_AC}
\begin{cases}
d u -\Delta u dt=\big(V(\cdot,u)-\Xint-_{\Dom}V(\cdot,u)dx\big) dt \\ \qquad \qquad \qquad \qquad + \big(\sum_{n\geq 1} \sum_{j=1}^d b_{nj} \partial_j u+ g_n(\cdot,u)\big)dw_t^n,&\text{on }\Dom,\\
\partial_{\n} u=0,&\text{on }\partial\Dom,\\
u(0)=u_0,&\text{on }\Dom.
\end{cases}
\end{equation}
Here $\Dom,V,g_n,b_{nj}$ verify Assumption \ref{ass:AC_W_first}, $\Xint-_{\Dom}\cdot dx :=1/|\Dom|\int_{\Dom}\cdot\,dx$ denotes the mean, $\n$ is the exterior normal field on $\partial\Dom$.

In literature (cf.\ \cite{ABBK16,FS19}) the problem \eqref{eq:mass_conservative_AC} is usually called mass-conservative since, at least formally, the `mass' $\E\int_{\Dom} u(t,x)dx$ is preserved under the flow, i.e.\ $\E\int_{\Dom} u(t,x)dx=\E\int_{\Dom}u_0(x)dx$.

To study \eqref{eq:mass_conservative_AC}, we employ the extrapolation-interpolation scale $(\HN^{s,q}(\Dom))_{s\in [-2,\infty)}$ of the Neumann Laplacian $\DN_{q}$. Such scale of Banach spaces fits the boundary condition appearing in \eqref{eq:Allen_Cahn_stochastic_potentials}. Indeed $\HN^{2,q}(\Dom)=\{u\in W^{2,q}(\Dom)\,:\,\partial_{\n}u=0\}$, $\HN^{1,q}(\Dom)=W^{1,q}(\Dom)$ and $\HN^{0,q}(\Dom)=L^q(\Dom)$. We refer to Example \ref{ex:Laplacian_Neumann_scales} for additional properties of such spaces.

Let $s\in [0,1)$ and $q\in [2,\infty)$. We rewrite \eqref{eq:mass_conservative_AC} in the form \eqref{eq:semilinearabstract} by setting $X_0:=\HN^{-1-s,q}(\Dom)$, $X_1:=\HN^{1-s,q}(\Dom)$ and for $u\in X_1$
\begin{align*}
A(t)u&=-\DN_{-1-s,q} u, &  B(t)u &=0,
\\ F(t,u)&=V(t,u)-\Xint-_{\Dom} V(t,u)dx, & G(t,u)&=G_1(t,u)+G_2(t,u),
\\ G_1(t,u)&=(g_n(t,u))_{n\geq 1} & G_2(t,u)&=\Big(\sum_{j=1}^d b_{nj}(t)\partial_j u\Big)_{n\geq 1}.
\end{align*}
As above, we say that $(u,\sigma)$ is a maximal local solution to \eqref{eq:mass_conservative_AC} if $(u,\sigma)$ is a maximal local solution to \eqref{eq:semilinearabstract} for the above choice of $A,B,F,G$.

Our main results concerning \eqref{eq:mass_conservative_AC}, reads as follows.

\begin{theorem}
Let the Assumption \ref{ass:AC_W_first} be satisfied. Then Theorems \ref{t:allen_cahn_first_local}, \ref{t:allen_cahn_first_local_critical} and Corollary \ref{cor:AC_local_L_d} hold for \eqref{eq:mass_conservative_AC} if the spaces $\HD$ and $\BD$ are replaced by $\HN$ and $\BN$, respectively.
\end{theorem}

\begin{proof}
Due to the results presented in Example \ref{ex:Laplacian_Neumann_scales} on the Neumann Laplacian and the $\HN$-scale, one can repeat the proof of the stated results literally.
\end{proof}

\subsection{Stochastic Cahn-Hilliard equations}
\label{ss:CH}
The Cahn-Hilliard equation has been derived as a phenomenological model for phase separation
of binary alloys. Stochastic versions of the Cahn-Hilliard equation have been proposed in the physics to model  external fields, impurities in the alloy, or may describe thermal
fluctuations or external mass supply see \cite{C70,HH77,L71}. For a mathematical perspective we refer to \cite{ABNP19,CW02,CW01,DPD96,EM91,S18} and the references therein. To the best of our knowledge the results presented below are new.

In this section we study the following stochastic Cahn-Hilliard equation for the unknown process $u:\I_T\times \O\times \Dom\to \R$:
\begin{equation}
\label{eq:Cahn_hilliard}
\begin{cases}
du + \Delta^2 u dt=\Delta (\phi(\cdot,u))dt +\sum_{n\geq 1}\Phi_n(\cdot,u)dw_{t}^n,&\text{ on }\Dom,\\
\partial_{\n}u=\partial_{\n}\Delta u =0,&\text{ on }\partial\Dom,\\
u(0)=u_0,&\text{ on }\Dom.
\end{cases}
\end{equation}
Here $\n$ denotes the exterior normal field on $\partial\Dom$. Reasoning as in Subsection \ref{ss:stochastic_reaction_diffusion_domains}, one could allow an additional multiplicative noise term $\big(b_{in} \partial_i u + b_{ijn} \partial_i\partial_j u)d w^n_t$ as long as $\|(b_{ijn})_{n\geq 1}\|_{\ell^2}$ is small. Note that since the operator is of fourth order, we do not need a smallness condition on the first order part $(b_{in})_{n\geq 1}$.

We study \eqref{eq:Cahn_hilliard} under the following assumption.
\begin{assumption} Let $d\geq 2$.
\label{ass:CH_Phi}
\begin{enumerate}[{\rm(1)}]
\item Assume that one of the following conditions holds:
\begin{itemize}
\item $q\in [2,\infty)$, $p\in (2,\infty)$ and $\a\in [0,p/2-1)$;
\item $q=p=2$ and $\a=0$.
\end{itemize}
\item\label{it:CH_regularity_domain} $\Dom\subseteq \R^d$ is a bounded domain with $C^4$-boundary.
\item\label{it:CH_phi_nonlinearities} The maps $\phi:\I_T\times \O\times\Dom\times \R\to \R$ and $\Phi:=(\Phi_n)_{n\geq 1}:\I_T\times \O\times\Dom\times \R\to \ell^2$ are $\Progress\otimes \Borel(\Dom)\otimes\Borel(\R)$-measurable, $\phi(\cdot, 0) \in L^{\infty}(\I_T\times \O;L^q(\Dom))$ and $\Phi(\cdot, 0) \in L^{\infty}(\I_T\times \O;L^q(\Dom;\ell^2))$. Moreover, we assume that there exist $h>1$ and a constant $C>0$ such that for all $y,y'\in \R$,
$$
|\phi(\cdot,y)-\phi(\cdot, y')|+\|(\Phi_n(\cdot, y)-\Phi_n(\cdot, y'))_{n\geq 1}\|_{\ell^2}\leq C (1+|y|^{h-1}+|y'|^{h-1})|y-y'|.
$$
\end{enumerate}
\end{assumption}

Note that the above condition for $h=3$ covers the case
\begin{equation}
\label{eq:CH_well_potential}
\phi(u)=\Psi'(u)=u^3-u,\qquad \text{where}\qquad \Psi(s)=\frac{1}{4}(s^2-1)^2.
\end{equation}
In the physical literature $\Psi$ is called the double well-potential. In Example \ref{ex:CH_h_d_3}, this case will be investigated in detail. By Assumption \ref{ass:CH_Phi}\eqref{it:CH_regularity_domain}, we can study \eqref{eq:Cahn_hilliard} in the scale $(\Hn^{s,q}(\Dom))_{s\geq -4}$ introduced in Example \ref{ex:Cahn_Hilliard_scale}. Such scale fits the boundary conditions required in \eqref{eq:Cahn_hilliard}. For instance,
\begin{align*}
\Hn^{4,q}(\Dom)& =\{u\in W^{4,q}(\Dom)\,:\,\partial_{\n} u|_{\partial\Dom}=\partial_{\n}\Delta u|_{\partial\Dom}=0 \},
\\ \Hn^{2,q}(\Dom)& =\{u\in W^{2,q}(\Dom)\,:\,\partial_{\n}u|_{\partial\Dom}=0\},
\end{align*}
and $\Hn^0(\Dom)=L^q(\Dom)$. We refer to Example \ref{ex:Cahn_Hilliard_scale} for further properties.

\subsubsection{Almost very weak setting}
Let $s\in [0,2)$ and $q\in [2,\infty)$. We rewrite \eqref{eq:Cahn_hilliard} in the form \eqref{eq:semilinearabstract} by setting $X_0:=\Hn^{-2-s,q}(\Dom)$, $X_1=\Hn^{2-s,q}(\Dom)$ and for $u\in X_1$
\begin{align*}
A(t)u&=\Dn_{-2-s,q}^2 u, &  B(t)u &=0,
\\ F(t,u)&=\Dn_{-2-s,q} (\phi(t,u)), & G(t,u)&=(\Phi_n(t,u))_{n\geq 1}.
\end{align*}
Here $\Dn_{\eta,q}^2$ is the extrapolated bi-Laplace operator \eqref{eq:Bilaplacian_extrapolated} and $\Dn_{\beta,q}$ is as in \eqref{eq:CH_Delta_extrapolated}.

As usual, we say that $(u,\sigma)$ is a maximal local solution to \eqref{eq:Cahn_hilliard} if $(u,\sigma)$ is a maximal local solution to \eqref{eq:semilinearabstract} in the sense of Definition \ref{def:solution2} with the above choice of $A,B,F,G$ and $H=\ell^2$.

To show local existence for \eqref{eq:Cahn_hilliard} we employ Theorem \ref{t:semilinear}. By Example \ref{ex:Cahn_Hilliard_scale} it follows that $\Dn^2_{-1-s,q}$ has a bounded $H^{\infty}$-calculus on $\Hn^{-s,q}(\Dom)$ of angle less than $\pi/2$. By Theorem \ref{t:H_infinite_SMR}, it follows that $\Dn^2_{-1-s,q}\in \MRta$. It remains to look at suitable bounds for the non-linearities $F,G$. Let us begin by looking at $F$:
\begin{equation}
\label{eq:estimate_Cahn_Hilliard_F}
\begin{aligned}
\|F(\cdot,u)-F(\cdot,v)&\|_{\Hn^{-2-s,q}(\Dom)}\\
&\stackrel{(i)}{\lesssim} \|\phi(\cdot,u)-\phi(\cdot,v)\|_{\Hn^{-s,q}(\Dom)}\\
&\stackrel{(ii)}{\lesssim} \|\phi(\cdot,u)-\phi(\cdot,v)\|_{L^{r}(\Dom)}\\
&\stackrel{(iii)}{\lesssim} (1+\|u\|_{L^{h r}(\Dom)}^{h-1}+\|v\|_{L^{h r}(\Dom)}^{h-1})\|u-v\|_{L^{hr}(\Dom)}\\
&\stackrel{(iv)}{\lesssim} (1+\|u\|_{\Hn^{\theta,q}(\Dom)}^{h-1}+\|v\|_{\Hn^{\theta,q}(\Dom)}^{h-1})\|u-v\|_{\Hn^{\theta,q}(\Dom)}.
\end{aligned}
\end{equation}
where in $(i)$ we used that $\Dn_{-2-s,q}:\Hn^{-s,q}(\Dom)\to \Hn^{-2-s,q}(\Dom)$ boundedly (see \eqref{eq:CH_Delta_extrapolated}), in $(ii)$ Sobolev embedding with $-s-d/q = -d/r$ (see \eqref{eq:Hn_Sob_embeddings}). In $(iii)$ we applied H\"{o}lder's inequality and Assumption \ref{ass:CH_Phi}\eqref{it:CH_phi_nonlinearities}, and in $(iv)$ we used Sobolev embeddings with
$$
\theta-\frac{d}{q}=-\frac{d}{h r}\Rightarrow   \theta=\frac{d}{q}-\frac{1}{h}\Big(s+\frac{d}{q}\Big).
$$
Note that $r\in (1,\infty)$ since we assume $q>d/(d-s)$. To ensure $\theta\in (0,2-s)$ we require
\begin{equation*}
\frac{d(h-1)}{2h-s(h-1)}<q<\frac{d(h-1)}{s}.
\end{equation*}
Setting
\begin{equation}
\label{eq:CH_beta_1_def}
\beta_1:=\frac{\theta+s+2}{4}=\frac{1}{4}\Big(\frac{d}{q}+s\Big)\Big(1-\frac{1}{h}\Big)+\frac{1}{2}
\end{equation}
we obtain $\Hn^{\theta,q}(\Dom)=[\Hn^{-2-s,q}(\Dom),\Hn^{2-s,q}(\Dom)]_{\beta_1}$ by \eqref{eq:Hn_complex_interpolation}. Obviously, $\theta<2-s$ implies that $\beta_1<1$. Summarizing, we have proved that $F:X_{\beta_1}\to X_0$ is locally Lipschitz. As usual, we split into three cases (see Remark \ref{r:non_linearities}\eqref{it:non_linearities_continuous_trace}-\eqref{it:non_linearities_varphi_equal_to_beta}):
\begin{enumerate}[{\rm(1)}]
\item If $1-(1+\a)/{p}>\beta_1$, \ref{HFcritical} holds with $F_{\Tr}=F$, $F_c\equiv F_L\equiv 0$.
\item If $1-(1+\a)/{p}= \beta_1$, \ref{HFcritical} holds with $F_{\Tr}\equiv 0$, $F_{c}:=F$, $F_L\equiv 0$, $m_F=1$, $\rho_1=h-1$ and $\varphi_1=\beta_1$.
\item If $1-(1+\a)/{p}<\beta_1$, \ref{HFcritical} holds with $F_{\Tr}\equiv 0$, $F_{c}:=F$, $F_L\equiv 0$, $m_F=1$, $\rho_1=h-1$ and $\varphi_1=\beta_1$ under the condition \eqref{eq:HypCritical}. The latter becomes
\begin{equation}
\label{eq:CH_critical_weights}
\frac{1+\a}{p}\leq \frac{\rho_1+1}{\rho_1}(1-\beta_1)=\frac{h}{2(h-1)}-\frac{1}{4}\Big(s+\frac{d}{q}\Big).
\end{equation}
\end{enumerate}
Next we estimate $G$. Reasoning as in \eqref{eq:estimate_Cahn_Hilliard_F}, using that $X_{1/2}=\Hn^{-s,q}(\Dom)$ by \eqref{eq:Hn_complex_interpolation} and Assumption \ref{ass:CH_Phi}\eqref{it:CH_phi_nonlinearities}, one has
\begin{align*}
\|G(\cdot,u)-G(\cdot,v)\|_{\g(\ell^2;\Hn^{-s,q}(\Dom))}
&\lesssim \|(\Phi_n(\cdot,u)-\Phi_n(\cdot,v))_{n\geq 1}\|_{\g(\ell^2;L^r(\Dom))}\\
&\stackrel{(i)}{\eqsim} \|(\Phi_n(\cdot,u)-\Phi_n(\cdot,v))_{n\geq 1}\|_{L^r(\Dom;\ell^2)}\\
&\lesssim (1+\|u\|_{L^{h r}(\Dom)}^{h-1}+\|v\|_{L^{h r}(\Dom)}^{h-1})\|u-v\|_{L^{h r}(\Dom)}\\
&\lesssim (1+\|u\|_{\Hn^{\theta,q}(\Dom)}^{h-1}+\|v\|_{\Hn^{\theta,q}(\Dom)}^{h-1})\|u-v\|_{\Hn^{\theta,q}(\Dom)}.
\end{align*}
where $r,\theta$ are as in \eqref{eq:estimate_Cahn_Hilliard_F} and in $(i)$ we used \eqref{eq:gammaidentity}. Thus, under the above assumptions, the same argument used for $F$ implies that $G$ verifies \ref{HGcritical}.

Finally, by \eqref{eq:def_Bn}, the trace space is given by
\[\Xap = (X_0, X_1)_{1-\frac{1+\kappa}{p},p} = \Bn^{2-s-\frac{4(1+\a)}{p}}_{q,p}(\Dom).\]
See \eqref{eq:Bn_identification} for more on $\Bn$-spaces. Thus, Theorem \ref{t:semilinear} gives the following result.
\begin{theorem}
\label{t:CH_local}
Suppose that Assumption \ref{ass:CH_Phi} holds. Let $s\in [0,2)$ and assume that
\begin{equation}
\label{eq:CH_local_limitation_q}
\max\Big\{\frac{d(h-1)}{2h-s(h-1)},\frac{d}{d-s}\Big\}<q<\frac{d(h-1)}{s}
\end{equation}
holds.\footnote{Here we used the convention $1/0:=\infty$.} Let $\beta_1$ be as in \eqref{eq:CH_beta_1_def}. Assume that one of the following conditions holds:
\begin{itemize}
\item $1-(1+\a)/p\geq \beta_1$.
\item $1-(1+\a)/p<\beta_1$ and \eqref{eq:CH_critical_weights} holds.
\end{itemize}
Then for any
\[u_0\in L^0_{\F_0}(\O;\Bn^{2-s-\frac{4(1+\a)}{p}}_{q,p}(\Dom))\]
the problem \eqref{eq:Cahn_hilliard} has a maximal local solution $(u,\sigma)$. Moreover, there exists a localizing sequence $(\sigma_n)_{n\geq 1}$ such that for each $n\geq 1$ and a.s.
$$
u\in L^p(\I_{\sigma_n},w_{\a};\Hn^{2-s,q}(\Dom))\cap C(\overline{I}_{\sigma_n};\Bn^{2-s-\frac{4(1+\a)}{p}}_{q,p}(\Dom))
\cap C(I_{\sigma_n};\Bn^{2-s-\frac{4}{p}}_{q,p}(\Dom)).
$$
\end{theorem}

\subsubsection{Critical spaces for \eqref{eq:Cahn_hilliard}}
As usual, we check when \eqref{eq:CH_critical_weights} becomes an equality.
We first look at the case $p\in (2,\infty)$. Since $\a\geq 0$ we have to assume
\begin{equation}
\label{eq:Cahn_Hilliard_p_limitation}
\frac{1}{p}\leq \frac{h}{2(h-1)}-\frac{1}{4}\Big(s+\frac{d}{q}\Big).
\end{equation}
Since $\a<p/2-1$ if and only if $(1+\a)/{p}<1/2$, we require
\begin{align*}
\frac{h}{2(h-1)}-\frac{1}{4}\Big(s+\frac{d}{q}\Big)<\frac{1}{2}\quad \Leftrightarrow \quad
q<\frac{d(h-1)}{2-s(h-1)}.
\end{align*}
Simple computations show that the previous is verified if and only if
\begin{equation}
\label{eq:Cahn_Hilliard_critical_condition}
h\geq \frac{2+s}{s},\quad \text{ or }\quad \Big[ h<\frac{2+s}{s}  \ \ \text{and}  \ \ q<\frac{d(h-1)}{2-s(h-1)}\Big].
\end{equation}
If \eqref{eq:Cahn_Hilliard_critical_condition} holds, then
\begin{equation}
\label{eq:CH_critical_weight_case_p_larger_than_2}
\a_{\crit}:=p\Big(\frac{h}{2(h-1)}-\frac{1}{4}\Big(s+\frac{d}{q}\Big)\Big)-1\in \Big[0,\frac{p}{2}-1\Big).
\end{equation}
By \eqref{eq:def_Bn}, the corresponding critical space is
\begin{equation}
\label{eq:CH_critical_spaces_p_large}
\begin{aligned}
\Xapcrit
&=\Bn^{2-s-\frac{4(1+\a_{\crit})}{p}}_{q,p}(\Dom)\\
&=\Bn^{2-s + \frac{2h}{h-1}+\frac{d}{q}+s}_{q,p}(\Dom)=\Bn^{\frac{d}{q}-\frac{2}{h-1}}_{q,p}(\Dom).
\end{aligned}
\end{equation}
Note that the trace space does not depend on the parameter $s\in [0,2)$ and depends on $p$ only through the microscopic parameter. In addition, it coincides with the critical space for \eqref{eq:Cahn_hilliard} in the deterministic setting (see \cite[Example 3]{CriticalQuasilinear}).

In the case $p=q=2$ and $\a=0$, equality in \eqref{eq:CH_critical_weights} holds if and only if
\begin{equation}
\label{eq:CH_h_critical_case_p_2_q_2}
\frac{1}{2}=\frac{h}{2(h-1)}-\frac{1}{4}\Big(s+\frac{d}{2} \Big) \Leftrightarrow h=1+\frac{4}{2s+d}.
\end{equation}
Thus, in this case the trace space becomes
$$
\Bn^{-s}_{2,2}(\Dom)=\Bn^{\frac{d}{2}-\frac{2}{h-1}}_{2,2}(\Dom),
$$
where we used \eqref{eq:CH_h_critical_case_p_2_q_2}. The latter space is consistent with \eqref{eq:CH_critical_spaces_p_large} in the case $q\in (2,\infty)$. The previous considerations and Theorem \ref{t:CH_local} give the following result.

\begin{theorem}
\label{t:CH_critical}
Let the Assumption \ref{ass:CH_Phi} be satisfied and let $s\in [0,2)$. Assume that one of the following conditions hold:
\begin{enumerate}[{\rm(1)}]
\item $q=p=2$, $h=1+\frac{4}{d+2s}$, $d>\max\{2s,2s^2/(2-s)\}$ and $\a_{\crit}=0$.
\item $q\in [2,\infty)$, $p\in (2,\infty)$, \eqref{eq:CH_local_limitation_q}, \eqref{eq:Cahn_Hilliard_p_limitation}, \eqref{eq:Cahn_Hilliard_critical_condition} hold and $\a_{\crit}$ is given by \eqref{eq:CH_critical_weight_case_p_larger_than_2}.
\end{enumerate}
Then for each
$$
u_0\in L^0_{\F_0}(\O;\Bn^{\frac{d}{q}-\frac{2}{h-1}}_{q,p}(\Dom)),
$$
there exists a maximal local solution $(u,\sigma)$ to \eqref{eq:Cahn_hilliard}. Moreover, there exists a localizing sequence $(\sigma_n)_{n\geq 1}$ such that for all $n\geq 1$ and a.s.
\begin{equation*}
u\in L^{p}(\I_{\sigma_n},w_{\a_{\crit}};\Hn^{2-s,q})
\cap C(\overline{I}_{\sigma_n};\Bn^{\frac{d}{q}-\frac{2}{h-1}}_{q,p})
\cap C((0,\sigma_n];\Bn^{2-s-\frac{4}{p}}_{q,p}).
\end{equation*}
\end{theorem}

\begin{proof}
It remains to check the condition of Theorem \ref{t:CH_local} for $q=p=2$ and $\a=0$. To see this, note that if $q=p=2$ and $h=1+\frac{4}{2s+d}$,
then the  condition \eqref{eq:CH_local_limitation_q} is equivalent to
\[\max\Big\{\frac{d}{d-s},\frac{2d}{d+4}\Big\}<2<\frac{4d}{s(d+2s)}.\]
Since $\frac{2d}{d+4}<2$ is always true, the remaining conditions imply $d>\max\{(2s), \frac{2s^2}{2-s}\}$.
\end{proof}

Let us give a concrete example to see what this condition becomes in an important special case.
\begin{example}
\label{ex:CH_h_d_3}
Let $h=3$ and $d=3$. By \eqref{eq:CH_local_limitation_q}-\eqref{eq:Cahn_Hilliard_critical_condition}, the limitations on $q\in (2,\infty)$ in Theorem \ref{t:CH_critical} are equivalent to
$$
\Big[s<1, \text{ and } \frac{3}{3-s}<q<\min\Big\{\frac{6}{s},\frac{3}{1-s}\Big\}\Big], \quad \text{ or }\quad
\Big[s\geq 1,  \text{ and } \frac{3}{3-s}<q<\frac{6}{s}\Big].
$$
Let us consider the first cases, i.e.\ $s<1$. Optimizing the upper bound on $q$ we obtain $2<q<9$ for  $s=\frac{2}{3}$. Thus, if $\frac{1}{p}\leq \frac{3}{4}-\frac{1}{4}\Big(s+\frac{3}{q}\Big)$ holds, then Theorem \ref{t:CH_critical} ensures existence for initial values $u_0\in\Bn^{\frac{3}{q}-1}_{q,p}(\Dom)$ a.s.\ with smoothness $\frac{3}{q} - 1>-\frac23$. In the case $s\geq 1$, $q<\frac{6}{s}$ and therefore $\frac{3}{q}-1>\frac{s}{2}-1\geq -\frac12$. A similar situation arises for the Allen-Cahn equations, see Remark \ref{r:smoothness_AC}. As in the case of Allen-Cahn equations, the optimality of the threshold $-\frac23$ is not know.
\end{example}

Due to Theorem \ref{t:CH_local} the same strategy used in Corollary \ref{cor:AC_local_L_d} leads to the following result.
\begin{corollary}
Let the Assumption \ref{ass:CH_Phi} be satisfied. Let $h>1+4/d$ and $q:=d(h-1)/2$. Assume that $p\in[q,\infty)$ and $p>2(h-1)$. The there exists $\bar{s}>0$ such that for any $s\in (0,\overline{s})$ and any
$$
u_0\in L^0_{\F_0}(\O;L^q(\Dom))
$$
there exists a maximal local solution $(u,\sigma)$ to \eqref{eq:Allen_Cahn_stochastic_potentials}, and there exists a localizing sequence $(\sigma_n)_{n\geq 1}$ such that for any $n\geq 1$ and a.s.
$$
u\in L^p(\I_{\sigma_n},w_{\a_{\crit}};\Hn^{2-s,q}(\Dom))\cap C(\overline{I}_{\sigma_n};\Bn^{0}_{q,p}(\Dom))\cap C(I_{\sigma_n};\Bn^{2-s-\frac{4}{p}}_{q,p}(\Dom)),
$$
where $\a_{\crit}=p(\frac12-\frac{s}{4})-1$.
\end{corollary}

We remark that the restriction $p>2(h-1)$ is due to \eqref{eq:Cahn_Hilliard_p_limitation}. Finally, as in Corollary \ref{cor:AC_local_L_d}, one sees that the solution immediately regularizes.

\appendix

\section{Interpolation-extrapolation scales}\label{sec:app}
In this appendix we present results on interpolation-extrapolation scales for sectorial operators which we need in the paper. For a more detailed presentation we refer to \cite[Chapter 5]{Am} and \cite[Section 5.3]{Haase:2}.

Definitions related to sectorial operators have been given in Subsection \ref{ss:sectorial_calculus} and will be used below. Let $\Aop$ be a sectorial operator on a Banach space $X$ such that $0\in \rho(\Aop)$. The latter implies that $(X,\|\Aop^{-1}\cdot\|_{X})$ is a normed space. We define the extrapolated space $X_{-1,\Aop}$ as the completion of $(X,\|\Aop^{-1}\cdot\|_{X})$, i.e.\
\begin{equation}
\label{eq:def_X_first_extrapolation}
X_{-1,\Aop}:=(X,\|\Aop^{-1}\cdot\|_{X})^{\sim};
\end{equation}
where $\sim$ denotes the completion of the space. For notational convenience we set $\|\cdot\|_{-1,\Aop}:=\|\cdot\|_{X_{-1,\Aop}}$. It is evident that $X\hookrightarrow X_{-1,\Aop}$ and if $x\in X$
$$
\|x\|_{-1,\Aop}=\|\Aop^{-1}x\|_{X}\leq \|\Aop^{-1}\|_{\calL(X)}\|x\|_{X}.
$$
Since $\overline{\Do(\Aop)}=X$, equality in the previous formula shows that $\Do(\Aop)\ni x\mapsto \Aop x\in X$ extends to a linear isometric isomorphism between $X$ and $X_{-1,\Aop}$. The extension of this map will be denoted by $T_{-1,\Aop}$ or simply $T_{-1}$ if no confusion seems likely.

To proceed further, let us note that $\Aop$ induces a closed linear operator $\Aop_{-1}$ on $X_{-1}$ given by
\begin{equation}
\label{eq:A_first_extrapolation}
\Aop_{-1}:=T_{-1}\Aop T_{-1}^{-1}.
\end{equation}
One can check that $\Aop_{-1}|_{X}=\Aop$. By \eqref{eq:A_first_extrapolation}, $\Aop$ is similar to $\Aop_{-1}$. These simple observations lead to the following.
\begin{proposition}
\label{prop:A_first_extrapolation}
Let $\Aop$ be a sectorial operator on $X$ such that and $0\in \rho(\Aop)$. Then $\Aop_{-1}$ is the closure of $\Aop$ in $X_{-1}$ with $\Do(\Aop_{-1})=X$.

Moreover, the following hold true:
\begin{enumerate}[{\rm(1)}]
\item $\Aop_{-1}$ is a sectorial operator on $X_{-1,\Aop}$ and $\om(\Aop_{-1})=\om(\Aop)$;
\item\label{it:A_extrapolated_BIP} If $\Aop\in \BIP(X)$, then $\Aop_{-1}\in \BIP(X_{-1,\Aop })$;
\item If $\Aop$ has a bounded $H^{\infty}$-calculus on $X$, so does $\Aop_{-1}$ and $\angH(\Aop_{-1})=\angH(\Aop)$.
\end{enumerate}
\end{proposition}

The previous proposition shows that if $\Aop_{-1}$ is sectorial, then the fractional power $(\Aop_{-1})^{\alpha}$ for $\alpha>0$ are well defined closed linear operators on $X_{-1,\Aop}$. Let us denote by $X_{1-\alpha,\Aop}$ the domain of $(\Aop_{-1})^{\alpha}$,
\begin{equation}
\label{eq:complete_scale}
X_{-1+\alpha,\Aop}:=\big(\Do((\Aop_{-1})^{\alpha}),\|(\Aop_{-1})^{\alpha}\cdot\|_{X}\big),\qquad \alpha\geq 0.
\end{equation}
By Proposition \ref{prop:A_first_extrapolation} one has $\Do(\Aop_{-1})=X$ and thus $X_{0,\Aop}=X$.

Let $\alpha\geq -1$ and let $\Aop_{\alpha}$ be the realization of $\Aop_{-1}$ on $X_{\alpha,\Aop}$, i.e.\
\begin{align*}
\Do(\Aop_{\alpha})&:=\{x\in X_{\alpha,\Aop}\,:\,\Aop_{-1}x \in X_{\alpha,\Aop} \},\\
\Aop_{\alpha}x&:=\Aop_{-1}x,\text{ if }x\in \Do(\Aop_{\alpha}),\qquad x\in \Do(\Aop_{\alpha}).
\end{align*}
Note that $\Aop_{0}=\Aop$ and $\Aop_{\alpha}=\Aop_{-1}$ if $\alpha=-1$. Under suitable assumptions, $(X_{\alpha,\Aop})_{\alpha\geq -1}$ becomes an interpolation scale with respect to complex interpolation (see \cite[Theorem 6.6.9]{Haase:2}):
\begin{proposition}
\label{prop:interpolation_extrapolation_scale}
Let $\Aop\in \BIP(X)$ be such that $0\in \rho(\Aop)$. Let $(X_{\alpha,\Aop})_{\alpha\geq -1}$ be as above. Then
\begin{equation*}
X_{\alpha(1-\theta)+\beta\theta,\Aop}=[X_{\alpha,\Aop},X_{\beta,\Aop}]_{\theta}, \quad \alpha,\beta\geq -1,\;\theta\in (0,1).
\end{equation*}
isomorphically.
\end{proposition}

If the assumption of the previous proposition holds, then we say that $(X_{\alpha,\Aop},\Aop_{\alpha})_{\alpha\geq -1}$ is the interpolated-extrapolated scale of $\Aop$.

In applications it will be useful to have the description of the spaces $X_{-\alpha,\Aop}$ for $\alpha\in (0,1)$ given in \cite[Chapter 5, Theorem 1.4.9]{Am}.
Here $\Aop^*$ denotes the adjoint operator of $\Aop$.

\begin{theorem}
\label{t:duality}
Let $\Aop$ a sectorial operator on a reflexive Banach space $X$ such that $0\in \rho(\Aop)$. Then for each $\vartheta\in [0,1]$, $X_{-\vartheta,\Aop}$ is isomorphic to the dual of the space $X_{\vartheta,\Aop^*}=(\Do((\Aop^*)^{\vartheta}),\|(\Aop^*)^{\vartheta}\cdot\|_X)$. More concisely,
$$
X_{-\vartheta,\Aop}=(X_{\vartheta,\Aop^*})^*.
$$
\end{theorem}

Due to Proposition \ref{prop:A_first_extrapolation}, the construction can be iterated again and one can construct a family of super-spaces $(X_{-n,\Aop})_{n\geq 1}$ and operators $(\Aop_{-n})_{n\geq 1}$ with analogous properties. Since it will be not used in the paper we refer to \cite{Am} and \cite{Haase:2} for the complete construction.

In the following examples we look at operators which will be used later.
\begin{example}[Dirichlet Laplacian]
\label{ex:extrapolated_Laplace_dirichlet}
In this example we specialize the above construction to the strong Dirichlet Laplacian $\Dd_{q}$ where $q\in (1,\infty)$. In this example we assume that $\Dom$ is a $C^2$-domain in $\R^d$ with compact boundary. Note that exterior domains are allowed. To begin, let us set $W^{1,q}_0(\Dom):=\{u\in W^{1,q}(\Dom)\,:\,u|_{\partial \Dom}=0\}$. The strong Dirichlet Laplacian is defined as
\begin{equation}
\label{eq:strong_Dirichlet_Laplacian}
\Dd_q:W^{2,q}(\Dom)\cap W^{1,q}_0(\Dom)\subseteq L^q(\Dom) \to L^q(\Dom),\qquad \Dd_q f:=\Delta f.
\end{equation}
By \cite{DDHPV}, there exists $c>0$ such that $A_q:=c \,I -\Dd_q$ has a bounded $H^{\infty}$-calculus on $L^q(\Dom)$ with angle $\angH(A_q)<\pi/2$. If $\Dom$ is bounded, then we may set $c=0$.

Thus, $A_q$ generates an extrapolated-interpolated scale, which will be denoted by
$$\big(\HD^{2\alpha,q}(\Dom),A_{2\alpha,q}\big)_{\alpha\in [-1,\infty)}.$$
Therefore, $\HD^{2,q}(\Dom)=W^{2,q}(\Dom)\cap W^{1,q}_0(\Dom)$ and $\HD^{0,q}(\Dom)=L^q(\Dom)$.
By Proposition \ref{prop:interpolation_extrapolation_scale}, for all $ -2\leq s_1<s_2<\infty$ one has
\begin{equation}
\label{eq:HD_complex_interpolation}
\HD^{s,q}(\Dom)=[\HD^{s_1,q}(\Dom),\HD^{s_2,q}(\Dom)]_{\theta},\qquad \vartheta\in (0,1),s:=(1-\theta) s_1+\theta s_2.
\end{equation}
Moreover, by Theorem \ref{t:duality},
\begin{equation}
\label{eq:HD_duality}
\HD^{-s,q}(\Dom)=(\HD^{s,q'}(\Dom))^*, \qquad s\in(0,2).
\end{equation}
We define the extrapolated Dirichlet Laplacian as:
\begin{equation}
\label{eq:extrapolated_Laplacian_D}
\Dd_{s,q}:=- A_{s,q}+c \I,\qquad s\geq -2.
\end{equation}
Note that $\Dd_{0,q}=\Dd_{q}$. By \cite{Se} and \eqref{eq:HD_complex_interpolation} one has the following identification:
\begin{equation}
\label{eq:HD_identification}
\HD^{s,q}(\Dom)=
\begin{cases}
H^{s,q}(\Dom)\qquad &\text{ if }s\in (0,1/q),\\
\{H^{s,q}(\Dom)\,:\,u|_{\partial\Dom}=0\}\qquad &\text{ if }s\in (1/q,2).
\end{cases}
\end{equation}
Here $H^{s,q}(\Dom)$ denotes the Bessel potential spaces on domains (see \cite[Section 4.3.1]{Tri95}).  Note that we avoided the $s=1/q$ as in this case the description is more complicated. The latter identity implies $\HD^{1,q}(\Dom)=W^{1,q}_0(\Dom)$, and by \eqref{eq:HD_duality} one has
$$
\HD^{-1,q}(\Dom)=(W^{1,q'}_0(\Dom))^*=:W^{-1,q}(\Dom).
$$
The above identities and an integration by parts argument show that $\Dd_{-1,q}$ is the `weak Dirichlet Laplacian', i.e.\ $\Dd_{-1,q}:W^{1,q}_0(\Dom)\subseteq W^{-1,q}(\Dom)\to W^{-1,q}(\Dom)$ and
\begin{equation}
\label{eq:weak_dirichlet_laplacian}
 -\l g,\Dd_{-1,q}f\r=\int_{\Dom}\nabla g\cdot \nabla f\,dx,\quad  f\in W^{1,q}_0(\Dom),g\in W^{1,q'}_0(\Dom);
\end{equation}
see \cite[Example 3]{CriticalQuasilinear} for details. The same integration by parts argument, allows us to consider the divergence operator $\div:=\sum_{j=1}^d \partial_j$ as a map $\div :L^q(\Dom;\R^d)\to \HD^{-1,q}(\Dom)$ defined by
\begin{equation}
\label{eq:weak_divergence_operator}
 - \l g,\div F\r:=\int_{\Dom} F\cdot \nabla g\,dx, \qquad \forall F\in L^q(\Dom;\R^d),g\in W^{1,q'}_0(\Dom).
\end{equation}

For later use, we discuss Sobolev type embedding results for $\HD$-spaces. For $s_0,s_1\geq -1$, and $1<q_0<q_1<\infty$ such that $s_0-d/q_0\geq s_1-d/q_1$, one has
\begin{equation}
\label{eq:HD_embedding}
\HD^{s_0,q_0}(\Dom)\hookrightarrow \HD^{s_1,q_1}(\Dom).
\end{equation}
Indeed, this follows from the steps below.
\begin{enumerate}[{\rm(1)}]
\item\label{it:HD_sob_proof_1} If $s_0,s_1\geq 0$, then \eqref{eq:HD_embedding} follows from \eqref{eq:HD_identification} and the embedding for $H$-spaces (see \cite[Theorem 4.6.1]{Tri95}).
\item\label{it:HD_sob_proof_2} If $s_0,s_1\leq 0$, then \eqref{eq:HD_embedding} follows from \eqref{eq:HD_identification} and the duality \eqref{eq:HD_duality}.
\item If $s_1<0<s_0$ are arbitrary, then let $p\in (q,\infty)$ be such that $s_0-d/q_0=-d/p$. Then \eqref{eq:HD_embedding} follows from
$$
\HD^{s_0,q_0}(\Dom)\hookrightarrow \HD^{0,p}(\Dom)\hookrightarrow \HD^{s_1,q_1}(\Dom).
$$
\end{enumerate}

We conclude this example by looking at real interpolation spaces. For any $\theta\in (0,1)$ and $q,p\in (1,\infty)$ we define
\begin{equation*}
\BD^{-2+4\theta}_{q,p}(\Dom):=(\HD^{-2,q}(\Dom),\HD^{2,q}(\Dom))_{\theta,p}.
\end{equation*}
By \cite[Theorem 4.7.2]{BeLo} and \eqref{eq:HD_complex_interpolation} for $-2\leq s_0<s_1$ and $\theta\in (0,1)$ one has
\begin{equation}
\label{eq:def_BD}
\BD^{s}_{q,p}(\Dom)=(\HD^{s_0,q}(\Dom),\HD^{s_1,q}(\Dom))_{\theta,p},\qquad \text{where }s:=(1-\theta)s_0+\theta s_1.
\end{equation}
The notation $\BD$ is motivated by the following identification (see \cite{Grisvard}):
\begin{equation}
\label{eq:BD_identifications}
\BD^{s}_{q,p}(\Dom)=
\begin{cases}
B^{s}_{q,p}(\Dom),&\qquad s\in(0,1/q),\\
\{u\in B^{s}_{q,p}(\Dom)\,:\,u|_{\partial\Dom}=0\},&\qquad s\in(1/q,2).
\end{cases}
\end{equation}
Here $B^{s}_{q,p}(\Dom)$ denotes the usual Besov spaces on domains (see \cite[Section 4.3.1]{Tri95}).
\end{example}

The same reasoning can be applied to other boundary conditions.

\begin{example}[Neumann Laplacian]
\label{ex:Laplacian_Neumann_scales}
In this example we look at the Neumann Laplacian $\DN_{q}$. Here we assume that $\Dom\subseteq \R^d$ is a bounded $C^2$-domain. As usual, $\n$ denotes the exterior normal field on $\partial\Dom$. Let $q\in (1,\infty)$ and $\Do(\DN_q):=\{u\in W^{2,q}(\Dom)\,:\,\partial_{\n} u=0\}$, the Neumann Laplacian is given by
$$
\DN_q:\Do(\DN_q)\subseteq L^q(\Dom)\to L^q(\Dom),\qquad \DN f:=\Delta f,\text{ for } f\in \Do(\DN_q).
$$
By \cite{DDHPV}, there exists $c>0$ such that $A_{q}^N:=c-\DN_q$ has a bounded $H^{\infty}$-calculus with angle $<\pi/2$. Therefore, $A_q^N$ generates an interpolated-extrapolated scale $(\HN^{2\alpha,q}(\Dom),A_{2\alpha,q}^N)_{\alpha\in [-1,\infty)}$. The extrapolation Neumann Laplacian is given by:
$$
\DN_{2\alpha,q}:=A_{2\alpha,q}^N-c \I,\qquad \forall \alpha\geq -1.
$$
Below we list the main properties and further definition related to the $\HN$-scale. Their proofs are similar to the one in Example \ref{ex:extrapolated_Laplace_dirichlet}.
\begin{itemize}
\item $\HN^{0,q}(\Dom)=L^q(\Dom)$ and $\HN^{2,q}(\Dom)=\Do(\DN_q)$.
\item Complex interpolation property: for all $ -2\leq s_1<s_2<\infty$ and $\theta\in (0,1)$,
\begin{equation*}
\HN^{s,q}(\Dom)=[\HN^{s_1,q}(\Dom),\HN^{s_2,q}(\Dom)]_{\theta},\quad s=(1-\theta) s_1+\theta s_2.
\end{equation*}
\item Duality: By Theorem \ref{t:duality},
\begin{equation*}
\HN^{-s,q}(\Dom)=(\HN^{s,q'}(\Dom))^*, \qquad s\in(0,2).
\end{equation*}
\item Identification of $\HN^{s,q}(\Dom)$:
\begin{equation*}
\HN^{s,q}(\Dom)=
\begin{cases}
H^{s,q}(\Dom) &\text{ if }s\in (0,1+1/q),\\
\{H^{s,q}(\Dom)\,:\,\partial_{\n}u|_{\partial\Dom}=0\} &\text{ if }s\in (1+1/q,2).\\
\end{cases}
\end{equation*}
\item Real interpolation: for $p\in (1,\infty)$, $\theta\in (0,1)$ and
$$
\BN^{-2+4\theta}_{q,p}(\Dom):=(\HN^{-2,q}(\Dom),\HN^{2,q}(\Dom))_{\theta,p}=(\HN^{-2,q}(\Dom),\HN^{2,q}(\Dom))_{\phi,p},
$$
provided $-2\leq s_0<s_1$, $\phi\in (0,1)$ and $-2+4\theta=(1-\phi)s_0+\phi s_1$.
\item Identification of $\BN^{s,q}(\Dom)$: For any $q,p\in (1,\infty)$
\begin{equation*}
\BN^{s}_{q,p}(\Dom)=
\begin{cases}
B^{s}_{q,p}(\Dom),& \quad s\in (-2,1+1/q),\\
\{u\in B^{s}_{q,p}(\Dom)\,:\,\partial_{\n}u|_{\partial\Dom}=0\},&\quad   s\in (1+1/q,2).
\end{cases}
\end{equation*}
\item Sobolev embeddings: For any $s_0,s_1\geq -1$, $1<q_0<q_1<\infty$ such that $s_0-d/q_0\geq s_1-d/q_1$ one has
\begin{equation}
\label{eq:HN_Sob_embeddings}
\HN^{s_0,q_0}(\Dom)\hookrightarrow \HN^{s_1,q_1}(\Dom).
\end{equation}
\end{itemize}
\end{example}

In the following example we look at powers of the Neumann Laplace operator.

\begin{example}[Bi-Laplacian]
\label{ex:Cahn_Hilliard_scale}
In this example we look at Bi-Laplace operators with Neumann-type boundary conditions. Here, $\Dom\subseteq \R^d$ is a bounded $C^4$-domain and $\n$ denotes the exterior normal derivative on $\partial\Dom$. Let $q\in (1,\infty)$ and set
$\Do(\Dn_{q}^2):=\{u\in H^{4,q}(\Dom)\,:\,\partial_{\n}u|_{\partial\Dom}=\partial_{\n}\Delta u|_{\partial\Dom}=0\}.$
The Bi-Laplacian with Neumann type boundary conditions is given by
$$
\Dn_{q}^2:\Do(\Dn_{q}^2)\subseteq L^q(\Dom)\to L^q(\Dom),\quad \Dn_{q}^2 f:=\Delta^2 f.
$$
By \cite{DDHPV}, it follows that there exists $c>0$ such that $A^{\n}_q:=cI+\Dn_q^2$ has a bounded $H^{\infty}$-calculus with angle $\angH(A^{\n}_q)<\pi/2$. In particular, $A^{\n}_q$ generates an extrapolated-interpolated scale
$$
(\Hn^{4\alpha,q}(\Dom),A_{4\alpha,q}^{\n})_{\alpha\in  [-1,\infty)}.
$$
Let us denote by $\Delta_{s,q}^2$ the extrapolated Bi-Laplacian
\begin{equation}
\label{eq:Bilaplacian_extrapolated}
\Dn_{s,q}^2:=A^{\n}_{s,q}-cI,\quad s\in (-4,0),\,q\in (1,\infty).
\end{equation}
Let us list some properties which will be needed in Subsection \ref{ss:CH}. The proofs are similar to the one given in Example \ref{ex:extrapolated_Laplace_dirichlet}.
\begin{itemize}
\item $\Hn^{0,q}(\Dom)=L^q(\Dom)$ and $\Hn^{2,q}(\Dom)=\Do(\Dn_q^2)$.
\item Complex interpolation property: for all $ -4\leq s_1<s_2<\infty$ and $\theta\in (0,1)$,
\begin{equation}
\label{eq:Hn_complex_interpolation}
\Hn^{s,q}(\Dom)=[\Hn^{s_1,q}(\Dom),\Hn^{s_2,q}(\Dom)]_{\theta},\quad s=(1-\theta) s_1+\theta s_2.
\end{equation}
\item Duality: By Theorem \ref{t:duality},
\begin{equation*}
\Hn^{-s,q}(\Dom)=(\Hn^{s,q'}(\Dom))^*, \qquad s\in(0,4).
\end{equation*}
\item Identification of $\Hn^{s,q}(\Dom)$:
\begin{equation}
\label{eq:Hn_identification}
\Hn^{s,q}(\Dom)=
\begin{cases}
H^{s,q}(\Dom) &\text{ if }s\in (0,1+1/q),\\
\{H^{s,q}(\Dom)\,:\,\partial_{\n}u|_{\partial\Dom}=0\} &\text{ if }s\in (1+1/q,3+1/q),\\
\{H^{s,q}(\Dom)\,:\,\partial_{\n}u|_{\partial\Dom}=0,\partial_{\n}\Delta u|_{\partial \Dom}=0\} &\text{ if }s\in (3+1/q,4).\\
\end{cases}
\end{equation}
\item Real interpolation: for $p\in (1,\infty)$, $\theta\in (0,1)$ and
\begin{equation}
\label{eq:def_Bn}
\Bn^{-4+8\theta}_{q,p}(\Dom):=(\Hn^{-4,q}(\Dom),\Hn^{4,q}(\Dom))_{\theta,p}
=(\Hn^{-4,q}(\Dom),\Hn^{4,q}(\Dom))_{\phi,p};
\end{equation}
provided $-4\leq s_0<s_1$, $\phi\in (0,1)$ and $-4+8\theta=(1-\phi)s_0+\phi s_1$.
\item Identification of $\Bn^{s,q}(\Dom)$: For any $q,p\in (1,\infty)$
\begin{equation}
\label{eq:Bn_identification}
\Bn^{s}_{q,p}(\Dom)=
\begin{cases}
B^{s}_{q,p}(\Dom),&\text{if } s\in (-2,1+1/q),\\
\{u\in B^{s}_{q,p}(\Dom)\,:\,u|_{\partial\Dom}=0\},&\text{if } s\in (1+1/q,3+1/q),\\
\{u\in B^{s}_{q,p}(\Dom)\,:\,u|_{\partial\Dom}=0,\partial_{\n}\Delta u|_{\partial\Dom}=0\},
&\text{if } s\in (3+1/q,4).
\end{cases}
\end{equation}
\item Sobolev embeddings: For any $s_0,s_1\geq -1$, $1<q_0<q_1<\infty$ such that $s_0-d/q_0\geq s_1-d/q_1$ one has
\begin{equation}
\label{eq:Hn_Sob_embeddings}
\Hn^{s_0,q_0}(\Dom)\hookrightarrow \Hn^{s_1,q_1}(\Dom).
\end{equation}
\end{itemize}
We conclude this example by looking at the Laplace operators on $\Hn^{-s,q}$-spaces. To this end, let us note that we can define $\Delta_{-4,q}:\Hn^{-2,q}(\Dom)\to\Hn^{-4,q}(\Dom)$ as
\begin{equation*}
\l \psi,\Delta_{-4,q} \phi\r:=\l \Delta \psi,\phi\r, \qquad \psi\in \Hn^{4,q'}(\Dom),\phi\in \Hn^{-2,q}(\Dom);
\end{equation*}
where we used that $\Hn^{-4,q}(\Dom)=(\Hn^{4,q'}(\Dom))^*$ and the fact that $\Delta\psi\in \Hn^{2,q'}(\Dom)$ by \eqref{eq:Hn_identification}. One can readily check that the above definition is consistent with the usual Laplacian provided $\phi\in \Hn^{2,q}(\Dom)$.

Since $\Delta_{0,q}:\Hn^{2,q}(\Dom)\to  L^q(\Dom)=\Hn^{0,q}(\Dom)$, by \eqref{eq:Hn_complex_interpolation} and interpolation, one gets
\begin{equation}
\label{eq:CH_Delta_extrapolated}
\Delta_{-2-s,q}:\Hn^{-s,q}(\Dom)\to \Hn^{-2-s,q}(\Dom),\qquad \text{boundedly for }s\in [-2,2].
\end{equation}
\end{example}

\bibliographystyle{alpha-sort}
\bibliography{literature}

\end{document}